\documentclass[a4paper,10pt]{article}
\usepackage[english]{babel}
\usepackage[utf8]{inputenc}
\usepackage{anyfontsize}
\usepackage{amsmath, amsfonts, amsthm, amssymb, dsfont, mathrsfs}
\usepackage{indentfirst}
\usepackage{xcolor} 
\usepackage{graphicx} 
\usepackage{enumerate}
\usepackage{braket}
\usepackage{xfrac}
\usepackage[T1]{fontenc}
\definecolor{darkblue}{RGB}{0,0,155} 
\definecolor{brightbordeaux}{RGB}{200, 0, 70} 
\usepackage{hyperref}
\usepackage{amsthm}
\usepackage{type1cm}
\usepackage{cleveref}

\hypersetup{ 
    colorlinks=true,
    linkcolor=darkblue,
    citecolor=brightbordeaux, 
    urlcolor=blue
}
\usepackage{calligra}
\usepackage[a4paper, left=2.5cm, right=2.5cm, top=2cm, bottom=2cm]{geometry}
\usepackage{microtype} 
\usepackage[safe]{tipa}
\usepackage{bm} 
\makeatletter
\newcommand*\bigcdot{\mathpalette\bigcdot@{.5}}
\newcommand*\bigcdot@[2]{\mathbin{\vcenter{\hbox{\scalebox{#2}{$\m@th#1\bullet$}}}}}
\makeatother
\newtheorem{lemma}{Lemma}[section]
\newtheorem{theorem}[lemma]{Theorem}
\newtheorem{proposition}[lemma]{Proposition}
\newtheorem{corollary}[lemma]{Corollary}

\newtheorem{definition}[lemma]{Definition}
\theoremstyle{remark}

\newtheorem{remark}[lemma]{Remark}
\numberwithin{equation}{section}
\renewcommand{\th}{\mathtt{h}}

\newcommand{\wt}[1]{{\widetilde{#1}}}

\newcommand{\wh}[1]{{\widehat{#1}}}

\newcommand{\C}{{\mathbb C}}

\newcommand{\N}{{\mathbb N}}

\newcommand{\R}{{\mathbb R}}

\newcommand{\T}{{\mathbb T}}
\newcommand{\Z}{{\mathbb Z}}

\newcommand{\cC}{{\mathcal C}}
\newcommand{\cD}{{\mathcal D}}
\newcommand{\cE}{{\mathcal E}}
\newcommand{\cF}{{\mathcal F}}

\newcommand{\cH}{{\mathcal H}}

\newcommand{\cL}{{\mathcal L}}
\newcommand{\cM}{{\mathcal M}}
\newcommand{\cN}{{\mathcal N}}

\newcommand{\cP}{{\mathcal P}}
\newcommand{\cQ}{{\mathcal Q}}
\newcommand{\cR}{{\mathcal R}}
\newcommand{\cS}{{\mathcal S}}

\newcommand{\cU}{{\mathcal U}}

\newcommand{\cW}{{\mathcal W}}

\newcommand{\cZ}{{\mathcal Z}}

\newcommand{\ta}{{\mathtt a}}
\newcommand{\tb}{{\mathtt b}}
\newcommand{\tc}{{\mathtt c}}
\newcommand{\td}{{\mathtt d}}

\newcommand{\tg}{{\mathtt g}}

\newcommand{\tm}{{\mathtt m}}

\newcommand{\tq}{{\mathtt q}}
\newcommand{\tr}{{\mathtt r}}
\newcommand{\ts}{{\mathtt s}}

\newcommand{\tB}{{\mathtt B}}

\newcommand{\tE}{{\mathtt E}}
\newcommand{\tF}{{\mathtt F}}
\newcommand{\tG}{{\mathtt G}}

\newcommand{\tK}{{\mathtt K}}
\newcommand{\tL}{{\mathtt L}}
\newcommand{\tM}{{\mathtt M}}

\newcommand{\tQ}{{\mathtt Q}}
\newcommand{\tR}{{\mathtt R}}
\newcommand{\tS}{{\mathtt S}}

\newcommand{\tV}{{\mathtt V}}

\newcommand{\tX}{{\mathtt X}}

\newcommand{\be}{{\bm{e}}}

\newcommand{\bA}{{\bm{A}}}
\newcommand{\bB}{{\bm{B}}}

\newcommand{\bF}{{\bm{F}}}
\newcommand{\bG}{{\bm{G}}}

\newcommand{\bR}{{\bm{R}}}

\newcommand{\bT}{{\bm{T}}}

\newcommand{\bPhi}{{\bf \Phi}}

\newcommand{\e}{{\epsilon}}
\newcommand{\s}{{\sigma}}

\newcommand{\hamvec}[1]{\mathcal{X}_{\scriptstyle #1}}
\newcommand{\bcQ}{\boldsymbol{\cQ}}
\renewcommand{\div}{\mathrm{div}}

\newcommand{\norm}[1]{\left\| #1 \right\|}

\newcommand{\uno}{\mathrm{Id}}
\newcommand{\im}{{ \mathrm{i}}}
\newcommand{\ii}{{\im}}

\newcommand{\und}[1]{\underline{#1}}

\newcommand{\di}{{\mathrm{d}}}

\newcommand{\vphi}{\varphi}
\newcommand{\vertiii}[1]{{\left\vert\kern-0.25ex\left\vert\kern-0.25ex\left\vert #1 
    \right\vert\kern-0.25ex\right\vert\kern-0.25ex\right\vert}}
\newcommand{\opnorm}[1]{{\left\vert\kern-0.25ex\left\vert\kern-0.25ex\left\vert #1 
    \right\vert\kern-0.25ex\right\vert\kern-0.25ex\right\vert}}

\newcommand{\Id}{\mathrm{Id}}
\newcommand{\zetina}{\text{\normalfont\textctyogh}}
\newcommand{\vomega}{{\boldsymbol{\Lambda}}}
\renewcommand{\be}{\begin{equation}}
\newcommand{\ee}{\end{equation}}
\newcommand{\sha}[1]{\, \#_{#1}\,}

\newcommand{\vare}{{\varepsilon}}
\renewcommand{\bar}{\overline}

\renewcommand{\Re}{\mathrm{Re}\, }

\newcommand{\supp}{{\mathrm{supp} }\,}
\newcommand{\fs}{\mathfrak{s}}

\newcommand{\normLa}[1]{\Big\| e^{a|y|} #1 \Big\|_{L^2_y}}

\newcommand{\kap}{\kappa} 
\newcommand{\zak}{\mathfrak{Z}} 
\newcommand{\sla}[2]{{#1\mkern-2mu/\mkern-1.5mu #2}}

\newcommand{\Opw}[1]{{\mathrm{Op}}^W\!\left(#1\right)}

\newcommand{\Opbw}[1]{{{{\mathrm{Op}}^{{\scriptscriptstyle{\mathrm BW}}}}\left(#1\right)}}

\newcommand{\cOpbw}[1]{{{{\mathrm{Op}}^{{\scriptscriptstyle{\mathrm BW}}}}\left(#1\right)}}

\newcommand{\vOpbw}[1]{{\mathrm{Op}}_{\mathtt{vec}}^{{\scriptscriptstyle{\mathrm BW}}}\!\left(#1\right)}

\newcommand{\zOpbw}[1]{{\mathrm{Op}}_{\mathtt{out}}^{{\scriptscriptstyle{\mathrm BW}}}\!\left(#1\right)}

\newcommand{\sym}[3]{#1_{\scs{#2}}^{{\scs{(#3)}}}}
\newcommand{\fun}[2]{#1_{\scs{#2}}}

\newcommand{\scs}[1]{\scriptscriptstyle{#1}}

\newcommand{\pa}{\partial}

\renewcommand{\s}{{\sigma}}
\newcommand{\x}{{\xi}}

\newcommand{\id}{{\mathrm{Id} }}
\newcommand{\vr}{\varrho}
\newcommand{\Ad}{\mathrm{Ad}}

\newcommand{\vect}[2]{{\begin{pmatrix}#1 \\ #2\end{pmatrix}}}
\newcommand{\ov}{\overline}

\newcommand{\pare}[1]{( #1)}
\renewcommand{\be}{\begin{equation}}
\renewcommand{\ee}{\end{equation}}

\newcommand{\mbra}[3][]{\left\{\!\!\left\{#2,#3\right\}\!\!\right\}_{\!#1}}

\title{{\fontsize{14}{10}\bf \scshape quasi-resonant normal form and quadratic lifespan for
3D gravity-capillary water waves}}
\author{ \fontsize{9.8}{12}\scshape Roberto Feola,\thanks{Universit\`a degli Studi RomaTre, Dept. of Mathematics and Physics, Largo San 
    Leonardo Murialdo 1, 00146 Roma, Italy. \textit{Email:} \texttt{roberto.feola@uniroma3.it}} \ Riccardo Montalto,\thanks{Università degli studi di Milano, Dept. of Mathematics Federigo Enriques,
	Via Saldini 50, 20133 Milano, Italy. \newline \textit{Emails:} {\texttt{riccardo.montalto@unimi.it}}, 
{\texttt{federico.murgante@unimi.it}}} 
\ federico murgante \thanks{Università degli studi di Milano, Dept. of Mathematics Federigo Enriques,
Via Saldini 50, 20133 Milano, Italy 
\newline
Fields Institute for Research in Mathematical Sciences, 222 College St, Toronto, Canada
 \newline
 \textit{Email:} {\texttt{murgante@fieldsinstitute.ca}}, 
{\texttt{federico.murgante@unimi.it}}}
    }
\date{}
\begin{document} 
	\maketitle

\begin{abstract}
\small 
\noindent We study the long-time dynamics of 
small-amplitude solutions to the 
three-dimensional gravity–capillary water waves 
equations
for an inviscid and irrotational fluid 
with periodic boundary conditions. 
We prove that, for almost all values of the 
surface tension parameter, solutions with initial size 
$\varepsilon$ exist and remain  small over 
time intervals of order $\varepsilon^{-2}$.

\noindent
A major difficulty arises from the loss 
of derivatives caused by the quasilinear 
nature of the equations combined 
with severe quadratic and cubic 
small-divisor interactions in high space dimensions.
Classical normal form methods applied 
to 3D water waves system  
typically fail to prevent derivative 
loss due to the accumulation of near-resonances.

\noindent To overcome this obstruction, we develop 
a new analytical strategy that combines 
a sharp frequency partition with 
a quasi-resonant normal form transformation  
acting only on selected interaction scales. 
Our microlocal analysis reveals that 
the potentially dangerous interactions terms exhibit 
a block-diagonal structure, 
which stems from both the geometric properties 
of the quasi-resonant frequency sets and 
the Hamiltonian structure of the 
water waves system. 
As a consequence, these operators preserve 
Sobolev norms and do not produce energy growth. 
This structural insight, together with the 
quasi-resonant normal-form transformation, 
allows us to prevent derivative-loss mechanisms 
while avoiding the accumulation of harmful small denominators.
\end{abstract}

{\footnotesize
\noindent
{\sc keywords}: Fluid Mechanics, Water Waves, capillarity, long-time stability, Microlocal Analysis, 
Energy methods, quasi-resonant normal forms, high/low frequencies decomposition.

}

\begingroup
\hypersetup{
linkcolor=black,
}
	\tableofcontents
    \endgroup
	\section{Introduction}
The gravity–capillary water waves equations describe the motion of an
incompressible, irrotational fluid with a free surface under the
combined action of gravity and surface tension.
Despite their long history, the long-time stability of small-amplitude
solutions remains largely open, especially on compact domains.
While local well-posedness is by now well established, understanding
the behaviour of solutions over nonlinear time scales poses
substantial additional challenges.
These difficulties become more pronounced in three spatial dimensions,
where the nonlinear interactions are
more intricate.

With this motivation in mind, we consider an
irrotational perfect fluid under the action of gravity and capillarity
forces, occupying at time $t$ a three-dimensional domain with infinite
depth, given by
\begin{equation*}
\Omega := \big\{ (x,y)\in \T^2 \times\R \, ;  
\ - \infty <y<\eta(t,x) \big\},  \quad 
  \T^2 := \R^2 \slash (2 \pi \Z)^2 \, ,
\end{equation*}
where $\eta$ 
is a regular enough function. 
The velocity field in the time dependent domain 
$ \Omega $ is the gradient of a harmonic function 
$\Phi$, called the velocity potential. 
The time-evolution of the fluid is determined 
by a system of equations for 
the two functions $(t,x)\to \eta(t,x) $, $ (t,x,y)\to \Phi(t,x,y)$.
Following Zakharov \cite{Zakharov1968} and Craig-Sulem \cite{CrSu}, 
given 
$\eta(t,x)$ and 
the restriction $ \psi(t,x) := \Phi(t,x,\eta(t,x))$ 
of the velocity potential at the top boundary, 
one can recover  
$\Phi(t,x,y)$ as the unique solution of the elliptic problem 
\begin{equation*} 
\Delta \Phi = 0  \ \text{in } \,\Omega\,, \quad
\partial_y \Phi \to 0  \ \text{as } y \to - \infty \, , \quad 
\Phi = \psi \  \text{on } \{y = \eta(t,x)\}. 
\end{equation*}
In terms of the variables $(\eta,\psi)$, the gravity–capillary water waves 
equations take the Zakharov–Craig–Sulem form:
\begin{equation}\label{eq:1.2}
\left\{\begin{aligned}
\partial_t \eta &= G(\eta) \psi
	\\
\partial_t \psi &=- g\eta  - \frac{1}{2} |\nabla \psi|^2 
+ \frac{1}{2} \frac{(\nabla \eta \cdot \nabla \psi 
+ G(\eta) \psi)^2}{1 + |\nabla \eta|^2}
+ \kap \div\left(\frac{\nabla \eta}{\sqrt{1+|\nabla\eta|^2}}\right)\, ,
\end{aligned}\right.
\end{equation}
where $g>0$ is the gravity constant, $\kap >0$ 
is the surface tension coefficient, 
and where $ G(\eta)\psi $ is the Dirichlet-Neumann operator  
\begin{equation}\label{eq:112aINTRO}
  G(\eta)\psi := 
  \sqrt{1+|\nabla\eta|^2 
   } (\partial_n\Phi)\vert_{y=\eta(t,x)}
  = (\partial_y\Phi 
  -\nabla\eta\cdot\nabla\Phi)(t,x,\eta(t,x))\,,
\end{equation}
and $n$ is the outward unit normal at the 
free interface $y= \eta(t,x)$.
As observed by Zakharov 
\cite{Zakharov1968}, the equations \eqref{eq:1.2} are the Hamiltonian system associated to the Hamiltonian energy functional
\begin{equation*}
H(\eta, \psi) := \frac12 \int_{\T^2} \psi \, 
G(\eta ) \psi \, dx + \frac{g}{2} \int_{\T^2} \eta^2  \, dx
+\kap\int_{\T^{2}}\big(\sqrt{1+|\nabla\eta|^2}-1\big) dx,
\end{equation*}
given by the sum of the kinetic 
and potential energy of the fluid 
due to gravity and surface tension.
The mass $\int_{\T^2} \eta \, dx$ 
is a prime integral of \eqref{eq:1.2} and, 
with no loss of generality, we can fix it to zero 
by shifting the $y$ coordinate. 
Moreover \eqref{eq:1.2} is invariant under spatial translations
and Noether's theorem implies that 
the momenta 
$ \mathtt M_i (\eta, \psi) := 
\int_{\T^2} \eta_{x_i} (x) \psi (x)  \, dx $, $i = 1,2 $
are prime integrals of \eqref{eq:1.2}. 
Let $ H^s (\T^2;\R)$, $ s\in \R $, 
be the standard Sobolev spaces of $ 2 \pi $-periodic 
functions of $ x\in\T^{2} $. 
It is natural to consider 
$ \eta $ in the subspace of zero average functions
$H^s_0 (\T^2;\R) \subset H^s (\T^2;\R)$, 
and $ \psi$ in the standard homogeneous 
Sobolev space $  {\dot H}^s  (\T^2;\R)  $\footnote{The spaces  
$ {\dot H}^s  (\T^2;\R) $ and $ H^s_0  (\T^2;\R) $ are 
isometric.
Thus we will conveniently identify 
$ \psi $ with a zero average function.}.

Our main goal is to establish an 
\emph{extended lifespan} result for the system 
\eqref{eq:1.2}. We show that solutions with 
sufficiently high Sobolev regularity can be extended 
far beyond the time threshold predicted by 
local existence theory, which, for initial 
data of size $0\leq\vare\ll1$, is of order $O(\vare^{-1})$. 
More precisely, we obtain a \emph{quadratic} 
lifespan, i.e. of order $O(\vare^{-2})$, for solutions 
evolving from $\vare$-small initial data, extending 
the previous result \cite{IoP}, where a time of existence 
of order $O(\vare^{-\sla{5}{3}^{+}})$ was proved for the system 
\eqref{eq:1.2}. 

The time scale $O(\varepsilon^{-2})$ is the natural stability scale 
for equations with non-resonant quadratic nonlinearities. 
Actually, in general, the gravity–capillary dispersion relation 
admits nontrivial three-wave resonances. These resonances are absent 
for almost every value of the surface tension coefficient, 
which explains why our result (as in \cite{IoP}) holds for 
$\kappa$ outside a zero-measure exceptional set.

However, even in this non-resonant regime, reaching the quadratic time scale 
remains highly nontrivial. The main obstruction is the presence of 
strong small divisor effects generated by quasi-resonant three- 
and four-wave interactions, whose strength increases with the largest 
frequency. These quasi-resonances produce derivative losses that 
prevent the closure of classical normal form reductions and 
limit the lifespan in \cite{IoP}.

The main novelty of the present work is the introduction of a
quasi-resonant normal form approach, tailored to quasilinear problems on 
higher-dimensional tori with weak small divisor control. This approach allows us to
reach the quadratic lifespan stated in the main theorem below. Given a Banach space $(X, \| \cdot \|_X)$ and a compact interval $I = [a, b]$, we denote by $C^0(I; X)$ the space of continuous functions $u : I \to X$ equipped with the norm
$\| u \|_{C^0([a, b]; X)} := \sup_{t \in I} \| u(t) \|_X$ and for any integer $k \geq 1$, we denote by $C^k([a, b]; X)$, the space of $C^k$ functions $u : I \to X$ equipped with the norm $\| u \|_{C^k([a, b]; X)} := \sum_{n = 0}^k \| \partial_t^n u \|_{C^0([a, b]; X)}$. 

\begin{theorem}
\label{thm:main}
    Let $g>0$. There are $s_0>0$ and a zero measure set 
    $\mathscr{N}\subset \R^+$ 
    such that for any $\kap \in \R^+ \setminus \mathscr{N}$ the following holds. 
    For any $s\geq s_0$ there are $\e_0=\e_0(s,\kap)>0$ 
    and two constants $\mathtt{c}=\mathtt{c}_s>0$, $C=C_s>0$ 
    such that for any initial data 
    $(\eta_0,\psi_0) \in H^{s+\sla{1}{4}}_0(\T^2;\R)\times \dot H^{s-\sla{1}{4}}(\T^2;\R)$ 
    with 
    \begin{equation}\label{piccolezzaDati}
        \| \eta_0\|_{H^{s_0+\sla{1}{4}}_0(\T^2;\R)}
        + \| \psi_0\|_{\dot H^{s_0-\sla{1}{4}}(\T^2;\R)}\leq \vare
        \leq \e_0\,,
    \end{equation}
    there is a unique long time classical solution 
    \begin{equation}\label{timetime}
   (\eta,\psi) \in C^0\left([0,T_\vare]; H^{s+\sla{1}{4}}_0(\T^2;\R)\times \dot H^{s-\sla{1}{4}}(\T^2;\R)\right), \qquad \text{with} \quad T_\vare\geq \mathtt{c} \vare^{-2}\,.
   \end{equation}   
    Moreover up to time $ T_\vare$ one has the bound 
    \begin{align}
       &\sup_{t \in [0,T_\vare]} 
       \|  \eta(t)\|_{H^{s_0+\sla{1}{4}}_0(\T^2;\R)}
       + \| \psi(t)\|_{\dot H^{s_0-\sla{1}{4}}(\T^2;\R)}
       \leq C\vare\,,\label{stimaFinalemain}
       \\
       &\sup_{t \in [0,T_\vare]} 
       \|  \eta(t)\|_{H^{s+\sla{1}{4}}_0(\T^2;\R)}
       + \| \psi(t)\|_{\dot H^{s-\sla{1}{4}}(\T^2;\R)}
       \leq C\left(\|  \eta_0\|_{H^{s+\sla{1}{4}}_0(\T^2;\R)}
       + \| \psi_0\|_{\dot H^{s-\sla{1}{4}}(\T^2;\R)}\right)\,.
       \label{stimaFinalemain2}
    \end{align}
\end{theorem}
Some comments  on \Cref{thm:main} are in order.
\noindent
The local well-posedness of the water waves and free boundary 
Euler equations has been addressed by many authors and it is presently well-understood.  
We mention for instance
\cite{Nali, Yosi, Crai, Wu0, Wu1, BeGun, CrisLin, AMS, Lannes, Sch, Lind1, Com, MZ, SZ1, SZ2}
and the more recent results
\cite{ABZ2011_2, ABZ2014, ABZnl}. In particular we remark that 
Local well-posedness theory yields an existence time depending on the size 
of the initial data.  
For data of size $\varepsilon$, one typically obtains a lifespan 
$T \sim \varepsilon^{-1}$. 

Theorem \ref{thm:main} shows that, for almost every value of the surface 
tension coefficient, this time scale can be improved to the quadratic 
scale $T \sim \varepsilon^{-2}$. As discussed above, this is the natural 
nonlinear stability scale in the absence of three-wave resonances. 
Achieving it on the torus is particularly delicate due to the lack 
of dispersive decay and the presence of strong small divisor effects.

Long-time stability for quasilinear PDEs on one-dimensional 
tori is by now well understood, following the paradigm 
introduced by Berti–Delort \cite{BD2018} and further developed 
in \cite{FI2, BFP, BFF, BMaMu, MMS24, MRS}. 
In higher dimensions, however, substantially new mechanisms are required. 
The principal obstruction comes from the superlinear 
dispersion relation, 
which generates strong small divisor effects through 
almost-resonant three- and four-wave interactions. 
These near-resonances induce derivative losses 
in the energy estimates and prevent the closure 
of classical normal form reductions or purely energy-based 
arguments at the quadratic time scale. 
In the three-dimensional water waves setting, 
the only previous result in this direction is \cite{IoP}, 
where an intermediate lifespan 
$T \sim \varepsilon^{-5/3+}$ was obtained.

The main novelty of the present work 
is a framework tailored to quasilinear problems on higher dimensional 
tori with superlinear dispersion and weak small divisor control. 
We combine a high–low frequency decomposition 
with a quasi-resonant normal form analysis. 
Rather than eliminating all quadratic interactions, 
as in a classical Birkhoff scheme, 
we isolate and retain the quasi-resonant component, 
which captures the effective long-time dynamics in the high-frequency regime 
and allows us to close the energy estimates. 
The analysis splits naturally into two parts. 
For low frequencies, a refined Bourgain-type 
block decomposition yields control up to the sharp frequency 
threshold $\tR \sim \varepsilon^{-1}$, significantly improving 
the frequencies scale $\tR \sim \varepsilon^{-2/3}$ achieved in \cite{IoP}. 
For high frequencies, we construct 
a quasilinear modified energy functional that controls the evolution.

The final outcome is a modified energy equivalent to the Sobolev norm 
and almost conserved over times of order $\varepsilon^{-2}$. 
Unlike the energies produced by classical normal form constructions, 
our functional is structurally non-perturbative 
at low frequencies (see \eqref{def:Energy_low}), 
reflecting the genuinely nonlinear 
character of the effective dynamics.

\begin{remark}
      Define, for $ \lambda >0$, 
\begin{equation*}
(\eta_\lambda(t,x), \psi_\lambda(t,x)):= 
( \eta(\lambda t,x), \lambda \psi(\lambda t , x)) \,.
\end{equation*}
If $(\eta,\psi)$ is a solution of \eqref{eq:1.2} 
then $(\eta_\lambda, \psi_\lambda)$ 
is a solution of \eqref{eq:1.2} with 
$(g,\gamma)\leadsto (\lambda^2 g , \lambda^2 \kap)$. 
For this reason, 
we henceforth fix $g = 1$, without loss of generality.

\end{remark}

\smallskip

{\sc Periodic setting vs Euclidean space.} 
Let us mention also that normal forms methods has been successfully applied 
to study long time existence of solutions of dispersive PDEs on $\R^d$, for 
smooth Cauchy data  that decay at infinity.
In the seminal work \cite{Sha1985} by Shatah, this approach 
has been introduced to prove global existence of solutions of nonlinear Klein-Gordon equations.
More recently this approach has been extended to several different situations. 
We mention for instance 
global 
in time results (for  
sufficiently small,  localized, 
and  smooth  
 initial data) 
\cite{Wu, IP, AlDe1, IFRT, IP3}  
have been obtained for irrotational water 
waves equations in  $2D$,
and \cite{GMS, GMS2, Wu2, DIPP} for the $3D$ case. 
In the aforementioned papers, the dispersive effects of the linear flow are exploited.
The periodic setting  is deeply different, 
as the linear waves  oscillate without decaying in time, 
and the long time dynamics of the solutions strongly depends
on 
the presence of $N$-wave resonant interactions and  
the Hamiltonian and reversible nature of the equations. 
In this periodic setting there are results about stability over a finite, but very long, time scale for the water waves on the one dimensional torus fro solution with initial data of size 
$O(\vare)$. We mention for instance 
\cite{ifrTat, IFT, HuIT, HarIT, BFF} 
where times of stability of order $O(\vare^{-2})$
have been proved for the 2D water waves in various contexts.
It is also worth mentioning \cite{BFP}
where the authors solved the Zakharov conjecture 
\cite{Zak2}
for 2D pure gravity Water Waves, 
namely they proved that for small data of size $O(\vare)$, 
the solutions stay of order $O(\vare)$ over 
a time interval of size $O(\vare^{- 3})$. See also \cite{Wu3} for an extension to more general intial data.
For the gravity capillary water waves we quote the paper
 \cite{BD2018}, Berti-Delort where the authors proved 
 almost global existence, namely time of stability $O(\vare^{-N})$
 for any $N\gg1$. This latter result
 has been extended in \cite{BMaMu}.
Regarding the 3D the situation is more involved and the only available result for the water waves equations, up to now, is the seminal work by Ionescu-Pusateri \cite{IoP}. 
We refer to \cite{FM2022, FGI} for a similar result on different models of quasi-linear PDEs in high space dimension.

\smallskip
\noindent 
{\sc Periodic and Quasi-periodic solutions:}
So far no global existence is  known for the solutions of \eqref{eq:1.2}  with periodic boundary conditions. However,
several families of time periodic/quasi-periodic 
 solutions of \eqref{eq:1.2} 
have been constructed in the last years in \cite{Wh2,BFM1,BM}.
Other KAM results for pure gravity water waves are  proved in
\cite{PlTo, IPT, BBHM, BFM2, FG}. The mentioned papers regards only the 2D case. Concerning the full 3d models which is much more complicated,
we refer to
\cite{CN, IP-Mem-2009, IP2} for the existence of bi-periodic traveling waves which are stationary in a moving frame. See also the more recent result \cite{GNPW1} and reference therein.
The case of general quasi-periodic travelling waves (not stationary in any moving frame)
has been addressed in \cite{FMTtravel}

\smallskip
\noindent
{\sc Long time regularity and normal forms 
for PDEs on compact manifolds.}
On tori (or more in general on compact manifolds) 
there are no dispersive effects that could help 
in controlling the behavior of 
the solutions
for long times. In order to clarify the main 
difficulties and to describe the main results 
let us consider a PDE on $\T^d$
\begin{equation*}
\partial_t u + \ii\Lambda(D) u 
+ {\mathcal Q}(u, u) = 0\,, \qquad x \in \T^d\,,
\end{equation*}
where $\Lambda(D)$ is a suitable dispersion 
relation with symbol $\Lambda(\xi)$, and 
${\mathcal Q}(u, u)$ is a quadratic 
nonlinearity of order smaller or equal to the 
order of $\Lambda$ (namely the same one of $\Lambda$). 
In order to prove that for small initial 
data of size $O(\vare)$, the solution stay
of order $O(\vare)$ over a time interval 
of size $O(\vare^{- 2})$, one has to perform a 
Birkhoff normal form step, 
namely to construct a change of coordinates 
$u = \Phi(w)$ such that the PDE in the 
new coordinates reads as 
\[
\partial_t w + \ii \Lambda(D) w = \text{cubic terms}\,. 
\]
In order to perform this change of coordinates, 
one has to control the "three waves interactions" 
\[
\Lambda (\xi_1) \pm \Lambda(\xi_2) \pm \Lambda(\xi_3) 
\qquad \text{with} 
\qquad \xi_1 \pm \xi_2 \pm \xi_3 = 0\,,
\]
where $\x_i\in\Z^{d}$.
In order to construct this change of coordinates,
one has to control in a good way the quantity
\begin{equation*}
\dfrac{q_{\pm \pm}(\xi_1, \xi_2, \xi_3)}{\pm \Lambda(\xi_1) 
\pm \Lambda(\xi_2) \pm \Lambda(\xi_3) }\,,
\end{equation*}
where $q_{\pm \pm}$ are some coefficients associated 
to the bilinear form ${\mathcal Q}$.  
Hence, one needs to control two features in order to obtain good energy estimates in the new variables:
 1) to provide a good control 
on the divisors 
$\pm \Lambda(\xi_1) \pm \Lambda(\xi_2) \pm \Lambda(\xi_3)$. 
2) To control the growth of the coefficients 
$q_{\pm \pm}$ w.r. to the Fourier indices. 
The case of semilinear PDEs is much easier 
than the case of PDEs containing derivatives 
in the nonlinearity, since the coefficients 
$q_{\pm \pm}$ are bounded w.r.t. the Fourier indices. 
Moreover, if one has a control of 
the small divisors as 
\begin{equation*}
|\Lambda (\xi_1) \pm \Lambda(\xi_2) \pm \Lambda(\xi_3)| 
\gtrsim \dfrac{1}{{\rm min}(|\xi_1|, |\xi_2|, |\xi_3|)^\tau}\,, 
\qquad \tau \gg 0\,,
\end{equation*}
one is able to construct a well 
defined transformation such that in the new coordinates 
one can perform good energy estimates. 
The formal argument explained above could also be 
performed at any order, by controlling higher 
order resonances 
\[
\Lambda(\xi_1) \pm \Lambda(\xi_2) \pm \ldots \pm \Lambda(\xi_k)\,, 
\qquad \xi_1 \pm \ldots \pm \xi_k = 0 \,,
\]
and the corresponding estimate \eqref{stima max 3 divisors} 
for higher order resonance will be given by 
\begin{equation}\label{stima max 3 divisors}
|\Lambda (\xi_1) \pm \ldots \pm \Lambda(\xi_k)| 
\gtrsim 
\dfrac{1}{{\rm max}_3(|\xi_1|, \ldots, |\xi_k|)^\tau}\,, 
\qquad \tau \gg 0
\end{equation}
where
\[
{\rm max}_3(|\xi_1|, \ldots, |\xi_k|) := 
\text{the third largest number among} 
\qquad |\xi_1|, \ldots, |\xi_k|\,. 
\]

This approach has been successfully and widely used 
in the past starting from the study of semi-linear, one dimensional PDEs, see \cite{Bou96, Bam1, BG1} and some extension to the high dimensional case \cite{BDGS, DelortSzeft1}.
%
The situation becomes much more complicated 
when one has a bad control of the small divisors, 
namely 
\begin{equation}\label{stima max divisors}
|\Lambda (\xi_1) \pm \ldots \pm \Lambda(\xi_k)| 
\gtrsim \dfrac{1}{{\rm max}(|\xi_1|, \ldots, |\xi_k|)^\tau}\,, 
\qquad \tau \gg 0\,,
\end{equation}
which provides a very strong loss of regularity, 
which is an obstruction in order to define 
a well defined normal form transformation. 
Lower bounds of the form \eqref{stima max divisors} 
typically occur on general manifolds and actually, 
already on ``irrational tori''. 
This bad lower bounds is typically related to 
``bad separation properties'' 
of the spectrum of the linear part of the equation.
For this case, we mention 
\cite{DelortSzeft1, Delort-Imrekaz, FIM}
and reference therein,
where
the authors improve the lifespan provided by the local theory. 
An almost global existence result in the case of 
``bad'' small divisors satisfying \eqref{stima max divisors} has been obtained much more recently  in \cite{BaFM}.

\medskip

\noindent
The situation dramatically changes when the 
nonlinearity contains derivatives 
(even not of maximal order) and 
the methods developed in the aforementioned 
papers do not work in this case. 
However, 
for quasi-linear PDEs on tori a very general 
approach has been developed in the last years 
based on the combination of Normal Form techniques 
and para-differential calculus, 
with the main long term objective to 
understand the long time dynamics 
of 2D and 3D Water Waves equations. 
This approach actually developed a general 
method to deal with one-dimensional 
quasi-linear PDEs but the problem of 
understanding quasi-linear PDEs in dimension 
greater than one is totally open. 

\noindent  
For quasi-linear Klein-Gordon equation 
on one dimensional torus and on the sphere, 
long time existence results have 
been proved by Delort \cite{Del2,Del3}. 
In this case the dispersion relation is 
linear $\Lambda(\xi) \sim |\xi|$ 
and one is able to get good lower bounds 
on the divisors like in \eqref{stima max 3 divisors}. 
In \cite{BD2018}, Berti-Delort proved long 
time existence for 2D 
gravity capillary Water Waves (1d interface) 
for the gravity capillary water waves on $\mathbb{T}$ (see also \cite{FI2}
for quasi-linear Schr\"odinger type equations). 
This result was extended to the 
Hamiltonian framework in \cite{BMaMu}, 
by removing additional assumption on initial data 
(see also \cite{BFF}, for any value of the surface tension). 
We also mention \cite{ifrTat, IFT, HuIT, HarIT} 
where times of stability of order $O(\vare^{-2})$
have been proved for the 2D water waves in various contexts. 
Furthermore in \cite{BFP}, 
Berti-Feola-Pusateri  
solved the Zakharov conjecture 
\cite{Zak2}  (see also \cite{CW})
for 2D pure gravity Water Waves 
($\Lambda (\xi) \sim |\xi|^{\frac12}$, 
sublinear dispersion), 
namely they proved that for small data of size $O(\vare)$, 
the solutions stay of order $O(\vare)$ over 
a time interval of size $O(\vare^{- 3})$. We also refer to \cite{Wu3} for long-time results on $\R$, and to \cite{DIP,DIP2}, where the system is considered on a large box.
The result in \cite{BFP} rigorously justify the formal 
integrability of 2D pure gravity Water Waves 
equations up to order four. The approach developed in \cite{BD2018} 
is based on the fact that one can reduce, through a 
paradifferential change of coordinates, 
the Water Waves equation to 
\begin{equation}\label{paradiff intro 1d}
\partial_t u + \ii \Lambda(u; D) u + {\mathcal R}(u) =0 \,,
\end{equation}
(they actually perform it at any order 
but due to our purposes we explain only 
the quadratic reduction) where $\Lambda(u; D)$ 
is a  Fourier multiplier with real symbol and 
${\mathcal R}(u)$ is a nonlinear smoothing 
operator satisfying the estimate
\[
\| {\mathcal R}(u) \|_{H^{s + \rho}} 
\lesssim \| u \|_{H^s}^2\,, 
\qquad \rho \gg 0, \quad s \gg 0\,. 
\]
Then by imposing non-resonance conditions 
of the form \eqref{stima max divisors} 
(with the maximal loss of derivatives), 
by using $\kappa$ as parameter, 
by choosing $\tau \sim \rho$, 
one can construct a bounded normal form 
transformation that transforms \eqref{paradiff intro 1d} to 
\[
\partial_t u + \ii \Lambda(u; D) u  = O(\| u \|_{H^s}^3)\,. 
\]
This is sufficient to ensure long-time stability
on a time scale of order 
$O(\e^{- 2})$. In order to get this reduction, 
it is fundamental the use of paradifferential calculus. The first point is to reduce to constant coefficients the highest order, which is of the form 
$\ii {\rm Op}^{\scriptscriptstyle{BW}}(A(u; x) |\xi|^{\frac32})$ (we refer the reader to \eqref{BWnon} for a  precise definition of the \emph{Bony-Weyl} quantization). 
This is done by means of a bounded and invertible map 
$ \Phi^{\theta}$ as the flow of the paradifferential operator
$ \ii\Opbw{ b (u; \theta, x) \xi} $ where
$ b(u; \theta, x ) = \frac{\beta (u; \theta, x) }{1 + \theta \beta_x (u; \theta, x)} $, which is basically the paracomposition operator induced by the diffeomorphism of the one dimensional torus 
$$
\T \to \T, \quad x \mapsto x + \beta(u; x)\,. 
$$
In this one-dimensional case, in order to reduce the highest order term to constant coefficients, $\beta(u; x, \xi)$ has to solve 
\begin{equation}\label{eq diffeo toro beta 1d}
\Big[ A(u; y)\Big(\big( 1 + \partial_y \beta(u; y) \big) |\xi| \Big)^{\frac32} \Big]_{y = x + \breve \beta(u; x)} = \mathtt c_{\frac32} |\xi|^{\frac32}\,,
\end{equation}
where 
$$
\T \to \T, \quad x \mapsto x + \breve \beta(u; x)
$$
is the inverse diffeomorphism of $y \mapsto y + \beta(u; y)$ and $\mathtt c_{\frac32} \in \R$ is a suitable constant to be chosen. 
 The latter equation is clearly easily solvable by a simple integration. In order to reduce to constant coefficients the lower order terms, one considers the bounded and invertible map 
$ \Phi^{\theta}$ defined as the flow of a paradifferential operator
$ \ii\Opbw{ f(u; x, \xi)} $ where $f(u; x, \xi)$ is a symbol 
depending nonlinearly on $u$ of order smaller than one. 
The equations that such symbols have to satisfy take the form 
\begin{equation}\label{eq omologica ordini bassi 1d}
\partial_\xi \Lambda(\xi) \partial_x f(u; x, \xi) 
+ p(u; x, \xi) = \lambda_p(\xi)\,,
\end{equation}
where 
$\Lambda(\xi) = \sqrt{|\xi|(1 + \kappa |\xi|^2)} 
\sim |\xi|^{\frac32}$, $p(u; x, \xi)$ is 
a symbol of order $m$ depending 
nonlinearly on $u$ and 
\[
\lambda_p(\xi) := \frac{1}{2 \pi} \int_\T p(u; x, \xi)\, d x\,. 
\]
Also this equation is  solvable by a simple integration and using that $\partial_\xi \Lambda(\xi) \sim |\xi|^{\frac12}$ if $p$ is a symbol of order $m + \frac12$, then $f$ is a symbol of order $m$. This ideas developed by  Berti-Delort \cite{BD2018} actually fail in higher space dimension, since the homological equations \eqref{eq diffeo toro beta 1d}, \eqref{eq diffeo toro beta 1d} are not solvable in higher space dimension. Concerning the first equation, this has even a geometric interpretation: on the one dimensional torus all the metrics close to the flat can be transformed into the flat metric by a suitable diffeomorphisms of the torus. This is actually not true on higher dimensional tori. The equation \eqref{eq diffeo toro beta 1d} on the torus $\T^2$, takes the form 
\begin{equation}\label{eq diffeo toro in 2d}
\Big[ A(u; y) \Big| \Big({\rm Id} + \nabla_y \beta(u; y) \Big)[\xi] \Big|^{\frac32} \Big]_{y = x + \breve \beta(u; x)} = \mathtt c_{\frac32} |\xi|^{\frac32}, \quad y \in \T^2\,, 
\quad  \xi \in \R^2\,.
\end{equation}
This is an equation whose unknown is $\beta(u; \cdot) : \T^2 \to \T^2$, 
which one is not able to solve. 
Concerning equations of the second form 
\eqref{eq omologica ordini bassi 1d}, for 
$(x, \xi) \in \T^2 \times \R^2$, takes the form  
($\nabla \Lambda(\xi) \sim |\xi|^{- \frac12} \xi$)
\begin{equation}\label{eq omologica ordini bassi 2d}
|\xi|^{- \frac12} \xi \cdot \nabla_x f(u; x, \xi) + p(u; x, \xi) 
= \lambda_p(\xi)\,.
\end{equation}
This equation is not even  solvable in higher dimension. 
Indeed, by expanding in Fourier series with respect to $x$, 
one has to control the resonances and the quasi resonances 
$|\xi \cdot k| \ll 1$, $\xi \in \R^2$, $k \in \Z^2 \setminus \{ 0 \}$. 
This kind of lower order normal form has been developed 
for the first time in \cite{BML2019, BML2, BML3}, 
for linear  of Schr\"odinger equations with unbounded perturbations 
and it has been used in \cite{FM2022}  for proving long time existence 
for nonlinear Schr\"odinger equations with unbounded nonlinearity. 
On the other hand, this is not sufficient 
for dealing with quasi-linear 
cases such as 3D gravity capillary water waves. 
In the seminal
work Ionescu-Pusateri \cite{IoP}, 
the authors developed a strategy for going beyond the lifespan 
of the local existence (see also the extension to other models 
Feola-Gr\'ebert-Iandoli \cite{FGI}). We also quote Delort-Szeftel 
that proved a quadratic lifespan for quasi-linear 
Klein-Gordon equations on tori. 
Due to  the presence of small divisors, the strategy in \cite{IoP} 
does not allow to reach the natural stability time of a 
quadratic normal form step. 
In the next section we shall describe in detail the main 
ideas of the proof of our result. 

\medskip

\noindent
{\bfseries Notations}. 
Throughout the paper we shall use the following convention for natural numbers $\N:=\{1,2, \ldots\}$ and  $\mathbb{N}_0:=\mathbb{N}\cup\{0\}$.
Given some parameters $\sigma_1, \ldots, \sigma_k > 0$, we shall 
use the notation $A\lesssim_{\sigma_1, \ldots, \sigma_k} B$ to denote 
$A\le C B$ where $C \equiv C(\sigma_1, \ldots, \sigma_k) > 0$ is a positive constant
depending on $\sigma_1, \ldots, \sigma_k$. If the constant depends on the lowest Sobolev index or the orders of the operators involved in our proof, we simply write $\lesssim$. 
 We will only emphasize the dependence on the Sobolev index $s$ by writing $\lesssim_s$. Similarly we will write $A\sim B$ if  $A\lesssim B\lesssim A$ and $A\sim_s B $ if the implicit constants depend on $s$.
 We shall also write $H_{x}^{s}$ or $H^{s}$ instead of $H^{s}(\T^2;\C)$. Given two Banach spaces $X, Y$, we denote by ${\mathcal L}(X; Y)$ the space of bounded linear operators from $X$ to $Y$ equipped with the standard operator norm $\| \cdot \|_{{\mathcal L}(X; Y)}$. If $X = Y$ we often write ${\mathcal L}(X)$ instead of ${\mathcal L}(X; Y)$. 

\subsection{Ideas of the proof and 
outline of the strategy}\label{idee proof}
%

Following \cite{ABZ2014,AlDe}, 
we begin our analysis by paralinearizing the water waves system \eqref{eq:1.2}, 
 expressed in the variable $(\eta,\omega)$, where $\omega$ in \eqref{def:good}
 is the so called ``Good unknown'' of Alinhac. 
 This is the content of Proposition \ref{prop:paraWW}
whose proof is based on a paralinearization formula 
for the Dirichlet-Neumann operator
$G(\eta)$ in \eqref{eq:112aINTRO} 
provided in Appendix \ref{app:DN}. In particular  we show that $G(\eta)$
admits an expansion like (for the precise definitions and properties of 
para-differential operators, we refer to Section \ref{sezione-functional setting})
\begin{equation}         \label{eq_para:DNINTRO}
     G(\eta)\psi= \Opbw{\lambda(\eta;x,\xi)}\omega+ \Opbw{-\ii \tV\cdot \xi- \tfrac12\div(\tV)}\eta
     + R_{\geq 1}(\eta)\psi\,,
\end{equation}
where $\omega=\psi-\Opbw{\mathtt{B}}\eta$ is the good unknown in \eqref{def:good}, the functions 
$\mathtt{V}$ and $\mathtt{B}$ are respectively the  horizontal  and vertical 
components of the velocity field $(\nabla\Phi,\pa_{y}\Phi)$ (see  \eqref{def:V-B}),
and where $\lambda(\eta;x,\x)$ is a symbol (see \Cref{def:simbolohold}) of order one admitting an expansion
as
\[
\lambda(\eta;x,\x)=\sqrt{(1+|\nabla \eta|^2)|\xi|^2 - (\nabla \eta \cdot \xi)^2 }+{\rm l.o.t.}
\]
Moreover the operator $R_{\geq1}(\eta)[\cdot]$ in \eqref{eq_para:DNINTRO}
is a \emph{smoothing remainder} in the sense that it satisfies the following tame estimates
           \begin{align*}
               \| R_{\geq 1}(\eta)\psi\|_{H^{s+\vr}}
               \lesssim_s \| \eta\|_{H^{s_0}}\| \psi\|_{H^{s}}+ \| \eta\|_{H^{s}}\| \psi\|_{H^{s_0}},
               \qquad \rho \gg 0, \quad s \gg 0\,, \;\;s_0>1\,.
           \end{align*}
We remark that, a paradifferential 
expansion like \eqref{eq_para:DNINTRO} is somehow already known in literature, we refer
for instance to \cite{AlM} and \cite{DIPP}.
However, for our aims we need precise and sharp information on symbols and  remainders.
In particular we prove that the symbol $\lambda$ and the remainder $R_{\geq1}(\eta)[\cdot]$
admits Taylor expansions in the variable $\eta$ close the flat surface, see for instance 
\eqref{est_R:DN2}-\eqref{est_R:DN}. Such an expansion on the remainders is  necessary 
for the normal form reduction of \Cref{sezione birkhoff smoothing}.

As a second step, 
we express the system \eqref{eq:eta}-\eqref{eq:omega}
in complex coordinates \eqref{def:cQ} (see  \Cref{diag.ord0}), and then
we perform a (actually quite standard) 
diagonalization of the paralinearized system \eqref{eq:U}
up to smoothing remainders, 
obtaining system \eqref{diagonale}, see \Cref{prop:blockdecoupling}.

It is worth noting  that the diagonalization procedure
is performed under the following assumption:
given an initial condition 
\[
\|\eta_{0}\|_{H^{s+1/4}}+\|\psi_0\|_{H^{s-1/4}}\leq \vare\,,
\]
then there is a time $T>0$ and a solution $(\eta(t),\psi(t))$ of \eqref{eq:1.2} defined over the time interval 
$[0,T]$, such that
\[
\sup_{t\in[0,T]}\big(\|\psi(t,\cdot)\|_{H^{s+\frac{1}{2}}}+\|\eta(t,\cdot)\|_{H^{s-\frac{1}{2}}}\big)\lesssim_s \vare\,.
\]
This assumption is verified in view of the well-established local well-posedness theory for water waves, see 
for instance \cite{ABZ2011_2, ABZ2014}.
Passing to the complex coordinates  \Cref{diag.ord0} (see also \eqref{zak} 
and the equivalence \eqref{equiv:U_etaomega})
the assumption above implies that we shall work with solutions $z=z(t,x)$ of the complex 
water waves system 
\eqref{eq:U}
satisfying $\|z(0,\cdot)\|_{H^s}\lesssim \varepsilon$ and 
\begin{equation}\label{ansatz local existence intro}
\sup_{t\in[0,T]}\|z(t,\cdot)\|_{H^s}\lesssim_s \varepsilon\,.
\end{equation}
In particular, local theory implies that $T\sim \vare^{-1}$.

\noindent
In brief, after the reduction of  Sections
\ref{sec good unknown}, \ref{sec complex coordinates},
\ref{section block-decoupling}, the water waves system \eqref{eq:1.2} is transformed into 
(here, 
for simplicity, we retain only the diagonal terms)
\begin{equation*}
\partial_t z+ \ii \Lambda(D) z
+ \ii \Opbw{\sym{\Sigma}{\geq 2}{\sla{3}{2}}+\sym{a}{\geq 1}{\sla{1}{2}}} z
+  \Opbw{\ii \mathtt{V}\cdot\x} z
+ {\mathcal R}(z) = 0\,,
\end{equation*}

\begin{itemize}
\item $\Lambda(D):=\Opbw{\Lambda(\x)}$ with $\Lambda(\x):=\sqrt{|\x|(1+\kap|\x|^2)}$;
\item $\sym{\Sigma}{\geq 2}{\sla{3}{2}}:=\sym{\Sigma}{\geq 2}{\sla{3}{2}}( x, \xi)$ is a symbol of order 
$\frac{3}{2}$, and of size at least $O(\vare^{2})$
(see \eqref{def:sd}-\eqref{def:sdBIS});
\item $\sym{a}{\geq 1}{\sla{1}{2}}:=\sym{a}{\geq 1}{\sla{1}{2}}( x, \xi)$ 
is a symbol of order $\frac{1}{2}$, and of size at least  $O(\vare)$;
\item ${\mathcal R}(z)$ is a smoothing operator satisfying
\begin{equation}\label{stimaquadIntro}
\|{\mathcal R}(z)\|_{H^{s+\vr}} \lesssim_s \|z\|_{H^s}^2,
\qquad s \gg \vr \gg 0\,,
\end{equation}

\item $\mathtt{V}:=\nabla\Phi(x,\eta(x))$; in particular one has 
the homogeneity expansion (see \eqref{def:V-B}, \eqref{V:exp1})
$\mathtt{V}=\nabla\psi+\mathtt{V}_{\geq2}$ where $\mathtt{V}_{\geq2}$ is 
at least quadratic in $u$.
\end{itemize}
In \Cref{section-riduzione-trasporto}
we start to perform a normal form reduction for the $1$-homogeneous 
component of the transport term $-\ii \tV \cdot \xi$. 
Note that such $1$-homogeneous part is explicitly given by $ - \ii \nabla \psi\cdot \xi$.
To perform such a reduction we reason as follows. We look for a nonlinear change of variables of the form
$u=\Phi^1(z)z$, where $\Phi^\tau(z)$ denotes the flow of the
linear, $u$-dependent, variable-coefficient equation
\[
\partial_\tau \Phi^\tau(z)
=
\ii \Opbw{\sym{g}{1}{\sla{1}{2}}(x,\xi)}\, \Phi^\tau(z)\,,
\qquad
\Phi^0(z)=\uno\,,
\]
for a symbol $\sym{g}{1}{\sla{1}{2}}(x,\xi)$ of order $1/2$, linear in $\psi$, 
to be determined. The equation satisfied by $u$ takes the form
\[
\partial_t u+ \ii \Lambda(D) u
+ \ii \Opbw{\sym{a}{\geq 2}{\sla{3}{2}}+\sym{b}{\geq 1}{\sla{1}{2}}} u
+  \Opbw{\ii \sym{\mathtt{h}}{1}{1}} u
+ {\mathcal R}(u) = 0\,,
\]
where $\sym{a}{\geq 2}{\sla{3}{2}}$ is a symbol of order $3/2$ at least quadratic in $u$, 
$\sym{b}{\geq 1}{\sla{1}{2}}$ is of order $1/2$ and at least linear in $u$, $\mathcal{R}$ is a quadratic 
smoothing remainder (see \eqref{stimaquadIntro})
and where 
\[
\sym{\mathtt{h}}{1}{1}:=\nabla\psi\cdot\x+\{\sym{g}{1}{\sla{1}{2}}(x,\x) ,\Lambda(\x)\}=
\nabla\psi\cdot\x-(\pa_{\x}\Lambda(\x))\cdot\nabla\sym{g}{1}{\sla{1}{2}}(x,\x)\,.
\]
Our aim is to find $\sym{g}{1}{\sla{1}{2}}(x,\x)$ in such a way 
$\sym{\mathtt{h}}{1}{1}\equiv0$ (see \eqref{homo:g12}).
We look for  $\sym{g}{1}{\sla{1}{2}}$ of the form 
$\sym{g}{1}{\sla{1}{2}}(x,\x)=\psi \sym{\mathtt{m}}{}{\sla{1}{2}}(\x)$ for some 
Fourier multiplier $\sym{\mathtt{m}}{}{\sla{1}{2}}(\x)$. 
Therefore, by an explicit computation, we have
\[
\begin{aligned}
0&=\sym{\mathtt{h}}{1}{1}=\nabla\psi\cdot\x-(\pa_{\x}\Lambda(\x))\cdot\nabla\sym{g}{1}{\sla{1}{2}}(x,\x)
=\nabla\psi\cdot\x\Big(1-\frac{\sym{\mathtt{m}}{}{\sla{1}{2}}(\x)(|\x|^{-1}+3\kappa|\x|)}{2\Lambda(\x)}\Big)
\end{aligned}
\]
provided we define
\[
\sym{\mathtt{m}}{}{\sla{1}{2}}(\x):=\frac{2\Lambda(\x)}{|\x|^{-1}+3\kappa|\x|}\,.
\]
We remark that we are able to 
completely remove the term $\nabla\psi\cdot\x$, 
without small divisors issues, thanks to its gradient structure 
which is reminiscent of the vector potential in electrodynamics.
In \cite{DIPP, IoP}, this structure is observed at the level of the energy estimates, where it gives rise to a favorable factor, referred to as the \emph{depletion factor}. In our approach, we instead exploit it at the level of the equation to eliminate the first-order term up to a symbol of order $\sla{1}{2}$. We refer to \Cref{prop:nfquadtra} for details.

To summarize, 
after the  sequence of symmetrizations and 
paradifferential normal form reductions  of Sections
\ref{sec good unknown}, \ref{sec complex coordinates},
\ref{section block-decoupling}, \ref{section-riduzione-trasporto},
the water waves system can be reduced 
to the paradifferential equation
\begin{equation}\label{paradiff eq intro}
\partial_t u + \ii \Lambda(D) u
+ \ii \Opbw{\sym{a}{\geq 2}{\sla{3}{2}}(u; x, \xi)} u
+ \ii \Opbw{\sym{b}{\geq 1}{\sla{1}{2}}(u; x, \xi)} u
+ {\mathcal R}(u) = 0,
\end{equation}
for some quadratic smoothing remainder $\mathcal{R}$ (see \eqref{stimaquadIntro}), and 
where $\sym{a}{\geq 2}{\sla{3}{2}}$ is a symbol 
of order $3/2$ 
and of size $O(\vare^2)$,
$\sym{b}{\geq 1}{\sla{1}{2}}$ is of order $1/2$ and of size $O(\vare)$.
An important properties is also that the symbols 
$\sym{a}{\geq 2}{\sla{3}{2}}, \sym{b}{\geq 1}{\sla{1}{2}}$ are (at least at positive order)
real valued. This is a consequence of  the \emph{Hamiltonian} structure of the original problem
that we preserve (at least approximatively) along the reduction procedure described above.
This property will be crucial in what follows.

We now recall again  that local theory implies that  the bound \eqref{ansatz local existence intro} 
holds true for a times scale $T\sim \vare^{-1}$.
Our goal is to extend this control up to the natural quadratic time scale 
$T\sim \varepsilon^{-2}$.

\medskip

To this end, the first attempt could be to introduce the quasilinear modified energy
\[
\mathcal{E}^{(s)}(u)
:= \Big\langle
\Big[\Lambda(D)
+ \Opbw{\sym{a}{\geq 2}{\sla{3}{2}}}
+ \Opbw{\sym{b}{\geq 1}{\sla{1}{2}}}
\Big]^{s \frac{4}{3}} u,\, u
\Big\rangle_{L^2}
\sim \|u\|_{H^s}^2.
\]
A direct computation yields the energy estimate
\begin{equation}\label{stimaenergiaINTRO}
\frac{d}{dt}\mathcal{E}^{(s)}(u)
\lesssim_s
\|u\|_{H^{s_0}}^2 \|u\|_{H^s}^2
+
\|u\|_{H^{s_0}} \|u\|_{H^{s-1}}\|u\|_{H^s}.
\end{equation}

The second term in \eqref{stimaenergiaINTRO} prevents an improvement of the lifespan beyond $T\sim \varepsilon^{-1}$.  
However, this obstruction is essentially a low-frequency phenomenon.

\medskip

Let $u_{\geq \tR}$ denote the high-frequency part of $u$, supported in $\{|\xi|\geq \tR\}$.  
Let $u_{\geq \tR}$ denote the projection of $u$ onto frequencies
$|\xi|\geq \tR$. When the cubic term in \eqref{stimaenergiaINTRO}
is restricted to this high-frequency component, it satisfies
\[
\|u\|_{H^{s_0}}
\|u_{\geq \tR}\|_{H^{s-1}}
\|u_{\geq \tR}\|_{H^s}
\lesssim
\tR^{-1}
\|u\|_{H^{s_0}}
\|u_{\geq \tR}\|_{H^s}^2.
\]
Restricting the modified energy analysis to frequencies $|\xi|\geq \tR$, one obtains the effective time scale
\begin{equation}\label{intro_epsR}
T \sim \min\{\varepsilon^{-2},\, \varepsilon^{-1}\tR\}.
\end{equation}
Thus, choosing $\tR\sim \varepsilon^{-1}$ yields the desired quadratic lifespan.

The difficulty is therefore to control the solution up to the sharp frequency threshold $|\xi|\lesssim \varepsilon^{-1}$.  
In \cite{IoP}, severe small divisor effects (see \eqref{stima max divisors} with $k=3$ and $\tau=(3/2)^+$) limited the analysis to the intermediate scale $|\xi|\lesssim \varepsilon^{-2/3+}$, which, together with \eqref{intro_epsR}, explains the lifespan $T\sim \varepsilon^{-5/3+}$ obtained there.

\medskip

To reach the optimal scale $\tR\sim \varepsilon^{-1}$, instead of performing a direct normal form or integration by parts argument, we develop a quasi-resonant normal form method based on a completely different paradigm.

The first main important observation to keep in mind is that for $|\xi|\lesssim \varepsilon^{-1}$, the quadratic quasilinear symbol $\sym{a}{\geq 2}{\sla{3}{2}}$ behaves like the lower-order symbol $\sym{b}{\geq 1}{\sla{1}{2}}$:
\begin{align*}
|\sym{b}{\geq 1}{\sla{1}{2}}(u; x, \xi)|
&\lesssim \varepsilon |\xi|^{\sla{1}{2}}\,, 
\\
|\sym{a}{\geq 2}{\sla{3}{2}}(u; x, \xi)|
&\lesssim \varepsilon^2 |\xi|^{\sla{3}{2}}
\lesssim \varepsilon |\xi|^{\sla{1}{2}}\,.
\end{align*}
Conversely, for $|\xi|\gtrsim \varepsilon^{-1}$, one infers
\begin{align*}
|\sym{b}{\geq 1}{\sla{1}{2}}(u; x, \xi)|
&\lesssim \varepsilon |\xi|^{\sla{1}{2}}
\lesssim \varepsilon^2 |\xi|^{\sla{3}{2}}\,, 
\\
|\sym{a}{\geq 2}{\sla{3}{2}}(u; x, \xi)|
&\lesssim \varepsilon^2 |\xi|^{\sla{3}{2}}\,.
\end{align*}

We therefore fix $\tR=\varepsilon^{-1}$ and split each symbol into low- and high-frequency components.  
Writing $\sym{a}{}{m} = \sym{a}{\leq \varepsilon^{-1}}{m} + \sym{a}{>\varepsilon^{-1}}{m}$ according to the support in $\{|\xi|\lesssim \varepsilon^{-1}\}$ and $\{|\xi|\gtrsim \varepsilon^{-1}\}$, equation \eqref{paradiff eq intro} can be rewritten as
\begin{equation*}
\partial_t u + \ii \Lambda(D) u
+ \ii \Opbw{a_{\varepsilon}(t,x,\xi)}u
+ \ii \Opbw{b_{\varepsilon}(t,x,\xi)}u
+ {\mathcal R}(u) = 0\,,
\end{equation*}
where
\[
\supp \big(a_{\varepsilon}(t,x,\xi) \big)\subset 
\{|\xi|\gtrsim \varepsilon^{-1}\}\,,
\qquad
\supp \big(b_{\varepsilon}(t,x,\xi) \big)\subset
\{|\xi|\lesssim \varepsilon^{-1}\}\,,
\]
and
\[
|a_{\varepsilon}(t,x,\xi)| \lesssim \varepsilon^2 
|\xi|^{\sla{3}{2}}\,,
\qquad
|b_{\varepsilon}(t,x,\xi)| \lesssim 
\varepsilon |\xi|^{\sla{1}{2}}\,.
\]

We perform a quasi-resonant normal form reduction on the
low-frequency symbol $b_{\varepsilon}$, while keeping the
high-frequency quasilinear symbol $a_{\varepsilon}$ unchanged.

At a conceptual level, and in the spirit of \cite{BD2018} for
one-dimensional quasilinear PDEs, the goal is to carry out a
normal form procedure in decreasing orders, eliminating the
paradifferential part of the equation up to a regularizing remainder.
Equivalently, this corresponds to a microlocal normal form in which
the effective small parameter is heuristically $|\xi|^{-1}$. This fact is actually not possible in higher dimension because of much stronger resonant phenomena. In the spirit of \cite{BML2019}, \cite{BML2}, \cite{BML3}, \cite{FM2022}, one normalizes the symbol $b_\varepsilon$ up to a resonant part and a smoothing remainder. Whereas it is not possible to normalize the symbol of maximal order $a_\varepsilon$, which is actually treated by constructing a modified energy, using that it is supported on high frequencies. This procedure is carried out in Section \ref{QR-NormalForm} and in particular in subsections \ref{sec:linsym}, \ref{sec:quadsym}, \ref{sec:cubicsym}. We now describe the first normalization step of the symbol $b_\varepsilon$ of order $\frac12$, which is supported on $|\xi| \lesssim \varepsilon^{- 1}$. 

We seek a nonlinear change of variables of the form
$v=\Phi^1(u)u$, where $\Phi^\tau(u)$ denotes the flow of the
linear, $u$-dependent, variable-coefficients equation
\[
\partial_\tau \Phi^\tau(u)
=
\ii \Opbw{g(u; x,\xi)}\, \Phi^\tau(u)\,,
\qquad
\Phi^0(u)=\uno\,,
\]
for a symbol $g$ of order $1 - \delta$, with $\delta > 0$ 
very close to $1$, $0 < 1 - \delta \ll 1$ and size $\varepsilon$,
to be determined. More precisely,
\[
|g(u; x,\xi)| \lesssim \varepsilon |\xi|^{1 - \delta}\,.
\]
\noindent
Writing the equation satisfied by $v$, one finds that it takes the form 
\begin{equation}\label{eq dopo primo step normal form}
\partial_t v + \ii \Lambda(D) v + \ii {\rm Op}^{BW}\big(b_1(t, x, \xi) \big)[v] + \ii {\rm Op}^{BW}\big(a_1(t, x, \xi) \big)[v] + {\mathcal R}_1(v) = 0
\end{equation}
where the transformed symbol of order $\frac12$ is given by 
\begin{equation}\label{homo_intro}
b_1(t, x, \xi)
=
b_\varepsilon(t,x,\xi)
+
\nabla_\xi \Lambda(\xi)\cdot \nabla_x g(t,x,\xi)\,,
\end{equation}
$a_1 (t, x, \xi)$ is a symbol of order $\frac32$, 
of size $\varepsilon^2$ and supported on 
$|\xi | \gtrsim \varepsilon^{- 1}$, i.e.
\begin{equation*}
\begin{aligned}
& |a_1(t, x, \xi)| \lesssim \varepsilon^2 |\xi|^{\frac32}\,, 
\qquad  
{\rm supp}\big( a_1 \big) 
\subset \big\{ |\xi| \gtrsim \varepsilon^{- 1} \big\}\,,
\end{aligned}
\end{equation*}
and ${\mathcal R}_1(v)$ is a smoothing quadratic remainder, 
\[
\| {\mathcal R}_1(v) \|_{H^{s + \vr}} \lesssim_s \| v \|_{H^s}^2\,, 
\qquad s, \vr \gg 0\,. 
\]

As we mentioned above, In contrast to the one-dimensional case, it is no longer possible
in higher dimensions to choose the generator $g$ so that
$b_\varepsilon^{(1)}$ becomes independent of $x$.
Indeed, the operator $\nabla_\xi \Lambda(\xi)\cdot \nabla_x \sim |\xi|^{-\sla{1}{2}} \xi \cdot \nabla_x$
is not invertible and, more importantly,
is affected by severe small divisor phenomena,
which prevent a complete elimination of the $x$-dependence of the symbol. Indeed, one is led to the relation
\[
b_{\varepsilon}
=
\sum_{k\in\Z^2} \widehat b(k,\xi)\, e^{\ii k\cdot x},
\qquad
g
=
- (\nabla_\xi \Lambda(\xi)\cdot \nabla_x)^{-1}
b_{\varepsilon}\,.
\]
At the Fourier level this yields formally
\[
g(t,x,\xi)
\sim
\sum_{k\in\Z^2}
\frac{\widehat b(k,\xi)}
{\ii (k\cdot \xi)\, |\xi|^{-\sla{1}{2}}}
\, e^{\ii k\cdot x}\,.
\]

The denominator
$\ii (k\cdot \xi)\, |\xi|^{-\sla{1}{2}}$
may vanish when $k\cdot \xi=0$ and it can be very small for many $\xi \in \R^2$. In the one-dimensional case,
however, $k\cdot \xi = k\xi$, so that for $k\neq0$ the denominator
is of size $\sim |\xi|^{\sla{1}{2}}$ and therefore large at high
frequencies. In higher dimensions the situation is drastically
different: the quantity $(k\cdot \xi)|\xi|^{-\sla{1}{2}}$ can be arbitrarily small along
large sets of frequencies, producing severe small divisors that
accumulate to zero very fast.
Then, in order to analyze the symbol in \eqref{homo_intro}, we introduce
two parameters
\[
\delta,\nu\in (0,1)\,, \qquad 0 < 1-\delta \ll 1\,, 
0 < \quad \nu \ll 1\,,
\]
and decompose the symbol $b_{\varepsilon}$ into four components
\[
b_{\varepsilon}
=
\langle b_{\varepsilon} \rangle
+ b_{\varepsilon}^{(\mathtt{res})}
+ b_{\varepsilon}^{(\mathtt{nr})}
+ b_{\varepsilon}^{(\mathtt{S})}.
\]
For some $\tau \gg 1$ the components are defined as follows:
\begin{itemize}

\item $\langle b_{\varepsilon} \rangle$ is the spatial average of the symbol,
\[
\langle b_{\varepsilon} \rangle
:=
\widehat b(0,\xi)
=
\frac{1}{(2\pi)^2}
\int_{\T^2} b_{\varepsilon}(x,\xi)\,dx;
\]

\item $b_{\varepsilon}^{(\mathtt{res})}$ is the \emph{resonant} component and satisfies
\[
\supp\!\left(
\widehat{b_{\varepsilon}^{(\mathtt{res})}}(k,\xi)
\right)
\subseteq
\Big\{
(k,\xi)\in\Z^2\times\R^2 :
0 < |k|\le |\xi|^\nu,\;
|k\cdot\xi|\lesssim |k|^{-\tau}|\xi|^\delta
\Big\};
\]

\item $b_{\varepsilon}^{(\mathtt{nr})}$ is the \emph{non–resonant} component and satisfies
\begin{equation}\label{supporto_nr}
\supp\!\left(
\widehat{b_{\varepsilon}^{(\mathtt{nr})}}(k,\xi)
\right)
\subseteq
\Big\{
(k,\xi)\in\Z^2\times\R^2 :
0 < |k|\le |\xi|^\nu,\;
|k\cdot\xi|\gtrsim |k|^{-\tau}|\xi|^\delta
\Big\};
\end{equation}

\item $b_{\varepsilon}^{(\mathtt{S})}$ is the \emph{smoothing} component and satisfies
\[
\supp\!\left(
\widehat{b_{\varepsilon}^{(\mathtt{S})}}(k,\xi)
\right)
\subseteq
\Big\{
(k,\xi)\in\Z^2\times\R^2 :
|k|\ge |\xi|^\nu
\Big\}.
\]
\end{itemize}
We refer to subsection \ref{preliminary resonant normal form}, for details on the latter splitting. 

We retain the average and the resonant part in the normal form,
\[
\langle b_{\varepsilon} \rangle
+
b_{\varepsilon}^{(\mathtt{res})},
\]
since these terms encode the genuine resonant interactions and the
small–divisor structure of the three–dimensional water waves problem.

The smoothing component produces a smoothing operator,
\[
\Opbw{b_{\varepsilon}^{(\mathtt{S})}}
= \text{smoothing operator},
\]
which can therefore be absorbed into the remainder $\cR_1(u)$ in
\eqref{eq dopo primo step normal form}.

Finally, we remove the non–resonant component by exploiting its spectral
support, which provides a lower bound on the denominator.
This leads to the generator
\[
g
=
-(\nabla_\xi\Lambda(\xi)\cdot\nabla_x)^{-1}
b_{\varepsilon}^{(\mathtt{nr})}
\sim
\sum_{k\in\Z^2}
\frac{\widehat{b_{\varepsilon}^{(\mathtt{nr})}}(k,\xi)}
{\ii (k\cdot\xi)\,|\xi|^{-\sla{1}{2}}}
\,e^{\ii k\cdot x}.
\]

Using the lower bound \eqref{supporto_nr} on the support of
$b_{\varepsilon}^{(\mathtt{nr})}$, and recalling that and supported on $|\xi| \lesssim \varepsilon^{- 1}$
$|b_{\varepsilon}|\lesssim \varepsilon |\xi|^{\sla{1}{2}}$, we obtain that $g$ is a symbol of order $1 - \delta$, of size $\varepsilon$ and supported on $|\xi| \lesssim \varepsilon^{- 1}$, namely
\[
|g(t, x, \xi)|
\lesssim
\varepsilon |\xi|^{1-\delta}
\ll
\varepsilon |\xi|^{\sla{1}{2}}, \quad {\rm supp}(g) \subseteq \big\{ |\xi| \lesssim \varepsilon^{- 1} \big\}
\]
since $\delta$ is very close to $1$ (in particular $\delta>1/2$).

As a consequence, the new symbol in \eqref{homo_intro} becomes
\[
b_1 
=
\langle b_{\varepsilon} \rangle
+
b_{\varepsilon}^{(\mathtt{res})}\,. 
\]
the same procedure can be iterated. After an arbitrary number of this iterative steps, one deals with a para-differential equation of the form 
\begin{equation}\label{eq dopo normal form resonant}
\partial_t v + \ii \Lambda(D) v + \ii {\rm Op}^{BW}\big(z(t, x, \xi) \big)[v] + \ii {\rm Op}^{BW}\big(c(t, x, \xi) \big)[v] + {\mathcal Q}(v) = 0
\end{equation}
where 
\begin{itemize}
\item $z(t, x, \xi)$ is a symbol of order $\frac12$, of size $\varepsilon$, supported on $|\xi| \lesssim  \varepsilon^{- 1}$, namely 
$$
|z(t, x, \xi)| \lesssim \varepsilon |\xi|^{\frac12}, \quad {\rm supp}(z) \subseteq \big\{ |\xi| \lesssim \varepsilon^{- 1} \big\}\,
$$
and it is a symbol in {\bf normal form}, namely 
\[
\widehat z(t, k, \xi) \neq 0 \quad \Longrightarrow \quad 
|k| \leq |\xi|^\nu\,, 
\quad | \xi \cdot k| \leq |\xi|^\delta |k|^{- \tau}
\]
for some $\tau \gg 0$ large enough (see definition \ref{def:normalform}). 
\item The symbol $c$ is of order $\frac32$, of size $\varepsilon^2$ and supported on $|\xi| \gtrsim \varepsilon^{- 1}$, i.e. 
$$
\begin{aligned}
|c(t, x, \xi)| \lesssim \varepsilon^2 |\xi|^{\frac32}, \quad {\rm supp}(c) \subseteq \big\{ |\xi| \gtrsim \varepsilon^{- 1} \big\}\,. 
\end{aligned}
$$
\item ${\mathcal Q}(v)$ is a quadratic smoothing remainder, i.e. 
$$
\| {\mathcal Q}(v) \|_{H^{s + \rho}} \lesssim_s \| v \|_{H^s}^2, \quad s, \rho \gg 0\,. 
$$
\end{itemize}
Finally after a standard Birkhoff normal form step on the quadratic smoothing remainder ${\mathcal Q}(v)$ (see subsection \ref{sezione birkhoff smoothing}), using that for full measure set of surface tension parameters $\kappa$, the three waves interactions satisfy the lower bound 
$$
\begin{aligned}
&|\Lambda (\xi_1) \pm \Lambda(\xi_2) \pm \Lambda(\xi_3)| \geq \frac{\gamma}{\langle \xi_1 \rangle^\tau \langle \xi_2 \rangle^\tau}\,, \quad  \xi_1 \pm \xi_2 \pm \xi_3 = 0,
\end{aligned}
$$
by choosing the order of regularization $\rho$ as $\rho \gg \tau$, one reduces the PDE \eqref{eq dopo normal form resonant} to 
\begin{equation}\label{eq prima della modified energy}
\partial_t v + \ii \Lambda(D) v + \ii {\rm Op}^{BW}\big(z(t, x, \xi) \big)[v] + \ii {\rm Op}^{BW}\big(c(t, x, \xi) \big)[v] + {\mathcal C}(v) = 0
\end{equation}
where ${\mathcal C}(v)$ is a cubic remainder satisfying the estimate
$$
\| {\mathcal C}(v) \|_{H^s} \lesssim_s \| v \|_{H^s}^3\,. 
$$
The last part of the proof consist in the construction of a modified energy which is taylored to the resonant structure of $z$ and to the fact that $c$ is a symbol of order $\frac32$, supported on $|\xi| \gtrsim \varepsilon^{- 1}$ and of size $O(\varepsilon^2)$. The construction of the modified energy to get the desired energy estimate is performed in Section \ref{sez-energy-estimate}. First of all, we split $v$ as 
        \begin{equation}\label{splitting-alti-bassi-intro}
        \begin{aligned}
        & v = v_\varepsilon + v_\varepsilon^\bot \quad \text{where} \\
        & v_\varepsilon (x) := \sum_{|\xi| \leq \varepsilon^{- 1}}  v_\xi e^{\ii x \cdot \xi}, \quad v_\varepsilon^\bot (x) := \sum_{|\xi| > \varepsilon^{- 1}}  v_\xi e^{\ii x \cdot \xi}
        \end{aligned}
        \end{equation}
        and we perform separate energy estimates for the low modes $v_\varepsilon$ and for the high modes $v_\varepsilon^\bot$. 
We estimate the low modes $v_\varepsilon$ in Section \ref{sec:low_est}. The key point is the following partition of the resonant frequencies. From Corollary 5.11 of \cite{BML2} and Lemma 4.6 in \cite{BML3}, we have the following geometric block decomposition: there is a partition $\cP= \{ \Omega_\alpha\}_{\alpha \in \N_0}$ of 
       $\Z^2$ such that
        \begin{enumerate}
            \item {\bfseries Dyadic blocks:} There 
            is a constant $R=R(\delta)>0$ such that
            \begin{align}
                \Omega_0 \subset B_{\Z^2}(R)\,, 
                \qquad  
                \max_{\xi \in \Omega_\alpha} |\xi|\leq 2 \min_{\xi \in \Omega_\alpha}|\xi|\,, 
                \qquad \text{for any }\alpha \in \N\,;
                \label{size:block-intro}
            \end{align}
            \item {\bfseries Invariance of normal form operators:} Let $\cZ(x,\xi)$ be a normal form symbol, 
            then 
            \begin{align}
            \Pi_{\Omega_{\alpha}} \Opbw{\cZ(x,\xi)}= 
            \Opbw{\cZ(x,\xi)}\Pi_{\Omega_{\alpha}}\,, \quad \forall \alpha \in \mathbb N_0 
            \label{block_invariant-intro}
            \end{align}
        \end{enumerate}
        where for any subset $\Omega \subseteq \Z^2$, we define 
        $$
\Pi_\Omega u (x) := \sum_{\xi \in \Omega}  u_\xi e^{\ii x \cdot \xi}\,.
        $$
Since our normal form symbols are supported for frequencies $|\xi| \lesssim \varepsilon^{- 1},$ we also define the set of indices corresponding to 
blocks that intersect the ball $ B_{\mathbb{Z}^2}\big(1/\varepsilon\big)$ 
as
\begin{align}\label{I_vare-intro}
\mathbb{I}_\vare := \left\{ \alpha \in \mathbb{N}_0 \,\middle|\, 
\Omega_\alpha \cap B_{\mathbb{Z}^2}\big( {\varepsilon^{-1}} \big) 
\neq \emptyset \right\}\,.
\end{align}
We show that for $\varepsilon \ll 1$ small enough, one has that  
\begin{align}\label{inclu-intro}
            \Omega_\alpha \subset B_{\mathbb{Z}^2}\big(2\varepsilon^{-1}\big)\,, \quad \forall \alpha \in \mathbb{I}_\vare\,. 
        \end{align}
        
        Note that  the Sobolev norm of $v_\varepsilon$, 
        thanks to the diadic property \eqref{size:block-intro},
        is given by  
        $$
        \| v_\varepsilon \|_{H^s} \sim \tE^{(s)}(v):= \sum_{\alpha\in \mathbb{I}_\vare} \tM_\alpha^{2s} 
	\| \Pi_{\Omega_\alpha}v\|_{L^2}^2\,, 
	\qquad 
	\tM_{\alpha}:= \max_{\xi \in \Omega_\alpha} |\xi|\,. 
        $$
        Furthermore, since the symbol $z(t, x, \xi)$ is 
        in normal form, 
         the property \eqref{block_invariant-intro} implies 
         that the operator ${\rm Op}^{BW}(z(t, x, \xi))$ 
         is invariant on the blocks 
         $\Omega_\alpha$. 
         Using these properties, one is able 
         to prove the energy estimate
        \begin{equation}\label{stima energia modi bassi}
\Big|\frac{d}{d t} \tE^{(s)}(v) \Big| 
\lesssim_s \varepsilon^4\,, 
\quad \forall t \in [0, T], 
        \end{equation}
        see Lemma \ref{lem:est_dt_E_low}. 
        Therefore, for any $t \in [0, T]$, 
        $$
        \begin{aligned}
        \| v_\varepsilon(t) \|_{H^s}^2 
        & \sim \tE^{(s)}(v(t)) 
        \lesssim_s \tE^{(s)}(v(0)) + T \varepsilon^4 
        \lesssim_s    
        \| v_\varepsilon(0) \|_{H^s}^2 + T \varepsilon^4 
        \\& 
        \lesssim_s \varepsilon^2 (1 + T \varepsilon^2) 
        \lesssim_s \varepsilon^2\,.
        \end{aligned}
        $$
For $T \sim \varepsilon^{- 2}$ (see Proposition \ref{prop:stima_low}). Then one is led to estimate 
the high modes $v_\varepsilon^\bot$. This is done 
in Section \ref{sec:high_est}. 
We define a modified energy ${\mathcal E}^{(s)}(v)$, 
whose precise definition is given 
in \eqref{def:mod_ene}. Basically, according 
to the equation 
\eqref{eq prima della modified energy}, we define 
$$
L(t, x,\xi) := \Lambda(\xi) + z(t, x, \xi) + c(t, x, \xi)
$$
and 
$$
{\mathcal E}^{(s)} (v) := \Big\langle {\rm Op}^{BW}\Big(L(t, x,\xi)^{\frac43 s}\chi_\varepsilon(t, x,\xi) \Big) v\,,\, v\Big\rangle_{L^2}
$$
where $\chi_\varepsilon$ is a suitable cut off symbol (it is chosen as a function of $L$, i.e. $\chi_\varepsilon \equiv \chi_\varepsilon(L)$) that is identically zero for $|\xi| \lesssim \varepsilon^{- 1}$. With this choice, since $z(t, x, \xi) = 0$ for $|\xi| \lesssim \varepsilon^{- 1}$, essentially one has 
$$
\begin{aligned}
L(t, x,\xi)^{\frac43 s}\chi_\varepsilon(t, x,\xi) & \sim \Big( \Lambda(\xi) + c(t, x, \xi) \Big)^{\frac43 s} \sim |\xi|^{2s}\big( 1 + r_L(t, x, \xi) \big)\,, \\
|r_L(t, x, \xi)| & \lesssim \varepsilon^2, \quad {\rm supp}(r_L) \subset \big\{|\xi| \gtrsim \varepsilon^{- 1} \big\}\,. 
\end{aligned}
$$
Hence, one has that 
$$
{\mathcal E}^{(s)}(v) \sim \| v_\varepsilon^\bot \|_{H^s}^2\,. 
$$
Then in Lemma \ref{lem:est_dt_E}, it is proved that 
$$
\Big| \frac{d}{d t} {\mathcal E}^{(s)}(v) \Big| \lesssim_s \varepsilon^4 , \quad \forall t \in [0, T]\,.
$$
This is the crucial ingredient to prove Proposition \ref{prop:stima_high}, in which it is proved that 
$$
\begin{aligned}
\| v_\varepsilon^\bot(t) \|_{H^s}^2 & \sim {\mathcal E}^{(s)}(v(t)) \lesssim_s {\mathcal E}^{(s)}(v(0)) + \int_0^t \Big| \frac{d}{d \tau}{\mathcal E}^{(s)}(v(\tau)) \Big| \, d \tau  \\
& \lesssim_s \| v(0) \|_{H^s}^2 + T \varepsilon^4 \lesssim_s   \varepsilon^2 (1 + T \varepsilon^2) \lesssim_s \varepsilon^2
\end{aligned}
$$
for times $T \sim \varepsilon^{- 2}$. The proof is then concluded with a standard bootstrap argument, see Section \ref{sezione bootstrap argument}.

  \bigskip

\noindent {\bf Acknowledgements.} R. Montalto 
and F. Murgante
are supported by the ERC STARTING GRANT 2021 
`Hamiltonian Dynamics, 
Normal Forms and Water Waves'' (HamDyWWa), Project Number: 101039762. 
Views and opinions expressed are however those of the authors only and do not necessarily reflect those of the European Union or the European Research Council. Neither the European Union nor the granting authority can be held responsible for them.

R. Feola is also supported 
by ``GNAMPA - INdAM'', CUP E53C25002010001,
and 
``GNAMPA - INdAM'', CUP E5324001950001.

\medskip

\noindent
The authors warmly thank Alberto Enciso, Alberto Maspero and Fabio Pusateri for very useful discussions and comments.


\section{ Functional Setting}\label{sezione-functional setting}
In this section we introduce the functional framework and the paradifferential calculus tools used throughout the paper. 
We first fix the Sobolev setting and basic notation, then define the classes of symbols and associated operators, and finally introduce the multilinear symbol classes and smoothing remainders that will be used in the normal form procedure.

Any function $u(x)$, $x\in \mathbb{T}^2$, is expanded in Fourier series as\footnote{We also use the notation
$u_n^+ := u_n := \widehat{u}(n)$ and
$u_n^- := \ov{u_n} := \ov{\widehat{u}(n)}$.}
\begin{equation}\label{def:Fou_esp}
u(x) = 
\sum_{n \in \Z^2 } u_{n}e^{\ii n\cdot x }, 
\qquad 
u_{n} := \frac{1}{(2\pi)^2} \int_{\mathbb{T}^2} u(x) e^{-\ii n \cdot x } \, dx.
\end{equation}

For any subset $\Theta \subset \Z^2$, we denote
\begin{align}
\Pi_\Theta u := \sum_{n \in \Theta } u_{n}e^{\ii n\cdot x },
\qquad 
\Pi_\Theta^\bot u := \Pi_{\Theta^c} u.
\label{def:PiLambda}
\end{align}
For any $s\in \R$, we denote by $H^{s}(\mathbb{T}^2;\mathbb{C})$ the usual $L^{2}$-based 
Sobolev space 
on $\mathbb{T}^2$.
We set $\langle j \rangle:=\sqrt{1+|j|^{2}}$ f
or $j\in \mathbb{Z}^2$, and we equip 
$H^{s}(\mathbb{T}^2;\mathbb{C})$ with the norm
\begin{equation}\label{Sobnorm}
\|u\|_{s}^{2}
:=
\sum_{j\in \mathbb{Z}^2}\langle j\rangle^{2s}|u_{j}|^{2}\,.
\end{equation}

\noindent
We also introduce the translation operator
\begin{align}
(\tau_\upsilon u)(x):= u(x+ \upsilon),
\qquad 
\upsilon \in \R^2\,.
\label{tautras}
\end{align}

\noindent
For $r>0$ and $s\in \R$, we denote by
\[
B_{s}(r):=\{u\in H^{s}(\T^{2};\C)\,:\, \|u\|_{s}\leq r\}
\]
the ball of $H^s$ with radius $r$ centered at the origin.

\medskip

It is convenient to work on the product space $H^{s}(\T^2;\C)\times H^{s}(\T^2;\C)$, endowed with the natural product norm. In particular, we will use the real subspace
\begin{equation}\label{Hcic}
H^s_\R
:=
\{(u^{+},u^{-})\in H^{s}(\T^2;\C)\times H^{s}(\T^2;\C)\; : \; 
u^{-}=\overline{u^{+}}\}\,.
\end{equation}

\noindent
We denote by $B_{s,\R}(r)$ the ball of $H_{\R}^s$ 
with radius $r$ centered at the origin.

\begin{remark}
 We recall that the norm \eqref{Sobnorm} satisfies the following ``tame'' estimate:
 for any $s\geq {s}_0>1$ one has
 \begin{equation}\label{tameHsx}
 \|ab\|_{s}\lesssim_{s}\|a\|_{s}\|b\|_{s_0}+\|a\|_{{s}_0}\|b\|_{s}\,,\qquad \forall\, a,b\in H^{s}\,.
 \end{equation}
 \end{remark}
  \noindent

	\subsection{Paradifferential calculus}
In this section we introduce the basic tools from paradifferential calculus that will be used throughout the paper. We follow a standard paradifferential approach based on Bony–Weyl quantization.

		\paragraph{Symbols and quantization}
	We define the class of symbols with \emph{finite} regularity we shall use throughout the paper.
	\begin{definition}{\bfseries (H\"older symbols).}\label{def:simbolohold}
			Let $ m  \in \R $, $\delta\in [0,1]$ and $\sigma\geq 0$.
	We denote by $\cN^{m,\delta}_{\sigma}$ the space of 
			functions   $a(x, \xi): \T^2\times \R^2\to \C $,  
			such that, for any multi-indices $ \alpha ,\beta \in \N_0^2 $ with ${|\alpha| + |\beta| \leq \sigma}$, 
			there exists a constant $C_{\alpha,\beta} >0$ such that
			\begin{equation}\label{stima:symbols}
			  |\partial_x^{\alpha} \partial_\xi^{\beta} a( x, \xi) 
				|\leq C_{\alpha,\beta} \langle \xi \rangle^{ m - \delta |\beta|}\,, \quad \forall (x,\xi) \in \T^2\times\R^2 
			\end{equation}
			We endow $\cN^{m,\delta}_{\sigma}$ with the  family of   semi-norms 
			\begin{equation}\label{seminormSimbo}
				|a|_{m, \sigma} := \max_{|\alpha| + |\beta| \leq \sigma} 
				\sup_{ \xi\in \R^2} \| \partial_x^{\alpha} \partial_\xi^{\beta} a(\cdot , \xi) \|_{L^\infty(\T^2)}
				\langle \xi \rangle^{- m + \delta |\beta|}\,.
			\end{equation} 
				We also denote $\mathcal{N}_{\s}^{m}\equiv\mathcal{N}_{\s}^{m,1}$.
			\end{definition}
		
\begin{remark}
    We consider symbols with finite regularity in $\xi \in \mathbb{R}^2$. This unusual class is intrinsic to the quasi-resonant normal form procedure. Indeed, the solution of the homological equation (see, e.g., \eqref{eq:homologica}) has the property that taking derivatives with respect to $\xi$ also entails a loss of regularity in $x$.  
\end{remark}

    \vspace{0.5em}
	\noindent
	{\bfseries Paradifferential quantization.}
	Given a smooth symbol  $ a (x, \xi) $ 
	we define its Weyl quantization  as the operator
	acting on a
	$ 2 \pi $-periodic function
	$u(x)$ (written as in \eqref{def:Fou_esp})
	as
	\[
	\Opw{a}u= \sum_{\ell \in \Z^2}
	\Big(\sum_{j \in\Z^2}\hat{a}\big(\ell-j, \tfrac{\ell+j}{2}\big) u_j \Big){e^{\im \ell x}}\,,
	\]
	where $\hat{a}(k,\xi)$ is the $k$-{th} Fourier coefficient of the $2\pi-$periodic function $x\mapsto a(x,\xi)$.
	
	\noindent
	In order to introduce its associated  \emph{para-differential} operator
	we consider a smooth cut-off function
	$\chi\in C^{\infty}(\R^2\times\R^2;\R)$, 
	even with respect to each of its arguments, satisfying, for $0<\delta_0\leq \tfrac{1}{10}$,
	\begin{equation}
		{\supp}\, \chi \subset\{(\xi',\xi)\in\R^{2}\times\R^{2}; 
        \ |\xi'|\leq\delta_0 \langle\xi\rangle\} \,,
        \qquad \quad
		\chi(\xi',\xi) \equiv 1\,\,\, \mathrm{ for } \,\,\, |\xi'|\leq \delta_0   \langle\xi\rangle / 2 \, . 
        \label{cut off defin}
	\end{equation}
	Moreover, we assume that 
\[
|\partial_{\xi}^{\alpha}\partial_{\xi'}^{\beta}\chi(\xi',\xi)|
\leq C_{\alpha,\beta}\langle\xi\rangle^{-|\alpha|-|\beta|}\,, 
\ \forall \alpha , \,\beta\in\N_0^2 \, .
\]

	\begin{definition}{\bfseries (Bony-Weyl quantization)}\label{quantizationtotale}
		Given a symbol $a(x,\xi)\in \cN_{s_0}^{m,\delta}$
		we set
\begin{equation*}
a_{\chi}(x,\xi) :=\sum_{k\in \Z^2} 
\chi (k,\xi )\hat{a}(k,\xi)e^{\im k \cdot  x}\,, 
\qquad 
\hat{a}(k,\xi):= \frac{1}{(2\pi)^2} 
\int_{\T^2} {a}(x,\xi)e^{-\ii k x }\, \di x  \, 
\end{equation*}
and we define the \emph{Bony-Weyl} quantization  as 
\begin{align}
&\Opbw{a(x,\xi)}u= \Opw{ a_{\chi}(x,\xi)}u=
	\!\!\!\!\!\! 
    \sum_{(j,k)\in \Z^2} 
    \!\!\!  
    \chi \left(j-k,\frac{j+k}{2} \right)
\hat{a}\left( j-k, \frac{j+k}{2}\right)
u_k {e^{\im j x}} \,.
\label{BWnon}
		\end{align}
	\end{definition}
\begin{remark} \label{rem:symbols_modozero}

\noindent $\bullet$
Paradifferential operators map zero average functions into zero average functions, indeed 
 
\begin{align}
    \Pi_0 \cOpbw{a}= \cOpbw{a}\Pi_0= \hat a(0,0) \Pi_0 , \qquad \Pi_0^\bot \cOpbw{a}= \cOpbw{a}\Pi_0^\bot\,,
    \label{comm_copbw}
\end{align}
where $\Pi_0$ is the $L^2$-projector on the zero-th mode and $\Pi_0^{\perp}:={\id}-\Pi_0$.

\noindent $\bullet$ The action of
$ \Opbw{a} $ on  the spaces $ \dot H^s $ only depends
on the values of the symbol $  a(x,\xi)$
for $|\xi|\geqslant1$.
Therefore, we may identify two symbols $ a( x,\xi)$ and
$ b( x,\xi)$ if they agree for $|\xi| \geqslant\sla{1}{2}$.
In particular, whenever we encounter a symbol that is not smooth at $\xi=0 $,
such as, for example, $a = g(x)|\x|^{m}$ for $m\in \R^*$,
we will identify it with its smoothed out version
$\pare{1-\chi(\xi)}a\pare{x, \xi}$, where $\chi$ is a smooth cut off function such that 
\begin{align*}
\chi(\xi)=
    \begin{cases}
        1 & |\xi| \leq \frac14\\
        0 & |\xi| \geq \frac12.
    \end{cases}
\end{align*}
In this case we shall, with a slight abuse of notation, indicate 
\[
|a(x,\xi)|_{m,\s}\equiv | (1-\chi(\xi)) a(x,\xi)|_{m,\s}.
\]
This convention will be used systematically without further mention.

\noindent{$\bullet$} We choose the Weyl quantization as one easily verifies the fundamental algebraic property:
\begin{align*}
\Opbw{a(x,\xi)}^*= \Opbw{\ov{a(x,\xi)}}.
\end{align*}

\end{remark}

	The following classical result about the action of paradifferential operators on the scale of Sobolev spaces follows by standard arguments (see, for instance, \cite{BMM2021,FI2022}, Lemma 2.3 and Theorem 2.4 respectively). 
	\begin{lemma}{\bfseries (Action of a paradifferential operator).}\label{thm:action}
		Let $m \in \R$, $\delta\in [0,1]$ and $s_0 > 2$. Then for any $s \in \R $, the linear map 
		\[
		{\cal N}^{m,\delta}_{s_0} \to {\cal L}(H^{s }, H^{s-m}), \quad a \mapsto \Opbw{a}
		\]
		is continuous, namely 
		\begin{align}
		    \| \Opbw{a} v\|_{H^{s - m}} \lesssim_s |a|_{m, s_0}\| v\|_{H^s}, \quad \text{for any } \ v \in H^s\,. 
            \label{actionSob}
		\end{align}

	\end{lemma}
		\paragraph{Compositions}\label{sez calcolo simbolico}
	In this section we provide some abstract lemmas on the classes 
	that we defined before that we shall apply in our normal form procedure. 
	We introduce the following differential operator
	\begin{equation}\label{sigmino}
	\s(D_{x},D_{\x},D_{y},D_{\eta}) := D_{\x}\cdot D_{y}-D_{x}\cdot D_{\eta}\,, 
	\end{equation}
	where $D_{x}:=\frac{1}{\ii}\nabla_{x}$ and $D_{\x},D_{y},D_{\eta}$ are similarly defined. 
    We introduce a truncated symbolic expansion of the composition, which is sufficient for our purposes.
	\begin{definition}{\bfseries (Asymptotic expansion of composition symbol).}
	Let $\vr\in \mathbb{N}$, $m_1,m_2\in \R$, $s_0 \geq\vr$ and 
    $a\in \cN^{m_1,\delta}_{s_0}$, 
	$b\in \cN^{m_2,\delta}_{s_0}$.
		We define the symbol
		\begin{align}
			(a\sha{\vr} b)(x,\x):=&\sum_{k=0}^{\vr-1}\frac{1}{k!}
			\left(
			\frac{\ii}{2}\s(D_{x},D_{\x},D_{y},D_{\eta})\right)^{k}
			\Big[a(x,\x)b(y,\eta)\Big]_{|_{\substack{x=y, \x=\eta}}}\, .
\label{espansione2}
		\end{align}
	\end{definition}
\begin{remark}\label{espansEsplic}
We collect some useful properties of the composition defined in \eqref{espansione2}.

\medskip
\noindent $\bullet$ {\bfseries Explicit expansion.}
In view of \eqref{sigmino}--\eqref{espansione2}, we have
\begin{equation}\label{def:pk}
(a\sha{\vr} b)(x,\x)
=
\sum_{n=0}^{\varrho-1} p_n(a,b),
\end{equation}
where
\[
p_n(a,b)
:=
\frac{1}{(2\ii)^n}
\sum_{|\alpha| + |\beta| = n}
\frac{(-1)^{|\beta|}}{\alpha!\beta!}
(\partial_x^\beta \partial_\xi^\alpha a)
(\partial_x^\alpha \partial_\xi^\beta b).
\]
Moreover, recalling \eqref{seminormSimbo}, one has
\begin{equation}\label{est:pk}
p_n(a,b) \in \cN_{s_0-n}^{m_1+m_2-\delta n}\,, 
\qquad
|p_n(a,b)|_{m_1+m_2-\delta n,\, s_0-n}
\lesssim_n |a|_{m_1,s_0} |b|_{m_2,s_0}\,.
\end{equation}
In particular,
\[
a\sha{\vr} b
=
ab + \frac{1}{2\ii}\{a,b\}
+ \cN_{s_0-\vr}^{m_1+m_2-2\delta},
\]
and
\begin{align*}
|a\sha{\vr} b - ab|_{m_1+m_2-\delta,\, s_0-\vr}
&\lesssim |a|_{m_1,s_0}|b|_{m_2,s_0}\,,
\\
\left| a\sha{\vr} b - ab - \frac{1}{2\ii}\{a,b\}
\right|_{m_1+m_2-2\delta,\, s_0-\vr}
&\lesssim |a|_{m_1,s_0}|b|_{m_2,s_0}\,.
\end{align*}

\medskip
\noindent $\bullet$ {\bfseries Algebraic properties.}
A direct computation shows that
\begin{align}
\ov{a\sha{\vr} b}
=
\ov{b}\sha{\vr} \ov{a}, 
\qquad 
\big(a\sha{\vr} b\big)^\vee
=
b^\vee \sha{\vr} a^\vee,
\label{prop:ov}
\end{align}
where $a^\vee(x,\xi):=a(x,-\xi)$. In particular,
\[
\overline{a^\vee}\sha{\vr} \overline{b^\vee}
=
\overline{a\sha{\vr} b}^\vee\,.
\]

\medskip
\noindent $\bullet$ {\bfseries Composition with Fourier multipliers.}
Let $\Lambda(\x):=\sqrt{|\x|(1+\kap|\x|^2)}\in \tilde \cN_{s_0}^{\sla{3}{2}, 1}$. Then
\[
p_k(a,\Lambda)\in \cN_{s_0-k}^{m_1+\sla{3}{2}-k, \delta}\,,
\qquad
|p_k(a,\Lambda)|_{m_1+\sla{3}{2}-k,\, s_0-k}
\lesssim |a|_{m_1,s_0}\,.
\]

\medskip
\noindent $\bullet$ {\bfseries Moyal bracket.}
For $\vr\geq 2$, we introduce the Moyal bracket, which corresponds to the commutator at the symbolic level,
\begin{align}\label{Moyal}
\mbra[\vr]{a}{b}
:=
\ii\big(a\sha{\vr} b - b\sha{\vr} a\big)
=
\sum_{\substack{k=1,\ldots,\vr-1 \\ k \ \mathrm{odd}}}
2\ii\, p_k(a,b)
=
\{a,b\} + \sym{r}{}{m_1+m_2-3\delta}\,,
\end{align}
where 
$\sym{r}{}{m_1+m_2-3\delta}\in \cN_{s_0-\vr}^{m_1+m_2-3\delta}$ and
\begin{align}
|\sym{r}{}{m_1+m_2-3\delta}|_{m_1+m_2-3\delta,\, s_0-\vr}
\lesssim |a|_{m_1,s_0}|b|_{m_2,s_0}\,.
\label{est:Moyal}
\end{align}
\medskip
\noindent $\bullet$ {\bfseries Iterated commutators.}
Let $p\in \N_0$. For any $\s \geq s_0$ and $b \in \cN_{\s+p \vr}^{m_1}$, define
\begin{align}
a \mapsto \mathtt{ad}_{\! \vr}(b)[a]
:=
\frac{1}{\ii} \mbra[\vr]{a}{b}\,,
\qquad
\mathtt{ad}^p_{\! \vr}(b)[a]
:=
\overbrace{\mathtt{ad}_{\! \vr}(b)\circ \ldots \circ \mathtt{ad}_{\! \vr}(b)}^{p\text{-times}}[a]\,.
\label{def:ad_sym}
\end{align}
Then
\[
\mathtt{ad}_{\! \vr}^p(b)\colon \cN_{\s+p\vr}^{m_2} \to \cN_{\s}^{m_2+p(m_1-\delta)},
\]
with estimate
\[
|\mathtt{ad}_{\! \vr}^p(b)[a]|_{m_2+p(m_1-\delta),\, \s}
\lesssim_\s |b|_{m_1,\s+\vr} |a|_{m_2,\s+\vr}\,.
\]
\end{remark}

	The following is a classical result about the composition of paradifferential operators, for a proof in the same setting as the present paper see \cite{FM2022}.
		
\begin{lemma}{\bfseries (Composition).}\label{compoparapara0}
Let $\delta \in (0,1)$, $\vr\geq 0$ and $m_1,m_2\in \R$. 
There is $s_0>\vr$ such that for any  
$\ta \in \cN_{s_0}^{m_1,\delta}$, $\tb\in \cN_{s_0}^{m_2,\delta}$, 
one has
\begin{align}
\cOpbw{\ta}\circ \cOpbw{\tb} = 
\cOpbw{\ta\sha{\vr}\tb}+ \mathcal{Q}(\ta,\tb)[\cdot]\,,
\label{espansionecompo}
\end{align}
where the bilinear map 
\begin{align}
\mathcal{Q}: \cN_{s_0}^{m_1}\times \cN_{s_0}^{m_2}
&\rightarrow 
\mathcal{L}\big(H^{s}(\mathbb{T}^{2};\mathbb{C}); H^{s+{\delta}\vr-m_1-m_2}(\mathbb{T}^{2};\mathbb{C}) \big)\,,
\qquad \forall s\in \R\,,
\nonumber
\\
(\ta,\tb) &\mapsto \mathcal{Q}(\ta,\tb)[\cdot]
\label{def:Q}
\end{align}
satisfies 
\begin{equation}\label{calma1}
\|\mathcal{Q}(\ta,\tb)[h]\|_{s-m_1-m_2+{\delta}\vr}\lesssim_{s}
|\ta|_{m_1,s_0}|\tb|_{m_2,s_0}
\|h\|_{{s}}\,,
\qquad 
\forall\, h\in H^{s}(\T^2;\C)\,.
\end{equation}
\end{lemma}
\begin{proof}
The result is the classical composition Theorem of para-differential operators. One can see Lemma $2.18$ in \cite{FM2022} for a proof in the very same setting of the present paper. 
\end{proof}
\begin{remark}
We make the following remarks.
\begin{itemize}
    \item{\bfseries Composition of two Fourier multipliers.} \label{rem:compo_Fou}   If $\ta\equiv \ta(\xi)$ and $\tb\equiv\tb(\xi)$ are Fourier multipliers, the smoothing remainder in \eqref{espansionecompo} vanishes, namely 
    \begin{align*}
        \cQ(\ta(\xi),\tb(\xi))\equiv 0\,.
    \end{align*}
    Moreover, if $\tb=\tb(\x)$ is such that $\pa_{\x}^{3}\tb(\x)\equiv0$, one has, for $\rho\geq 3$,
    \[
   \mathcal{Q}(a(x,\x),\tb(\xi))\equiv0\,. 
    \]

\item {\bfseries Composition of three Bony-Weyl operators.}
  It follows directly from \eqref{espansionecompo} that 
  there is $s_0>2\vr$ such that, for symbols 
  $\ta\in \cN_{s_0}^{m_1,\delta}$, 
  $ \tb\in \cN_{s_0}^{m_2,\delta}$ and 
  $\tc\in  \cN_{s_0}^{m_3,\delta}$ one has
\begin{equation*}
    \Opbw{\ta}\Opbw{\tb}\Opbw{\tc}= 
    \Opbw{\ta\sha{\vr} \tb\sha{\vr} \tc}+  \cQ^{(\mathtt{t})}(\ta,\tb,\tc)\,,
    \end{equation*}
    where 
    \begin{equation*}
    \begin{aligned}
   &\ta\sha{\vr} \tb\sha{\vr} \tc:= 
   \tfrac{1}{2}( \ta\sha{\vr}(\tb\sha{\vr} \tc)
   + (\ta\sha{\vr} \tb)\sha{\vr} \tc)\,, 
   \\
&\cQ^{(\mathtt{t})}(\ta, \tb,\tc):=
\tfrac12\big(\Opbw{\ta}\cQ(\tb,\tc) 
+ \cQ(\ta, \tb\sha{\vr} \tc)+ \cQ(\ta,\tb)\Opbw{\tc} 
+ \cQ(\ta\sha{\vr} \tb,\tc)\big)\,.
\end{aligned}
\end{equation*}
Combining \eqref{calma1} and \eqref{actionSob},  the trilinear map 
\begin{align*}
\mathcal{Q}^{(\mathtt{t})}: 
\cN_{s_0}^{m_1}
\times \cN_{s_0}^{m_2}\times \cN_{s_0}^{m_3}
&\rightarrow 
\mathcal{L}\big(H^{s}(\mathbb{T}^{2};\mathbb{C}); 
H^{s+{\delta}\vr-m_1-m_2-m_3}(\mathbb{T}^{2};\mathbb{C}) \big)\,,
\qquad \forall s\in \R\,,
\nonumber
\\
(\ta,\tb,\tc) &\mapsto 
\mathcal{Q}^{(\mathtt{t})}(\ta,\tb,\tc)[\cdot]
\end{align*}
satisfies 
\begin{equation*}
\|\mathcal{Q}^{(\mathtt{t})}(\ta,\tb,\tc)[h]
\|_{s-m_1-m_2-m_3+{\delta}\vr}
\lesssim_{s}
|\ta|_{m_1,s_0}|\tb|_{m_2,s_0}|\tc|_{m_3,s_0}
\|h\|_{{s}}\,,
\qquad 
\forall\, h\in H^{s}(\T^2;\C)\,.
\end{equation*}
Applying twice \eqref{prop:ov}, the symbol $\ta\sha{\vr}\tb \sha{\vr}\tc$ belongs to $ \cN_{s_0-2\vr}^{m_1+m_2+m_3}$ and  satisfies the algebraic property
\begin{align}
\overline{\ta\sha{\vr} \tb\sha{\vr} \tc}
= \ov{\tc}\sha{\vr} \ov{\tb}\sha{\vr} \ov{\ta}\,.
\label{prop:ov3}
\end{align}
\end{itemize}
\end{remark}

Given two operators $A$ and $B$ we define 
\begin{align*}
    \mathrm{Ad}_B[A]:= \left[B, A\right]= BA-AB\,.
\end{align*}

	\begin{corollary}
		Let $\delta\in (0,1)$,  $m_1<\delta$, $m_2\in \mathbb{R}$, $s_0>1$. For any  $p\in \N$ and  $N\geq \frac{2 m_1 p+m_2+p}{\delta}$
	 there is $ \td:= \td(p,N)>pN$ such that if $g\in \cN_{s_0+\td}^{m_1}$,  $a\in \cN_{s_0+\td}^{m_2}$, one has the commutator expansion 
      \begin{align*}
            \mathrm{Ad}^p_{\Opbw{g}}\left[\Opbw{a}\right]= \Opbw{\mathtt{ad}_{\! N}^p(g)[a]}+ \cC_p(a,g)\,,
        \end{align*}
     where 
     $ \cC_p(a,g) \in \cL\left(H^s;
     H^{s-m_2- p m_1 +N\delta}\right)$ 
     for any $s\in \R$, with estimates 
    \begin{align}
\| \cC_p(a,g)\|_{\cL\left(H^s;
H^{s-m_2- p m_1 + \delta N}\right)}
\lesssim_s |g|_{m_1,s_0+\td}^p|a|_{m_2,s_0+\td}\,.
         \label{stima:Cp}
     \end{align}
In particular one has 
			\begin{align}			
    \left\|\mathrm{Ad}^p_{\Opbw{g}}\left[\Opbw{a}\right]v\right\|_{s-m_2-p m_1+\delta p}
    \lesssim_s 
    | g |^p_{m_1, s_0+\tq}| a |_{m_2, s_0+\tq} \| v \|_s\, .
    \label{stima:AdpG}
			\end{align}
            \end{corollary}
	\begin{proof}
The proof is by induction on $p$.

\noindent { \bfseries Case $p=1$.} Using \eqref{espansionecompo} with $ \vr \leadsto N $ and recalling \eqref{Moyal}, \eqref{def:ad_sym},  we get
\begin{align*}
    \Ad_\Opbw{g}\left[\Opbw{a}\right]
    = \Opbw{\mathtt{ad}_{\! N}(g)[a]}+\cC_1(a,b)\,,
    \qquad \cC_1(a,g):=\cQ(a,g)-\cQ(g,a)\,.
\end{align*}
Thanks to \eqref{def:Q} and \eqref{calma1} we obtain estimate \eqref{stima:Cp} with $p=1$.

\noindent {\bfseries Case $p\implies p+1$.} For $N\geq \frac{2 m_1 (p+1)+m_2+p+1}{\delta}$ and consider the parameter $ \tq(N)$ obtained applying \Cref{compoparapara0} with $\vr \leadsto N$. We set $ \td\equiv \td(p+1,N):=  Np+\tq(N)$. 
First we note that if $a\in \cN_{s_0+\td }^{m_1}$ and $g\in \cN_{s_0+\td}^{m_1}$  then $\mathrm{ad}_{\! N}^p(g)[a]\in \cN_{s_0+\tq}^{pm_1 + m_2- p \delta}$ with estimates
\begin{align}
    \left| \mathrm{ad}_{\! N}^p(g)[a]
    \right|_{pm_1+m_2 -\delta p, s_0+\tq}
    \lesssim | g |_{m_1, s_0+\td}^p | a |_{m_2,s_0+\td}\,. 
    \label{stima:adp}
\end{align}
Then we expand $\Ad_\Opbw{g}^{p+1}\left[\Opbw{a}\right]$ using \Cref{compoparapara0} and the inductive hypothesis 
\begin{align*}
\Ad_\Opbw{g}^{p+1}\left[\Opbw{a}\right]&= 
\left[\Opbw{g}, \Ad_\Opbw{g}^{p}\left[\Opbw{a}\right]\right]
\\
&= 
\left[\Opbw{g}, 
\Opbw{\mathtt{ad}_{\! N}^p(g)[a]}\right]
+\left[\Opbw{g}, \cC_p(g,a)\right]
\\
&= \Opbw{\mathtt{ad}_{\! N}^{p+1}(g)[a]}+ \cC_{p+1}(g,a)\,,
\end{align*}
where 
\[
\cC_{p+1}(g,a):= 
\cQ\left(g,\mathtt{ad}_{\! N}^p(g)[a]\right)
-\cQ\left(\mathtt{ad}_{\! N}^p(g)[a],g\right)
+ \left[\Opbw{g}, \cC_p(g,a)\right]\,.
\]
Then combining \eqref{stima:Cp} for $\cC_p$, \eqref{actionSob} for $\Opbw{g}$, \eqref{calma1} and \eqref{stima:adp} we eventually conclude  that $\cC_{p+1}(g,a)$ satisfies \eqref{stima:Cp} with $p\leadsto p+1$.
	\end{proof}
\subsection{Multilinear symbols and operators}	

	We define  the class of symbols which we will use throughout the paper. 
	They correspond to the autonomous symbols of Definition 3.3 in \cite{BD2018}, where the dependence on time enters only through the function $U=U(t)$. In view of this, we do not need to keep track on the regularity indexes in time and we fix $K = K' =0$ with respect to Definition 3.3 of \cite{BD2018}.
	\begin{definition}[Symbols]\label{def:sfr}
		Let $ m  \in \R $, $\delta\in [0,1]$, $N \in \N_0$, $p \in \N$, $r>0$.
		\begin{enumerate} 
			\item {\bfseries $p$-Homogeneous symbols.} We denote by $\wt {\Gamma}^{m,\delta}_p$ 
			the space of $\zak$-dependent symbols $\sym{a}{p}{m}(\zak; x,\x)$ with the following property: there is $s_0>0$ such that  for any $\sigma\geq s_0$ there is $\mu \equiv \mu(\sigma)\geq 0$ such that $a_p(\zak; x,\x)$ is a  $p$-homogeneous polynomial 
           \begin{align}\label{def:pol_sym}
               H_\R^{\s+\mu}\ni \zak \mapsto
 \sym{a}{p}{m}(\zak; x,\x)\in \mathcal{N}_{\s}^{m,\delta}\,, 
           \end{align}
with estimates 
            \begin{align}
                \left| \sym{a}{p}{m}(\zak; x, \xi)\right|_{m,\sigma}\lesssim_\sigma \| \zak\|_{\sigma+\mu}^p.
                \label{bound:homosym}
            \end{align}
			Moreover we require the  translations invariant property
        \begin{align}
             \tau_\upsilon (\sym{a}{p}{m})(\zak;x,\xi):=\sym{a}{p}{m}(\zak;x+\upsilon,\xi) = \sym{a}{p}{m}(\tau_\upsilon\zak;x,\xi)\,, 
             \qquad \forall \upsilon \in \R^2\,,
             \label{def:sym_mome}
        \end{align}
        where the operator $\tau_\upsilon$ is defined in \eqref{tautras}.
			\item {\bfseries Non-homogeneous symbols.} We denote by $\Gamma_{\geq N}^{m,\delta}[r]$ 
			the space of symbols  
\[
B_{s_0+\mu_0,\R}(r)\ni \zak \mapsto
 \sym{a}{\geq N}{m}(\zak; x,\x)\in \mathcal{N}_{s_0}^{m,\delta}\,, \qquad 
\mathrm{for\,\,some\,\,}s_0>0 \ \text{and } \mu_0\geq 0,
\] 
 having the following property: for any $ \sigma \geq s_0$ there exist
 $\mu\geq 0$ and $0<r'\leq r$ such that  for any $ \zak \in {H}_\R^{\sigma +\mu}\cap B_{s_0,\R}(r')$ one has  
			\begin{align}
			    |\sym{a}{\geq N}{m}(\zak; x, \xi)|_{m,\sigma}
                \lesssim_{\s} \| \zak\|_{\sigma +\mu}^N\,.
			     \label{stima:nonhom}
            \end{align}
            
            \item {\bfseries Symbols.} Given positive integers $0\leq p\leq N$, we denote by 
            $\Sigma\Gamma_{p}^{m,\delta}[r,N]$,
            the class of symbols which admits the expansion
            \[
                    \sym{a}{}{m}(\zak;x,\xi)= \sum_{q=p}^{N-1} \sym{a}{q}{m}(\zak;x,\xi)+ \sym{a}{\geq N}{m}(\zak;x,\xi)\,, 
                    \qquad  a_q\in \wt \Gamma^{m,\delta}_q, \quad a_{\geq N} \in \Gamma^{m,\delta}_{\geq N}[r]\,.
             \]

\item {\bfseries Functions}.
		We  denote by $\wt \cF_p$ (respectively $\cF_{\geq N}[r]$)  the subspace of $\wt \Gamma^{0,1}_p$ (respectively $\Gamma^{0,1}_{\geq N}[r]$)   made of those
		symbols which are independent of $\xi$.		
		\end{enumerate}
	\end{definition}
\begin{remark}
    The space $H^{\s+\mu}_\R$ is a \emph{real} Hilbert space, whereas $\cN_\s^{m,\delta}$ is a complex Banach space. Accordingly, the polynomial map in \eqref{def:pol_sym} should be understood as a $p$-homogeneous polynomial in the real sense.
\end{remark}
\begin{remark}
    Let $ \tau_\upsilon$ be the translation operator in \eqref{tautras}. Then, by definition \eqref{BWnon} and by the fact that 
\begin{align}
    \tau_\upsilon (e^{\ii k\cdot x}) (x)= e^{\ii k \cdot \upsilon} e^{\ii k \cdot x}, \qquad \forall k \in \Z^2\,, \quad x \in \T^2\,,\quad \upsilon\in\R^2\,,
    \label{trans:fou}
\end{align}
one has 
    \begin{align*}
        \tau_\upsilon \Opbw{a(x,\xi)}= \Opbw{(\tau_\upsilon a) (x,\xi)}\tau_\upsilon\,.
    \end{align*}
    As a consequence if $\sym{a}{p}{m}$ is a translation invariant symbol,  the operator $A(\zak):= \Opbw{\sym{a}{p}{m}(\zak;x,\xi)}$ is translation invariant, namely 
    \begin{align*}
        \tau_\upsilon A(\zak)= A(\tau_\upsilon \zak)\tau_\upsilon
    \end{align*}
\end{remark}

We also need the following definition.

\begin{definition}{\bfseries (Real symbols).}
      We say that a symbol  $a(\zak;x,\xi)\in \Sigma\Gamma_{p}^{m,\delta}[r,N] $ 
      is \emph{real} if it is real valued for any
		$ \zak \in B_{s_0,\R}(I;r)$.
		We denote  by $\wt \cF^\R_p$ (respectively $\cF^\R_{\geq N}[r]$)  the subspace of  functions in 
		$\wt \cF_p$ (respectively $\cF_{\geq N}[r]$)  which are real valued.
\end{definition}	
	
	\begin{remark}
We shall deal with symbols depending on the real variables 
$(\eta, \psi) \in H^{s_0}(\T^2; \R^2)$.
With a slight abuse of notation, we shall write 
$a(\eta,\psi; x,\xi) \in \Sigma \Gamma_p^m[r,N]$ whenever 
$\tilde a(\zak;x,\xi):=a(\frac12(\zetina+\ov{\zetina}), \frac{1}{2\ii}(\zetina-\ov{\zetina}); x,\xi)$ 
belongs to $\Sigma \Gamma_p^m[r,N]$ according to \Cref{def:sfr}.

If $a(\eta,\psi; x,\xi)$ fulfills \eqref{bound:homosym} or \eqref{stima:nonhom} with $\zak \leadsto (\eta,\psi)$ then $\tilde a(\zak;x,\xi)$  defined above is a symbol in $ \wt \Gamma_p^m$ or $\Gamma_{\geq p}^m[r]$ respectively.
\end{remark}
	\begin{remark}\label{rmk:inclusioneclassi}
	    Some comments are in order:
        \begin{itemize}
        \item If $m\leq m'$ then $\Sigma\Gamma_p^{m}[r,N]\subseteq \Sigma\Gamma_p^{m'}[r,N]$;
            \item If $a\in \Sigma\Gamma_p^m[r,N]$ and $b\in \Sigma\Gamma_{p'}^{m'}[r,N]$ then, recalling \eqref{def:pk}, one has 
            \begin{align*}
    p_n(a,b) \in \Sigma\Gamma_{p+p'}^{m+m'-n}[r,N]\,,
\qquad a\sha{\vr} b \in \Sigma\Gamma_{p+p'}^{m+m'}[r,N]\,, 
\quad \forall n \in \N_0\,.
            \end{align*}
            Indeed we note that $ \tau_\upsilon p_n(a,b) = p_n(\tau_\upsilon a, \tau_\upsilon b)$.
        \end{itemize}
	\end{remark}

		We now define the classes of remainders that we use in our procedure. 
	\begin{definition}\label{def:smoothing}
		Let $\vr\geq 0$, $ \vare_0>0$ and $r>0$. 
		
		\begin{itemize}
			\item {\bfseries $1$-homogeneous smoothing operators: }  We denote by
			$\wt \cR^{-\vr}_1$ the class of 
			$\zak$-dependent operators $R_1(\zak)$ with the following property: there is $s_0>0$ such that, for any $s\geq s_0$, the map 
            \begin{align*}
                H^s_\R \ni \zak \mapsto R_1(\zak) \in \cL\left(\dot H^s(\T^2;\C);\dot H^{s+\vr}(\T^2;\C)\right),
            \end{align*}
           is linear and continuous, with estimate
                    \begin{align}
                        \| R_1(\zak)v \|_{s+\vr} \lesssim_s \left(\| \zak\|_{s_0}\| v\|_s+ \| \zak\|_{s}\| v\|_{s_0}\right), \qquad \forall \zak \in H^s_\R, \ v \in \dot H^s(\T^2;\C).
                        \label{bound:smoo}
                    \end{align}
                    Moreover we require the translation invariant property 
            \begin{align}
                \tau_\upsilon R_1(\zak)= R_1(\tau_\upsilon \zak)\tau_\upsilon, \qquad \forall \upsilon \in \R^2.
                \label{tra:smoo}
            \end{align}
			\item {\bfseries Quadratic smoothing remainders:} We denote by $ \cR^{-\vr}_{\geq 2}[r]$ 
            the space of functions 
$ B_{\geq 2}(\zak): B_{s_0,\R}(r)  
\mapsto \cL(\dot H^s,\dot H^{s+\vr})$ for any 
$s\geq s_0$ with estimates 
\[
\| B_{\geq 2}(\zak)v \|_{s+\vr}
\lesssim_s 
\| \zak\|_{s_0}^2 \| v\|_s+ \| \zak\|_{s_0} 
\| \zak \|_{s}\| v\|_{s_0}\,, 
\quad \text{for any } \  
\zak \in B_{s_0,\R}(r)\,, \ v \in \dot H^{s}\,.
\]
		\end{itemize}
	\end{definition}
	The following lemma characterizes a $1$-homogeneous smoothing operator 
    \begin{lemma}
    \label{lem:smoohomo}
        Let $\vr\geq 0$ and $R_1(\zak)$ a $1$-homogeneous, smoothing operator in $ \wt \cR^{-\vr}_1$. Then one has the Fourier expansion
			\begin{align}
				\label{Mp}
				R_1(\zak)v:= 
				\sum_{\substack{  j, k \in \Z^2\setminus \{0\} 
                \\ \sigma\in \{\pm\} } }
				R_{k, j}^{\sigma} \,
                \zetina_k^\sigma\,
                v_{j}  \, {e^{\ii (\sigma k+j) \cdot  x} }\,, 
			\end{align}  
			that satisfy the following. There are $\mu\geq0$, $C>0$ such that 
			for any $k,j \in \Z^2 $,   $  \sigma\in \{ \pm \} $,  one has 
			\be\label{smoocara}
			|R_{k, j }^{ \sigma} |\leq C \, 
			\mathrm{min}\{ |k|,|j|\}^{\mu}\, \max\{|k|,|j|\}^{-\vr}   \, . 
			\ee
            In general, if a family of complex coefficients $\{R_{j,k}^\sigma\}_{j,k,\sigma}$ satisfies the estimate \eqref{smoocara} for some $\mu > 0$, then the expression in \eqref{Mp} defines a $1$-homogeneous, smoothing operator belonging to $\wt{\cR}_1^{-\vr}$.
    \end{lemma}
    \begin{proof}
We first prove that the $\R$-linearity of $R_1$ and the translation invariance property \eqref{tra:smoo} give the expression \eqref{Mp}, then we use the bound \eqref{bound:smoo} to obtain \eqref{smoocara}. Finally we prove the converse, stated in the final part of the statement. 
We start writing the Fourier expansions   
\begin{align*}
   R_1(\zak)v= \sum_{\ell \in \Z^2 \setminus \{ 0\}}  (R_1(\zak)v)_\ell e^{\ii \ell\cdot x}, \qquad \zak= \sum_{k \in \Z^2 \setminus \{ 0\}} \vect{\zetina_k e^{\ii k \cdot x}}{\ov{\zetina_k} e^{-\ii k \cdot x}}, \qquad  v=\sum_{j\in \Z^2\setminus \{0\}} v_j e^{\ii j\cdot x}.
\end{align*}
Using also the $\R$-linearity with respect to $ \zak$ and the $\C$-linearity with respect to $v$, we get 
\begin{align}
    R_1(\zak)v= 
    \sum_{\substack{\ell, k,\,j \in \Z^2 \setminus \{ 0\}\\ \sigma\in \{\pm\}}} M_{j, k, \ell}^\sigma 
    \zetina_k^\sigma v_j e^{\ii \ell \cdot x},
    \label{esp:proof}
\end{align}
where, for $\lambda \in \C$ we used the notation $ \lambda^+:=\lambda$, $\lambda^-:= \ov{\lambda}$, and 
\begin{align}
    M_{j, k, \ell}^\sigma:= \frac{1}{(2\pi)^2} \int_{\T^2} \frac12 \left[ R_1\left( \vect{e^{\ii k \cdot x }}{e^{-\ii k \cdot x }}\right)-\sigma \ii R_1\left( \vect{\ii e^{\ii k \cdot x }}{-\ii e^{-\ii k \cdot x }}\right)\right] e^{\ii j\cdot x}\, e^{-\ii \ell \cdot x} \, \di x.
    \label{R1coeff}
\end{align}
Note that to get \eqref{esp:proof} we also used that, by $\R$-linearity, one has 
\begin{align}
     R_1\left( \vect{\lambda e^{\ii k \cdot x }}{\ov{\lambda}e^{-\ii k \cdot x }}\right)-\sigma \ii R_1\left( \vect{\ii \lambda e^{\ii k \cdot x }}{-\ii \ov{\lambda} e^{-\ii k \cdot x }}\right)= \lambda^\sigma \left[ R_1\left( \vect{e^{\ii k \cdot x }}{e^{-\ii k \cdot x }}\right)-\sigma \ii R_1\left( \vect{\ii e^{\ii k \cdot x }}{-\ii e^{-\ii k \cdot x }}\right)\right].
     \label{ClinR1}
\end{align}
Using \eqref{R1coeff}, \eqref{trans:fou} and the $\C$-linearity in \eqref{ClinR1},
we obtain the restriction 
\begin{align}
    M_{j, k, \ell}^\sigma\not=0 \quad \implies \quad \ell= \sigma k + j.
    \label{Mjkell}
\end{align}
Then the expansion in \eqref{Mp} follows with 
\begin{align*}
    R_{j,k}^\sigma:= \left.M_{j, k, \ell}^\sigma\right|_{\ell=\sigma k + j}.
\end{align*}
Then, applying \eqref{bound:smoo} with $s=s_0$,  $ \zak\leadsto \vect{e^{\ii k \cdot x }}{e^{-\ii k \cdot x }}$ and $ \zak\leadsto \vect{\ii e^{\ii k \cdot x }}{-\ii e^{-\ii k \cdot x }}$, $v\leadsto e^{\ii j \cdot x}$ and Cauchy-Schwarz inequality to \eqref{Mjkell} with $\ell= \sigma k +j$, we obtain 
\begin{align*}
    |R_{j,k}^\sigma| \lesssim \langle\sigma k +j\rangle^{-(s_0+\vr)}  |k|^{s_0}|j|^{s_0}\leq 2^{2s_0+\vr} \mathrm{min}\{ |j|,|k|\}^{2s_0+\vr}\max\{ |j|,|k|\}^{-\vr},
\end{align*}
proving \eqref{smoocara} with $\mu:= 2 s_0+\vr$.

 \noindent
 Finally assuming \eqref{Mp} and the bound \eqref{smoocara} for the 
 coefficients $R_{k, j}^{\sigma}$, by  
 Young's inequality for convolutions of sequences, we get 
  \begin{align*}
 \| R_1(\zak)v \|_{s+\vr}^2 &=   \sum_{\ell \in \Z^2\setminus\{ 0\}}
|\ell|^{2(s+\vr)}
\left| \sum_{\substack{\sigma\in \{\pm\}\\ \sigma k+j=\ell}} 
R_{k, j}^{\sigma} \, \zetina_{k}^{ \sigma}\,  v_{j}\right|^2
\\
& \lesssim_s 
\sum_{\sigma\in \{\pm\}} \sum_{\ell \in \Z^2\setminus\{ 0\}}
\left| \sum_{ \sigma k+j=\ell}\mathrm{min}\{ |k|,|j|\}^{\mu}\, \max\{|k|,|j|\}^{s} 
|\zetina_{k}^{ \sigma}|\,  |v_{j}|\right|^2
\\
&\lesssim_s  
\sum_{\sigma\in \{\pm\}} \sum_{\ell \in \Z^2\setminus\{ 0\}}
\left| \sum_{  k+j=\ell} |k|^\mu |j|^{s-\vr} |\zetina_{\sigma k}^{ \sigma}|\,  |v_{j}|\right|^2
+ \sum_{\sigma\in \{\pm\}} \sum_{\ell \in \Z^2\setminus\{ 0\}}\left| 
\sum_{  k+j=\ell} |j|^\mu |k|^{s} |\zetina_{\sigma k}^{ \sigma}|\,  
|v_{j}|\right|^2
\\
&= 
\sum_{\sigma\in \{\pm\}} \||k|^\mu |\zetina_{k}^\sigma|\|_{\ell^1_k(\Z^2)}^2 \||k|^{s} v_k\|_{\ell^2_k(\Z^2)}^2+ \||k|^s|\zetina_{k}^\sigma|\|_{\ell^2_k(\Z^2)}^2 \||k|^{\mu} v_k\|_{\ell^1_k(\Z^2)}^2 
\\&
\lesssim_s 
\| \zak\|_{2+\mu}^2\| v\|_s^2+ \| \zak\|_s^2 \| v\|_{2+\mu}^2\,,
\end{align*}
proving that \eqref{Mp} defines a $1$-homogeneous smoothing operator satisfying \eqref{bound:smoo}  with $s_0:= 2+\mu$. Moreover the translation invariant property follows applying \eqref{trans:fou} to the expression \eqref{Mp}.
    \end{proof}

    \paragraph{Real to real matrices of operators} A $\zak$-dependent matrix of operators of the form  
     \begin{align*}
        \bA(\zak)= \begin{pmatrix}
            A^{+,+}(\zak) & A^{-,+}(\zak) \\ A^{+,-}(\zak)& A^{-,-}(\zak)
        \end{pmatrix},
    \end{align*}
    is said to be \emph{real to real} if 
    \begin{align*}
\ov{A^{\sigma , \sigma'}(\zak)}= 
A^{-\sigma , -\sigma'}(\zak)\,, 
\qquad \forall \sigma, \sigma'\in \{\pm\}\,, 
\  \zak= (\zetina, \ov{\zetina})^\top\,.
    \end{align*}
   A real to real operator leaves invariant the space of couple of complex functions $V=(v^+, v^-)$ such that $ \ov{v^+}=v^-$.
If $\bA(\zak)$ is a paradifferential matrix of real-to-real operators  then it has the form 
\begin{align*}
    \Opbw{\begin{bmatrix}a(x,\xi) &b(x,\xi)\\ \ov{b(x,-\xi)}& \ov{a(x,-\xi)}\end{bmatrix}}=\vOpbw{a(x,\xi)}+ \zOpbw{b(x,\xi)},
\end{align*}
where we used the notation 
\begin{equation}\label{def:vec_out}
\begin{aligned}
    \vOpbw{a(x,\xi)}&:= \Opbw{\begin{bmatrix}a(x,\xi) &0\\ 0& \ov{a(x,-\xi)}\end{bmatrix}}\,,
    \\
    \zOpbw{b(x,\xi)}&:= \Opbw{\begin{bmatrix}0 &b(x,\xi)\\ \ov{b(x,-\xi)}& 0\end{bmatrix}}\,.
\end{aligned}
\end{equation}

       \begin{remark}
       \label{rem:stima_homo}
        Let $\vr\geq 0$. In view of \Cref{lem:smoohomo}, if $\bR_1(\zak)$ is a matrix of real to real smoothing operators in $ \wt \cR^{-\vr}_1$ then it has the expansion 
   \begin{align}
       (\bR_1(\zak)V)^{\sigma}= 
				\sum_{\substack{  j, k \in \Z^2\setminus \{0\} \\ \sigma_1,\sigma_2,\sigma\in \{\pm\} } }
				R_{k, j}^{\sigma_1,\sigma_2,\sigma} \, \zetina_{k}^{ \sigma_1}\,  v_{j}^{\sigma_2}  \, {e^{\ii\sigma (\sigma_1 k+\sigma_2 j) \cdot  x} },\quad \text{with} \quad \ov{R_{k, j}^{\sigma_1,\sigma_2,\sigma}} =R_{k, j}^{-\sigma_1,-\sigma_2,-\sigma}.
                \label{esp:mat_smoo}
   \end{align}
   Note that, for $ \sigma_2, \sigma\in \{\pm\}$, the coefficients $R_{k, j}^{\sigma_1,\sigma_2,\sigma}$ coincide with the coefficients $R_{k,j}^{\sigma_1}$ in the expansion \eqref{Mp} relative to the smoothing operator $ (R_1(\zak))^{\sigma_2, \sigma}$.
        \end{remark}
        
   The following lemma provides the expansions of the compositions between some real to real operators. 
   \begin{lemma}[\textbf{Compositions}]
       Let $\vr \geq 0$, $m\in \R$, $\delta \in (0,1]$ and $ N\in \N$ with $N\geq \vr+m+\frac32$. If $\sym{g}{1}{m,\delta}=\sym{g}{1}{m,\delta}(\zak;x,\xi)$ is a $1$-homogeneous symbol in $\wt \Gamma_1^{m,\delta}$ then 
       \begin{align}
           &\cQ(\sym{g}{1}{m,\delta}, \Lambda(\xi))=\Opbw{\sym{g}{1}{m}}\Lambda(D)- \Opbw{\sym{g}{1}{m,\delta}\sha{N}\Lambda(\xi)}\in \wt \cR^{-\vr}_1,\label{eq:Gomega}
           \\
           &\cQ(\Lambda(\xi), \sym{g}{1}{m,\delta})=\Lambda(D)\Opbw{\sym{g}{1}{m,\delta}}-\Opbw{\Lambda(\xi)\sha{N}\sym{g}{1}{m,\delta}}\in \wt \cR^{-\vr}_1.
           \label{eq:omegaG}
       \end{align}
   \end{lemma}
   \begin{proof}
       Thanks to \Cref{compoparapara0} we know that $\cQ(\sym{g}{1}{m,\delta}, \Lambda(\xi))$ and $\cQ(\Lambda(\xi), \sym{g}{1}{m,\delta})$ are smoothing operators in $\cL(H^s;H^{s+\vr})$. The estimate \eqref{bound:smoo} follows combining \eqref{calma1} with \eqref{bound:homosym}. 
   \end{proof}
   \subsection{$\tR$-localized symbols}

In the normal form procedure developed in \Cref{QR-NormalForm}, we act only on the low--frequency component of the symbols, while leaving the high--frequency modes untouched. More precisely, the normal form transformation is performed on Fourier modes satisfying
\[
|\xi| \leq 11\tR,
\]
where $\tR \gg 1$ is a large parameter fixed throughout the analysis.  
This motivates the introduction of a class of symbols whose spectral support is localized to this low--frequency region.

In this subsection we introduce the notion of $\tR$-localized symbols and establish their basic properties, which will be repeatedly used in the construction and control of the normal form transformation.
\begin{definition}[$\tR$-localized symbols]
\label{def:Rloc}
		Let $s_0$, $m\in \R$  and $\delta \in (0,1)$. 
        A $\tR$-localized symbol in $\cN_{s_0}^{m,\delta}$ is a family of symbols $$[1,+\infty)\ni\tR\mapsto b_\tR(x,\xi)\in L^{\infty}\left([1,+\infty);\cN_{s_0}^{m,\delta}\right),$$
      such that, for any $ \tR\geq 1$, one has 
      \begin{align}
\supp b_\tR(x,\xi) \subset \left\{( x, \xi) \in \T^2 \times \R^2    \colon | \xi| \leq 11 \tR\right\}.
          \label{spec:supp}
      \end{align}
\end{definition}
 \begin{remark}
\label{rem:stimaesp}
Let $s_0>0$, for any $\mu\in [0,1]$ and  $\alpha\in (0,2)$ there are constants $C_\mu>0$  and $C_\alpha >0$ such that 
    \begin{align}
        \left(\tR \| \zak\|_{s_0}\right)^\alpha \leq C_\alpha \mathrm{exp}\Big({(\tR \|\zak\|_{s_0})^2}\Big), \qquad \mathrm{exp}\Big({(\tR \|\zak\|_{s_0})^\mu}\Big)\leq C_\mu \mathrm{exp}\Big({(\tR \|\zak\|_{s_0})^2}\Big).
        \label{est:power}
    \end{align}
\end{remark}
	The fundamental property of $\tR$-localized symbols is that they 
	are symbols in $\cN^{m-\sigma, \delta}_{s_0}$ for any $\sigma\geq 0$. 
	The price to pay is that their $m-\sigma$ norm  amplificates by 
	a factor $\tR^\sigma$. This property will be 
	used to bound positive order symbols with small norm and  
	is the content of the following lemma.
	\begin{lemma}
    \label{lem:perdita}
    Let $s_0>1$, $m\in \R$  and $\delta \in (0,1)$.
		Let $b_\tR(x,\xi)$be a $\tR$-localized symbol in $\cN^{m,\delta}_{s_0}$  according to \Cref{def:Rloc}. Then, for any $ \tR>1$, $b_\tR(x,\xi)$ belongs to $ \cN^{m-\sigma,\delta}_{s_0}$ for any $ \sigma\geq 0$ with estimate 
		\be \label{guadagno:loc}
		| b_\tR(x,\xi)|_{m-\sigma,s_0}
        \leq C\tR^\sigma	| b_\tR(x,\xi)|_{m,s_0}\,,
		\ee
        for some $C=C_\s>0$.  Moreover if the $\tR$-localized symbol $b_\tR$ depends on $\zak$ and belongs to $ \Gamma_{\geq p}^m[r]$ then
	\begin{align}
	| b_\tR(\zak;x,\xi)|_{m-\sigma,s_0}
    \lesssim_{\s} \tR^{\sigma}\| \zak\|_{s_0}^p\,.
    \label{est:symspec}
	\end{align}
    In addition, for $m\geq 0$ and $p\geq 2+m$,  we have 
    \begin{align}
        \| \vOpbw{b_\tR(\zak;x,\xi)}V\|_{s}
        \lesssim_s (\tR\| \zak\|_{s_0})^m 
        \| \zak\|_{s_0}^2\|V\|_{s}
        \lesssim_s 
    e^{(\tR \| \zak\|_{s_0})^2}\| \zak\|_{s_0}^2 \| V\|_s\,.
        \label{stima:exp_opbw}
    \end{align}
	\end{lemma} 
	\begin{proof}
From the support property \eqref{spec:supp} of $b_\tR(x,\xi)$ we have, for any multi-index $ \beta \in \N^2_0$,  
\begin{align*}
    \pa_\xi^\beta b_\tR(x,\xi)=0, \qquad \forall |\xi|>11 \tR.
\end{align*}
Therefore we get
\begin{align*}
| b_\tR(x,\xi)|_{m-\sigma,s_0}&=
\max_{|\alpha|+|\beta|\leq s_0} \sup_{|\xi| \leq 11 \tR } 
\left| \pa_x^\alpha \pa_\xi^\beta b_\tR(x,\xi)\right| 
\langle \xi\rangle^{-m+ \sigma + \delta | \beta|}
\\
&\leq 
\langle 11 \tR\rangle^{\sigma}| b_\tR(x,\xi)|_{m,s_0}
\leq 
22^\sigma \tR^\sigma | b_\tR(x,\xi)|_{m,s_0}\,,
\end{align*}
which implies \eqref{guadagno:loc}. Estimate \eqref{est:symspec} is a direct consequence of \eqref{guadagno:loc} and \eqref{stima:symbols}. Finally \eqref{stima:exp_opbw} follows combining \eqref{est:symspec} (with $ \sigma\leadsto m$) and \eqref{actionSob}.
	\end{proof}
A consequence of \Cref{rem:stimaesp} and \Cref{lem:perdita} is the following lemma. 
    \begin{lemma}
    \label{lem:stima_comm}
       Let $r>0$ and  $p_1, p_2 \in \N$, $m_1,m_2 \in \R$ and $\delta\in (\frac56,1)$ and define $ \sigma:= \max\{0,m_1+m_2-\delta\}$. If $\sym{a}{\geq p_1}{m_1} \in \Gamma_{\geq p_1}^{m_1,\delta}[r]$ is a $\tR$-localized symbol and $\sym{a}{\geq p_2}{m_2,\delta} \in \Gamma_{\geq p_2}^{m_2}[r]$ and $p_1+p_2\geq 2+\sigma$, then
        \begin{itemize}
            \item for any $s\in \R$ there is $C_s>0$ such that for any $ \tR>0$ one has 
        \begin{align}
\| \left[\vOpbw{\sym{a}{\geq p_1}{m_1}},
\vOpbw{\sym{a}{\geq p_2}{m_2}}\right]\|_{\cL(H^s_\R;H^s_\R)}
\leq C_s 
\mathrm{exp}\Big({C_s(\tR \| \zak\|_{s_0})^2} \Big)
\| \zak\|_{s_0}^2\,;
            \label{eq:stima_comm}
        \end{align}
        \item if $\vr \geq m_2$ and $ \bR_1(\zak)$ belongs to $ \wt \cR^{-\vr}_1$ there is $s_0>0$ such that, for any $s\geq s_0$, there is a constant $C_s>0$ such that 
        \begin{align}
            \| \left[\bR_1(\zak),\Opbw{\sym{a}{\geq p_2}{m_2}}\right]V\|_s
            \leq 
            C_s \left( \| \zak\|_{s_0}^2 \| V\|_s + \| \zak\|_{s_0} \| \zak\|_s \| V\|_{s_0}\right)\,.
            \label{eq:stima_comm_smoothing}
        \end{align}
        \end{itemize}
    \end{lemma}
    \begin{proof}
We apply estimate \eqref{stima:AdpG} to the symbols $\sym{a}{\geq p_2}{m_2}$ and $\sym{a}{\geq p_1}{m_1}$, which belong to $\cN_{s_0}^{m_1-\sigma}$ by \Cref{lem:perdita}. Using in addition the fact that $m_1-\sigma+m_2-\delta \leq 0$, we obtain
\begin{equation*}
\begin{aligned}
 \| \left[\vOpbw{\sym{a}{\geq p_1}{m_1}},\vOpbw{\sym{a}{\geq p_2}{m_2}}\right]V \|_s
 &\lesssim_s  | \sym{a}{\geq p_1}{m_1}|_{m_1-\sigma,s_0}
 | \sym{a}{\geq p_2}{m_2}|_{m_2,s_0}\|V\|_s
 \lesssim_s \tR^{\sigma} \| \zak\|_{s_0}^{p_1+p_2}\|V\|_s
 \\
  &\lesssim_s  
  (\tR \| \zak\|_{s_0})^{\sigma}\| \zak\|_{s_0}^2\|V\|_s
  \leq C_s \mathrm{exp}\Big({C_s(\tR \| \zak\|_{s_0})^2}\Big)\| \zak\|_{s_0}^2\| V\|_s\,,
    \end{aligned}
    \end{equation*}
    for some $C_s \gg 0$ large enough, where in the last inequality we used \eqref{est:power}.
    This proves \eqref{eq:stima_comm}. Moreover 
    estimate \eqref{eq:stima_comm_smoothing} 
    follows by combining \eqref{bound:smoo} 
    with \eqref{actionSob} and \eqref{stima:nonhom}. 
    \end{proof}

  \subsection{Flows and conjugations}
  In this section we collect all the conjugation rules needed in the paper.
 Consider general flow  $\bm{\Phi}^\tau(\zak)$, $\tau\in[-1,1]$ 
\be\label{flusso2para}
\begin{cases} 
	\partial_\tau { \bm{\Phi}}^\tau(\zak)= \bG(\zak){\bm{\Phi}}^\tau(\zak)\\
	{\bm{\Phi}}^0(\zak)=\uno\, .
\end{cases}
\ee
\begin{lemma}[Flow of a standard para-differential operator]\label{lem:flusso}
    Let $r>0$. Given a real valued symbol  $f(\zak; x, \xi)\in \Gamma^{\sla{1}{2}}_{\geq 1}[r]$,   and a symbol $ g \in \Gamma_{\geq 1}^0[r]$,  consider the flow in \eqref{flusso2para} with 
\[
    \bG(\zak):= \vOpbw{\im f(\zak;x,\xi) } 
    \qquad \text{or}  
    \qquad\bG(\zak):=\zOpbw{ g(\zak;x,\xi) }\,.
\] 
There is $s_0>0$ such that for any 
$ s\in \R$ there is $r_s\in (0,\frac14)$ and 
$C_s>0$ such that for any 
$ \tau\in [-1,1]$ and $ \zak \in B_{s_0,\R}(r_s)$, 
the linear flow $ \Phi^\tau(\zak)$  
and  its inverse  $\Phi^{-\tau}(\zak)$ 
belong to $\cL(H^s_\R;H^s_\R)$ with estimates
\be
\| \Phi^{\pm \tau}(\zak)\|_{\cL(H^s_\R;H^s_\R)}\leq C_s\,.
\label{stima:phistd}
\ee  
\end{lemma}
\begin{proof}
It is standard (see e.g.  Lemma 3.22 in \cite{BD2018}) that, for any  
	$\zak \in B_{s_0,\R}(r)$ with $s_0>0$ sufficiently large and $r>0$ sufficiently small, 
	the operator 
	$ {\bm{\Phi}}^\tau (\zak)\in\cL\pare{H^s(\T, \C^2)} $ for any $s \in \R$ with the quantitative estimate:  there is a constant $C_s>0$ such that	for any 
    $ W\in H^{s}(\T, \C^2)$,
    \begin{align}
        \norm{ {\bm{\Phi}}^\tau (\zak) W }_{s}+\norm{ {\bm{\Phi}}^\tau (\zak)^{-1} W }_{s}
		 \leq   C_s  \| W \|_{s}\,.
         \label{stima:flusso}
    \end{align}
    Since, by construction the generator $\bG(\zak)$ is a real to real matrix of operators, its flow $\Phi^{\pm \tau}(\zak)$ is real-to-real as well and thus preserves the real subspace $H^s_\R$.
\end{proof}
	Let $ \delta \in (0,1)$ and $m\in [0,2]$.
	We now  study the properties of the flow \eqref{flusso2para} when the generator is the diagonal paradifferential operator associated to a $\tR$-localized symbol $\sym{g}{p}{m}\in \wt \Gamma_p^m$. In particular we prove the following lemma.

\begin{lemma}[Flow of a $\tR$-localized para-differential operator]
\label{lem:flussoR}
Let $\tR>1$, $\delta\in (0,1)$, $ m \in [0,2]$ and  $p\in \N$ with $ p> m+\tfrac12$. Let $\sym{g}{p}{m}\in  \wt \Gamma_{p}^{m, \delta}$ a $\tR$-localized symbol and consider the flow $\bPhi^{\tau}_{g_p}(\zak)$ generated as in \eqref{flusso2para} 
with generator 
\[
\bG(\zak):= \vOpbw{\sym{g}{p}{m}(\zak)}\,.
\]
		There is $s_0>0$ such that for any $ s\in \R$ there is $r\in (0,\tfrac14)$ such that for any $ \tau\in [-1,1]$ and $ \zak \in B_{s_0,\R}(r)$, the linear flow $ \bPhi^\tau_{\fun{g}{p}}(\zak)$ and  its inverse  $\bPhi^{-\tau}_{\fun{g}{p}}(\zak)$ belong to $\cL(H^s_\R;H^s_\R)$ with estimates
\be
\| \bPhi^{\pm \tau}(\zak)\|_{\cL(H^s_\R;H^s_\R)}
\leq 
3\exp\left(2(\tR \|\zak\|_{s_0})^2
\| \zak\|_{s_0}^{\sla{1}{2}}\right)
\leq 3\exp\left((\tR \|\zak\|_{s_0})^2\right)\,.
            \label{stima:phiexp}
\ee
			\end{lemma}
		\begin{proof}
Thanks to \Cref{lem:perdita} and \Cref{thm:action}, 
the generator $\vOpbw{\ii \sym{g}{p}{m}}$ belongs to $\cL(H^s_\R;H^s_\R)$ and
\begin{align*}
\| \vOpbw{\ii \sym{g}{p}{m}}\|_{\cL(H^s_\R;H^s_\R)} 
&\leq C_s | \sym{g}{p}{m}|_{0, s_0} 
\leq C_s \tR^m\| \zak\|_{s_0}^p
\\&\leq C_s r^{m-p-\sla{1}{2}}( \tR \| \zak\|_{s_0})^m
\| \zak\|_{s_0}^{\sla{1}{2}}\leq ( \tR \| \zak\|_{s_0})^m\| \zak\|_{s_0}^{\sla{1}{2}}\,,
\end{align*}
provided $C_s r^{m-p}\leq 1$. Then standard Banach spaces 
ODEs theory implies that  $ \bPhi^{\pm \tau}_{\fun{g}{p}}(\zak) \in\cL(H^s_\R;H^s_\R) $ 
with estimate 
\begin{align*}
    \| \bPhi^{\pm \tau}(\zak)\|_{\cL(H^s_\R;H^s_\R)}
    \leq 
    \exp\left( ( \tR \| \zak\|_{s_0})^m\| \zak\|_{s_0}^{\sla{1}{2}}\right)
    \leq 
    3\exp\left( 2( \tR \| \zak\|_{s_0})^2\| \zak\|_{s_0}^{\sla{1}{2}}\right)\,,
\end{align*}
where to get the last bound we used that, 
as $ \| \zak\|_{s_0}^{\sla{1}{2}}\leq r^{\sla{1}{2}}\leq \frac12$, 
one has 
\[
( \tR \| \zak\|_{s_0})^m\| \zak\|_{s_0}^{\sla{1}{2}}
\leq 
1+ 2( \tR \| \zak\|_{s_0})^2\| \zak\|_{s_0}^{\sla{1}{2}}\,,
\]
proving the first inequality in \eqref{stima:phiexp}. 
The second inequality follows using again the bound 
$ \| \zak\|_{s_0}^{\sla{1}{2}}\leq1/2$.
\end{proof}
Finally we have a bound for the flow of a $1$-homogeneous smoothing remainder.
\begin{lemma}
\label{lem:flow_smoo}
    Let $\vr>0$ and $\bG_1(\zak)$ be a real-to-real  matrix of smoothing operators in $\wt \cR^{-\vr}_1$. Consider the flow $\bPhi^\tau(\zak)$ generated as in \eqref{flusso2para} with generator
\[
\bG(\zak)=\bG_1(\zak)\,.
\]
There is $s_0 >0$ such that for any $s\geq s_0$ 
there is $r=r_s>0$ such that for any 
$\zak \in B_{s_0,\R}(r)\cap H^{s}_{\R}$ the maps 
$\bPhi^\tau(\zak)$  and its inverse $\bPhi^{-\tau}(\zak)$  
belong to $\cL\left( H^s_\R;H^s_\R\right)$ with estimates 
\begin{align}
\|\bPhi^{\pm \tau}(\zak)V\|_{s}
\lesssim_s 
\| V\|_s+ \| \zak\|_{s}\| V\|_{s_0}\,.
\label{stima:flussosmoothing}
\end{align} 
\end{lemma}
\begin{proof}
    As the operator $\bG_1(\zak)$ belongs to $\cL\left( H^s_\R;H^s_\R\right)$ with estimates \eqref{bound:smoo}, standard Banach spaces ODEs theory implies that the flow $\bPhi^\tau(\zak)$ belongs to $\cL\left( H^s_\R;H^s_\R\right)$. The proof of the tame estimates \eqref{stima:flussosmoothing} can be found in \cite[Lemma A.4]{MMS24}. 
The only difference is that in \cite[Lemma A.4]{MMS24} 
the generator is a scalar real operator, 
while the generator $\bG_1(\zak)$ is a matrix of 
smoothing operators. This does not affect 
the proof since, by hypothesis, 
$\bG_1(\zak)$ preserves the real subspace $H^s_\R$.

\end{proof}
\section{ Background on the Water Waves Problem}
In this section, we collect the basic properties of the water waves problem that will be used in our analysis. This includes the linearization at $(0,0)$, the analytic structure of the water waves system in complex coordinates, and the analysis of the relevant physical quantities, namely the horizontal and vertical components of the velocity field at the free boundary $y=\eta(x)$,
\begin{align}
\tV := \nabla \Phi (x,\eta(x))= 
\nabla\psi- \tB \nabla \eta\,, 
\qquad
\tB := \partial_y \Phi(x,\eta(x))
= \frac{G(\eta)\psi+\nabla \eta \cdot \nabla \psi}{1+ 
| \nabla \eta|^2}\,,
\label{def:V-B}
\end{align}
and the classical Taylor coefficient,
\begin{align}
\ta :=  \pa_t \tB + \tV\cdot \nabla \tB \,.
\label{def:a-Tay}
\end{align}
Moreover, we define the function 
\begin{align}\label{funz:bpiccolo}
    \tb:= \frac{1}{1+|\nabla \eta|^2}\,.
\end{align}
\subsection{Complex Hamiltonian formulation and non-resonance conditions}
We now collect some general properties 
of the water waves system. We begin the section 
with the linearization at $(\eta,\psi)=(0,0)$.

\vspace{0.5em}
\noindent
{\bfseries Linear water waves system.} The formal linearization of \eqref{eq:1.2} at the flat surface leads to 
the linear system\footnote{Here we denote by $|D| =|-\ii \nabla|$ the Fourier multiplier acting as $|D|e^{\ii j\cdot x}=|j|e^{\ii j\cdot x}$ for any $j\in \Z^{2}$.}
\begin{equation}\label{lineq1.2}
\left\{
\begin{aligned}
\pa_{t}\eta&=|D|\psi\,,
\\
\pa_{t}\psi&=-(1+\kap|D|^2)\eta\,.
\end{aligned}\right.
\end{equation}
By defining the \emph{linear dispersion relation}
\begin{equation}\label{def:dispersion}
\Lambda(D):=\sqrt{|D|(1+\kap|D|^2)}\,,
\end{equation}
and, introducing the complex variable
\begin{equation}
    \zetina:= \frac{1}{\sqrt{2}}\left(\tM(D)^{-1} \eta+ \ii \tM (D) \psi\right), \qquad 
\mathtt{M}(D):=\sqrt{\frac{|D|}{\Lambda(D)}}=\left(\frac{|D|}{1+\kap|D|^2}\right)^{\frac{1}{4}}\,, 
\label{zetina}
\end{equation}
we get that \eqref{lineq1.2} is equivalent to
\[
\pa_{t}\zetina=-\ii \Lambda(D)\zetina\,.
\]
\begin{remark}
In view of \Cref{rem:symbols_modozero}, 
the functions $\tM(\xi)$, $\tM^{-1}(\xi)$ and 
$\Lambda(\xi)$ are Fourier multipliers, namely 
    \begin{align*}
        \tM(\xi)\in \wt \Gamma_0^{\sla{1}{4}}\,, 
        \qquad  
        \tM^{-1}(\xi)\in \wt \Gamma_0^{-\sla{1}{4}}
        \quad \text{and} 
        \quad \Lambda(\xi)\in\wt \Gamma_0^{\sla{3}{2}}\,.
    \end{align*} 
\end{remark}

\vspace{0.5em}
\noindent
{\bfseries Nonlinear water waves system in complex coordinates.}
In this paragraph, we study the full nonlinear water waves system \eqref{eq:1.2} written with respect to the complex coordinates $\zetina$. Since the full nonlinear system for $ \zetina$ depends on both $\zetina$ and  $\ov{\zetina}$, it is convenient to introduce the vector of complex variables 
\begin{equation}
    \zak:= \vect{\zetina}{\ov{\zetina}} = \cM \vect{\eta}{\psi}, \qquad \cM=\frac{1}{\sqrt{2}}\begin{pmatrix}
        \tM(D)^{-1} &   \ii \tM(D) \\
        \tM(D)^{-1} &    -\ii\tM(D)
    \end{pmatrix}, \qquad \cM^{-1}:=\frac{1}{\sqrt{2}}\begin{pmatrix}
        \tM(D) &    \tM(D) \\
        -\ii \tM(D)^{-1} &    \ii\tM(D)^{-1}
    \end{pmatrix}\,,
    \label{zak}
\end{equation}
 We have the following result. 
\begin{proposition}[\textbf{Water waves equations for the complex Zakharov variables}]
\label{prop:compZak}
For any $ s\in \R$ the linear operators $\cM$ and $\cM^{-1}$ 
in \eqref{zak} belong respectively to \\ 
$\cL\left(H^{s+\sla{1}{4}}_0\times \dot H^{s+\sla{1}{4}};\dot H^{s}_\R(\T^2;\C^2)\right)$ 
and $\cL\left(\dot H^{s}_\R(\T^2;\C^2);H^{s+\sla{1}{4}}_0\times \dot H^{s+\sla{1}{4}}\right)$ 
with estimates
        \begin{align}\label{equiequi}
            \| \zak\|_s \lesssim \| \eta\|_{s+\sla{1}{4}}+ \| \psi\|_{s-\sla{1}{4}}\lesssim \| \zak\|_s.
        \end{align}
        Moreover there is $s_0>0$ such that for any $ s\geq s_0$ there is $r>0$  such that if $(\eta(t, \psi(t)))$ with  
        $\eta(t)\in B_{s_0}(r)\cap H^{s+\sla{1}{4}}_0$, $\psi(t)\in B_{s_0}(r)\cap  \dot H^{s+\sla{1}{4}}$ is a solution  of the water waves system \eqref{eq:1.2}, the variable $\zak\in H^{s}_\R(\T^2;\C^2)$ in \eqref{zak} solves 
        \begin{align}
            \pa_t \zak= \hamvec{\cH}(\zak)\,, 
            \qquad 
            \text{where} \quad 
            \hamvec{\cH}:  
B_{s_0}(r)\cap H^{s}_\R\to H^{s-\sla{3}{2}}_\R  
            \label{eq:zak}
        \end{align}
        is analytic with Taylor expansion 
        \begin{align}
\hamvec{\cH}(\zak)= -\ii \vomega(D) \zak + \hamvec{2}(\zak,\zak)+\hamvec{\geq 3}(\zak)\,, 
\label{X:zak_esp}
        \end{align}
        where 
        \begin{equation}
         \hamvec{2}(\cdot, \cdot )\colon H^s_\R \to H^{s-\sla{3}{2}}_\R    \quad  \ \text{is quadratic and }\quad \| \hamvec{\geq 3}(\zak)\|_{s-\sla{3}{2}}\lesssim_s \| \zak\|_{s}^3.
         \label{espansione_campoanal}
        \end{equation}
        In particular one has 
        \begin{align}
            \| \hamvec{\cH}(\zak)\|_{s-\sla{3}{2}}\lesssim_s \| \zak\|_s.
            \label{est:XH}
        \end{align}
Moreover the vector field 
\begin{equation*}
\hamvec{\geq 2}(\zak):= 
\hamvec{2}(\zak) + \hamvec{\geq 3}(\zak)\,,
\end{equation*}
satisfies,  for any 
$\zak \in B_{s_0, \R}(r) \cap H^{s}_\R$, 
        the quadratic tame estimates
        \begin{equation}
            \| \hamvec{\geq 2}(\zak) \|_{s-\sla{3}{2}} 
            \lesssim_{s} \| \zak\|_{s_0}\| \zak\|_s.
            \label{tameest:campo}
        \end{equation}
            	\end{proposition}
\begin{proof}
We first recall that (see \eqref{eq:1.2}, \eqref{funz:bpiccolo} and \eqref{def:V-B})
\begin{align*}
    &\vect{\pa_t \eta}{\pa_t \psi}= \vect{\tX^{(\eta)}(\eta,\psi)}{\tX^{(\psi)}(\eta,\psi)}, \qquad \;\;
     \begin{aligned}
    \tX^{(\eta)}(\eta,\psi)&= G(\eta)\psi, \\
    \tX^{(\psi)}(\eta,\psi)&= 
    - \eta+ \kap \div \left( \tb^{\sla{1}{2}} \nabla \eta\right)-\tfrac12 |\nabla \psi|^2 + \tfrac12 \tb^{-1}\tB^2.
    \end{aligned}
\end{align*}
Thanks to \Cref{lem:an_DN}, for any $s\geq s_0$ there is $r'>0$ such that  
\begin{align}
    &\tX^{(\eta)}: B_{s_0}(r')\cap H^{s+\sla{1}{4}}_0(\T^2;\R)\times \dot H^{s-\sla{1}{4}}(\T^2;\R)\to H^{s-\sla{5}{4}}_0(\T^2;\R),\nonumber\\
    &\tX^{(\psi)}: B_{s_0}(r')\cap H^{s+\sla{1}{4}}_0(\T^2;\R)\times \dot H^{s-\sla{1}{4}}(\T^2;\R)\to \dot H^{s-\sla{7}{4}}(\T^2;\R).
    \label{XetaXpsi}
\end{align}
Moreover, by \eqref{zak}, we have $ \vect{\eta}{\psi}= \cM^{-1} \zak$, where 
\begin{align}
    \cM^{-1} \in \cL \left(H^s_\R;  H^{s+\sla{1}{4}}_0(\T^2;\R)\times \dot H^{s-\sla{1}{4}}(\T^2;\R)\right).
    \label{emme-1}
\end{align}
By boundedness of $ \cM^{-1}$ there is $ r>0$ such that we have also
\begin{align}
    \cM^{-1} \colon B_{s_0}(r)\cap H^s_\R \to  B_{s_0}(r')\cap H^{s+\sla{1}{4}}_0(\T^2;\R)\times \dot H^{s-\sla{1}{4}}(\T^2;\R).
    \label{emme-2}
\end{align}
Moreover we have 
\begin{align}
\pa_t \zak= \vect{\pa_t \zetina}{\pa_t \ov{\zetina}} 
= \frac{1}{\sqrt{2}}\vect{\tM(D)^{-1}\circ \tX^{(\eta)}\circ \cM^{-1}(\zak)+ \ii \tM(D) \circ \tX^{(\psi)}\circ \cM^{-1}(\zak)}{\tM(D)^{-1}\circ \tX^{(\eta)}\circ \cM^{-1}(\zak)- \ii \tM(D)\circ \tX^{(\psi)}\circ \cM^{-1}(\zak)}=: \hamvec{\cH}(\zak).
    \label{hamvecX}
\end{align}
Gathering \eqref{emme-2} and \eqref{XetaXpsi}, and using the very definition \eqref{hamvecX}, we deduce that
\begin{align}
    \hamvec{\cH}: B_{s_0}(r)\cap H^s_\R
    \to H^{s-\sla{3}{2}}_\R\,,
    \qquad \text{is analytic.}
    \label{ana:hamvec}
\end{align}
We finally prove the expansion in \eqref{X:zak_esp}. 
First one has $ \hamvec{X}(0)=0$. Moreover 
\begin{align*}
    \di_\zak \hamvec{X}(0)= \cM \begin{pmatrix}
        \di_\eta \tX^{(\eta)}(0)&\di_\psi \tX^{(\eta)}(0) \\
        \di_\eta \tX^{(\psi)}(0)& \di_\psi \tX^{(\psi)}(0)
    \end{pmatrix}\cM^{-1}= \begin{pmatrix}
        -\ii \Lambda(D)& 0 \\ 0& \ii \Lambda(D)
    \end{pmatrix}= -\ii \vomega(D).
\end{align*}
Then the expansion in \eqref{X:zak_esp} follows from the analyticity established in \eqref{ana:hamvec} which automatically provides the estimates in \eqref{espansione_campoanal}- \eqref{est:XH} for $\zak \in B_{s}(r)$. In order to prove the expansion and tame estimate \eqref{tameest:campo} for  $\zak \in B_{s_0}(r) \cap H^{s}_\R$ it is sufficient to combine the paralinearization of the system in \Cref{prop:paraWW}, the definition of the good unknown \eqref{def:good}, estimate \eqref{actionSob} for the paradifferential operators, \eqref{bound:smoo} with \eqref{emme-1}-\eqref{emme-2}. This concludes the proof. 
\end{proof}  
We conclude this section with recalling  the non-resonance conditions with loss of derivatives proved in \cite[Formula (2.3)]{IoP}.
\begin{lemma}
\label{lem:small_divisors}
    There is a zero measure set $\mathscr{N}\subset \R_+$ such that for any $\kap \in \R_{+}\setminus \mathscr{N}$ there is a constant $c>0$ such that the following holds. For any $j,k \in \Z^2\setminus \{0\}$ and $\sigma_1,\sigma_2 \in \{\pm\}$ one has the lower bound 
    \begin{align}
        | \sigma_1 \Lambda(j)+ \sigma_2\Lambda(k)- \Lambda(\sigma_1 j+ \sigma_2 k)| \geq \frac{c}{\max\{|j|,|k|\}^{2}\min\{|j|,|k|\}^{4}}\,.
        \label{eq:lower_bound}
    \end{align}
\end{lemma}
\begin{proof}
    We distinguish three cases: 
    \begin{enumerate}
        \item If $|j|\leq \min\{|k|, |\sigma_1 j+\sigma_2 k|\}$. Then applying formula \cite[Formula (2.3)]{IoP} with $ \eta\leadsto \sigma_2k$ and $\xi \leadsto \sigma_1 j + \sigma_2 k$ we get \eqref{eq:lower_bound};
         \item If $|k|\leq \min\{|j|, |\sigma_1 j+\sigma_2 k|\}$. Then applying formula \cite[Formula (2.3)]{IoP} with $ \eta\leadsto \sigma_1 j$ and $\xi \leadsto \sigma_1 j + \sigma_2 k$ we get \eqref{eq:lower_bound};
          \item If $|\sigma_1 j+\sigma_2 k|\leq \min\{|k|,|j| \}$. Then applying formula \cite[Formula (2.3)]{IoP} with $ \eta\leadsto -\sigma_2k$ and $\xi \leadsto \sigma_1 j$ we get \eqref{eq:lower_bound}.
    \end{enumerate}
    This concludes the proof.
\end{proof}
\begin{remark}
   In fact, \cite[Formula (2.3)]{IoP} yields a sharper estimate, where the large factor $\max{ |j|,|k|}^{2}$ in the denominator on the right-hand side of \eqref{eq:lower_bound} is replaced by
$\max{ |j|,|k|}^{3/2}\log(\max{ |j|,|k|})^{-1-\alpha}$ for some small $\alpha>0$.
Since this refinement is not needed in our analysis, we state the weaker bound \eqref{eq:lower_bound} for the sake of simplicity.
\end{remark}
	    \subsection{Water waves related bounds}
    In this section we collect all the properties of the fundamental water waves quantities. Let us recall (see \eqref{def:V-B}) 
    the functions
\be \label{V_B:def}
\tV:= (\nabla \Phi)_{|y=\eta(t,x)}= 
\nabla \psi - \tB \nabla \eta\,, 
\qquad 
\tB:=(\pa_y \Phi)_{|y=\eta(t,x)}= 
\frac{G(\eta)\psi+ \nabla \eta\cdot \nabla \psi}{1+|\nabla \eta|^2}\,,
\ee
    which are the horizontal and vertical component of the fluid velocity at the boundary $\{ y=\eta(t,x)\}$. 
    \begin{lemma}[Expansion and bounds for $\tV$ and $\tB$]
    \label{V_B:lemma}
 There exists  $s_0 > 0$ such that for every 
 $\sigma \geq s_0+\frac14$, 
 there exists $r > 0$ with the property that for all 
$
\eta \in B_{s_0}(r) \cap H^{\sigma+\tfrac{1}{4}}_0$ 
{and} $\psi \in \dot{H}^{\sigma - \tfrac{1}{4}},
$ the functions $\tV$ and $\tB$ in \eqref{V_B:def} belong to $ H^{\s-\sla{5}{4}}$ and:
\begin{itemize}
    \item {\bf Boundedness:}
    \begin{align}\label{est:V+B}
    \| \tV\|_{\s-\sla{5}{4}}+\| \tB\|_{\s-\sla{5}{4}}
    \lesssim_{\s} 
    \| \psi\|_{\sigma-\sla{1}{4}}
    + \| \eta\|_{\sigma+\sla{1}{4}} \|\psi\|_{s_0}\,;
\end{align}
\item {\bf Taylor's expansions:} For any $\s\geq s_0+\frac14$ there is $r'>0$ such that for any  $\eta \in B_{\sigma+\sla{1}{4}}(r')$ and $\psi\in \dot H^{\s-\sla{1}{4}}$ we have
\begin{align}
&\tV=\tV_1+ \tV_{\geq 2}\,, 
\qquad &&\tV_1=\nabla \psi\,, 
\qquad &&\| \tV_{\geq 2}\|_{\s-\sla{5}{4}}
\lesssim_{\s} 
\| \eta\|_{\sigma+\sla{1}{4}} \| \psi\|_{\sigma-\sla{1}{4}}\label{V:exp1}
\\
&\tV=\tV_1+\tV_2+ \tV_{\geq 3}\,, 
\quad &&\tV_2=|D|\psi \,\nabla \eta\,, 
\qquad &&\| \tV_{\geq 3}\|_{\s-\sla{5}{4}}
\lesssim_{\s} \| \eta\|_{\sigma+\sla{1}{4}}^2
\| \psi\|_{\sigma-\sla{1}{4}}\,,\label{V:exp2}
\\
&\tB=\tB_1+ \tB_{\geq 2}\,, 
\qquad &&\tB_1:=|D| \psi\,, 
\qquad &&\| \tB_{\geq 2}\|_{\s-\sla{5}{4}}
\lesssim_{\s} \| \eta\|_{\sigma+\sla{1}{4}} 
\| \psi\|_{\sigma-\sla{1}{4}}\,.\label{B:exp}
\end{align}
\end{itemize}
    \end{lemma}
    \begin{proof}
        First of all it is classical (see e.g. \cite[Theorem 1.2]{BMV22}) that under the hypothesis of \Cref{V_B:lemma} the function
        \[
        \eta \mapsto G(\eta)\,, 
        \qquad 
        B_{s_0}(r)\cap H^{\s+\sla{1}{4}}_0
        \mapsto \cL\left(H^{\sigma-\sla{1}{4}};H^{\s-\sla{5}{4}}\right)\,,
        \]
        is analytic and satisfies tame estimates
        \begin{align}\label{stima:DN}
            \| G(\eta)\psi\|_{\s-\sla{5}{4}}\lesssim_\s 
            \| \psi\|_{\s-\sla{1}{4}}
            + \| \eta \|_{\s+\sla{1}{4}} \| \psi\|_{s_0}\,.
        \end{align}
        We start estimating $\tB$. From \eqref{V_B:def}, tame estimates for the product, 
        Moser's tame estimates and \eqref{stima:DN}, we get 
\begin{align}
\| \tB\|_{\s-\sla{5}{4}}&\lesssim_\s 
\left\| \frac{1}{1+|\nabla\eta|^2}\right\|_{s_0}
\left( \|G(\eta)\psi\|_{\s-\sla{5}{4}}
+\| \nabla \eta\cdot \nabla \psi\|_{\s-\sla{5}{4}}\right)
\nonumber
\\&\quad 
+\left\| \frac{1}{1+|\nabla\eta|^2}\right\|_{\s-\sla{5}{4}} 
\left( \|G(\eta)\psi\|_{2}
+\| \nabla \eta\cdot \nabla \psi\|_{2}\right)
\nonumber
\\
&\lesssim_\sigma 
\| \psi\|_{\s-\sla{1}{4}}
+ \| \eta \|_{\s+\sla{1}{4}} \| \psi\|_{s_0}
+ (1+ \| \eta\|_{3}\| \eta\|_{\s+\sla{1}{4}}) 
\| \psi\|_{s_0}
\nonumber
\\&
\lesssim_\sigma 
\| \psi\|_{\s-\sla{1}{4}}
+ \| \eta \|_{\s+\sla{1}{4}} \| \psi\|_{s_0}\,.
\label{est:B_proof}
\end{align}
This proves the estimate \eqref{est:V+B} 
for $ \tB$. 
Using tame estimates for the product 
and \eqref{est:B_proof} for $ \tB$, 
we get \eqref{est:V+B} also for $\tV$.

\noindent
We now prove the expansion in \eqref{V:exp1}. 
By \eqref{V_B:def}, we define 
$V_{\geq 2}:= \tB \nabla \eta$ that satisfies 
the quadratic estimate in \eqref{V:exp1} 
thanks to the algebra estimates for 
the product and \eqref{est:V+B}. 

\noindent
We then prove \eqref{B:exp}. We use the analyticity 
of $G(\eta)$ and we Taylor expand it 
\begin{align}\label{exp:DN}
G(\eta)= |D|+ G_{\geq 1}(\eta)\,, 
\qquad 
\| G_{\geq 1}(\eta)\|_{\cL(H^{\s-\sla{1}{4}};H^{\s-\sla{5}{4}})}
\lesssim_\sigma \| \eta\|_{\s+\sla{1}{4}}\,.
\end{align}
        Then, by \eqref{V_B:def}, we expand
        \begin{align}\label{B:exp_proof}
            \tB= |D|\psi+ \tB_{\geq 2}\,, 
            \qquad \tB_{\geq 2}:= G_{\geq 1}(\eta)\psi + \frac{\nabla \eta\cdot \nabla \psi-|\nabla \eta|^2 G(\eta)\psi}{1+|\nabla \eta|^2}\,.
        \end{align}
Using the algebra estimates of $H^{\s-\sla{5}{4}}$, the estimate in \eqref{exp:DN} and \eqref{stima:DN}, we get
\begin{align*}
    \| \tB_{\geq 2}\|_{\s-\sla{5}{4}}&\lesssim_\s 
    \| \eta \|_{\s+\sla{1}{4}} \| \psi\|_{\s-\sla{1}{4}} 
    + \left\|\frac{1}{1+|\nabla \eta|^2}\right\|_{\s-\sla{5}{4}}
    \left( \| \eta\|_{\s+\sla{1}{4}}\| \psi\|_{\s-\sla{1}{4}}
    + \| \eta\|_{\s+\sla{1}{4}}^2
    \| G(\eta)\psi\|_{\s-\sla{5}{4}} \right)
    \\
    &\lesssim_\s   
\| \eta \|_{\s+\sla{1}{4}} \| \psi\|_{\s-\sla{1}{4}}\,,
\end{align*}
        proving \eqref{B:exp}. We finally prove \eqref{V:exp2}. Thanks to \eqref{V_B:def} and \eqref{B:exp} we write
        \begin{align*}
\tV= \nabla \psi-\tB \nabla \eta= \tV_1+ \tV_2 +\tV_{\geq 3}\,, 
\qquad \tV_{\geq 3}:= \tB_{\geq 2}\nabla \eta\,.
    \end{align*}
        Then the estimate \eqref{V:exp2} for $\tV_{\geq 3}$ follows from the algebra estimate of $H^{\s-\sla{5}{4}}$ and  \eqref{B:exp}.
    \end{proof}
    

    For $\s\in \R$ we define also 
    $X^\s:= H^{\s+\sla{1}{4}}_0\times \dot H^{\s-\sla{1}{4}}$ with norm
    $$
    \| (\eta,\psi)\|_{X^\s}= \| \eta\|_{\s+\sla{1}{4}}+\| \psi\|_{\s-\sla{1}{4}}\,.
    $$
\begin{lemma}[Expansion and bound for $\ta$]
There is $s_0>0$ such that for 
any $ \s\geq s_0+\frac14$ there is $r>0$ 
and a function
\begin{align}\label{def:Tayfunctional}
\ta: \Big( B_{\s+\sla{1}{4}}(r) \cap H^{\s+ \sla{1}{4}}_0 \Big) \times \Big( B_{\s- \sla{1}{4}}(r) \cap
\dot H^{\s-\sla{1}{4}} \Big) \mapsto H^{\s-\sla{11}{4}}\,, 
\qquad (\eta,\psi) \mapsto \ta(\eta,\psi;x) \,,
\end{align}
with the property that for any solution 
$(\eta(t),\psi(t)) \in \Big( B_{\s+\sla{1}{4}}(r) \cap H^{\s+ \sla{1}{4}}_0 \Big) \times \Big( B_{\s- \sla{1}{4}}(r) \cap \dot H^{\s-\sla{1}{4}} \Big)$ of the system \eqref{eq:1.2}, 
the function $\ta$ defined in \eqref{def:a-Tay} 
is in $H^{\s-\sla{11}{4}}$, 
coincides with \eqref{def:Tayfunctional} and 
\begin{itemize}
    \item {\bf Boundedness:}
\begin{align*}
\| \ta\|_{\s-\sla{11}{4}}
\lesssim_\sigma 
\| (\eta,\psi)\|_{X^\s}+\| (\eta,\psi)\|_{X^\s}^2 \,;
\end{align*}
     \item {\bf Taylor's expansion:}
\begin{align*}
&\ta= \ta_1 +\ta_{\geq 2}\,, 
&& \ta_1:= -(|D|+\kap|D|^3)\eta\,, 
&& \| \ta_{\geq 2}\|_{\s-\sla{11}{4}}
\lesssim_\s \| (\eta,\psi)\|_{X^\s}^2\,.
\end{align*}
\end{itemize}
\end{lemma}
	\begin{proof}
    First, recalling \eqref{eq:1.2} we introduce 
    the notation
    \begin{align}\label{def:meancurv}
\div\left(\frac{\nabla \eta}{\sqrt{1+|\nabla\eta|^2}}\right)
= \Delta \eta- \div\left(\frac{\nabla \eta|\nabla \eta|^2}{\sqrt{1+|\nabla\eta|^2}
\left(1+\sqrt{1+|\nabla\eta|^2}\right)}\right)
=\Delta \eta+ H_{\geq 3}(\eta)\,.
\end{align}
Then, using the Taylor expansion in \eqref{B:exp} and
the explicit formula \eqref{B:exp_proof}, 
we compute the time derivative 
    \begin{align}
\pa_t \tB&= 
| D |\pa_t \psi+ \di_\psi \tB_{\geq 2}(\eta,\psi)[\pa_t \psi]
+ \di_\eta \tB_{\geq 2}(\eta,\psi)[\pa_t \eta]
\nonumber
\\
&= -(|D| +\kap |D|^3) \eta
+ | D| \tX_{\geq 2}^{(\psi)}(\eta,\psi)
+ \di_\psi \tB_{\geq 2}(\eta,\psi)[\tX^{(\psi)}(\eta,\psi)]
+\di_\eta \tB_{\geq 2}(\eta,\psi)[G(\eta)\psi]
\nonumber
\\
&= -(|D| +\kap |D|^3) \eta+R_1+R_2+R_3\,,
\label{exp:patB_proof}
\end{align}
       where (recall also \eqref{eq:1.2}-\eqref{def:meancurv})
       \begin{align*}
           \tX_{\geq 2}^{(\psi)}(\eta,\psi):=\kap  H_{\geq 3}(\eta) + \frac12 | \nabla \psi|^2
           + (1+|\nabla \eta|^2)\tB^2\,.
       \end{align*}
       Using the algebra estimate for the product, the 
Moser's composition estimate and \eqref{est:V+B}, we have
\begin{align}
\| \tX_{\geq 2}^{(\psi)}(\eta,\psi)\|_{\sigma-\sla{7}{4}}
&\lesssim_\sigma 
\left\| 
\frac{\nabla \eta|\nabla \eta|^2}
{\sqrt{1+|\nabla\eta|^2}(1+\sqrt{1+|\nabla\eta|^2})}
\right\|_{\sigma -\frac34}
+ \| \psi\|_{\s-\sla{1}{4}}^2+ \| \tB\|_{\s-\sla{7}{4}}^2
\notag
\\
&\lesssim_\sigma 
\| \eta\|_{\s+\sla{1}{4}}^3
+ \| \psi\|_{\s-\sla{1}{4}}^2
\lesssim_\s \| (\eta,\psi)\|_{X^\s}^2\,, 
\label{est:Xpsi}
\end{align}
         and
         \begin{align}
             \| \tX^{(\psi)}(\eta, \psi)\|_{\sigma-\sla{7}{4}}\lesssim_\s \| \eta+\kap\Delta \eta\|_{\sigma-\sla{7}{4}}+ \| \tX_{\geq 2}^{(\psi)}\|_{\sigma-\sla{7}{4}}\lesssim_\s \|\eta\|_{\s+\sla{1}{4}}+ \| (\eta,\psi)\|_{X^\s}^2\lesssim_\s \| (\eta,\psi)\|_{X^\s}.\label{est:Xpsi2}
         \end{align}
         We now estimate each term $R_j$, $j=1,\, 2,\, 3$. First of all, using \eqref{est:Xpsi}, we get
         \begin{align}
         \| R_1\|_{\s-\sla{11}{4}}\lesssim_\s \| \tX_{\geq 2}^{(\psi)}(\eta,\psi)\|_{\sigma-\sla{7}{4}}\lesssim_\s  \| (\eta,\psi)\|_{X^\s}^2.\label{R1}
         \end{align}
Then, using the explicit expression of $ \tB_{\geq 2}$ in \eqref{B:exp_proof}, we get 
\begin{align}
\| R_2\|_{\s-\sla{11}{4}}&= 
\left\|  G_{\geq 1}(\eta)\tX^{(\psi)} 
+ \frac{\nabla \eta\cdot \nabla \tX^{(\psi)}
-|\nabla \eta|^2 G(\eta)\tX^{(\psi)}}{1+|\nabla \eta|^2}
\right\|_{\s-\sla{11}{4}}
\nonumber
    \\
& \lesssim_\s  
\|  G_{\geq 1}(\eta)\tX^{(\psi)}\|_{\s-\sla{11}{4}}
+ \left\|\frac{\nabla \eta\cdot \nabla \tX^{(\psi)}
-|\nabla \eta|^2 G(\eta)\tX^{(\psi)}}{1+|\nabla \eta|^2}
\right\|_{\s-\sla{11}{4}}
\nonumber
\\
&\overset{\eqref{exp:DN}}{\lesssim_{\s}}  
\| \eta\|_{\s+\sla{1}{4}}\| \tX^{(\psi)}\|_{\s-\sla{7}{4}}
+ \| \eta\|_{\s-\sla{7}{4}}\| \tX^{(\psi)}\|_{\s-\sla{7}{4}}
+ \| \eta\|_{\s-\sla{7}{4}}^2
\| \tX^{(\psi)}\|_{\s-\sla{7}{4}} 
\nonumber
\\
& \overset{\eqref{est:Xpsi2}}{\lesssim_{\s}}  
\| (\eta,\psi)\|_{X^\s}^2\,.
\label{R2}
\end{align}
To estimate $ R_3$
we first observe that,  
by the analyticity of $G(\eta)$, we get
\begin{align*}
\|\di_\eta G_{\geq 1}(\eta)[\hat \eta] 
\psi\|_{ \s-\sla{11}{4}} 
\lesssim_{\s} 
\| \hat \eta \|_{\s- \sla{7}{4}} 
\|\psi \|_{\s- \sla{7}{4}}\,.
\end{align*}
Then, by \eqref{B:exp_proof}, $\tB_{\geq2}$ 
is analytic with respect to $ \eta$ with estimate
\begin{align*}
\| \di_\eta\tB_{\geq2}[ \hat \eta] \|_{\s-\sla{11}{4}} 
\lesssim_{\s} 
\|\hat \eta\|_{\s-\sla{7}{4}}\| (\eta,\psi)\|_{X^\s}\,.
\end{align*}
Then, using also \eqref{stima:DN}, we finally get
\begin{align}
\| R_3 \|_{\s-\sla{11}{4}}\lesssim_{\s} 
\|G(\eta)\psi\|_{\s-\sla{7}{4}}\| (\eta,\psi)\|_{\s}
\lesssim_{\s} \| (\eta,\psi)\|_{X^\s}^2.\label{R3}
\end{align}
    Finally, by \eqref{exp:patB_proof}, we have 
    \begin{equation*}
        \ta= \ta_1+ \ta_{\geq 2}\,, \qquad \ta_{\geq 2}:= R_1+R_2+R_3+ \tV\cdot \nabla \tB\,. 
    \end{equation*}
    Moreover, by algebra estimate of $H^{\s-\sla{11}{4}} $, we get 
    \begin{align}
    \| \tV\cdot \nabla \tB\|_{\s-\sla{11}{4}} 
        \lesssim_{\s} 
        \| \tV \|_{\s-\sla{11}{4}}
        \|\tB\|_{\s-\sla{7}{4}}
        \overset{\eqref{est:V+B}}{\lesssim_{\s}} 
        \| (\eta,\psi)\|_{X^\s}^2\,.
        \label{VB_ultima}
    \end{align}
    Gathering \cref{R1,R2,R3,VB_ultima} we obtain the thesis.
	\end{proof}

	\section{The Complex Good Variable}\label{sec:NFclassica}

In this section, we reformulate the water waves system in paradifferential form and introduce a suitable complex change of variables that symmetrizes the principal part of the equations. 

This allows us to reduce the system to a diagonal form up to lower order and smoothing remainders, which is the starting point for the analysis of the dynamics.

In addition, we perform a normal form reduction of the quadratic transport term, removing its $1$-homogeneous component without encountering small divisor issues, thanks to its gradient structure.
	\subsection{Paralinearization and Alinhac's good unknown}\label{sec good unknown}
    Following \cite{ABZ2011_2}, we introduce Alinhac’s celebrated good unknown:
\begin{equation}
\omega:= \psi- \Opbw{\tB}\eta\,,
\label{def:good}
\end{equation}
where $ \tB$ is the vertical component 
of the velocity field defined in \eqref{def:V-B}.  
The variable $\omega$ enables us to 
recast the water waves equations 
in paradifferential form. 
In particular, the following proposition, 
whose proof is contained in \Cref{app:paraWW}, holds.
    \begin{proposition}\label{prop:paraWW}
        The water waves system, for the variables $(\eta,\omega)$ defined in \eqref{def:good}, reads
\begin{align}
\pa_t \eta&= \Opbw{\lambda}\omega
+ \Opbw{-\ii \tV\cdot \xi-\tfrac12 \div (\tV)}\eta
+ R_1(\eta) \omega 
+ R_{\geq 2}(\eta)\omega
+R_{\geq 2} (\eta,\psi)\eta
\label{eq:eta}
\\
\pa_t \omega &= - \Opbw{{\kap} \sym{\th}{}{2}+1
+\ta+ \kap \sym{\th}{\geq 2}{0}}\eta
+\Opbw{-\ii \tV\cdot \xi+\tfrac12\div(\tV)}\omega
+ R_1(\eta, \psi)\omega 
+ \breve{R}_{\geq 2}(\eta,\psi)\eta
\label{eq:omega}
\end{align}
where 
        \begin{itemize}
            \item The symbol 
            $\fun{\lambda}{}(\eta;x,\xi)$ expands as 
            \begin{align}
    \lambda(\eta;x,\xi)= 
    \sym{\lambda}{}{1}(\eta;x,\xi)
    + \sym{\lambda}{}{0}(\eta;x,\xi)
    +\sym{\lambda}{}{-1}(\eta;x,\xi)\,,
           \label{esp:lambda}
            \end{align}
           where:
           \begin{itemize}
            \item $\sym{\lambda}{}{1}$ 
            belongs to $\Sigma \Gamma_0^{1}[r,4]$ 
having the explicit form
            \begin{align}
                \sym{\lambda}{}{1}:= 
                \sqrt{(1+|\nabla \eta|^2)|\xi|^2 
                - (\nabla \eta \cdot \xi)^2 }\,;   
                \label{def:lam1}
            \end{align}
               \item  $\sym{\lambda}{}{0}$ belongs to $\Sigma \Gamma_1^{0}[r,4]$ having the explicit form
               \begin{align}
                   \sym{\lambda}{}{0} := 
\frac{(1 + |\nabla \eta|^2)^2}{2\sym{\lambda}{}{1}}
\left\{
\frac{\sym{\lambda}{}{1}}{1 + |\nabla \eta|^2},
\frac{\xi \cdot \nabla \eta}{1 + |\nabla \eta|^2}
\right\}
+ \tfrac{1}{2} \Delta \eta\,; 
\label{def:lam0}
               \end{align}
               \item $\sym{\lambda}{}{-1}$ belongs to $ \Sigma\Gamma_1^{-1}[r,2]$, with expansion 
               \begin{align*}
                   \sym{\lambda}{}{-1}=\sym{\lambda}{1}{-1}+ \sym{\lambda}{\geq 2}{-1}\,,  
               \end{align*}
where $\sym{\lambda}{1}{-1} \in \wt \Gamma_1^{-1}$ 
is real and 
$\sym{\lambda}{\geq 2}{-1} \in \Gamma_{\geq 2}^{-1}[r]$;
    \end{itemize}
        
            \item the symbol $\sym{\th}{}{2}=\sym{\th}{}{2}(\eta;x,\xi)$  belongs to $\Sigma \Gamma_0^2[r,4]$ and is explicitly given by 
      \begin{align}
          \th^{(2)}:=\frac{1}{(1+|\nabla\eta|^2)^\sla{1}{2}}\left( | \xi|^2 - \frac{(\nabla \eta \cdot \xi)^2}{1+|\nabla\eta|^2}\right)=\frac{[\lambda^{(1)}]^2}{(1+|\nabla\eta|^2)^\sla{3}{2}}\,;\label{def:h}
      \end{align}      
      
            \item the function $ \ta=\ta(\eta,\psi;x) \in \Sigma\cF^{\R}_{1}[r,2] $ is the Taylor coefficient defined in \eqref{def:a-Tay} and the function $\sym{\th}{\geq 2}{0}$ belongs to the class $\cF_{\geq 2}^{\R}[r]$ (see formula \eqref{def:h0});
            \item $R_1(\eta)$ and $R_1(\eta,\psi)$ are $1$-homogeneous smoothing remainders in $\wt \cR^{-\vr}_1$;
            \item $R_{\geq 2} (\eta)$, $R_{\geq 2} (\eta,\psi)$ and $\breve{R}_{\geq 2} (\eta,\psi)$ are non-homogeneous smoothing remainders in $ \cR^{-\vr}_{\geq 2}[r]$.
        \end{itemize}
    \end{proposition}
    The next lemma provides some specific properties of the symbols appearing 
    in \Cref{prop:paraWW}.
\begin{lemma}[Expansion and bounds for the symbols 
$\sym{\lambda}{}{1}$, $\sym{\lambda}{}{0}$ and $\sym{\th}{}{2}$]\label{Lem:lamh}
There is $r>0$ such that $\sym{\lambda}{}{1}\in \Sigma \Gamma_{0}^1[r,4]$, 
$\sym{\lambda}{}{0}\in \Sigma \Gamma_{1}^0[r,4]$ and 
$\sym{\th}{}{2}\in \Sigma \Gamma_{0}^2[r,4]$ admit the expansions
\begin{align}
&\sym{\lambda}{}{1}= 
| \xi| + \sym{\lambda}{2}{1}+ \sym{\lambda}{\geq 4}{1}\,, 
    && 
    \sym{\lambda}{2}{1}\in \wt \Gamma_2^1\,, 
\quad 
    \sym{\lambda}{\geq 4}{1}\in \Gamma_{\geq 4}^1[r]\,;
    \label{exp:lam12}
    \\
    &\sym{\lambda}{}{0}=
    \sym{\lambda}{1}{0}+ \sym{\lambda}{\geq 3}{0}\,, 
    && 
    \sym{\lambda}{1}{0}\in \wt\Gamma_1^0\,,
    \quad \sym{\lambda}{\geq 3}{0}\in 
    \Gamma_{\geq 3}^0[r];
    \label{exp:lam0}
    \\
    &\sym{\th}{}{2}= |\xi|^2 + \sym{\th}{2}{2}+ \sym{\th}{\geq 4}{2}\,, 
    && \sym{\th}{2}{2}\in \wt \Gamma_2^2\,, 
    \quad \sym{\th}{\geq 4}{2}\in \Gamma_{\geq 4}^2[r]\,.
    \label{exp:h2}
\end{align}

In addition, there is $s_0>2$ such that for any 
$ \s \geq s_0$ there is $0<r'\leq r$ 
such that the symbols $\sym{\lambda}{}{1}(\eta;\bigcdot)$, $\sym{\lambda}{}{0}(\eta;\bigcdot)$ and $\sym{\th}{}{2}(\eta;\bigcdot)$ are analytic with respect to  $\eta\in B_{\s+\sla{1}{4}}(r')$, fulfilling the estimates
            \begin{align}
                &|\di_\eta \sym{\lambda}{}{1}(\eta;x,\xi)[\hat \eta]|_{1,\s-\sla{7}{4}}+|\di_\eta\sym{\th}{}{2}(\eta;x,\xi)[\hat \eta]|_{2,\s-\sla{7}{4}}\lesssim_\s \| \hat \eta \|_{\s+\sla{1}{4}}\|  \eta \|_{\s+\sla{1}{4}}\nonumber\\
               &|\di_\eta \sym{\lambda}{\geq 3}{0}(\eta;x,\xi)[\hat \eta]|_{0,\s-\sla{7}{4}}\lesssim_{\s}\| \eta\|_{\s+\sla{1}{4}}^2\| \hat \eta\|_{\s+\sla{1}{4}}.
               \label{stima:difflambda}
            \end{align}
    \end{lemma}
     \begin{proof}
It follows by direct inspection using formulas\, 
\eqref{def:lam1}, \eqref{def:lam0} and \eqref{def:h}.
         \end{proof}
    
	\subsection{Symmetrization and complex coordinates}\label{sec complex coordinates}
    
	In this section we symmetrize equation \eqref{eq:eta}-\eqref{eq:omega}. To this end we introduce the symbol 
    \begin{align}
\tQ\equiv \tQ(\eta;x,\xi):=
\left( \frac{\sym{\lambda}{}{1}(\eta;x,\xi)}{\kap \sym{\th}{}{2}(\eta;x,\xi)+1}\right)^{\sla{1}{4}}\,,
        \label{def:MQ}
    \end{align}
	which is a perturbation of the symmetrizer 
    $ \tM(\xi)$ defined in \eqref{zetina}. 
    The symbol $\tQ$ symmetrizes the  principal 
    matrix of symbols in the system  
    \eqref{eq:eta}-\eqref{eq:omega}, namely 
\begin{align*}
    \begin{pmatrix}
        \tQ^{-1} & 0 \\ 0 & \tQ
    \end{pmatrix}\begin{pmatrix}
        0& \sym{\lambda}{}{1} \\ -(1+ \kap \sym{\th}{}{2}) &0
    \end{pmatrix} \begin{pmatrix}
        \tQ & 0 \\ 0 &\tQ^{-1}
    \end{pmatrix}= \begin{pmatrix}
        0 & \sym{\Sigma}{}{\sla{3}{2}}\\ 
        - \sym{\Sigma}{}{\sla{3}{2}} & 0
    \end{pmatrix}
\end{align*}
where the modified dispersion relation 
$\sym{\Sigma}{}{\sla{3}{2}}$ is defined as 
  \begin{align}
            \sym{\Sigma}{}{\sla{3}{2}}:=\sqrt{\sym{\lambda}{}{1} (\kap\sym{\th}{}{2}+1)}\,. 
            \label{def:sd}
        \end{align}
       In view of the expansions 
       \eqref{exp:lam12}-\eqref{exp:h2} 
       in Lemma \ref{Lem:lamh}, one checks that the symbols in \eqref{def:sd}
admits the expansion
\begin{equation}\label{def:sdBIS}
\sym{\Sigma}{}{\sla{3}{2}}=\Lambda(\x)+\sym{\Sigma}{\geq2}{\sla{3}{2}}\,,
\qquad 
\sym{\Sigma}{\geq2}{\sla{3}{2}}\in 
\Gamma^{\sla{3}{2}}_{\geq 2}[r]\,,
\end{equation}
where $\Lambda(\x)$ is given in \eqref{def:dispersion}.

The main  result of the section is the following.
\begin{proposition}[Symmetrization] 
 \label{diag.ord0}
 Let $\vr >1$. There is  $s_0> 0$ and a bounded and 
 invertible linear transformation $\bcQ(\eta)$ such 
 that for any $s \in \R$  there is  
 $r=r_s  >0 $ such that for any $\eta \in B_{s_0}(r)$ 
  \begin{align}
\bcQ(\eta)\colon H_0^{s+\sla{1}{4}}(\T^2;\R)\times 
\dot H^{s-\sla{1}{4}}(\T^2;\R) 
\to H^s_\R\,,
\label{QQspazietti}
\end{align}
and if $(\eta, \psi)(t) \in B_{s_0}(r) \times B_{s_0}(r)$ solves \eqref{eq:1.2} then the variable
  $U:= \bcQ(\eta)\vect{\eta}{\omega}$, 
  with $\omega$ defined in  \eqref{def:good},  solves (recall \eqref{zak} and the notation in \eqref{def:vec_out})
  \begin{equation}\label{eq:U}
\begin{aligned}
\pa_t U&= \vOpbw{-\ii \left[\sym{\Sigma}{}{\sla{3}{2}}
+ \tV\cdot \xi+\tfrac12 \tQ^{-2} \sym{\lambda}{}{0}
+ \sym{a}{1}{0}\right]}U+ \zOpbw{\tfrac{\ii}{2}  \tQ^{-2} \sym{\lambda}{}{0}
+ \sym{b}{1}{0} } U 
\\
&+ \bR_1(\zak)U+ \bB_{\geq 2}(\zak)U\,, 
\end{aligned}
\end{equation}
where the real vector of functions  $ \tV \in  \cF^{\R}_{\geq 1}[r]$ is defined in \eqref{def:V-B}, the real symbol  $ \sym{\lambda}{}{0} \in \Gamma_{\geq 1}^0[r]$ is defined in \eqref{def:lam0} and 
\begin{itemize}
\item the real symbol $\sym{\Sigma}{}{\sla{3}{2}}$ is defined in 
\eqref{def:sd}-\eqref{def:sdBIS}, 
the symbol $\fun{\tQ}{}$ in \eqref{def:MQ} belongs to $\Sigma \Gamma_0^{-\sla{1}{4}}[r,4]$;
    \item $\sym{a}{1}{0}$, $\sym{b}{1}{0}$ are $1$-homogeneous  symbols in $\wt \Gamma^0_{1}$ and  $\sym{a}{1}{0}$ is real;
    \item $\bR_1(\zak)$  is a real-to-real matrix of smoothing remainders in $ \wt\cR^{-\vr}_1$;
      \item For any $\zak\in B_{s_0, \R}(r)\cap H^s_\R$, $\bB_{\geq 2}(\zak)$ is a real-to-real a matrix of bounded operators in $ \cL(H^s_\R, H^s_\R)$ satisfying
      \begin{equation}
      \| \bB_{\geq 2}(\zak) U \|_s \lesssim_s 
      \| \zak\|_{s_0}^2 \| U\|_s
      + \| \zak\|_s\| \zak\|_{s_0} \| U\|_{s_0} \,, 
      \qquad \forall U\in H^s_\R\,.
      \label{stima_quadratica0}
      \end{equation}
\end{itemize}
Moreover the map $\Opbw{\fun{\tQ}{}}$ is invertible and  the complex variable $U$ has the form
\begin{equation} \label{def:cQ}
U=\vect{u}{\bar{u}}=\bcQ(\eta)\vect{\eta}{\omega}\qquad {\rm with }\qquad 
u:= \frac{1}{\sqrt{2}}\left[\Opbw{\tQ(\eta;x,\xi)}^{-1}\eta+\ii \Opbw{\tQ(\eta;x,\xi)}\omega\right]\,.
\end{equation}
In particular, for any $s\geq s_0$, the following are equivalent: $(\eta,\omega)\in H^{s+\sla{1}{4}}\times H^{s-\sla{1}{4}}$,  $U\in H^s_\R$, $\zak\in H^s_\R$, with estimates
\begin{align}
\| U\|_{s}\sim_s 
\| \eta\|_{s+\sla{1}{4}}
+ \| \omega\|_{s-\sla{1}{4}}
\sim_s \| \zak\|_s\,.
   \label{equiv:U_etaomega}
\end{align}
\end{proposition}
In order to prove the proposition above 
we need a preliminary result about the  
properties of the symbol $\tQ$, $\tQ^{-1}$ and 
$\sym{\Sigma}{}{\sla{3}{2}}$ 
in the sense of  \Cref{rem:symbols_modozero}. 
In particular in the next lemma  
we construct the conjugating map $\bcQ(\eta)$.
\begin{lemma}[Conjugating map]\label{lem:Q-Qinv}
There is $r>0$ such that (recall \eqref{def:MQ}, \eqref{def:sd}) 
the symbol $\tQ(\eta;x,\xi)$ belongs to 
$ \Sigma \Gamma_{0}^{-\sla{1}{4}}[r,4]$,  
the symbol $\tQ^{-1}(\eta;x,\xi)$ belongs to 
$\Sigma \Gamma_0^{\sla{1}{4}}[r,4]$ 
and the symbol $ \sym{\Sigma}{}{\sla{3}{2}}$ 
belongs to $ \Sigma \Gamma_0^{\sla{3}{2}}[r,4]$ 
and they admit the expansions
\begin{align}
&\tQ(\eta;x,\xi)= 
\tM(\xi)+ \fun{\tQ}{\geq 2}(\eta;x,\xi)\,, 
&& \fun{\tQ}{\geq 2}(\eta;x,\xi)
\in \Sigma \Gamma_2^{-\sla{1}{4}}[r,4];
\label{esp:Q}
    \\
&\tQ^{-1}(\eta;x,\xi)= \tM^{-1}(\xi)
+ \fun{\tQ^-}{\geq 2}(\eta;x,\xi)\,, 
&& \fun{\tQ^-}{\geq 2}(\eta;x,\xi)
\in \Sigma \Gamma_2^{\sla{1}{4}}[r,4]\,;
\nonumber 
\\
&\sym{\Sigma}{}{\sla{3}{2}}(\eta;x,\xi)= 
\Lambda(\xi)+ \sym{\Sigma}{2}{\sla{3}{2}}(\eta;x,\xi)
+\sym{\Sigma}{\geq 4}{\sla{3}{2}}(\eta;x,\xi)\,, 
&& \sym{\Sigma}{2}{\sla{3}{2}}\in  \wt\Gamma_2^{\sla{3}{2}}\,, 
\quad 
\sym{\Sigma}{\geq 4}{\sla{3}{2}} 
\in  \Sigma \Gamma_{\geq 4}^{\sla{3}{2}}[r]\,.
\label{esp:sigma}
\end{align}
Moreover  there is $s_0> 0$ large enough such that for any 
$ \s \geq s_0$ there is $0<r'\leq r$ such that the symbols 
$\fun{\tQ}{}(\eta;\bigcdot)$, $\fun{\tQ^{-1}}{}(\eta;\bigcdot)$ 
and $\sym{\Sigma}{}{\sla{3}{2}}(\eta;\bigcdot)$ 
are analytic with respect to  
$\eta\in B_{\s+\sla{1}{4}}(r')$, fulfilling the estimates
\be \label{stima:dtQ}
\begin{aligned}
| \di_\eta \tQ(\eta;x,\xi)[ \hat \eta] |_{-\sla{1}{4},\s-\sla{7}{4}}
+| \di_\eta \tQ^{-1}(\eta;x,\xi)[ \hat \eta] |_{\sla{1}{4},\s-\sla{7}{4}}
&+| \di_\eta \sym{\Sigma}{}{\sla{3}{2}}(\eta;x,\xi)[ \hat \eta] |_{\sla{3}{2},\s-\sla{7}{4}}
\\
&\qquad \quad\qquad 
\lesssim_\s \| \eta\|_{\s+\sla{1}{4}}
\|\hat \eta \|_{\s+\sla{1}{4}}\,.
\end{aligned}
\ee
             In addition to this,
               for any $ s\in \R$, there is $ 0<r_s\leq r$ such that if 
               $\eta \in B_{s_0}(r_s)$, the operator 
               $ \Opbw{\tQ}\in \cL\left(\dot H^{s-\sla{1}{4}}(\T^2;\R); H^{s}_0(\T^2;\R)\right) $ 
               is invertible with inverse in 
               $\cL\left( H^{s}_0(\T^2;\R);\dot H^{s-\sla{1}{4}}(\T^2;\R)\right)$. 
               Moreover one has the estimates 
               \begin{align}
          \| \Opbw{\tQ}\|_{\cL(\dot H^{s-\sla{1}{4}};H^s_0)} + \| \Opbw{\tQ}^{-1}\|_{\cL(H^s_0;\dot H^{s-\sla{1}{4}})} \lesssim_s 1 \,,
          \qquad \textnormal{for any } s \in \R\,.\label{stima:Qinv}
       \end{align}
Finally  one has the expansion 
\be 
\Opbw{\tQ(\eta;x,\xi)}^{-1}= 
\Opbw{ \tQ^{-1}(\eta;x,\xi)}
+ R^{(-\sla{7}{4})}_{\geq 2}(\eta)\,, 
\qquad  
R^{(-\sla{7}{4})}_{\geq 2}(\eta)
\in\cL(H^s_0,\dot{H}^{s+\sla{7}{4}})\,,
\label{esp:Q-1}
\ee
         where the operator $R^{(-\sla{7}{4})}_{\geq 2}$ satisfies the estimates
        \begin{align}   
\| R^{(-\sla{7}{4})}_{\geq 2}(\eta)
\|_{\cL( H^s_0,\dot H^{s+\sla{7}{4}})} 
    \lesssim_s \|\eta\|_{s_0}^2\,,
    \qquad \textnormal{for any } s\in \R\,.
                \label{stima:RestoQinv}
        \end{align}
         \end{lemma}
         \begin{proof}
         The expansions \eqref{esp:Q}-\eqref{esp:sigma} 
         and the estimate \eqref{stima:dtQ} follow by direct inspection using the explicit formulas\, 
         \eqref{def:MQ}-\eqref{def:sd}.

\noindent
We now consider the operator $\Opbw{\tQ}$.
Since $\eta$ is small, the invertibility 
follows from a Neumann series argument. 
Indeed, recalling \eqref{esp:Q}, we have 
\[         
\Opbw{\tQ}= \tM(D) \left( \Id + A_0\right)\,, 
\qquad A_0:=\tM^{-1}(D)\Opbw{\fun{\tQ}{\geq 2}}\,.
\]
Note that, by \Cref{thm:action}, 
we get that $A_0\in \cL(\dot H^s; H^s_0)$ and, 
by \eqref{esp:Q},\eqref{stima:nonhom}, for some $s_0>0$, 
we have
\begin{align}
\| A_0\|_{\cL(\dot H^s;H^s_0)}\lesssim_s 
| \fun{\tQ}{\geq 2}|_{-\sla{1}{4},3} 
\stackrel{\eqref{stima:nonhom}}{\lesssim_s} 
\| \eta\|^2_{s_0}\,, 
\qquad \text{for any } s \in \R\,, \label{Qmeno}
\end{align}
for some $C_s>0$.
       Thus, for $\eta \in B_{s_0}(r)$ and $C_s r\leq \tfrac12$,  we invert $ \Id+A_0$ by Neumann series, proving that $ \Opbw{\tQ}^{-1}= (\Id+A_0)^{-1}\tM^{-1}(D)$ belongs to $\cL(H^s;H^{s-\sla{1}{4}})$ with estimate \eqref{stima:Qinv}. The corresponding estimate \eqref{stima:Qinv} for $\Opbw{\fun{\tQ}{}}$ follows directly from \Cref{thm:action}.
         Moreover, by symbolic calculus \Cref{compoparapara0}
\begin{align}
\Opbw{\tQ}\Opbw{\tQ^{-1}}= \Id+ 
\Opbw{\sym{\tQ}{\geq 2}{-2}}
+ R_{\geq 2}(\eta)\,, 
\quad 
\sym{\tQ}{\geq 2}{-2}\in \Gamma_{\geq 2}^{-2}[r]\,, 
\quad  
R_{\geq 2}(\eta)\in \cR^{-\vr }_{\geq 2}[r]\,. 
\label{tricketto} 
\end{align}
In particular $R_{\geq 2}(\eta)$ satisfies 
\begin{align}
\| R_{\geq 2}(\eta)\|_{\cL(H^s;H^{s+\vr})}\lesssim_s 
\| \eta \|_{s_0}^2\,.
    \label{errone}
\end{align}
Then, applying $\Opbw{\tQ}^{-1}$ to \eqref{tricketto}, we get
\begin{gather*}
\Opbw{\tQ}^{-1}= \Opbw{\tQ^{-1}}+ R^{(-\sla{7}{4})}_{\geq 2}\,, 
\\
R^{(-\sla{7}{4})}_{\geq 2}:= -\Opbw{\tQ}^{-1}\Opbw{\sym{\tQ}{\geq 2}{-2}}-\Opbw{\tQ}^{-1} R_{\geq 2}(\eta)\,.
\end{gather*}
Finally, combining \eqref{Qmeno}, \eqref{actionSob}, \eqref{stima:nonhom} for $\sym{\tQ}{\geq 2}{-2}$ and \eqref{errone}, we obtain \eqref{stima:RestoQinv}. This concludes the proof.
     \end{proof}
We are now in position to prove the main result of this section.

\begin{proof}[{\bf Proof of Proposition \ref{diag.ord0}}]
In view of Lemma \ref{lem:Q-Qinv} the map (see \eqref{def:cQ})
    \begin{gather}
U:= \vect{u}{\bar u}= \bcQ(\eta)\vect{\eta}{\omega}\,, 
\qquad 
 \bcQ(\eta):=\frac{1}{\sqrt{2}} \begin{pmatrix}
     \Opbw{\tQ}^{-1}&  \ii \Opbw{\tQ}   \\
     \Opbw{\tQ}^{-1}& -\ii \Opbw{\tQ}
 \end{pmatrix}\notag
     \end{gather}
     is well-defined and satisfies \eqref{QQspazietti}.
Note also that $ \bcQ$ is invertible with inverse 
\begin{align*}
\bcQ(\eta)^{-1}= \frac{1}{\sqrt{2}} 
\begin{pmatrix}
\Opbw{\tQ}&  \Opbw{\tQ}   \\
     -\ii \Opbw{\tQ}^{-1}& \ii \Opbw{\tQ}^{-1}
 \end{pmatrix}\colon H^s_\R(\T^2;\C^2) 
 \to H_0^{s+\sla{1}{4}}(\T^2;\R)\times 
 \dot H^{s-\sla{1}{4}}(\T^2;\R)
\end{align*}

 Therefore, by setting   
 \begin{equation}\label{def:AA12}
 \begin{aligned}
 &A_1:= \Opbw{-\ii \tV\cdot \xi-\tfrac{1}{2}\div(\tV)}\,, 
 \qquad 
 A_4:=\Opbw{-\ii \tV\cdot \xi+\tfrac{1}{2}\div(\tV)}\,,
 \\
 &A_2:= \Opbw{\lambda}\,,
 \qquad 
 A_3:=- \Opbw{ \kap \sym{\th}{}{2} +1+\ta+\sym{\th}{\geq 2}{0}}\,, 
 \end{aligned}
 \end{equation}
 in view a equations \eqref{eq:eta}-\eqref{eq:omega} for $(\eta,\omega)$, the variable $U$ solves the equation 
         \begin{equation}
         \pa_t U=
\bcQ
\begin{bmatrix}
A_1 & A_2 \\
A_3 & A_4
\end{bmatrix}
 \bcQ^{-1}U+\bcQ
\begin{bmatrix}
R_{\geq 2}(\eta,\psi) & R_1(\eta)+ R_{\geq 2}(\eta) \\
\breve{R}_{\geq 2}(\eta,\psi) & R_1(\eta,\psi)
\end{bmatrix}
 \bcQ^{-1}U +\pa_t \bcQ \bcQ^{-1}U\,.\label{eq:QQ-1}
  \end{equation}
  Explicit computations give
  \begin{align}
    \bcQ
\begin{bmatrix}
A_1 & A_2 \\
A_3 & A_4
\end{bmatrix}
 \bcQ^{-1}  = \frac12\begin{bmatrix}
A_+ & A_- \\
\overline{A_-} & \overline{A_+}
\end{bmatrix}\,, \qquad \bcQ
\begin{bmatrix}
R_{\geq 2}(\eta,\psi) & R_1(\eta)+ R_{\geq 2}(\eta) \\
\breve{R}_{\geq 2}(\eta,\psi) & R_1(\eta,\psi)
\end{bmatrix}
 \bcQ^{-1}= \frac12\begin{bmatrix}
R_+ & R_- \\
\overline{R_-} & \overline{R_+}
\end{bmatrix}
         \label{eq:conj_complex_Matrix}
  \end{align}
  where we defined 
  \begin{align*}
A_{\s}&:=
 \Opbw{\tQ}^{-1} A_1  \Opbw{\tQ} + \s \Opbw{\tQ} A_4  \Opbw{\tQ}^{-1} 
 \\&\qquad + \ii  \Opbw{\tQ} A_3  \Opbw{\tQ} 
 - \s \ii  \Opbw{\tQ}^{-1} A_2  \Opbw{\tQ}^{-1}\,,\\
 R_{\s}&:=
 \Opbw{\tQ}^{-1} R_{\geq 2}(\eta,\psi) \Opbw{\tQ} + \s \Opbw{\tQ} R_1(\eta,\psi)  \Opbw{\tQ}^{-1} 
 \\&\qquad + \ii  \Opbw{\tQ} \breve{R}_{\geq 2}(\eta,\psi)  \Opbw{\tQ} 
 - \s \ii  \Opbw{\tQ}^{-1} [R_1(\eta)+ R_{\geq 2}(\eta)]  \Opbw{\tQ}^{-1}\,,
 \qquad \s\in\{\pm\}\,,
  \end{align*}
       
We analyze each term.  
Recalling \eqref{def:AA12} we have, using 
 \Cref{compoparapara0}, 
\begin{align}
&\Opbw{\tQ}^{-1} A_1 \Opbw{\tQ}
+\s  \Opbw{\tQ}A_4\Opbw{\tQ}^{-1}
\notag
\\
&\qquad
= \Opbw{\tQ^{-1}\sha{\vr} (-\ii \tV\cdot \xi
-\tfrac12\div(\tV))\sha{\vr} \tQ+\s \tQ\sha{\vr} 
(-\ii \tV\cdot \xi+\tfrac12\div(\tV))\sha{\vr} 
\tQ^{-1}}+R_{\geq 1}(\zak) 
\label{intermedia_real}
\\
&\qquad
+ R^{(-\sla{7}{4})}_{\geq 2}(\eta) A_1\Opbw{\tQ}
+\sigma\Opbw{\tQ}A_4
R^{(-\sla{7}{4})}_{\geq 2}(\eta)
\notag
\\
&\qquad
= \Opbw{-\ii \tV\cdot \xi (1+\s)-\ii \sym{r}{\geq 1}{0,\sigma}}
+ R_1(\zak)+ R_{\geq 2}^{(-1)}(\zak)\,, 
\notag
\end{align}
where  $ \sym{r}{\geq 1}{0,+}$ is a  symbol in 
$\Sigma \Gamma_{1}^{-1}[r,2]$, $\sym{r}{\geq 1}{0,-}$ 
is a symbol in $\Sigma \Gamma_1^0[r,2]$, $R_1\in \wt \cR^{-\vr}_1$ 
and $R_{\geq 2}^{(-1)}\in \cR_{\geq 2}^{-1}[r]$ is an operator in $ \cL(H^s; H^{s+1})$, fulfilling the estimate
\begin{align*}
    \| R_{\geq 2}^{(-1)}(\zak)\|_{\cL(H^s, H^{s+1})}
    \lesssim_s \|\zak\|_{s_0}^2\,, 
    \quad \text{for any } s\in \R\,.
\end{align*}
In view of \eqref{prop:ov3}, the symbol $\tQ^{-1}\sha{\vr} (-\ii \tV\cdot \xi-\tfrac12\div(\tV))\sha{\vr} \tQ+ \tQ\sha{\vr} (-\ii \tV\cdot \xi+\tfrac12\div(\tV))\sha{\vr} \tQ^{-1}$,  appearing in \eqref{intermedia_real} when $\s=+$, is real. Then the symbol $\sym{r}{\geq 1}{0,+}$ is real as well by difference.
Then, recalling \eqref{def:AA12}, expansion \eqref{esp:lambda} and using again \Cref{compoparapara0}, we get
\begin{align*}
    \Opbw{\tQ}^{-1} A_2  \Opbw{\tQ}^{-1} &= \Opbw{\tQ^{-1}\sha{\vr}(\sym{\lambda}{}{1}+\sym{\lambda}{}{0}+\sym{\lambda}{1}{-1})\sha{\vr} \tQ^{-1} }\\
    &+\Opbw{\tQ}^{-1} \Opbw{\sym{\lambda}{\geq 2}{-1}}  \Opbw{\tQ}^{-1}+  R^{(-\sla{7}{4})}_{\geq 2}(\eta)A_2  \Opbw{\tQ}^{-1}\\
    &+ \Opbw{\tQ^{-1}}A_2R^{(-\sla{7}{4})}_{\geq 2}(\eta)+ R_1(\eta)+R_{\geq 2}(\eta)\\
   =& \Opbw{\tQ^{-1}\sha{\vr}(\sym{\lambda}{}{1}+\sym{\lambda}{}{0}+\sym{\lambda}{1}{-1})\sha{\vr} \tQ^{-1} }+ R_{\geq 2}^{(-\sla{1}{2})}(\eta)\\
   \stackrel{\eqref{def:MQ},\eqref{def:sd}}{=}& \Opbw{\sym{\Sigma}{}{\sla{3}{2}}+\tQ^{-2} \sym{\lambda}{}{0}+\sym{r}{1}{-\sla{1}{2}}}+ R_1(\zak)+\breve{R}_{\geq 2}^{(-\sla{1}{2})}(\zak)\,,
 \end{align*}
 where $\sym{r}{1}{-\sla{1}{2}} \in \wt \Gamma_1^{-\sla{1}{2}}$.
 In view of \eqref{prop:ov3}, 
 the symbol $\tQ^{-1}\sha{\vr}(\sym{\lambda}{}{1}+\sym{\lambda}{}{0}+\sym{\lambda}{1}{-1})\sha{\vr} \tQ^{-1}$ which appear in the paradifferential operator above is real valued as well as $\sym{\lambda}{}{1}$, $\sym{\lambda}{}{0}$ and $\sym{\lambda}{1}{-1}$. Then the symbol $ \sym{r}{1}{-\sla{1}{2}}$ is real valued as well by difference. Moreover $R_1(\zak)$  is obtained by substituting $\eta$ with \eqref{zak} to the $1$-homogeneous component of the remainder coming from \Cref{compoparapara0} and is thus a smoothing remainder in $ \wt \cR^{-\vr}_1$. In the same way the remainder  $\breve{R}_{\geq 2}^{(-\sla{1}{2})}(\zak)$ is obtained collecting all the quadratic non-homogeneous regularizing  contributions and substituting $\eta$ with \eqref{zak}. Thus $\breve{R}_{\geq 2}^{(-\sla{1}{2})}$ satisfies the estimates
\begin{align}
    \|\breve{R}_{\geq 2}^{(-\sla{1}{2})}(\zak)\|_{\cL(H^s;H^{s+\sla{1}{2}})}\lesssim_s \| \zak \|_{s_0}^2\,, 
    \qquad \forall s \in \R\,.
    \label{Rbreveest}
\end{align}

Finally, recalling again \eqref{def:AA12},
\begin{align*}
    \Opbw{\tQ} A_3\Opbw{\tQ}=& -\Opbw{ \tQ\sha{\vr} (\kap \sym{\th}{}{2} +1+\ta+\sym{\th}{\geq 2}{0})\sha{\vr} \tQ}+ R_1(\eta,\psi)+ R_{\geq 2}(\eta,\psi)
    \\
    \stackrel{\eqref{def:MQ},\eqref{def:sd}}{=}& 
    - \Opbw{\sym{\Sigma}{}{\sla{3}{2}}+\sym{\breve{r}}{1}{-\sla{1}{2}}}+ S_1(\zak) +{S}_{\geq 2}^{(-\sla{1}{2})}(\zak)\,,
\end{align*}
 In view of \eqref{prop:ov3}, the symbol $\tQ\sha{\vr} (\kap \sym{\th}{}{2} +1+\ta+\sym{\th}{\geq 2}{0})\sha{\vr} \tQ$, which appear in the paradifferential operator above, is real valued as well as $\sym{\th}{}{2}$, $\ta$, $\sym{\th}{\geq 2}{0}$ and $\sym{\Sigma}{}{\sla{3}{2}}$. Then $\sym{\breve{r}}{1}{-\sla{1}{2}}$ is real valued as well by difference.
Moreover $S_1(\zak)$  is obtained by substituting $(\eta,\psi)$ with \eqref{zak} to the $1$-homogeneous component of the remainder coming from \Cref{compoparapara0} and is thus a smoothing remainder in $ \wt \cR^{-\vr}_1$. In the same way the remainder  $S_{\geq 2}^{(-\sla{1}{2})}(\zak)$ is obtained collecting all the quadratic non-homogeneous regularizing  contributions and substituting $(\eta,\psi)$ with \eqref{zak}. Thus $Q_{\geq 2}^{(-\sla{1}{2})}$ satisfies the same estimates as in \eqref{Rbreveest}.
Gathering all the above  we get 
\begin{equation}
\begin{aligned}
A_+&= 2 \Opbw{ -\ii\Big[ \sym{\Sigma}{}{\sla{3}{2}}+ \tV\cdot\x 
+\sym{r}{\geq 1}{0,+}+ \tfrac12\tQ^{-2} \sym{\lambda}{}{0}+\sym{r}{1}{-\sla{1}{2}}+ \tfrac{1}{2}\tQ^2 \ta\Big]}+ \cR_1^+(\zak) + B_{\geq 2}^+(\zak)\,,
\\
A_{-}&= 2\Opbw{\tfrac{\ii}{2} \tQ^{-2} \sym{\lambda}{}{0}+\tfrac{\ii}{2} \Big[ -\sym{r}{\geq 1}{0,-}
+\sym{r}{1}{-\sla{1}{2}}-  \tQ^2 \ta\Big]}+\cR_1^-(\zak)+ B_{\geq 2}^+(\zak)\,,
\end{aligned}
\label{comp:Apm}
\end{equation}
where $\cR^\pm_1(\zak)\in \wt \cR^{-\vr}$ and 
\begin{align}
    \|B_{\geq 2}^\pm(\zak)U\|_{s}
    \lesssim_s \|\zak\|_{s_0}^2\| U\|_s\,,
    \qquad \forall s\in \R\,.
    \label{stima:Bpm}
\end{align}
On the other hand (recall also \eqref{zak})
\begin{gather}
    \pa_t \bcQ \bcQ^{-1}= \begin{pmatrix}
        0& \cD(\zak)\\ \cD(\zak)&0
    \end{pmatrix},\nonumber\\ \cD(\zak):=-\sla{1}{2} \left(\Opbw{\pa_t \tQ}\Opbw{\tQ}^{-1}+\Opbw{\tQ}^{-1}\Opbw{\pa_t \tQ}\right)_{|(\eta,\psi)= \cM^{-1} \zak}.
    \label{patQ}
\end{gather}
Using the equation \eqref{eq:1.2} for $\eta$ we have 
$ \pa_t \tQ= \di_\eta \tQ (\eta;x,\xi)[ G(\eta)\psi]$. 
Then, substituting $(\eta,\psi)=\cM^{-1}\zak$ and gathering \cref{stima:dtQ,stima:Qinv} and \Cref{thm:action}, we get 
\be 
\| \cD(\zak)\|_{\cL(H^s_\R;H^s_\R)} \lesssim_s 
\| \zak \|_{s_0+\mu}^2\,.
\label{stima:cD}
\ee

Moreover by \eqref{esp:Q}, \eqref{esp:Q-1} and 
substituting again $(\eta,\psi)=\cM^{-1}\zak$, we also have
\begin{align}
    \bcQ
\begin{bmatrix}
R_{\geq 2}(\eta,\psi) & R_1(\eta)+ R_{\geq 2}(\eta) \\
\breve{R}_{\geq 2}(\eta,\psi) & R_1(\eta,\psi)
\end{bmatrix}
 \bcQ^{-1}= \frac12\begin{bmatrix}
R_+ & R_- \\
\overline{R_-} & \overline{R_+}
\end{bmatrix}= \bR_1(\zak)+\bR_{\geq 2}(\zak)
\label{conj:Rsmoo}
\end{align}
where $\bR_1(\zak)$ is a matrix of $1$-homogeneous smoothing remainders in $ \wt\cR^{-\vr}_1$ and $\bR_{\geq 2}(\zak)$ is a matrix in $\cR_{\geq 2}^{-\vr}[r]$.
Thus \eqref{eq:QQ-1}, combined with 
\eqref{eq:conj_complex_Matrix}, \eqref{comp:Apm}, 
\eqref{patQ}, \eqref{conj:Rsmoo} and the estimates 
\eqref{stima:Bpm}, \eqref{stima:cD}, concludes the proof of  \Cref{diag.ord0}.
\end{proof}

     \subsection{Para-differential block-decoupling}\label{section block-decoupling}
In this section, we diagonalize the matrix of 
paradifferential operators in \eqref{eq:U} through two steps: first, 
we apply a transformation to remove the term 
$\Opbw{\tfrac{\ii}{2}\tQ^{-2}\sym{\lambda}{}{0}}$ of order $\tfrac12$ 
(see Lemma \ref{diagoPosorders}); 
then, we introduce a second transformation to 
eliminate the off-diagonal symbol of order $0$, 
up to a smoothing remainder (see Lemma \ref{diagoNegativeorders}).
The main outcome is the following proposition. 
     \begin{proposition}[Block-decoupling up to smoothing reminders]
     \label{prop:blockdecoupling}
     Let $\vr >1$. 
There are $s_0>0 $ and a bounded and invertible linear transformation $\bF(\zak)$ such that for any $s\geq s_0$ there is $r=r_s>0$ such that for any $\zak \in B_{s_0,\R}(r)$ the maps $\bF(\zak)$ and $\bF^{-1}(\zak)$ belong to $\cL\left( H^s_\R;H^s_\R\right)$ with estimates 
\begin{align}
\|\bF(\zak)\|_{\cL(H^s_\R;H^s_\R)}+ \|\bF^{-1}(\zak)\|_{\cL(H^s_\R;H^s_\R)}\lesssim_s 1\,, 
\label{stima:adm}
\end{align}
and if $\zak(t) \in B_{s_0, \R}(r)$ solves 
\eqref{eq:zak} then the variable
  $U_{1}= \bF(\zak)U$, 
  with $U$ defined in  \eqref{def:cQ} solving \eqref{eq:U}, 
   solves 
 \be \label{diagonale}
\pa_t U_{1} = \vOpbw{-\ii \left[\Sigma^{(\sla{3}{2})} + \tV\cdot \xi 
+ \tfrac12 \tQ^{-2} \sym{\lambda}{}{0} + \sym{a}{1}{0}\right]} U_{1} 
+ \bR_1(\zak)U_{1} + \bB_{\geq 2}(\zak)U_{1}\,,
        \ee
where  the real symbol 
$\Sigma^{(\sla{3}{2})}$ is defined in \eqref{def:sd}-\eqref{def:sdBIS}, 
the real vector of functions  $ \tV \in  \cF^{\R}_{\geq 1}[r]$ is defined in \eqref{def:V-B}, 
the real symbol  $ \sym{\lambda}{}{0} \in \Gamma_{\geq 1}^0[r]$ is defined in \eqref{def:lam0}. Moreover: 
\begin{itemize}
    \item $\sym{a}{1}{0}\in\wt \Gamma_1^{0}$ is  the real symbol defined in \Cref{diag.ord0};
    \item $\bR_1(\zak)$ is a real-to-real matrix of smoothing remainders in $ \wt\cR^{-\vr}_1$;
      \item $\bB_{\geq 2}(\zak)$ is a real-to-real matrix of bounded operators satisfying \eqref{stima_quadratica0} (with a possibly larger $s_0$).
\end{itemize}
     \end{proposition}
     The first step is to remove the positive order out diagonal term $ \zOpbw{\frac{\ii}{2} Q^{-2} \sym{\lambda}{}{0}}$ in \eqref{eq:U}. This is the content of the following lemma.

  \begin{lemma}[Diagonalization of the positive orders]\label{diagoPosorders}
      Let  $\vr \geq  0$. There exist $s_0 >0$ and a bounded 
      and invertible transformation $\Psi_0(\zak)$   
      such that for any $s\geq s_0$ there is $r=r_s>0$ such that for any 
      $\zak \in B_{s_0,\R}(r)$ the map ${\bf \Psi}_0(\zak)$ verifies \eqref{stima:adm}. 
      Moreover  if $\zak(t) \in B_{s_0, \R}(r)$ solves \eqref{eq:zak} then the variable
  \begin{align*}
      U_0:= \mathbf{\Psi}_0(\zak)U\,, 
      \qquad \text{where  $U$ solves \eqref{eq:U}\,,} 
  \end{align*}
   solves the system
     \begin{align}
\pa_t U_0 = 
\vOpbw{-\ii \left[\sym{\Sigma}{}{\sla{3}{2}} 
+ \tV\cdot \xi + \tfrac12 \tQ^{-2} \lambda^{\scs{(0)}}
+ \sym{a}{1}{0}\right]} U_0
+ \zOpbw{ \sym{b}{1}{0}}U_0 + \bR_1(\zak)U_0 
+ \bB_{\geq 2}(\zak)U_0\,,
         \label{NuovoParaprod}
\end{align}
     where 
    \begin{itemize}
        \item $\sym{a}{1}{0}\in\wt \Gamma_1^{0}$ is  the real symbol defined in \Cref{diag.ord0}  while $\sym{b}{1}{0}$ is a symbol in $\wt \Gamma_1^{0}$;
        \item $\bR_1(\zak)$ is a real-to-real matrix of real-to-real smoothing operators in $ \wt \cR^{-\vr}_1$;
        \item $\bB_{\geq 2}(\zak)$ is a real-to-real  matrix of bounded operators, satisfying \eqref{stima_quadratica0} (with a possibly larger $s_0$).
        \end{itemize}
  \end{lemma}
  \begin{proof}
      We consider the flow 
\begin{gather}
\pa_\tau \Phi^\tau(\zak)= \bG(\zak) \Phi^\tau(\zak)\,, 
\qquad \Phi^0(\zak)=\uno\,, 
\notag
\\
\bG(\zak):=\zOpbw{\sym{g}{\geq 1}{-1}(\eta;x,\xi)}\,,
\qquad 
\sym{g}{\geq 1}{-1}:= 
\frac{\tQ^{-2}\lambda^{\scs{(0)}}}{4 \sym{\Sigma}{}{\sla{3}{2}}}\,.
\label{flusso:-1out}
\end{gather}
Note that by its definition, the symbol $\sym{g}{\geq q}{-1}$ 
solves the homological equation
\begin{align}
-\ii \left[\sym{\Sigma}{}{\sla{3}{2}}
\sha{\vr} \sym{g}{\geq 1}{-1}
+\sym{g}{\geq 1}{-1}\sha{\vr} \sym{\Sigma}{}{\sla{3}{2}} \right]
+ \tfrac{\ii}{2} \tQ^{-2}\sym{\lambda}{}{0} 
\in \Sigma \Gamma_{1}^{-\sla{3}{2}}[r,2]\,.
    \label{homo:g-1}
\end{align}
Note that, by \Cref{lem:Q-Qinv} and \Cref{Lem:lamh} and 
recalling also the discussion in \Cref{rem:symbols_modozero}, one has 
$\sym{g}{\geq 1}{-1}\in \Sigma \Gamma_{1}^{-1}[r,3]$ with expansion 
\[
\sym{g}{\geq 1}{-1}= \sym{g}{1}{-1}+ \sym{g}{\geq 3}{-1}\,, 
\qquad 
\sym{g}{1}{-1}\in\wt{\Gamma}_1^{-1}, \qquad \sym{g}{\geq 3}{-1}
\in \Gamma_{\geq 3}^{-1}[r]\,.
\] 
Moreover
there is  $\mu>0$ such that
\begin{align}
\big| \di_\eta \sym{g}{\geq 1}{-1}[\hat\eta]\big|_{-1,\s}
\lesssim_{\s} 
\| \hat \eta\|_{\s+\mu}\,,
\qquad
\big| \di_\eta \sym{g}{\geq 3}{-1}[\hat\eta]\big|_{-1,\s}\lesssim_{\s}\| \eta\|_{\s+\mu}^2
\| \hat \eta\|_{\s+\mu}\,.
\label{stima:g1d}
\end{align}
Then \Cref{lem:flusso} ensures that $\Psi^\tau(\zak)$ 
is well defined and real-to-real and, 
by \eqref{stima:phistd}, satisfies 
the required estimates \eqref{stima:adm}. 
If $U$ solves \eqref{eq:U}, the variable 
\begin{align}\label{mappaPsi0}
    U_0:= {\bf \Psi}_0(\zak) U\,, 
    \qquad 
    \text{where} 
    \qquad  {\bf \Psi}_0(\zak):=\Phi^\tau(\zak)_{|\tau=1}\,,
\end{align}
solves 
\begin{equation}\label{coniugata0}
\begin{aligned}
    \pa_t U_0 &= {\bf \Psi}_0(\zak) \bA(\zak) {\bf \Psi}_0(\zak)^{-1}U_0 
    + \pa_t {\bf \Psi}_0(\zak){\bf \Psi}_0(\zak)^{-1} U_0
    +{\bf \Psi}_0(\zak) 
    \bR_1(\zak) {\bf \Psi}_0^{-1}(\zak)U_0
    \\&\qquad 
    + {\bf \Psi}_0(\zak) \bB_{\geq 2}(\zak) {\bf \Psi}_0^{-1}(\zak)U_0\,,
\end{aligned}
\end{equation}
where 
 \begin{align}
 \bA(\zak)&:= \bA_1(\zak)+\bA_2(\zak)+\bA_3(\zak)\,,\nonumber
 \\
 \bA_1(\zak)&:=\vOpbw{-\ii \sym{\Sigma}{}{\sla{3}{2}}}\,,\label{bfA1}
 \\
 \bA_2(\zak)&:=\vOpbw{-\ii \left[\tV\cdot \xi + \tfrac12 \tQ^{-2} \sym{\lambda}{}{0}
 +\sym{a}{1}{0}\right]} \,,\label{bfA2}
 \\
 \bA_3(\zak)&:=
 \zOpbw{\tfrac{\ii}{2}\tQ^{-2} \sym{\lambda}{}{0}+\sym{b}{1}{0}}\,.\label{bfA3}
\end{align}
         Note that, since $\bA(\zak)$, $\bR_1(\zak)$ and $\bB_{\geq 2}(\zak)$ 
         are real-to-real as well as ${\bf \Psi}_0(\zak)$ and ${\bf \Psi}_0^{-1}(\zak)$, 
         the operator in the right hand side of \eqref{coniugata0} is real-to-real. 
         We now expand it term by term.
         
         \smallskip
         \noindent{\bfseries Expansion of 
         ${\bf \Psi}_0(\zak) \bA(\zak) {\bf \Psi}_0(\zak)^{-1} $.}
        Classical Lie expansion (see e.g \cite[Formula A.3]{BFP}), 
        gives (see \eqref{mappaPsi0})
\begin{align}
\left(\Phi^{\tau}(\zak) \bA(\zak) \Phi^{\tau}(\zak)^{-1}\right)_{| \tau=1} &= 
\bA(\zak) + \int_0^{1} \Phi^\theta(\zak)\left[ \bG(\zak),  \bA(\zak)\right]  
\Phi^{-\theta}(\zak)\, \di \theta 
\label{Lie:1}
\\
&= \bA(\zak) + \left[ \bG(\zak),  \bA(\zak)\right] 
+\int_{0}^1 (1-\theta) \Phi^\theta(\zak){\mathrm{ Ad}}_{\bG}^2[ \bA(\zak)]  
\Phi^{-\theta}(\zak)\, \di \theta\,.
\label{Lie:2}
\end{align}
First, thanks to \Cref{thm:action} and estimates \eqref{bound:homosym} 
and \eqref{stima:nonhom} for $\sym{g}{\geq 1}{-1}$,  the generator 
$\bG(\zak)$ in \eqref{flusso:-1out} belongs to 
$\cL(H^s_\R;H^{s+1}_\R)$ for any $s\in \R$, with estimates 
    \be\label{stima:gen-1}
        \| \bG(\zak)\|_{\cL(H^s_\R;H^{s+1}_\R)}
        \lesssim_s \| \zak\|_{s_0}\,.
        \ee
We then compute the conjugation rule \eqref{Lie:1}, 
        \eqref{Lie:2} starting from 
        $ \bA(\zak)\leadsto\bA_1(\zak)$ in \eqref{bfA1}.
First, note that, by \Cref{thm:action}, it  belongs to 
        $\cL(H^s_\R;H^{s-\sla{3}{2}}_\R)$ with estimates 
        \begin{align}
            \| \bA_1(\zak)\|_{\cL(H^s_\R;H^{s-\sla{3}{2}}_\R)}\lesssim_s 1\,.
            \label{stima:A11}
        \end{align}
        Then, by \eqref{stima:A11} and \eqref{stima:gen-1}, we have $\left[ \bG(\zak),  \bA_1(\zak)\right] \in \cL(H^s_\R;H^{s-\sla{1}{2}}_\R)$ with estimates
        \begin{align}\label{stima:commGA1}
        \| \left[ \bG(\zak),  \bA_1(\zak)\right]\|_{\cL(H^s_\R;H^{s-\sla{1}{2}}_\R)}
        \lesssim_s \| \zak\|_{s_0}, \quad \textnormal{for any }s \in \R\,.
        \end{align}
        Finally, by \eqref{stima:A11} and \eqref{stima:commGA1}, we have that 
        ${\mathrm{Ad}}_{\bG}^2[ \bA_1(\zak)]\in \cL(H^s_\R;H^{s+\sla{1}{2}}_\R)$ 
        with estimates 
        \begin{align}\label{stima:2commGA1} 
        \| {\mathrm{Ad}}_{\bG}^2[ \bA_1(\zak)]\|_{\cL(H^s;H^{s+\sla{1}{2}})}
        \lesssim_s \| \zak\|_{s_0}^2\,.
      \end{align}
        Gathering \eqref{Lie:2},  \eqref{stima:2commGA1} and \Cref{compoparapara0}, we get
        \begin{align}
 \Phi^{1}(\zak) \bA_1(\zak) \Phi^{1}(\zak)^{-1}&=  
 \bA_1(\zak)+ \zOpbw{-\ii \left[\sym{\Sigma}{}{\sla{3}{2}} \sha{\vr} \sym{g}{\geq 1}{-1}
 +\sym{g}{\geq 1}{-1}\sha{\vr} \sym{\Sigma}{}{\sla{3}{2}} \right]}
 + \bR_1(\zak)+ \bB_{\geq 2}(\zak)\notag
 \\
   &\stackrel{\eqref{homo:g-1}}{=}\vOpbw{-\ii \sym{\Sigma}{}{\sla{3}{2}} }
   + \zOpbw{- \tfrac{\ii}{2} \tQ^{-2}\sym{\lambda}{}{0} 
   + \sym{r}{1}{-\sla{3}{2}}}
            + \bR_1(\zak)+ \breve{\bB}_{\geq 2}(\zak)\,,
            \label{conj:A1}
        \end{align}
        where $\bR_1(\zak) $ is a $1$-homogeneous smoothing remainder 
        in $\wt \cR^{-\vr}_1$ and $ \bB_{\geq 2}(\zak)$ and 
        $\breve{\bB}(\zak)$ belong to $ \cL(H^s_\R;H^s_\R)$ 
        for any $s\in \R$ with estimates 
        \begin{align}\label{stima:bione}
        \|  \bB_{\geq 2}(\zak)\|_{\cL(H^s_\R;H^s_\R)}
        \lesssim_s \| \zak\|_{s_0}^2\,.
        \end{align}
        We now compute the conjugation rule for 
        $\bA_{2}(\zak)$ defined in \eqref{bfA2},
        which, 
        thanks to \Cref{thm:action},  belongs to 
        $\cL(H^s_\R; H^{s-1}_\R)$ for any $s\in \R$ with estimates 
        \begin{align}
            \| \bA_2(\zak)\|_{\cL(H^s; H^{s-1})}\lesssim_s \| \zak\|_{s_0}.
            \label{stima:A21}
        \end{align}
        Then, by \eqref{stima:gen-1} and \eqref{stima:A21}, we obtain that $\left[ \bG(\zak),  \bA_2(\zak)\right] \in \cL(H^s_\R;H^{s}_\R)$ with estimates
        \begin{align}
            \| \left[ \bG(\zak),  \bA_2(\zak)\right] \|_{\cL(H^s_\R;H^{s}_\R)}\lesssim_s \| \zak\|_{s_0}^2.
            \label{stima:commGA2}
        \end{align}
        Then, by \eqref{Lie:1}, \eqref{stima:commGA2} and estimate \eqref{stima:adm} for $ \Phi^\tau(\zak)$, we get 
        \begin{align}
            \Phi^{1}(\zak) \bA_2(\zak) \Phi^{1}(\zak)^{-1}= 
            \vOpbw{-\ii \left[\tV\cdot \xi 
            + \tfrac12 \tQ^{-2} \sym{\lambda}{}{0}\right]} + \bB_{\geq 2}(\zak)\,,
            \label{conj:A2}
        \end{align}
        where $\bB_{\geq 2}(\zak)\in \cL(H^s_\R;H^s_\R)$ 
fulfills \eqref{stima:commGA2}. 
        Similarly, the operator $\bA_3(\zak)$ in \eqref{bfA3}
        is in $\cL(H^s_\R; H^{s-\sla{1}{2}}_\R)$ 
        for any $s\in \R$ with estimates 
        \begin{align*}
            \| \bA_3(\zak)\|_{\cL(H^s_\R; H^{s-\sla{1}{2}}_\R)}
            \lesssim_s \| \zak\|_{s_0}\,.
        \end{align*}
        Then we obtain that $\left[ \bG(\zak),  \bA_3(\zak)\right] \in \cL(H^s_\R;H^{s}_\R)$ with the same estimates as in \eqref{stima:commGA2} and, by \eqref{Lie:1}, we get
        \begin{align}
               \Phi^{1}(\zak) \bA_3(\zak) \Phi^{1}(\zak)^{-1}=  
               \zOpbw{\tfrac{\ii}{2}\tQ^{-2} \sym{\lambda}{}{0}+\sym{b}{1}{0}}+ \bB_{\geq 2}(\zak)\,,
               \label{conj:A3}
        \end{align}
        where $\bB_{\geq 2}(\zak)\in \cL(H^s_\R;H^s_\R)$ with 
        estimates \eqref{stima:bione}. 
        

\smallskip
\noindent
{\bfseries Expansion of $\pa_t \Phi(\zak) \Phi(\zak)^{-1}$.}
         We then expand the time derivative using its 
         corresponding Lie expansion (see \cite[Formula A.4]{BFP}).
         Recalling \eqref{flusso:-1out}
 we get        
 \begin{align}
&\pa_t {\bf \Psi}_0(\zak) {\bf \Psi}_0^{-1}(\zak)= 
\pa_t \bG(\zak)+ \int_0^1 (1-\theta)\Phi^\theta(\zak) 
\left[\pa_t \bG(\zak), \bG(\zak)\right] \Phi^{-\theta}(\zak)\, \di \theta \label{Lie:dt}
\\
 =& \zOpbw{\di_\eta \sym{g}{\geq 1}{-1}[G(\eta)\psi]}
 + \int_0^1 (1-\theta)\Phi^\theta(\zak) 
 \left[\zOpbw{\di_\eta \sym{g}{\geq 1}{-1}[G(\eta)\psi]}, \zOpbw{\sym{g}{\geq 1}{-1}}\right] 
 \Phi^{-\theta}(\zak)\, \di \theta\,.
            \label{exp:dt-1}
        \end{align}
        By \eqref{stima:g1d}, \eqref{exp:DN} and \eqref{stimaDNa perdere}, 
        $\di_\eta \sym{g}{\geq 1}{-1}[G(\eta)\psi]$ is a symbol in 
        $\Sigma \Gamma_1^{-1}[r,2]$ expanding as 
        $\di_\eta \sym{g}{\geq 1}{-1}[G(\eta)\psi]= \sym{r}{1}{-1}+\sym{r}{\geq 2}{-1}$. 
        Moreover, by  \eqref{exp:dt-1} and 
        using \eqref{actionSob}, we get 
\begin{align*}
    \pa_t {\bf \Psi}_0(\zak){\bf \Psi}_0^{-1}(\zak)= 
    \zOpbw{\sym{r}{1}{-1}}+ \bB_{\geq 2}(\zak)\,,
    \end{align*}
   where the operator $ \bB_{\geq 2}(\zak)$ belongs to $\cL(H^s_\R;H^s_\R)$ 
   with estimates \eqref{stima:bione}. Note that we included the operator 
   $\zOpbw{\sym{r}{\geq 2}{-1}}$ in the quadratic remainder, 
   as it satisfies the quadratic estimate \eqref{stima:bione} by \eqref{actionSob} and \eqref{stima:nonhom}.

\smallskip
\noindent
{\bfseries Expansion of ${\bf \Psi}_0(\zak) 
    \bR_1(\zak) {\bf \Psi}_0^{-1}(\zak)$.} Using the Lie expansion we get 
    \begin{align}
       {\bf \Psi}_0(\zak) 
    \bR_1(\zak) {\bf \Psi}_0^{-1}(\zak)= \bR_1(\zak)+ \int_0^{1} \Phi^\theta(\zak)\left[ \bG(\zak),  \bR_1(\zak)\right]  \Phi^{-\theta}(\zak)\, \di \theta
    = \bR_1(\zak)+ \bR_{\geq 2}(\zak)\,,
    \label{conj:R1}
    \end{align}
    where, thanks to estimates \eqref{bound:smoo} for $ \bR_1(\zak)$,  \eqref{stima:adm} for $\Phi^\tau(\zak)$ and \eqref{stima:gen-1} for $\bG(\zak)$,  $\bR_{\geq 2}(\zak)$ belongs to $ \cL( H^s_\R; H^{s+\vr}_\R)$ with estimates 
    \begin{align*}
        \| \bR_{\geq 2}(\zak)V\|_{s+\vr}\lesssim_s 
        \| \zak\|_{s_0}^2\| V \|_s+ \| \zak\|_{s_0}\|\zak\|_s\|V\|_{s_0}\,, 
        \qquad \forall s\geq s_0\,.
    \end{align*}
    
    \smallskip
    \noindent
    {\bfseries Estimate of $ {\bf \Psi}_0(\zak) 
    \bB_{\geq 2}(\zak) {\bf \Psi}_0^{-1}(\zak)$.} 
    It follows from \eqref{mappaPsi0}, \eqref{stima:flusso}  and \eqref{stima_quadratica0} that 
    \begin{equation}\label{stima:B2}
    \begin{aligned}
        \| {\bf \Psi}_0(\zak) 
    \bB_{\geq 2}(\zak) {\bf \Psi}_0^{-1}(\zak)V \|_s &\lesssim_s
    \| \bB_{\geq 2}(\zak){\bf \Psi}_0^{-1}(\zak)V\|_s
    \\
   & \lesssim_s \| \zak\|_{s_0}^2\|  
     {\bf \Psi}_0^{-1}(\zak)V\|_s+ \| \zak \|_{s_0}
     \|\zak\|_s \| \|{\bf \Psi}_0^{-1}(\zak)V\|_{s_0}
     \\
   & \lesssim_s \| \zak\|_{s_0}^2\|  
     V\|_s+ \| \zak \|_{s_0}\|\zak\|_s  \|V\|_{s_0}\,.
    \end{aligned}
    \end{equation}
    Gathering \eqref{conj:A1},\eqref{conj:A2}, \eqref{conj:A3}, \eqref{conj:R1} and \eqref{stima:B2} we obtain the thesis.
  \end{proof}
The next step of our procedure is to eliminate, up to smoothing remainders, the
term $\zOpbw{ \sym{b}{1}{0}}$ appearing in \eqref{NuovoParaprod}.
We first need the following technical lemma.

\begin{lemma}\label{equaHomodiago}
Let $\sym{b}{1}{0}$ be  a symbol in $\wt \Gamma_1^{0}$.
    There exists a $1$-homogeneous symbol $ \sym{g}{1}{-\sla{3}{2}}$ in 
    $ \wt \Gamma^{-\sla{3}{2}}_1$ such that 
    \begin{align}
        \sym{b}{1}{0}- \ii \left[\Lambda(\xi) \sha{\vr} \sym{g}{ 1}{-\sla{3}{2}}+\sym{g}{ 1}{-\sla{3}{2}}\sha{\vr} \Lambda(\xi)\right]+\sym{g}{1}{-\sla{3}{2}}(-\ii \vomega(D)\zak;x,\xi)\in \wt\Gamma^{-\vr}_1.
        \label{eq:homo_out}
    \end{align}
\end{lemma}
\begin{proof}
    Let $m\in \R$. For any symbol 
    $\sym{g}{1}{m}$ in $\wt \Gamma_1^m$ we define 
    \begin{align*}
 \ts[ \sym{g}{1}{m}]&:= \Lambda(\xi) \sha{\vr} \sym{g}{ 1}{m}
 +\sym{g}{ 1}{m}\sha{\vr} \Lambda(\xi)- 2 \sym{g}{1}{m} \Lambda(\xi)
 = 
 \sum_{\ell=2}^{\vr/2} 2 p_{2\ell}( \Lambda(\xi), \sym{g}{ 1}{m}) 
 \in \wt \Gamma_1^{m-\sla{1}{2}}
 \\
\td[ \sym{g}{1}{m}](\zak;x,\xi)&:= 
\sym{g}{1}{m}(-\ii \Lambda(D)\zak;x,\xi) \in \wt \Gamma_1^m\,.
\end{align*}
    Then equation \eqref{eq:homo_out} becomes 
    \begin{align}
        \sym{\tr}{1}{-\vr}:= \sym{b}{1}{0}-2 \ii \Lambda(\xi) \sym{g}{1}{-\sla{3}{2}}
        - \ii \ts[ \sym{g}{1}{-\sla{3}{2}}]+ \td[ \sym{g}{1}{-\sla{3}{2}}] \in \wt \Gamma^{-\vr}_1\,.
        \label{eq:homo_outproof}
    \end{align}
    We solve \eqref{eq:homo_outproof} iteratively. 
    Let $m:= \min\{ n\in \N \colon \sla{3}{2} n \geq \vr\}$, we define 
    \begin{align*}
        \sym{g}{1}{-\sla{3}{2}}=\sum_{\ell=1}^m \sym{h}{1}{-\sla{3}{2} \ell }\,, 
    \end{align*}
    where (recall the discussion in \Cref{rem:symbols_modozero})
    \begin{align*}
        \sym{h}{1}{-\sla{3}{2}}:= 
        \frac{\sym{b}{0}{1}}{2\ii \Lambda(\xi)}\in \wt \Gamma_1^{-\sla{3}{2}}\,, 
        \qquad 
        \sym{h}{1}{-\sla{3}{2}\ell}:= 
        \frac{- \ii \ts[ \sym{h}{1}{-\sla{3}{2}(\ell-1)}]
        + \td[ \sym{h}{1}{-\sla{3}{2}(\ell-1)}]}{2\ii \Lambda(\xi)}\in\wt \Gamma_1^{-\sla{3}{2}\ell}\, , 
        \quad \ell=2, \ldots, m\,.
    \end{align*}
    With this choice we get 
    \begin{align*}
         \sym{\tr}{1}{-\vr}= - \ii \ts[ \sym{h}{1}{-\sla{3}{2} m }]
         + \td[ \sym{h}{1}{-\sla{3}{2} m }]\in \wt \Gamma_1^{-\sla{3}{2}m} 
         \subset \wt \Gamma_1^{-\vr}\,. 
    \end{align*}
    This concludes the proof.
\end{proof}

\noindent
   We are now in position to complete the block-decoupling at negative orders.
\begin{lemma}[Diagonalization of the negative orders]
\label{diagoNegativeorders}
     Let $\vr>1$. 
There are $s_0  >0 $ and a bounded and invertible transformation  $\mathbf{\Psi}_1(\zak)$ such that
  such that for any $s\geq s_0$ there is $r=r_s>0$ such that for any $\zak \in B_{s_0,\R}(r)$ the map $\mathbf{\Psi}_1(\zak)$ verifies \eqref{stima:adm}. Moreover, if $\zak(t) \in B_{s_0, \R}(r)$ solves \eqref{eq:zak} then the variable
  $U_1:= \mathbf{\Psi}_1(\zak)U_0$, 
  with $U_0$ solving \eqref{NuovoParaprod},  solves the system
  \begin{align}
         \pa_t U_1 = \vOpbw{-\ii \left[\Sigma^{(\sla{3}{2})} + \tV\cdot \xi + \tfrac12 \tQ^{-2} \sym{\lambda}{}{0}+ \sym{a}{1}{0}\right]} U_1+ \bR_1(\zak)U_1 + \bB_{\geq 2}(\zak)U_1,
         \label{NuovoParaprod1}
    \end{align}
    where 
    \begin{itemize}
        \item $\sym{a}{1}{0}\in\wt \Gamma_1^{0}$ is  the real symbol defined in \Cref{diag.ord0};
  \item $\bR_1(\zak)$ is a real-to-real matrix of real-to-real smoothing operator in $ \wt \cR^{-\vr}_1$;
        \item $\bB_{\geq 2}(\zak)$ is a real-to-real  matrix of bounded operators, satisfying \eqref{stima_quadratica0} (with a possibly larger $s_0$).
    \end{itemize}
\end{lemma}
\begin{proof}
We consider the flow 
\begin{equation*}
\pa_\tau \Phi^\tau(\zak)= \bG(\zak) \Phi^\tau(\zak), \quad \Phi^0(\zak)=\uno, \qquad
\bG(\zak):=\zOpbw{\sym{g}{1}{-\sla{3}{2}}(\zak;x,\xi)},
\end{equation*}
where $\sym{g}{1}{-\sla{3}{2}}$ is the  $1$-homogeneous symbol in 
$\wt \Gamma_1^{-\sla{3}{2}}$
obtained by applying 
Lemma \ref{equaHomodiago} with $\sym{b}{1}{0}\in \wt \Gamma_1^{0}$ the symbol 
appearing in \eqref{NuovoParaprod}.
Then \Cref{lem:flusso} ensures that $\Phi^\tau(\zak)$ is well defined and real-to-real and, by \eqref{stima:phistd}, satisfies the required estimates \eqref{stima:adm}. 
If $ U_0$ solves \eqref{NuovoParaprod}, the variable 
\begin{align*}
    U_1:= {\bf \Psi}_1(\zak) U_0\,, 
    \quad \text{where} \quad  {\bf \Psi}_1(\zak):=\Phi^\tau(\zak)_{|\tau=1}\,,
\end{align*}
solves 
\begin{align*}
    \pa_t U_1 &= {\bf \Psi}_1(\zak) \bA(\zak) {\bf \Psi}_1^{-1}(\zak)U_1 
    + \pa_t {\bf \Psi}_1(\zak){\bf \Psi}_1^{-1}(\zak)U_1
    \\&\qquad 
    +{\bf \Psi}_1(\zak) 
    \bR_1(\zak) {\bf \Psi}_1^{-1}(\zak)U_1+ {\bf \Psi}_1(\zak) \bB_{\geq 2}(\zak) 
    {\bf \Psi}_1^{-1}(\zak)^{-1}U_1\,,
\end{align*}
where 
 \begin{align}
 \bA(\zak)&:=  \bA_1(\zak)+\bA_2(\zak)
 +\widetilde{\bA}_3(\zak)\,,
    \qquad 
      \widetilde{\bA}_3(\zak):=  \zOpbw{\sym{b}{1}{0}}\,,\label{bfA3tilde}
%
         \end{align}
         and where $\bA_1(\zak), \bA_2(\zak)$ are defined in 
         \eqref{bfA1}, \eqref{bfA2}.
         
         \noindent
      {\bf Expansion of}
     ${\bf \Psi}_1(\zak) \bA_{1}(\zak) {\bf \Psi}_1^{-1}(\zak)$.
        We use the expansions \eqref{Lie:1} and
         \eqref{Lie:2} with $\bG(\zak)=\zOpbw{\sym{g}{1}{-\sla{3}{2}}}$.
        First note that, by \Cref{thm:action}, 
        $\bG\in \cL(H^s_\R;H^{s+\sla{3}{2}}_\R)$ for any $s\in \R$ with estimates 
        \be\label{stima:gen-10}
        \| \bG(\zak)\|_{\cL(H^s_\R;H^{s+\sla{3}{2}}_\R)}\lesssim_s \| \zak\|_{s_0}\,.
        \ee
       
       \noindent
We start from  $\bA_1$ in \eqref{bfA1}.
By estimate \eqref{stima:A11} for $\bA_1$ and \eqref{stima:gen-10} for $\bG$,  
one has $\left[ \bG(\zak),  \bA_1(\zak)\right] \in \cL(H^s_\R;H^{s}_\R)$ 
with estimates
        \begin{align}\label{stima:commGA10}
        \| \left[ \bG(\zak),  \bA_1(\zak)\right]\|_{\cL(H^s_\R;H^{s}_\R)}
        \lesssim_s \| \zak\|_{s_0}\,, 
        \qquad \textnormal{for any }s \in \R\,.
        \end{align}
        Finally, by \eqref{stima:commGA10} and \eqref{stima:gen-10},  
        ${\mathrm{Ad}}_{\bG}^2[ \bA_1(\zak)]\in \cL(H^s_\R;H^{s+\sla{3}{2}}_\R)$ 
        with estimates 
        \begin{align}\label{stima:2commGA10} 
        \| {\mathrm{Ad}}_{\bG}^2[ \bA_1(\zak)]\|_{\cL(H^s;H^{s+\sla{3}{2}})}
        \lesssim_s \| \zak\|_{s_0}^2\,.
      \end{align}
        Gathering \eqref{Lie:2},  \eqref{stima:2commGA10}, 
        \Cref{compoparapara0} and 
        recalling that by \eqref{esp:sigma},  
        $\sym{\Sigma}{}{\sla{3}{2}}- \Lambda(\xi)\in \Gamma_{\geq 2}^{\sla{3}{2}}[r]$, 
        we get
        \begin{equation}\label{conj:A10}
        \begin{aligned}
            \Psi_{1}(\zak) \bA_1(\zak) \Psi_{1}(\zak)^{-1}&=  
            \bA_1(\zak)+ \zOpbw{-\ii \left[\sym{\Sigma}{}{\sla{3}{2}} \sha{\vr} 
            \sym{g}{ 1}{-\sla{3}{2}}+\sym{g}{ 1}{-\sla{3}{2}}\sha{\vr} \sym{\Sigma}{}{\sla{3}{2}} \right]}
            \\&\quad+ \bR_1(\zak)+ \bB_{\geq 2}(\zak)
          \\&  =\vOpbw{-\ii \sym{\Sigma}{}{\sla{3}{2}} }+ \zOpbw{-\ii\left[\Lambda(\xi) \sha{\vr} \sym{g}{ 1}{-\sla{3}{2}}+\sym{g}{ 1}{-\sla{3}{2}}\sha{\vr} \Lambda(\xi) \right] }
            \\&\quad + \bR_1(\zak)+ \bB_{\geq 2}(\zak)\,,
        \end{aligned}
        \end{equation}
        where $\bR_1(\zak) $ is a $1$-homogeneous smoothing remainder in 
        $\wt \cR^{-\vr}_1$ and $ \bB_{\geq 2}(\zak)$ belongs to $ \cL(H^s;H^s)$ 
        for any $s\in \R$ with estimates \eqref{stima:bione}. 
        Note that we have included the paradifferential operators 
        with symbols in $\Gamma_{\geq 2}^0[r]$ in the bounded quadratic 
        remainder $\bB_{\geq 2}(\zak)$, as they satisfy the bound \eqref{stima_quadratica0}.
        
        Since the operator $\bA_2(\zak)$ in \eqref{bfA2} satisfies the bound \eqref{stima:A21},
        we obtain that the commutator 
        $\left[ \bG(\zak),  \bA_2(\zak)\right]$ belongs to $\cL(H^s_\R;H^{s+\sla{1}{2}})$ 
        with estimates
        \begin{align}
            \| \left[ \bG(\zak),  \bA_2(\zak)\right] \|_{\cL(H^s;H^{s})}
            \lesssim_s \| \zak\|_{s_0}^2\,.
            \label{stima:commGA20}
        \end{align}
        Then, by \eqref{Lie:1}, \eqref{bfA2} and \eqref{stima:commGA20} we get 
        \begin{align}
            {\bf \Psi}_1(\zak) \bA_2(\zak) {\bf \Psi}_1^{-1}(\zak)=
             \vOpbw{-\ii \left[\tV\cdot \xi + \tfrac12 \tQ^{-2} \sym{\lambda}{}{0}+\sym{a}{1}{0}\right]} 
             + \bB_{\geq 2}(\zak)\,,
            \label{conj:A20}
        \end{align}
        where $\bB_{\geq 2}(\zak)\in \cL(H^s;H^s)$ 
        fulfilling \eqref{stima:commGA20}. 
        Similarly, the term $\widetilde{\bA}_3(\zak)=\zOpbw{ \sym{b}{1}{0}}$  in \eqref{bfA3tilde}
        belongs to $\cL(H^s; H^{s})$ for any $s\in \R$ with estimates 
        \begin{align*}
            \| \widetilde{\bA}_3(\zak)\|_{\cL(H^s; H^{s})}
            \lesssim_s \| \zak\|_{s_0}\,.
        \end{align*}
        Then we obtain that $\left[ \bG(\zak),  \widetilde{\bA}_3(\zak)\right] 
        \in \cL(H^s;H^{s+\sla{3}{2}})$ with the same estimates 
        as in \eqref{stima:commGA20} and, by \eqref{Lie:1}, we get
        \begin{align}
               {\bf \Psi}_1(\zak) \widetilde{\bA}_3(\zak) {\bf \Psi}_1^{-1}(\zak)=  
               \zOpbw{ \sym{b}{1}{0}}+ \bB_{\geq 2}(\zak)\,,
               \label{conj:A30}
        \end{align}
        where $\bB_{\geq 2}(\zak)\in \cL(H^s;H^s)$ with estimates \eqref{stima:bione}. 

\smallskip
\noindent
{\bfseries Expansion of $\pa_t {\bf \Psi}_1(\zak) {\bf \Psi}_1^{-1}(\zak)$.}
        Then we expand the time derivative using the expansion in \eqref{Lie:dt}. 
        Using also equation \eqref{eq:zak} for $\zak$, we get 
        \begin{align*}
            \pa_t {\bf \Psi}_1(\zak) {\bf \Psi}_1^{-1}(\zak)= &\zOpbw{\sym{g}{1}{-\sla{3}{2}}(\pa_t\zak;x,\xi)}\notag \\
            &+ \int_0^1 (1-\theta)\Phi^\theta(\zak) \left[\zOpbw{\sym{g}{1}{-\sla{3}{2}}(\pa_t\zak;x,\xi)}, \zOpbw{ \sym{g}{1}{-\sla{3}{2}}(\zak;x,\xi)}\right] \Phi^{-\theta}(\zak)\, \di \theta\notag\\
            =& \zOpbw{\sym{g}{1}{-\sla{3}{2}}(-\ii \vomega(D)\zak;x,\xi)}+ \bB_{\geq 2}(\zak),
        \end{align*}
where $\bB_{\geq 2}(\zak)\in \cL(H^s;H^s)$ with estimates \eqref{stima:bione}. 

\smallskip
\noindent{\bfseries Expansion of $ {\bf \Psi}_1(\zak) 
    \bR_1(\zak) {\bf \Psi}_1^{-1}(\zak)$.} Using the Lie expansion \eqref{Lie:1}, we get 
    \begin{align}
        {\bf \Psi}_1(\zak) 
    \bR_1(\zak) {\bf \Psi}_1^{-1}(\zak)= 
    \bR_1(\zak)+ \int_0^{1} \Phi^\theta(\zak)\left[ \bG(\zak),  
    \bA(\zak)\right]  \Phi^{-\theta}(\zak)\, \di \theta= \bR_1(\zak)+ \bR_{\geq 2}(\zak)\,. 
    \label{conj:R10}
    \end{align}
    where $\bR_{\geq 2}(\zak)$ belongs to $ \cL( H^s_\R; H^{s+\vr}_\R)$ with estimates 
    \begin{align*}
        \| \bR_{\geq 2}(\zak)V\|_{s+\vr}\lesssim_s \| \zak\|_{s_0}^2\| V \|_s+ \| \zak\|_{s_0}\|\zak\|_s\|V\|_{s_0}, \qquad \forall s\geq s_0.
    \end{align*}
    
    \smallskip    
    \noindent{\bfseries Estimate of ${\bf \Psi}_1(\zak) 
    \bB_{\geq 2}(\zak){\bf \Psi}_1^{-1}(\zak)$.} Arguing as in \eqref{stima:B2}, we get 
    \begin{align}
        \| {\bf \Psi}_1(\zak) 
    \bB_{\geq 2}(\zak) {\bf \Psi}_1^{-1}(\zak)V \|_s 
     \lesssim_s\| \zak\|_{s_0}^2\|  
     V\|_s+ \| \zak \|_{s_0}\|\zak\|_s  \|V\|_{s_0}\,.
     \label{stima:B20}
    \end{align}
    Gathering \eqref{conj:A10},\eqref{conj:A20}, \eqref{conj:A30}, 
    \eqref{conj:R10} and \eqref{stima:B20} we obtain 
    \begin{align*}
        \pa_t U_1 &= \vOpbw{-\ii \left[\Sigma^{(\sla{3}{2})} 
        + \tV\cdot \xi + \tfrac12 \tQ^{-2} \sym{\lambda}{}{0}+ \sym{a}{1}{0}\right]} U_1
        \\&
        + \zOpbw{\sym{b}{1}{0}- \ii \left[\Lambda(\xi) \sha{\vr} \sym{g}{ 1}{-\sla{3}{2}}
        +\sym{g}{ 1}{-\sla{3}{2}}\sha{\vr} \Lambda(\xi)\right]
        +\sym{g}{1}{-\sla{3}{2}}(-\ii \vomega(D)\zak;x,\xi)} 
        \\&+\bR_1(\zak)U_1 + \bB_{\geq 2}(\zak)U_1\,.
    \end{align*}
In view of Lemma \ref{equaHomodiago} (see \eqref{eq:homo_out})
we have chosen the symbol 
    $\sym{g}{1}{-\sla{3}{2}}$ in such a way that 
    \[
    \zOpbw{\sym{b}{1}{0}- \ii \left[\Lambda(\xi) \sha{\vr} \sym{g}{ 1}{-\sla{3}{2}}
    +\sym{g}{ 1}{-\sla{3}{2}}\sha{\vr} \Lambda(\xi)\right]
    +\sym{g}{1}{-\sla{3}{2}}(-\ii \vomega(D)\zak;x,\xi)}\,,
    \] 
    is a smoothing remainder in $ \wt \cR^{-\vr}_1$. 
    Therefore formula \eqref{NuovoParaprod1} follows. This concludes the proof.
    \end{proof}
    
\begin{proof}[{\bf Proof of Proposition \ref{prop:blockdecoupling}}.]
By combining Lemmata \ref{diagoPosorders} and \ref{diagoNegativeorders}, we define 
${\bf F}(\zak)[\cdot]:={\bf \Psi}_0(\zak)\circ {\bf \Psi}_{1}(\zak)$. 
Therefore the variable 
  $U_{1}= \bF(\zak)U$,  with $U$ solving  \eqref{eq:U}, solves the problem
  \eqref{diagonale} thanks to
  \eqref{NuovoParaprod1}, \eqref{NuovoParaprod}.
\end{proof}

	\subsection{Normal form for the quadratic transport}\label{section-riduzione-trasporto}

In this subsection, we perform a normal form reduction of the paradifferential diagonal operator $\vOpbw{\bigcdot}$ in \eqref{diagonale}, focusing on the $1$-homogeneous component of the transport term $-\ii \tV \cdot \xi$. 

By \eqref{V:exp1}, this contribution is explicitly given by
\[
- \ii \nabla \psi \cdot \xi\,.
\]
We show that this term can be completely removed without encountering small divisor issues, thanks to its gradient structure.
	\begin{proposition}
    \label{prop:nfquadtra}
		    Let $\vr>1$. 
There are $s_0  >0 $ and a bounded and invertible transformation  
transformation $\mathbf{\Psi}_2(\zak)$ such that 
for any $s\geq s_0$ there is $r=r_s>0$ such that for any 
$\zak \in B_{s_0,\R}(r)$ the map $\mathbf{\Psi}_2(\zak)$ 
verifies \eqref{stima:adm}. Moreover, 
if $\zak(t) \in B_{s_0, \R}(r)$ solves \eqref{eq:zak} then the variable then the variable
  $U_2:= \mathbf{\Psi}_2(\zak)U_1$, 
  with $U_1$ solving \eqref{diagonale},  
  solves the system
  \begin{align}
         \pa_t U_2 = \vOpbw{-\ii \left[\sym{\Sigma}{}{\sla{3}{2}} 
         + \sym{b}{\geq 2}{1}+  \sym{b}{1}{\sla{1}{2}}\right]} U_2
         + \bR_1(\zak)U_2 + \bB_{\geq 2}(\zak)U_2\,,
         \label{NuovoParaprod2}
    \end{align}
    where the real symbol 
$\Sigma^{(\sla{3}{2})}$
is defined in \eqref{def:sd}-\eqref{def:sdBIS}, and 
    \begin{itemize}
        \item $\sym{b}{\geq 2}{1}$ is a real symbol of the form 
        \begin{equation}\label{def:b1mag}
            \sym{b}{\geq 2}{1}= 
            \sym{b}{2}{1}+ \sym{b}{\geq 3}{1}\,, 
            \qquad \sym{b}{2}{1}\in \wt \Gamma_2^1\,, 
            \qquad 
            \sym{b}{\geq 3}{1}\in \Gamma_{\geq 3}^1[r]\,.
        \end{equation}
        Moreover $\sym{b}{\geq 3}{1}(\zak;\bigcdot)$ is differentiable and there is $\mu>0$ such that 
        \begin{align}
            | \di_{\zak} \sym{b}{\geq 3}{1}(\zak;\bigcdot)[\hat \zak]|_{1,\sigma}\lesssim_\sigma \| \zak\|_{\s+\mu}^2\| \hat \zak\|_{\s+\mu};
            \label{stima:diffb3}
        \end{align}
        \item $\sym{b}{1}{\sla{1}{2}}$ is a real symbol in $\wt\Gamma_1^\sla{1}{2}$;
        \item $\bR_1(\zak)$ is a real-to-real matrix of  smoothing operators in $ \wt \cR^{-\vr}_1$;
        \item $\bB_{\geq 2}(\zak)$ is a real-to-real matrix of  bounded operators, satisfying 
        \eqref{stima_quadratica0} (with a possibly larger $s_0$).
    \end{itemize}
	\end{proposition}
	\begin{proof}
First of all we define the symbol (recall \Cref{rem:symbols_modozero})
$$\sym{\tm}{}{\sla{1}{2}}(\xi):= \frac{2 \Lambda(\xi)}{3 \gamma | \xi|+ |\xi|^{-1}} \in \wt \Gamma_0^\sla{1}{2}\,.$$ 
Next we define the flow 
\begin{gather}
\pa_\tau \Phi^\tau(\zak)= \bG(\zak) \Phi^\tau(\zak)\,, 
\qquad \Phi^0(\zak)=\uno\,, 
\notag
\\
\bG(\zak):=\vOpbw{\ii \sym{g}{1}{\sla{1}{2}}(\zak;x,\xi)}\,,
\qquad 
\sym{g}{1}{\sla{1}{2}}:= \psi \sym{\tm}{}{\sla{1}{2}}(\xi)
= 
\tM(D)^{-1} \frac{(\zetina-\ov \zetina)}{\ii \sqrt{2}}\sym{\tm}{}{\sla{1}{2}}(\xi)
\in \tilde \Gamma_1^{\sla{1}{2}}\,.
\notag
\end{gather}
Since the symbol $\sym{g}{1}{\sla{1}{2}}$ is real valued, \Cref{lem:flusso} 
ensures that $\Phi^\tau(\zak)$ is well defined and real-to-real and, 
by \eqref{stima:phistd}, satisfies the required estimates \eqref{stima:adm}. 

\noindent
We note also that the generator $\sym{g}{1}{\sla{1}{2}}$ satisfies the homological equation 
\begin{align}
\nabla \psi\cdot \xi 
+ \left\{\sym{g}{1}{\sla{1}{2}},\Lambda(\xi)\right\}=0\,.
    \label{homo:g12}
\end{align}
Define $ {\bf \Psi}_2(\zak):= (\Phi^\tau)_{| \tau=1} $. 
As $U_1$ solves \eqref{diagonale},  the variable $U_2= \Psi_2(\zak)U_1$ 
solves 
		\begin{align}
\pa_t U_2  &= {\bf \Psi}_2(\zak)  
\bA(\zak){\bf \Psi}_2^{-1}(\zak)  \,  U_2  
+   \pa_t{\bf \Psi}_2(\zak)  {\bf \Psi}_2^{-1}(\zak)\, U_2
\notag
\\
&+ {\bf \Psi}_2(\zak)\bR_1(\zak) {\bf \Psi}_2^{-1}(\zak)\, U_2 
+{\bf \Psi}_2(\zak)\bB_{\geq 2}(\zak) 
{\bf \Psi}_2^{-1}(\zak)\, U_2\, , 
	\label{U2.eq}
		\end{align}
        where 
        \begin{align}
            \bA(\zak)&= \vOpbw{-\ii \sym{\Sigma}{}{\sla{3}{2}}}+ \vOpbw{-\ii \tV\cdot \xi}+  \vOpbw{-\ii \left[ \tfrac12 \tQ^{-2}\sym{\lambda}{}{0}+ \sym{a}{1}{0}\right]}\notag
            \\
            &= \bA_1(\zak)+\bA_2(\zak)+\bA_3(\zak)\,.
            \label{def:A}
        \end{align}
		We analyze equation \eqref{U2.eq} expanding each term with its Lie expansion.
        
        \smallskip
		\noindent{\bfseries Expansion of ${\bf \Psi}_2(\zak) \bA(\zak) {\bf \Psi}_2^{-1}(\zak) $.}
        Recalling  $ \bA$ in \eqref{def:A} we expand  
        ${\bf \Psi}_2(\zak) \bA_q(\zak) {\bf \Psi}_2^{-1}(\zak)$, $q=1,2,3$, separately.
        In order to expand ${\bf \Psi}_2(\zak) \bA_1(\zak) {\bf \Psi}_2^{-1}(\zak)$ , 
        we use  the following third order Lie's expansion 
        \begin{equation}\label{Lie:3}
        \begin{aligned}
\big(\Phi^\tau(\zak) \bA_1(\zak) \Phi^\tau(\zak)^{-1}\big)_{|\tau=1}
&= \bA_1(\zak)+ [\bG(\zak),\bA_1(\zak)]+ \frac12{\mathrm{Ad}}_\bG^2[ \bA_1](\zak)
\\
&+ \int_0^1 \frac{(1-\theta)^2}{2} \Phi^\theta(\zak){\mathrm{Ad}}_\bG^3[ \bA_1](\zak)\Phi^{-\theta}(\zak)\, \di\theta\,.
        \end{aligned}
        \end{equation}
        In view of the Lie expansion \eqref{Lie:3}, our first task is to expand the iterated commutators between  
$\bA_1=\vOpbw{-\ii \sym{\Sigma}{}{\sla{3}{2}}}$ and $\bG= \vOpbw{\ii \sym{g}{1}{\sla{1}{2}}}$. 
To this end, we employ the composition rule for paradifferential operators stated in \Cref{compoparapara0}, which allows us to compute compositions and commutators in terms of their symbols obtaining
        \begin{align}
    &{\mathrm{Ad}}_\bG[ \bA_1]= 
[  \bG, \bA_1 ]= 
\vOpbw{-\ii \left\{\sym{g}{1}{\sla{1}{2}}, 
\sym{\Sigma}{}{\sla{3}{2}} \right\}
+ \ii \sym{a}{1}{-1}+ \ii \sym{a}{\geq 2}{{-1}}}
+ \bR_1(\zak)+ \bR_{\geq 2}(\zak)\,,
\label{ad1}
\\
&{\mathrm{Ad}}^2_{\bG} [ \bA_1]= 
\vOpbw{-\ii \left\{\sym{g}{1}{\sla{1}{2}}, 
\left\{\sym{g}{1}{\sla{1}{2}}, \sym{\Sigma}{}{\sla{3}{2}} \right\}\right\}
+\ii \sym{a}{\geq 2}{-\sla{3}{2}} }+ \bR_{\geq 2}(\zak)\,,
\label{ad2}
\\
&{\mathrm{Ad}}^3_{\bG} [ \bA_1]= 
\vOpbw{\ii \sym{a}{\geq 3}{0} }+ \bR_{\geq 2}(\zak)\,,
\label{ad3}
\end{align}
        where $\bR(\zak)$ is a real-to-real matrix of smoothing remainders in $\wt{\cR}^{-\vr}_1$ and $\bR_{\geq 2}(\zak)$ is a real-to-real matrix of smoothing remainders in 
        ${\cR}^{-\vr}_{\geq 2}[r]$ which may change 
        from line to line.
Combining \eqref{Lie:3} with \eqref{ad1}, \eqref{ad2} 
and \eqref{ad3}, and  making use of 
the expansion \eqref{esp:sigma}, we get
\begin{align}
 {\bf \Psi}_2(\zak)  \bA_1(\zak){\bf \Psi}_2^{-1}(\zak)&=   
    \vOpbw{-\ii \left[ \sym{\Sigma}{}{\sla{3}{2}}
+\left\{\sym{g}{1}{\sla{1}{2}}, \sym{\Sigma}{}{\sla{3}{2}} \right\}
+ \left\{\sym{g}{1}{\sla{1}{2}}, \left\{\sym{g}{1}{\sla{1}{2}}, 
\sym{\Sigma}{}{\sla{3}{2}} \right\}\right\}\right]
+ \ii \sym{a}{1}{-1}}+ \bB_{\geq 2}(\zak)\,,
\notag
    \\
     &= \vOpbw{-\ii \left[ \sym{\Sigma}{}{\sla{3}{2}}+\left\{\sym{g}{1}{\sla{1}{2}},
      \Lambda(\xi)\right\}
      + \sym{a}{\geq 2}{1} \right]+ \ii \sym{a}{1}{-1}}
      + \bB_{\geq 2}(\zak)\,,
      \label{conj:A1_2}
\end{align}
where, defining 
$\sym{\Sigma}{\geq 2}{\sla{3}{2}}= 
\sym{\Sigma}{2}{\sla{3}{2}}+ \sym{\Sigma}{\geq 4}{\sla{3}{2}}$,
\begin{align*}
    \sym{a}{\geq 2}{1}:= 
    \left\{\sym{g}{1}{\sla{1}{2}}, \sym{\Sigma}{\geq 2}{\sla{3}{2}}\right\}
    + \left\{\sym{g}{1}{\sla{1}{2}}, \left\{\sym{g}{1}{\sla{1}{2}}, 
    \sym{\Sigma}{}{\sla{3}{2}} \right\}\right\},
\end{align*}
while the operator $\bB_{\geq 2}$ collects all bounded 
quadratic contributions arising in \eqref{ad1}, \eqref{ad2}, and \eqref{ad3}, namely
\begin{align*}
\bB_{\geq 2}(\zak) := \vOpbw{\ii \Big[ \sym{a}{\geq 2}{-1} 
+ \tfrac12\sym{a}{\geq 2}{-\sla{3}{2}} 
+ \sym{a}{\geq 3}{0} \Big]} + \bR_{\geq 2}\,.
\end{align*}
We next apply the second order Lie's expansion \eqref{Lie:2} to the term $\Psi_2(\zak)  \bA_2(\zak)\Psi_2(\zak)^{-1} $. To this end we compute the iterated commutators between the generator $ \tG$ and the operator $ \bA_2= \vOpbw{-\ii \tV \cdot \xi}$ using the composition rule of \Cref{compoparapara0} and the estimate \eqref{stima:AdpG}, we get 
\begin{align}
&{\mathrm{Ad}}_\bG[ \bA_2]= [  \bG, \bA_2 ]
= \vOpbw{-\ii \left\{\sym{g}{1}{\sla{1}{2}}, 
\tV\cdot\xi \right\}+ \ii \sym{a}{\geq 2}{-\sla{3}{2}}}
+ \bR_{\geq 2}(\zak)\,,
            \label{ad1_2}
\\
&{\mathrm{Ad}}^2_{\bG} [ \bA_1]= \bB_{\geq 2}(\zak)\,,
\label{ad2_2}
\end{align}
        where $\bR_{\geq 2}(\zak)\in \cR_{\geq 2}^{-\vr}[r]$ and $\bB_{\geq 2}(\zak)$ satisfies \eqref{stima_quadratica0}.
Then \eqref{Lie:2} with $ \eqref{ad1_2}$ and \eqref{ad2_2} 
imply that 
        \begin{align}
            {\bf \Psi}_2(\zak)  \bA_2(\zak){\bf \Psi}_2^{-1}(\zak)= &\vOpbw{-\ii \left[ \tV\cdot \xi 
            +\left\{\sym{g}{1}{\sla{1}{2}}, \tV\cdot\xi \right\}\right]}+\bB_{\geq 2}(\zak)\notag
            \\
            \stackrel{\eqref{V:exp1}}{=}& \vOpbw{-\ii \left[ \nabla \psi\cdot \xi
            +\fun{\tV}{\geq 2}\cdot \xi +\left\{\sym{g}{1}{\sla{1}{2}}, 
            \tV\cdot\xi \right\}\right]}+\bB_{\geq 2}(\zak)\,,
            \label{conj:A2_2}
        \end{align}
where $\fun{\tV}{\geq 2}= \tB \nabla \eta\in \Sigma\cF_{2}^\R[r,4]$.

\noindent
We finally apply the first order Lie's expansion \eqref{Lie:1} 
to the term ${\bf \Psi}_2(\zak)  \bA_3(\zak){\bf \Psi}_2^{-1}(\zak) $. 
To this end we 
compute the commutator between the generator 
$ \tG$ and the operator 
$ \bA_3= \vOpbw{-\ii \left[ \tfrac12 \tQ^{-2}\sym{\lambda}{}{0}+ \sym{a}{1}{0}\right]}$ 
using the commutator estimate \eqref{stima:AdpG} we get 
\begin{align}
& [  \bG, \bA_3 ]=\bB_{\geq 2}(\zak)\,.
\label{ad1_3}
\end{align}
Then \eqref{Lie:1} with \eqref{ad1_3} imply
\begin{align}
{\bf \Psi}_2(\zak)  \bA_3(\zak){\bf \Psi}_2^{-1}(\zak) = 
\vOpbw{-\ii \left[ \tfrac12 \tQ^{-2}\sym{\lambda}{}{0}
+ \sym{a}{1}{0}\right]}+ \bB_{\geq 2}(\zak)\,. 
\label{conj:A3_2}
\end{align}

\smallskip
\noindent{\bfseries Expansion of $\pa_t {\bf \Psi}_2(\zak) {\bf \Psi}_2^{-1}(\zak)$.} 
We apply the first order Lie's expansion \eqref{Lie:dt}. To this end 
we first compute the time derivative $\pa_t \bG$ of the generator, 
using that $\bG$ is linear with respect to the variable $\zak$ 
and the equation \eqref{eq:zak}-\eqref{X:zak_esp} for $\pa_t\zak$, we get 
\begin{align}
    \pa_t \bG(\zak)= \vOpbw{\ii \sym{g}{1}{\sla{1}{2}}(\hamvec{\cH}(\zak);x,\xi)}
    =\vOpbw{\ii \left[\sym{a}{1}{\sla{1}{2}}+ \sym{a}{2}{\sla{1}{2}}+\sym{a}{\geq 3}{\sla{1}{2}} \right]} \,,
    \label{dtG12}
\end{align}
	where
    \begin{gather*}
        \sym{a}{1}{\sla{1}{2}}(\zak;x,\xi):= 
        \sym{g}{1}{\sla{1}{2}}(-\ii \vomega(D)\zak;x,\xi)
        \in \wt \Gamma_1^\sla{1}{2}\,, 
        \qquad 
        \sym{a}{2}{\sla{1}{2}}(\zak;x,\xi):= 
        \sym{g}{1}{\sla{1}{2}}(\hamvec{2}(\zak);x,\xi)\in \wt \Gamma_2^\sla{1}{2}\,,
        \\
        \qquad \sym{a}{\geq 3}{\sla{1}{2}}(\zak;x,\xi):=
        \sym{g}{1}{\sla{1}{2}}(\hamvec{\geq 3}(\zak);x,\xi) \in \Gamma_{\geq 3}^\sla{1}{2}[r]\,.
    \end{gather*}
Then, using the expansion \eqref{dtG12} and the commutator estimate \eqref{stima:AdpG}, we get 
\begin{align}
[ \bG, \pa_t \bG]= \bB_{\geq 2}(\zak)\,.
\label{comm:dtGG}
\end{align}
Then \eqref{Lie:dt} with \eqref{dtG12} and \eqref{comm:dtGG} imply 
\begin{align}
    \pa_t {\bf \Psi}_2(\zak){\bf \Psi}_2^{-1}(\zak)= 
    \vOpbw{\ii \left[\sym{a}{1}{\sla{1}{2}}
    + \sym{a}{2}{\sla{1}{2}}+\sym{a}{\geq 3}{\sla{1}{2}} \right]}
    + \bB_{\geq 2}(\zak)\,.
    \label{conj:dt_2}
\end{align}

\smallskip
\noindent{ \bfseries Expansion of $ {\bf \Psi}_2(\zak)\bR_1(\zak){\bf \Psi}_2^{-1}(\zak)$}. 
Thanks to \eqref{Lie:1} we get 
		\begin{align}
			{\bf \Psi}_2(\zak)\bR_1(\zak){\bf \Psi}_2^{-1}(\zak)
			=\bR_1(\zak) + \bB_{\geq 2}(\zak)\,. 
            \label{conj:R1_2}
		\end{align}
		
\smallskip
\noindent {\bfseries Estimate of $ {\bf \Psi}_2(\zak)\bB_{\geq 2}(\zak) {\bf \Psi}_2^{-1}(\zak)$.} 
Arguing as in \eqref{stima:B2} we get 
		\begin{align}
			\|{\bf \Psi}_2(\zak)\bB_{\geq 2}(\zak){\bf \Psi}_2^{-1}(\zak)V \|_s 
			\lesssim_s 
			\| \zak \|_{s_0}^2 \|V\|_{s}+ \| \zak\|_{s_0}\| \zak\|_s \| V\|_{s_0}\,.
            \label{est:B2_2}
		\end{align}
In conclusion, combining \eqref{conj:A1_2}, \eqref{conj:A2_2}, 
\eqref{conj:A3_2}, \eqref{conj:dt_2}, \eqref{conj:R1_2}, and \eqref{est:B2_2}, 
and recalling that $\sym{g}{1}{\sla{1}{2}}$ solves the 
homological equation \eqref{homo:g12}, we obtain
\begin{align*}
    \pa_t U_2 =  
    \vOpbw{-\ii \left[\sym{\Sigma}{}{\sla{3}{2}} 
    +\sym{b}{\geq 2}{1}+ \sym{b}{1}{\sla{1}{2}}\right]}U_2 
    + \bR_1(\zak)U_2+ \bB_{\geq 2}(\zak)U_2\,,
\end{align*}
where the symbol $ \sym{b}{\geq 2}{1}\in \Gamma_{\geq 2}^1$ 
gathers all the quadratic contribution of order 
at least $1$ coming from \eqref{conj:A1_2}, 
\eqref{conj:A2_2}, \eqref{conj:A3_2}, \eqref{conj:dt_2}, namely
\begin{align}
    \sym{b}{\geq 2}{1}:= \sym{a}{\geq 2}{1}+\fun{\tV}{\geq 2}\cdot \xi +\left\{\sym{g}{1}{\sla{1}{2}}, \tV\cdot\xi \right\}+ \left[\tfrac12 \tQ^{-2}\sym{\lambda}{}{0}\right]_{\geq 2}+ \sym{a}{2}{\sla{1}{2}}+\sym{a}{\geq 3}{\sla{1}{2}}.
    \label{def:bgeq21}
\end{align}
While the symbol $ \sym{b}{1}{\sla{1}{2}}\in \wt \Gamma_1^\sla{1}{2}$ gathers all the linear contribution coming from \eqref{conj:A1_2}, \eqref{conj:A3_2} and \eqref{conj:dt_2}, namely 
\begin{align*}
    \sym{b}{1}{\sla{1}{2}}:= - \sym{a}{1}{-1}+ \left[\tfrac12 \tQ^{-2}\sym{\lambda}{}{0}\right]_1+ \sym{a}{1}{0}- \sym{a}{1}{\sla{1}{2}}.
\end{align*}
Finally note that the estimate \eqref{stima:diffb3} for the differential  follows 
by the explicit expression \eqref{def:bgeq21}, estimates \eqref{stima:dtQ}, 
\eqref{stima:difflambda} and the analyticity of $\tV$.  
This concludes the proof. 
\end{proof}

	\section{ Quasi-Resonant Normal Form}
    \label{QR-NormalForm}
In this section, we perform a quasi-resonant normal form reduction of the paradifferential system obtained in \Cref{sec:NFclassica}. 

The aim is to conjugate the operator to one whose symbol is in normal form, by removing non-resonant interactions up to smoothing remainders.

To this end, we decompose the symbols into resonant, non-resonant, and smoothing components, and solve suitable homological equations to eliminate the non-resonant contributions. The corresponding transformations are implemented through paradifferential flows generated by symbols of lower order.

This procedure is carried out iteratively at increasing homogeneity orders (linear, quadratic, and cubic), leading to a reduced system where only quasi-resonant interactions are retained.

More precisely, the following definition provides the notion of normal form symbols.
\begin{definition}[Normal form symbols]\label{def:normalform}
Let $s_0>0$, $ m\in \R$, $\delta\in (0,1)$,  $\nu\in (0,1)$ and $\tau\geq 0$. A symbol $z(x,\xi)$ in $\cN^{m,\delta}_{s_0}$ is said to be in normal form
 if
\[
z(x,\xi) = \sum_{k\in\mathbb{Z}^2} \hat z(k,\xi) e^{i k \cdot x}
\]
satisfies
\[
\hat z(k,\xi) \neq 0 \;\;\;\;\Longrightarrow\;\;\;\;
|k\cdot \xi| \leq \langle \xi \rangle^{\delta} |k|^{-\tau}
\quad \text{and} \quad
|k| \leq \langle \xi \rangle^{\nu},
\]
for any $k \neq 0$, $\xi \in \mathbb{R}^2$.
\end{definition}

We shall fix appropriately the parameters $\delta $,
$\tau$ and $\nu$ as
\begin{equation}\label{410}
 \delta \in (\tfrac{7}{8},1), 
\quad \tau > 2, 
\quad 0 < \nu < \tfrac{\delta}{\tau+1}, \qquad \mu:=1-\delta\in (0,\tfrac18) .
\end{equation}

Define a smooth cut-off function
\begin{align*}
    \theta(y):=\begin{cases}
        1& |y|\leq 10\\ 0& |y|\geq 11
    \end{cases}.
\end{align*}

The main result of  this section is the following quasi-resonant normal form result. 
\begin{proposition}
\label{prop:NF_finale}
      Let $\vr>1$ and $\delta$, $\nu$, $\tau$ fixed as in \eqref{410}. 
           Let $\kap \in \R_{+}\setminus \mathscr{N}$ where $\mathscr{N}$ is the zero measure set
     given by Lemma \ref{lem:small_divisors}.
There is $s_0  >\vr $ such that for any $\tR>1$ there is a bounded and invertible transformation  $\bT(\zak)$ 
such that for any $s\geq s_0$ there is $r=r_s>0$ such that for any $\zak \in B_{s_0,\R}(r)\cap H^{s}_{\R}$ the maps $\bT(\zak)$ and $\bT^{-1}(\zak)$ belong to $\cL\left( H^s_\R;H^s_\R\right)$ with estimates 
\begin{align}
\|\bT(\zak)V\|_{s}+ \|\bT^{-1}(\zak)V\|_{s}\lesssim_s \mathrm{exp}\Big((C\tR \|\zak\|_{s_0})^2 \| \zak\|_{s_0}^{\sla{1}{2}}\Big)\left[\| V\|_s+ \| \zak\|_{s}\| V\|_{s_0}\right], 
\label{stima:admR_nonspec}
\end{align}
for any $V \in H^s_\R$, and such that, if $\zak(t) \in B_{s_0, \R}(r)$ solves \eqref{eq:zak} then the variable
  $Z:= \bT(\zak)U_2$, 
  with $U_2$ solving \eqref{NuovoParaprod2},  solves the system
  \begin{equation}\label{QR-normalform}
  \begin{aligned}
   \pa_t Z&= \vOpbw{-\ii \left[\Lambda(\xi)+\sym{\cZ}{1}{1/ 2}+ 
         \sym{\tm}{2}{\sla{3}{2}}(\xi)
         +\sym{\cZ}{2}{\sla{3}{2}}+ \sym{\cZ}{3}{2-\delta}\right]}
         \\&\quad 
         +\vOpbw{- \ii  \left[\sym{\Sigma}{\geq 2}{\sla{3}{2}} 
         + \sym{b}{\geq 2}{1}
         +  \sym{b}{1}{\sla{1}{2}}\right]
         \left(1-\theta\left(\tR^{-1}\langle \xi\rangle \right)\right)} Z
         + \bB_{\geq 2}(\tR;\zak)Z\,,
    \end{aligned}
\end{equation}
        where   $\sym{\Sigma}{\geq 2}{\sla{3}{2}}\in \Gamma_{\geq 2}^{\sla{3}{2}}[r]$ is  in 
    \eqref{def:sdBIS},
    $ \sym{b}{\geq 2}{1} \in \Gamma_{\geq 2}^{\sla{1}{2}}[r]$ 
  is  in \eqref{def:b1mag}, $\sym{b}{1}{\sla{1}{2}}\in \wt \Gamma_1^{\sla{1}{2}}$ is in \eqref{NuovoParaprod2}
 and where
    \begin{itemize}
    \item  $\sym{\tm}{2}{\sla{3}{2}}$ is a $2$-homogeneous, $x$-independent,  
    $\tR$-localized,  real symbol in $\wt \Gamma_2^{\sla{3}{2}, \delta}$;
    
        \item for $(m,p)=(\sla{1}{2},1),\, (\sla{3}{2},2),\,(2-\delta,3)$,  
        $\sym{\cZ}{p}{m}$ is a $p$-homogeneous, $\tR$-localized,  
        real symbol in normal form (see \Cref{def:normalform}), 
        belonging to $\wt \Gamma_p^{m,\delta}$;
        
        \item $\bB_{\geq 2}(\tR;\zak)$ is a family of  matrices of real-to-real bounded operators, 
        satisfying the following: for any $s\geq s_0$ 
        there is a constant $C=C_s>0$ such that for any 
        $\tR>1$, $\zak(t) \in B_{s_0, \R}(r)\cap H^s_\R(\T^2;\C^2)$, 
        $V\in H^s_\R(\T^2;\C^2)$, one has 
        \begin{align}
            \| \bB_{\geq 2}(\tR;\zak) V \|_s \leq 
            C \mathrm{exp}\Big({C(\tR \| \zak\|_{s_0})^2}\Big) 
            \left( \| \zak\|_{s_0}^2\| V\|_s+ \| \zak\|_{s_0}\| \zak\|_s\| V\|_{s_0}\right)\,.
            \label{est:Bexp}
        \end{align}
    \end{itemize}
	\end{proposition}
    The rest of the section is devoted to the proof of \Cref{prop:NF_finale}.

\subsection{Preliminaries}\label{preliminary resonant normal form}
In order to prove \Cref{prop:NF_finale}, we need some further symbolic calculus. 
First of all, we consider an even smooth cut-off function $\chi : \R \to [0,1]$ with the property 
that $\chi(y) = 1$ for all $|y| \leq \tfrac12$ and $\chi(y) = 0$ for all $|y| \geq 1$. Following \cite{FM2022} we introduce the following decomposition.

    \begin{definition}\label{def:45}
Given $\nu, \delta, \tau$ as in \eqref{410}, define the following functions:
\[
\chi_k(\xi) = \chi\!\left(\frac{2|k|^\tau k\cdot\xi}{\langle \xi \rangle^\delta}\right), 
\qquad
\tilde \chi_k(\xi) = \chi\!\left(\frac{|k|}{\langle \xi \rangle^\nu}\right), \quad k \in \mathbb{Z}^2 \setminus \{0\}.
\]
Correspondingly, given a symbol 
\[
\sym{a}{}{m}=\sum_{k \in \Z^2} \hat a(k,\xi)e^{\ii k\cdot x} \in \cN^{m,\delta}_{s_0}\,,
\] 
we decompose it as
\[
\sym{a}{}{m} = \langle \sym{a}{}{m} \rangle +\sym{a}{}{m, \mathtt res}+ \sym{a}{}{m, \mathtt nr}   
+ \sym{a}{}{m, \mathtt S}\,,
\]
where $\langle a \rangle$ is the $x$–average of the symbol, namely
\begin{align*}
\langle \sym{a}{}{m} \rangle(\xi) = \frac{1}{(2\pi)^2}\int_{\mathbb{T}^2} \sym{a}{}{m}(x,\xi)\, \di x\,,
\end{align*}
and
\begin{gather}
    \sym{a}{}{m, \mathtt res}(x,\xi) = \sum_{k \neq 0} \chi_k(\xi)\tilde\chi_k(\xi)\, \hat a(k,\xi)e^{\ii k\cdot x}, \qquad \sym{a}{}{m, \mathtt nr}(x,\xi) = \sum_{k \neq 0} (1-\chi_k(\xi))\tilde\chi_k(\xi)\, \hat a(k,\xi)e^{\ii k\cdot x}\label{def:res}\\
\sym{a}{}{m,\mathtt S}(x,\xi) = \sum_{k \neq 0} (1-\tilde \chi_k(\xi))\, \hat a(k,\xi)e^{ik\cdot x}.
\label{def:smoo}
\end{gather}
We also define
\begin{align}
\tg[\sym{a}{}{m}](x,\xi) := - \sum_{k \neq 0} \frac{1}{(k\cdot \xi)} 
(1-\chi_k(\xi))\tilde\chi_k(\xi)\, \hat a(k,\xi)e^{\ii k\cdot x}\,.
\label{def:g_pare}
\end{align}

\end{definition}
In \cite[Lemma 4.7]{FM2022} the following lemma was proved.
\begin{lemma}
\label{lem:decomposition_symbols}
Let $m \in \mathbb{R}$. Then, for any $s \geq 0$, the linear map
\[
\cN^m_s \to \cN^m_s, \quad \sym{a}{}{m}\mapsto \langle \sym{a}{}{m} \rangle
\]
is linear and continuous. For any $s \geq 0$, there exists $\sigma_s > s$ large enough such that the maps
\[
\cN^m_{\sigma_s} \to \cN^m_s\,, 
\qquad \sym{a}{}{m} \mapsto \sym{a}{}{m,\mathtt nr}\,, \;\;\; \sym{a}{}{m} \mapsto \sym{a}{}{m,\mathtt res}\,,
\]
\[
\cN^m_{\sigma_s} \to \cN^{m-\delta}_s, \quad \sym{a}{}{m} \mapsto g[\sym{a}{}{m}]\,,
\]
are linear and continuous. Moreover, let $N \in \mathbb{N}_0$. 
Then there exists $s_0 = s_0(N) > 0$ large enough such that, for any $s \geq s_0$, the map
\begin{align}
\cN^m_{s_0} \to \mathcal{L}(H^s,H^{s+N}), \quad \sym{a}{}{m} 
\mapsto R^{({\mathtt S})}(\sym{a}{}{m}) := \Opbw{\sym{a}{}{m,\mathtt S}}\,,
\label{lin:cont_smoo}
\end{align}
is linear and continuous.
\end{lemma}
\begin{lemma}
\label{lem:deco_homo}
    Let $m\in \R$ and $p\in \N$. If $\sym{a}{p}{m} \in \wt \Gamma_p^m$ then: 
    \begin{enumerate}[(i)]
        \item $\sym{a}{p}{m,\mathtt res}$ belongs to $ \wt \Gamma_p^m$; 
         \item\label{item:deco_g} $ \tg\left[ \sym{a}{p}{m}\right]$ 
         belongs to $\wt \Gamma_p^{m- \delta}$;
         \item \label{item:smooOpbw} If $p\geq 2$ the operator 
         $\vOpbw{\sym{a}{p}{m,\mathtt S }}$ belongs to $ \cL(H^s_\R;H^s_\R)$ 
         for any $ s\geq 0$ with estimates 
         \begin{align*}
             \| \vOpbw{\sym{a}{p}{m,\mathtt S }}V \|_s \lesssim_s \| \zak\|_{s_0}^2\| V\|_s\,,
         \end{align*}
         for any $ \zak\in H^{s_0}_\R$ and $ V \in H^s_\R$. 
         \item \label{item:smoothingS} for $p=1$ the matrix 
         $\vOpbw{\sym{a}{1}{m,\mathtt S }}$ is a matrix of real-to-real operators in 
         $ \wt \cR^{-\vr}_1$ for any $ \vr>1$.
    \end{enumerate}
\end{lemma}
	\begin{proof}
Items $(i)$ and $(ii)$ follow applying \Cref{lem:decomposition_symbols} and 
recalling \eqref{bound:homosym}. 
The translation invariant property \eqref{def:sym_mome} follows 
noting that by their definitions \eqref{def:res}-\eqref{def:g_pare} 
one has 
\[
\tau_\upsilon\big( \sym{a}{p}{m,\mathtt res}\big)= 
\sym{(\fun{\tau_\upsilon a}{p})}{}{m,\mathtt res}\,, 
\qquad 
\tau_\upsilon\big(\tg[ \sym{a}{p}{m}]\big)= 
\tg[\tau_\upsilon \sym{a}{p}{m}]\,.
\]
 Item $(iii)$ follows from \eqref{lin:cont_smoo} 
 with $N=0$ and the fact that the matrix 
 $\vOpbw{\sym{a}{p}{m,\tS}}$ is real-to-real.
 Finally item $(iv)$ follows from 
 \eqref{lin:cont_smoo} with $N=\vr$ and using also 
 \[
\tau_\upsilon\big( \vOpbw{\sym{a}{p}{m,\tS}}\big)=  
\vOpbw{\tau_\upsilon(\sym{a}{p}{m,\tS})}\tau_\upsilon
=\vOpbw{\sym{\big[\tau_\upsilon(\sym{a}{p}{m})\big]}{}{\tS}}\tau_\upsilon\,.
 \]
This concludes the proof. 
	\end{proof}
In the following we will collect some useful results that will be used in subsections 
\ref{sec:linsym}, \ref{sec:quadsym}, \ref{sec:cubicsym} to transform 
$R$-localized symbols into symbols in normal form.

\begin{lemma}[Homological equation]\label{homo:solve0}
 Fix $ \delta, \nu, \tau, \mu $ as in \eqref{410}.
		Let $ m\in \R$. For any $\tR$-localized real symbol 
		$\sym{a}{1}{m}\in \wt \Gamma_1^{m,\delta}$, 
		there is a $\tR$-localized real symbol 
		$\sym{g}{1}{m+\sla{1}{2}-\delta}$ 
		in $ \wt  \Gamma^{m+\sla{1}{2}-\delta,\delta}_{1}$
        such that the non-resonant component of $\sym{a}{1}{m}$ is removed, namely
		\be\label{eq:homologica}
		\sym{a}{1}{m}+ \{ \sym{g}{1}{m+\sla{1}{2}-\delta}, \Lambda(\xi) \}  
		= \sym{a}{1}{m, \mathtt res}+ \sym{a}{1}{m, \mathtt S}\,.
		\ee
		\end{lemma}

    \begin{proof}
    As $ \sym{a}{1}{m} \in \wt \Gamma_1^{m,\delta}$, we expand it 
    \begin{align*}
        \sym{a}{1}{m}(\zak;x,\xi)= 
        \sum_{\substack{k\in \Z^2\setminus\{0\}\\ \sigma \in \{\pm\}}} 
        a_k^\sigma(\xi) \zetina_k^\sigma e^{\ii \sigma k\cdot x}
        = 
        \sum_{k\in \Z^2\setminus\{0\}}\hat a(k,\xi)  e^{\ii k \cdot x }\,, 
    \end{align*}
    where 
    $ \hat{a}(k,\xi)= 
    a_k^+(\xi) \zetina_k^+ +a_{-k}^-(\xi) \ov{ \zetina_{-k}}$. 
    We note then that $\langle \sym{a}{1}{m} \rangle=0$. 
    We define, recalling \eqref{def:g_pare},
    \begin{align*}
        \sym{g}{1}{m+\sla{1}{2}-\delta}(\zak;x,\xi):=
        \frac{2|\xi|\Lambda(\xi)}{3\kap |\xi|^2+1 } \tg[ \sym{a}{1}{m}]
        \end{align*}
    Thanks to \Cref{lem:deco_homo} and the fact that  
    $ \frac{2|\xi|\Lambda(\xi)}{3\kap |\xi|^2+1 }$ belongs to 
    $ \wt \Gamma_0^\sla{1}{2}$, the symbol 
    $\sym{g}{1}{m+\sla{1}{2}-\delta}(\zak;x,\xi)$ belongs to 
    $\wt \Gamma_1^{m+\sla{1}{2}-\delta,\delta}$. 
    Moreover, since $\sym{a}{1}{m}$ is $\tR$-localized, 
    the symbol $\sym{g}{1}{m+\sla{1}{2}-\delta}$ is $\tR$-localized as well. 
    Finally an explicit computation, noting also that 
    $ \nabla_\xi \Lambda(\xi)= \left[\frac{3\kap |\xi|^2+1 }{2|\xi|\Lambda(\xi)}\right] \xi$ ,  
    gives 
        \begin{align*}
           \sym{a}{1}{m}+ \left\{ \sym{g}{1}{m+\sla{1}{2}-\delta}, 
           \Lambda(\xi)\right\}
           = \sym{a}{1}{m}- \nabla_x \sym{g}{1}{m+\sla{1}{2}
           -\delta}\cdot \nabla_\xi \Lambda(\xi)= \sym{a}{1}{m}
           - \sym{a}{1}{m,\mathtt nr}=\sym{a}{1}{m, \mathtt res}+ \sym{a}{1}{m, \mathtt S}\,,
        \end{align*}
        thus proving \eqref{eq:homologica}.
    \end{proof}
    
    Using \Cref{homo:solve0} we prove the following result. 
    \begin{lemma}[Linear $R$-localized symbols]
    \label{homo:solve0_finale0}
 Fix $ \delta, \nu, \tau, \mu $ as in \eqref{410}. Let $ \sym{a}{1}{\sla{1}{2}}$ 
 be a $\tR$-localized, real symbol in $\wt \Gamma_1^{\sla{1}{2},\delta}$ 
 then for any $ \vr>1$ and $ N\in \N$ with $N\geq 2$, 
 there exists a $\tR$-localized symbol  
 $\sym{g}{1}{\mu} \in \wt \Gamma_{1}^{\mu,\delta}$ such that 
 \begin{equation}\label{eq:homo_finale1}
 \begin{aligned}
&\vOpbw{
-\ii \Big[\sym{a}{1}{\sla{1}{2}}+\sym{g}{1}{\mu}(-\ii \vomega(D)\zak;x,\xi)
+ \ii \left(\sym{g}{1}{\mu}\sha{N} \Lambda(\xi)- \Lambda(\xi) \sha{N} \sym{g}{1}{\mu}\right)
\Big]}
\\&
\qquad \qquad\qquad\qquad\qquad \qquad\qquad\qquad\qquad
\qquad \qquad
= \vOpbw{-\ii \sym{\cZ}{1}{\sla{1}{2}}}+\bR_1(\zak)\,,
        \end{aligned}
        \end{equation}
        where $\sym{\cZ}{1}{\sla{1}{2}} $ is a $\tR$-localized, normal form, real symbol in $\wt \Gamma_1^{\frac{1}{2},\delta}$ and $\bR_1(\zak)$ is a matrix of real-to-real smoothing remainders in 
        $\wt \cR_1^{-\vr}$.
    \end{lemma}
\begin{proof}
 For any $\tR$-localized, real  symbol $\sym{a}{1}{m}\in \Gamma^{m,\delta}_1$ we define (recall \eqref{Moyal}) the $\tR$-localized, real symbol
    \begin{align*}
        \tL[\sym{a}{1}{m}](\zak;,x,\xi):=  \sym{a}{1}{m}(-\ii \vomega(D)\zak;x,\xi)+ \ii \left[\sym{a}{1}{m}\sha{N} \Lambda(\xi)- \Lambda(\xi) \sha{N} \sym{a}{1}{m}\right]- \left\{ \sym{a}{1}{m}, \Lambda(\xi)\right\}\,.
    \end{align*}
By \eqref{Moyal}--\eqref{est:Moyal} and the fact that the substitution $\zak \mapsto -\ii \vomega(D)\zak$ preserves homogeneity and $\tR$-localization, one has $\tL[\sym{a}{1}{m}] \in \wt \Gamma_1^{m,\delta}$.
     Then the symbol in the left hand side of \eqref{eq:homo_finale1} becomes 
    \begin{align*}
        -\ii \left[\sym{a}{1}{\sla{1}{2}}+ \left\{ \sym{g}{1}{\mu}, \Lambda(\xi)\right\}+ \tL\left[ \sym{g}{1}{\mu}\right]\right].
    \end{align*}
    Let $ \tq\in \N$ such that $ \mu-\tq(\delta- \tfrac12)\leq \vr$ and we define
\begin{align*}
\sym{g}{1}{\mu}&:= \sym{f}{1}{\mu}+\sym{f}{1}{\mu-(\delta-\sla{1}{2})}
+ \ldots+ \sym{f}{1}{\mu-\tq(\delta-\sla{1}{2})}\,, 
\\
\sym{\cZ}{1}{\sla{1}{2}}&:= \sym{a}{1}{\sla{1}{2},\mathtt res}
+\sym{a}{1}{\mu,\mathtt res}+ \ldots+ \sym{a}{1}{\mu-(\tq-1)(\delta-\sla{1}{2}),\mathtt res} \,,
\\
 \sym{b}{1}{\sla{1}{2}, \mathtt S}&:= \sym{a}{1}{\sla{1}{2},\mathtt S}
+\sym{a}{1}{\mu,\mathtt S}+ \ldots
+ \sym{a}{1}{\mu-(\tq-1)(\delta-\sla{1}{2}),\mathtt S}\,.
    \end{align*}
    We now determine inductively the symbols in the expansions above. 
    First we define the $\tR$-localized, real symbol  
    $h\left[\sym{a}{1}{m}\right]$ as the solution in $\wt\Gamma_1^{m-(\delta-\sla{1}{2})}$ 
    of the homological equation \eqref{eq:homologica}. 
Then,  $\sym{f}{1}{\mu-n(\delta-\sla{1}{2})}$ are 
determined inductively by 
\begin{align*}
 \sym{f}{1}{\mu}:= h\left[\sym{a}{1}{\sla{1}{2}}\right]\in \wt \Gamma_1^{\mu,\delta}\,, 
 \qquad 
 \sym{f}{1}{\mu-n(\delta-\sla{1}{2})}:= 
 h\left[\tL\left[\sym{f}{1}{\mu-(n-1)(\delta-\sla{1}{2})}\right]\right]
 \in \wt \Gamma_1^{\mu-n(\delta-\sla{1}{2}), \delta}, \quad n=1, \ldots, \tq\,,
\end{align*}
and $ \sym{a}{1}{\sla{1}{2}, \mathtt res}$  and $ \sym{a}{1}{\sla{1}{2}, \mathtt S}$ are the resonant part and the smoothing part of $\sym{a}{1}{\sla{1}{2}}$ (see \eqref{def:res} and \eqref{def:smoo}) and 
\begin{align*}
     &\sym{a}{1}{\mu-n(\delta-\sla{1}{2}), \mathtt res}:= \sym{\left( \tL\left[ \sym{f}{1}{\mu-n(\delta-\sla{1}{2})}\right]\right)}{}{\mathtt res} \in \wt \Gamma_1^{\mu-n(\delta-\sla{1}{2}), \delta}, \quad n=1, \ldots, \tq-1,\\
     &\sym{a}{1}{\mu-n(\delta-\sla{1}{2}), \mathtt S}:= \sym{\left( \tL\left[ \sym{f}{1}{\mu-n(\delta-\sla{1}{2})}\right]\right)}{}{\mathtt S} \in \wt \Gamma_1^{\mu-n(\delta-\sla{1}{2}), \delta}, \quad n=1, \ldots, \tq-1.
\end{align*}
With these definitions one has 
\begin{align*}
\vOpbw{-\ii \left[\sym{a}{1}{\sla{1}{2}}+ \left\{ \sym{g}{1}{\mu}, \Lambda(\xi)\right\}+ \tL\left[ \sym{g}{1}{\mu}\right]\right]}= \vOpbw{-\ii \sym{\cZ}{1}{\sla{1}{2}}}+ \vOpbw{-\ii \left[\tL\left[\sym{f}{1}{\mu-\tq(\delta-\sla{1}{2})}\right]+ \sym{b}{1}{\sla{1}{2}, \mathtt S}\right] }.
\end{align*}
Since $\tL\left[\sym{f}{1}{\mu-\tq(\delta-\sla{1}{2})}\right]\in \wt \Gamma_1^{-\vr}$ and the operator $ \vOpbw{ -\ii \sym{b}{1}{\sla{1}{2}, \mathtt S}}\in \wt \cR^{-\vr}_1$ by \Cref{item:smoothingS} of \Cref{lem:deco_homo}, we obtain \eqref{eq:homo_finale1} with $ \bR_1(\zak):=\vOpbw{-\ii \left[\tL\left[\sym{f}{1}{\mu-\tq(\delta-\sla{1}{2})}\right]+ \sym{b}{1}{\sla{1}{2}, \mathtt S}\right] }\in \wt \cR^{-\vr}_1$, thus concluding the proof.
\end{proof}

 \begin{lemma}[Quadratic $R$-localized symbols]
    \label{homo:solve0_finale1}
        Fix $ \delta, \nu, \tau, \mu $ as in \eqref{410}. 
        Let $\sym{a}{2}{\sla{3}{2}}$ a $\tR$-localized, real symbol in 
        $\wt \Gamma_2^{\sla{3}{2},\delta}$ then there exists a $\tR$-localized, 
        real symbol  $\sym{g}{2}{1+\mu} \in \wt \Gamma_{2}^{1+\mu,\delta}$ 
        with zero average such that
        \begin{equation}\label{eq:homo_finale2}
        \begin{aligned}
&\vOpbw{-\ii \left[\sym{a}{2}{\sla{3}{2}}
+\di_\zak\sym{g}{2}{1+\mu}[-\ii \vomega(D)\zak]
+ \ii \left(\sym{g}{2}{1+\mu}\sha{4} \Lambda(\xi)
- \Lambda(\xi) \sha{4} \sym{g}{2}{1+\mu}\right)\right]}
\\&
\qquad\qquad\qquad\qquad\qquad
\qquad\qquad\qquad
= \vOpbw{-\ii \left\langle \sym{a}{2}{\sla{3}{2}}\right\rangle -\ii\sym{\cZ}{2}{\sla{3}{2}}}
+\bB_{\geq 2}(\tR;\zak)\,,
        \end{aligned}
       \end{equation}
        where $\sym{\cZ}{2}{\sla{3}{2}} $ is a $\tR$-localized, normal form, 
        real symbol in $\wt \Gamma_2^{\sla{3}{2},\delta}$ and 
        $\bB_{\geq 2}(\tR;\zak)\in \cL(H^s_\R;H^s_\R)$ 
        satisfies the admissible quadratic bound \eqref{est:Bexp}.
    \end{lemma}	
    \begin{proof}
We first note that, by \eqref{Moyal}-\eqref{est:Moyal}, one has 
\begin{align*}
    \left|\ii \left[\sym{g}{2}{1+\mu}\sha{4} \Lambda(\xi)
    - \Lambda(\xi) \sha{4} \sym{g}{2}{1+\mu}\right]
    - \{ \sym{g}{2}{1+\mu}, \Lambda(\xi)\}\right|_{-\sla{1}{2}+4\mu, s_0-\mu-\vr}
    \lesssim 
    |  \sym{g}{2}{1+\mu}|_{1+\mu,s_0-\mu}\lesssim \| \zak\|_{s_0}^2\,.
\end{align*}
Then, using also \Cref{thm:action}, one has the expansion  
\begin{align}
    \vOpbw{ \left[\sym{g}{2}{1+\mu}\sha{4} \Lambda(\xi)- \Lambda(\xi) \sha{4} \sym{g}{2}{1+\mu}\right]}= \vOpbw{\tfrac{1}{\ii}\{\sym{g}{2}{1+\mu},\Lambda(\xi)\}}+ \bB_{\geq 2}(\zak),
    \label{approx_Moyal}
\end{align}
where \( \bB_{\geq 2}(\zak) \) satisfies the quadratic bound \eqref{est:Bexp}, and is therefore included on the right-hand side of \eqref{eq:homo_finale2}.
For any symbol $ \sym{a}{2}{m}\in \wt \Gamma_2^m$, we define $\th[ \sym{a}{2}{m}]:= \frac{2|\xi|\Lambda(\xi)}{3\kap |\xi|^2+1 } \tg[ \sym{a}{2}{m}]$, where $\tg[ \sym{a}{2}{m}]$ is defined in \eqref{def:g_pare}. 
We note that 
\begin{align}
    \sym{a}{2}{m}+ \{ \th[ \sym{a}{2}{m}], \Lambda(\xi)\}= \langle \sym{a}{2}{m}\rangle +\sym{a}{2}{m,\mathtt res}+ \sym{a}{2}{m,\mathtt S},
    \label{eq:homoproof2}
\end{align}
where $\langle \sym{a}{2}{m}\rangle$,  $\sym{a}{2}{m,\mathtt res}$  and $\sym{a}{2}{m,\mathtt S}$ are respectively the average, the resonant part and the smoothing part of $\sym{a}{2}{m}$ (see \eqref{def:res} and \eqref{def:smoo}). 
Moreover, thanks to \Cref{item:deco_g} in \Cref{lem:deco_homo} and the fact that $ \frac{2|\xi|\Lambda(\xi)}{3\kap |\xi|^2+1 }\in \wt \Gamma_0^{\sla{1}{2}}$, we have that 
$\mathtt{h}[ \sym{a}{2}{m}]\in \wt \Gamma_2^{m+\sla{1}{2}-\delta}$.  Then we define
\begin{align*}
\sym{g}{2}{1+\mu}&:= 
\sym{f}{2}{1+\mu}+ \sym{f}{2}{\sla{1}{2}+2\mu}+\sym{f}{2}{3\mu}+ \sym{f}{2}{-\sla{1}{2}+4\mu}, 
\\
\sym{\cZ}{2}{\sla{3}{2}}&:= \sym{a}{2}{\sla{3}{2},\mathtt res}
+ \sym{a}{2}{1+\mu,\mathtt res}
+ \sym{a}{2}{\sla{1}{2}+2\mu,\mathtt res}
+\sym{a}{2}{3\mu,\mathtt res}\,,
\\
 \sym{b}{2}{\sla{3}{2}, \mathtt S}&:= 
 \sym{a}{2}{\sla{3}{2},\mathtt S}
 +\sym{a}{2}{1+\mu,\mathtt S}
 + \sym{a}{2}{\sla{1}{2}+2\mu,\mathtt S}+\sym{a}{2}{3\mu,\mathtt S}\,,
 \end{align*}
where, $\sym{f}{2}{1+\mu-n(\delta-\sla{1}{2})}$, $ n=0,1,2,3$,  are determined inductively by 
\begin{equation*}
\begin{aligned}
    \sym{f}{2}{1+\mu}&:= \th\left[\sym{a}{2}{\sla{3}{2}}\right]\in \wt \Gamma_2^{1+\mu,\delta}\,, 
    \\
    \sym{f}{2}{1+\mu-n(\delta-\sla{1}{2})}&:=
    \th\left[\di_\zak\sym{f}{2}{\mu-(n-1)(\delta-\sla{1}{2})}[-\ii \vomega(D)\zak]\right]
    \in \wt \Gamma_2^{1+\mu-n(\delta-\sla{1}{2}), \delta}\,.
\end{aligned} 
\end{equation*}
The symbols $ \sym{a}{2}{\sla{3}{2}, \mathtt res}$  and $ \sym{a}{2}{\sla{3}{2}, \mathtt S}$ are the resonant part and the smoothing part of $\sym{a}{2}{\sla{3}{2}}$  and 
\begin{align*}
&\sym{a}{2}{1+\mu-n(\delta-\sla{1}{2}), \mathtt res}
     := \sym{\left( \di_\zak \sym{f}{2}{1+\mu-n(\delta-\sla{1}{2})}[-\ii \vomega(D)\zak]\right)}{}{\mathtt res} 
     \in \wt \Gamma_2^{\mu-n(\delta-\sla{1}{2}), \delta}\,,
     \quad n=0,1,2\,,
     \\
     &\sym{a}{2}{1+\mu-n(\delta-\sla{1}{2}), \mathtt S}
     := \sym{\left( \di_\zak \sym{f}{2}{1+\mu-n(\delta-\sla{1}{2})}[-\ii \vomega(D)\zak]\right)}{}{\mathtt S} 
     \in \wt \Gamma_2^{\mu-n(\delta-\sla{1}{2}), \delta}\,, \quad n=0,1,2\,.
\end{align*}
With these definitions, and using also \eqref{approx_Moyal} 
and \eqref{eq:homoproof2} and the fact that 
$\langle \di_\zak \sym{f}{2}{1+\mu-n(\delta-\sla{1}{2})}[-\ii \vomega(D)\zak] \rangle=0$, 
we get 
\begin{align*}
 &\vOpbw{\sym{a}{2}{\sla{3}{2}}+\di_\zak\sym{g}{2}{1+\mu}[-\ii \vomega(D)\zak]
 + \ii \left[\sym{g}{2}{1+\mu}\sha{\vr} \Lambda(\xi)- \Lambda(\xi) \sha{\vr} \sym{g}{2}{1+\mu}\right]}
 \\
    &\qquad \qquad\qquad \qquad  
    = \vOpbw{\sym{\cZ}{2}{\sla{3}{2}}}
    +\vOpbw{\di_\zak \sym{f}{2}{-\sla{1}{2}+4\mu}[ -\ii \vomega(D)\zak]
    + \sym{b}{2}{\sla{3}{2},\mathtt S}}+\bB_{\geq 2}(\tR;\zak)\,.
\end{align*}
Applying \Cref{lem:deco_homo} to $\sym{b}{2}{\sla{3}{2},\mathtt S}$  and \Cref{thm:action} to $\di_\zak \sym{f}{2}{-\sla{1}{2}+4\mu}[ -\ii \vomega(D)\zak]\in \wt \Gamma_2^{-\sla{1}{2}+4\mu}\subset \wt \Gamma_2^{0}$, we prove that  the operator $\vOpbw{\di_\zak \sym{f}{2}{-\sla{1}{2}+4\mu}[ -\ii \vomega(D)\zak]+ \sym{b}{2}{\sla{3}{2},\mathtt S}}$  satisfies the quadratic bound \eqref{est:Bexp}. Hence, we include it in the remainder $ \bB_{\geq 2}(\tR;\zak)$ and we establish \eqref{eq:homo_finale2}.
    \end{proof}
    
\begin{lemma}[Cubic $R$-localized symbols]
    \label{homo:solve0_finale2}
        Fix $ \delta, \nu, \tau, \mu $ as in \eqref{410}. Let $ \sym{a}{3}{1+\mu}$ a $\tR$-localized, real symbol in $\wt \Gamma_3^{1+\mu,\delta}$ with zero average, then  there exists a $\tR$-localized real symbol  $\sym{g}{3}{\sla{1}{2}+2\mu} \in \wt \Gamma_{3}^{\sla{1}{2}+2\mu,\delta}$ such that 
\begin{align}
\vOpbw{-\ii \left[\sym{a}{3}{1+\mu}+ \ii \left(\sym{g}{3}{\sla{1}{2}+2\mu}\sha{3} \Lambda(\xi)- \Lambda(\xi) \sha{3} \sym{g}{3}{\sla{1}{2}+2\mu}\right)\right]}
= \vOpbw{-\ii \sym{\cZ}{3}{1+\mu}}+\bB_{\geq 2}(\tR;\zak)\,,
\label{eq:homo_finale3}
\end{align}
        where $\sym{\cZ}{3}{1+\mu} $ is a $\tR$-localized, normal form, real symbol in $\wt \Gamma_3^{1+\mu,\delta}$ and $\bB_{\geq 2}(\tR;\zak)\in \cL(H^s_\R;H^s_\R)$ satisfies the admissible quadratic bound \eqref{est:Bexp}.
    \end{lemma}	
\begin{proof}
    We define $\sym{g}{3}{\sla{1}{2}+2\mu}:= \frac{2|\xi|\Lambda(\xi)}{3\kap |\xi|^2+1 } \tg[ \sym{a}{3}{1+\mu}]$, where $g[ \sym{a}{3}{1+\mu}]$ is defined in \eqref{def:g_pare}. First of all \Cref{item:deco_g}  of \Cref{lem:deco_homo} and the fact that $\frac{2|\xi|\Lambda(\xi)}{3\kap |\xi|^2+1 }\in \wt \Gamma_0^{\sla{1}{2}}$, ensure that $\sym{g}{3}{\sla{1}{2}+2\mu}$ is a real symbol in $\wt \Gamma_3^{\sla{1}{2}+2\mu}$. Moreover, it is $\tR$-localized, since $\sym{a}{3}{1+\mu}$ is $\tR$-localized by hypothesis. 
    The symbol
    \begin{align*}
        \ii \left(\sym{g}{3}{\sla{1}{2}+2\mu}\sha{3} \Lambda(\xi)- \Lambda(\xi) \sha{3} \sym{g}{3}{\sla{1}{2}+2\mu}\right)-\{ \sym{g}{3}{\sla{1}{2}+2\mu},\Lambda(\xi)\} \in \wt \Gamma_{3}^{-1+5\mu}\subset \wt \Gamma_{3}^{0} ,
    \end{align*}
    therefore its corresponding Bony-Weyl operator can be included in the quadratic remainder $\tB_{\geq 2}(\tR;\zak)$, in view of \eqref{stima:exp_opbw}. We note that, since $\sym{a}{3}{1+\mu}$ has zero average, one has the following algebraic identity 
    \begin{align*}
\sym{a}{3}{1+\mu}+ 
\{\sym{g}{3}{\sla{1}{2}+2\mu}, \Lambda(\xi)\}
=\sym{a}{3}{1+\mu,\mathtt{res}}+ \sym{a}{3}{1+\mu,\tS}\,,
    \end{align*}
 where $\sym{a}{3}{1+\mu,\mathtt{res}}$ and 
 $ \sym{a}{3}{1+\mu,\tS}$ are respectively the resonant 
 and smoothing component of $\sym{a}{3}{1+\mu}$ 
 defined as in \eqref{def:res}-\eqref{def:smoo}. 
 Thanks to \Cref{item:smooOpbw} of \Cref{lem:deco_homo}, 
 the symbol $\sym{a}{3}{1+\mu,\tS}$ defines a 
 paradifferential operator which we include in 
 $\bB_{\geq 2}(\tR;\zak)$ 
and the thesis follows with 
$\sym{\cZ}{3}{1+\mu}:= \sym{a}{3}{1+\mu,\mathtt{res}}$.
\end{proof}

	\subsection{Normal form of linear symbol}\label{sec:linsym}
	In this section we put the low-frequencies component of the quadratic quasi-linear symbol $ \sym{b}{1}{\sla{1}{2}}(\zak;x,\xi)$, appearing in \eqref{NuovoParaprod2}, into its quasi-resonant normal form. Precisely, we prove the following. 
	\begin{proposition}
\label{prop:NF_linear}
      Let $\vr>1$ and $\delta$, $\tau$, $\nu$ as in \eqref{410}. 
There is $s_0 >0 $ such that for any $\tR>1$ there is a bounded and invertible transformation $\bF_1(\zak)$ 
such that for any $s\geq s_0$ there is $r=r_s>0$
such that  for any $\zak \in B_{s_0,\R}(r)$ the maps $\bF_1(\zak)$ and $\bF_1^{-1}(\zak)$ belong to $\cL\left( H^s_\R;H^s_\R\right)$ with estimates 
\begin{align}
\|\bF_1(\zak)\|_{\cL(H^s_\R;H^s_\R)}+ \|\bF_1^{-1}(\zak)\|_{{\cL(H^s_\R;H^s_\R)}}\lesssim_s \mathrm{exp}\Big((C\tR \|\zak\|_{s_0})^2 \| \zak\|_{s_0}^\sla{1}{2}\Big), 
\label{stima:admR}
\end{align}
and such that
if $\zak(t) \in B_{s_0, \R}(r)$ solves \eqref{eq:zak} then the variable
  $V_1:= \bF_1(\zak)U_2$, 
  with $U_2$ solving \eqref{NuovoParaprod2},  solves the system
  \begin{align}
\pa_t V_1 &= 
\vOpbw{-\ii 
\left[\Lambda(\xi)+ \sym{\cZ}{1}{1\mkern-2mu/\mkern-1.5mu 2}\right]
-\ii\sym{b}{2}{\sla{3}{2},\tR}
- \ii   \left[ \sym{\Sigma}{\geq 2}{\sla{3}{2}}
+ \sym{b}{\geq 2}{1}+ \sym{b}{1}{\sla{1}{2}}\right]\left(1-\theta\left(\tR^{-1}\langle\xi\rangle\right)\right)} V_1\notag
\\
&+ \bR_1(\zak)V_1+ \bB_{\geq 2}(\tR;\zak)V_1,
\label{QR-linear_normalform}
    \end{align}
    where   $\sym{\Sigma}{\geq 2}{\sla{3}{2}}\in \Gamma_{\geq 2}^{\sla{3}{2}}[r]$  in 
    \eqref{def:sdBIS},
    $ \sym{b}{\geq 2}{1} \in \Gamma_{\geq 2}^{\sla{1}{2}}[r]$ 
    in \eqref{def:b1mag} and $\sym{b}{1}{\sla{1}{2}}\in \wt \Gamma_1^{\sla{1}{2}}$ 
    are the real symbols in  appearing in \eqref{NuovoParaprod2} and
    \begin{itemize}
     \item  $\sym{\cZ}{1}{\sla{1}{2}}$ is a  $1$-homogeneous, $\tR$-localized, real symbol in $\wt \Gamma_1^{\sla{1}{2}, \delta}$, in normal form (see \Cref{def:normalform});
     \item $\sym{b}{2}{\sla{3}{2},\tR}$ is a $2$-homogeneous, $ \tR$-localized,  real symbol in $\wt \Gamma_2^{\sla{3}{2},\delta}$; 
        \item $\bR_1(\zak)$ is a matrix of real-to-real smoothing operators in $ \wt \cR_1^{-\vr}$;
        \item $\bB_{\geq 2}(\tR;\zak)$ is a family of  matrices of real-to-real bounded operators, satisfying \eqref{est:Bexp}.
    \end{itemize}
	\end{proposition}

\begin{proof}
We define the flow
		\be \label{flow_1}
		\begin{cases}
			\pa_\tau \Phi_{g_1}^\tau(\zak) =  \bG(\zak)\Phi^\tau_{g_1}(\zak)\\
			\Phi^0(\zak)=\id
		\end{cases}, \quad \bG(\zak):= \vOpbw{\ii \sym{g}{1}{\mu}(\zak;x,\xi)},
		\ee
        where $\sym{g}{1}{\mu}(\zak;x,\xi)\in \wt \Gamma_1^{\mu,\delta}$ is the $\tR$-localized symbol obtained by applying \Cref{homo:solve0_finale0} to the $\tR$-localized symbol $\sym{a}{1}{\sla{1}{2}}:=\sym{b}{1}{\sla{1}{2}} \theta(\tR \langle \xi \rangle) \in \wt \Gamma_1^{\sla{1}{2}}$. 
		We define $ \bF_1(\zak):=  \Phi_{g_1}^\tau(\zak)_{|\tau=1}$.
        Thanks to \Cref{lem:flussoR}, $\bF_1(\zak)$ is a bounded and invertible transformation, satisfying \eqref{stima:admR}.  
        
        Since $U_2$ solves \eqref{NuovoParaprod2}, the variable $V_1=\bF_1(\zak)U_2$  solves 
		\begin{align} 
			\partial_t V_1= \bF_1(\zak) \bA(\zak)\bF_1(\zak)^{-1} V_1+\pa_t \bF_1(\zak)\bF_1(\zak)^{-1}V_1+\bF_1(\zak)\big[\bR_1(\zak)+\bB_{\geq 2}(\zak)\big]\bF_1(\zak)^{-1}V_1
            \label{proof:V1:1}. 
		\end{align}
 where 
        \begin{align*}
     \bA(\zak)&=
\vOpbw{-\ii \Lambda(\xi)}
+\vOpbw{ -\ii \sym{b}{1}{\sla{1}{2}}}
+ \vOpbw{-\ii \left[\sym{\Sigma}{\geq 2}{\sla{3}{2}} 
+ \sym{b}{\geq 2}{1}\right]}\notag
\\
&= -\ii \vomega(D) +\bA_1(\zak)+\bA_2(\zak)\,.
\label{def:AV1}
\end{align*}
We analyze equation \eqref{proof:V1:1} expanding each term with its Lie expansion.
        
\smallskip
\noindent
{\bfseries Expansion of $\bF_1(\zak) [-\ii \vomega(D)] \bF_1(\zak)^{-1} $.}
In order to expand 
$\bF_1(\zak) [-\ii \vomega(D)] \bF_1(\zak)^{-1} $ , 
we use  the  second order 
Lie's expansion \eqref{Lie:2}, getting  
\begin{equation}\label{esp:V1omega}
   \begin{aligned}
 \bF_1(\zak) [-\ii \vomega(D)] \bF_1(\zak)^{-1}&=  
 -\ii \vomega(D) 
 + \left[ \bG(\zak),  -\ii \vomega(D) \right] 
 \\&+\int_{0}^1 (1-\theta) \Phi_{g_1}^\theta(\zak){\mathrm{Ad}}_{\bG}^2[ -\ii \vomega(D)]  
 \Phi_{g_1}^{-\theta}(\zak)\, \di \theta\,.
        \end{aligned}
\end{equation}
Then we expand the commutator $\left[ \bG(\zak),  -\ii \vomega(D) \right] $ using \eqref{eq:Gomega}-\eqref{eq:omegaG} with $N\geq \vr +\mu+\tfrac32$, obtaining  
\begin{align}
 \left[ \bG(\zak),  -\ii \vomega(D) \right]= \vOpbw{-\ii \left[\ii \sym{g}{1}{\mu}\sha{N} \Lambda(\xi)- \Lambda(\xi) \sha{N} \ii \sym{g}{1}{\mu} \right]}   + \bR_1(\zak), \qquad \bR_1 \in \wt \cR^{-\vr}_1.
 \label{primo:V1}
\end{align}
Then applying \Cref{lem:stima_comm} to the symbol  
$-\ii \left[\ii \sym{g}{1}{\mu}\sha{N} \Lambda(\xi)
- \Lambda(\xi) \sha{N} \ii \sym{g}{1}{\mu} \right] \in \wt \Gamma_1^{\sla{1}{2}+2\mu}$, 
the $\tR$-localized symbol $\sym{g}{1}{\mu}\in \tilde \Gamma_1^{\mu,\delta}$ 
and   parameters $p_1=p_2=1$, $m_1= \mu$, 
$m_2= \tfrac{1}{2}+2\mu$ and 
$\sigma:= \max\{0, -\tfrac{1}{2}+4\mu\}=0$, 
obtaining 
\begin{equation}\label{doppio:V1}
\begin{aligned}
\left[\bG(\zak), [ \bG(\zak), -\ii \vomega(D)] \right]
&=
\left[\bG(\zak), \vOpbw{-\ii \left[\ii \sym{g}{1}{\mu}\sha{N} \Lambda(\xi)
- \Lambda(\xi) \sha{N} \ii \sym{g}{1}{\mu} \right]}\right]
+ [ \bG(\zak), \bR_1(\zak)]
\\&=\bB_{\geq 2}(\zak)\,,
\end{aligned}
\end{equation}
satisfying \eqref{est:Bexp}. Gathering \eqref{esp:V1omega}, \eqref{primo:V1} 
and \eqref{doppio:V1}  and using the bound \eqref{stima:phiexp} for 
$\Phi_{\fun{g}{1}}^\theta$, we get 
\begin{align}
     \bF_1(\zak) [-\ii \vomega(D)] \bF_1(\zak)^{-1}
     = \vOpbw{-\ii \Lambda(\xi)-\ii \left[\ii \sym{g}{1}{\mu}\sha{N} \Lambda(\xi)
     - \Lambda(\xi) \sha{N} \ii \sym{g}{1}{\mu} \right]}+\bR_1(\zak)+ \bB_{\geq 2}(\zak)\,,
     \label{conj:omeghino_mu}
\end{align}
where $\bR_1\in \wt\cR^{-\vr}_1$ and  $\bB_{\geq 2}$ satisfies \eqref{est:Bexp}.

\smallskip        
\noindent
{\bfseries Expansion of $\bF_1(\zak) \bA_1(\zak)\bF_1(\zak)^{-1} $.}
We use \eqref{Lie:1} getting 
\begin{align}
\bF_1(\zak) [\bA_1(\zak)] \bF_1(\zak)^{-1} 
= \vOpbw{-\ii \sym{b}{1}{\sla{1}{2}}}+ \int_0^1\Phi_{g_1}^\theta (\zak)
\left[ \vOpbw{\sym{g}{1}{\mu}},\vOpbw{-\ii \sym{b}{1}{\sla{1}{2}}}\right]
\Phi_{g_1}^{-\theta}(\zak)\, \di \theta\,.
\end{align}
 Then we apply \Cref{lem:stima_comm} with $ m_1=\mu$, $m_2= \frac12$ and $\sigma= \max\{0, \mu+\frac12-\delta\}=0$ and we get that $\left[ \vOpbw{\sym{g}{1}{\mu}},\vOpbw{-\ii \sym{b}{1}{\sla{1}{2}}}\right] $ satisfies \eqref{eq:stima_comm}, then, using also the bound \eqref{stima:phiexp} for $\Phi_{\fun{g}{1}}^\theta$, we obtain
 \begin{align}
     \bF_1(\zak) [\bA_1(\zak)] \bF_1(\zak)^{-1}= \Opbw{-\ii \sym{b}{1}{\sla{1}{2}}}+ \bB_{\geq 2}(\zak),
     \label{conj:b12_mu}
 \end{align}
 where $ \bB_{\geq 2}(\zak)$ is a bounded quadratic map satisfying \eqref{est:Bexp}. 

\smallskip 
\noindent
{\bfseries Expansion of $\bF_1(\zak) \bA_2(\zak)\bF_1(\zak)^{-1} $.}
We introduce the notation  
$\sym{b}{\geq 2}{\sla{3}{2}}:= \sym{\Sigma}{\geq 2}{\sla{3}{2}}+ \sym{b}{\geq 2}{1}$. 
Then, using \eqref{Lie:1}, get 
\begin{align}
\bF_1(\zak) [\bA_2(\zak)] \bF_1(\zak)^{-1} = \vOpbw{-\ii \sym{b}{\geq 2}{\sla{3}{2}}}+ \int_0^1\Phi_{g_1}^\theta (\zak)\left[ \vOpbw{\sym{g}{1}{\mu}},\vOpbw{-\ii \sym{b}{\geq 2}{\sla{3}{2}}}\right]\Phi_{g_1}^{-\theta}(\zak)\, \di \theta.
\end{align}
 Then we apply \Cref{lem:stima_comm} with $p_1=1$, $p_2=2$, 
 $ m_1=\mu$, $m_2= \frac32$ and 
 $\sigma= \mu+\frac32-\delta=\frac12+2\mu >0$. Since $p_1+p_2=3 > 2+ \sigma$, 
 we get that $\left[ \vOpbw{\sym{g}{1}{\mu}},\vOpbw{-\ii \sym{b}{\geq 2}{\sla{3}{2}}}\right] $ 
 satisfies \eqref{eq:stima_comm}. 
 Then, using the bound \eqref{stima:phiexp} for $\Phi_{\fun{g}{1}}^\theta$, we obtain
 \begin{align}
 \bF_1(\zak) [\bA_2(\zak)] \bF_1(\zak)^{-1}
 = \Opbw{-\ii [\sym{\Sigma}{\geq 2}{\sla{3}{2}}+ \sym{b}{\geq 2}{1}]}
 + \bB_{\geq 2}(\tR;\zak)\,,
 \label{conj:sigma_mu}
 \end{align}
 where $ \bB_{\geq 2}(\tR;\zak)$ is a bounded quadratic map satisfying \eqref{est:Bexp}. 

\smallskip
\noindent
{\bfseries Expansion of $\pa_t \bF_1(\zak) \bF_1(\zak)^{-1} $.} 
We apply the first order Lie's expansion \eqref{Lie:dt}. 
To this end we first compute the time derivative $\pa_t \bG$ of the generator, 
using that $\bG$ is linear with respect to the variable $\zak$ and the equation \eqref{eq:zak}-\eqref{X:zak_esp} for $\pa_t\zak$, we get 
\begin{align}
    \pa_t \bG(\zak)= \vOpbw{\ii \sym{g}{1}{\mu}(\hamvec{\cH}(\zak);x,\xi)}
    =\vOpbw{\ii \left[\sym{a}{1}{\mu}+ \sym{a}{2}{\mu}+\sym{a}{\geq 3}{\mu} \right]} \,,
    \label{dtGmu}
\end{align}
	where
	\begin{equation}\label{def:apiccolini}
    \begin{aligned}
        \sym{a}{1}{\mu}(\zak;x,\xi)&:= \sym{g}{1}{\mu}(-\ii \vomega(D)\zak;x,\xi)
        \in \wt \Gamma_1^{\mu,\delta}\,, 
        \qquad 
        \sym{a}{2}{\mu}(\zak;x,\xi):= 
        \sym{g}{1}{\mu}(\hamvec{2}(\zak);x,\xi)\in \wt \Gamma_2^{\mu,\delta}\,,
        \\ 
        \sym{a}{\geq 3}{\mu}(\zak;x,\xi)&:=
        \sym{g}{1}{\mu}(\hamvec{\geq 3}(\zak);x,\xi) 
        \in \Gamma_{\geq 3}^{\mu,\delta}[r]\,,
    \end{aligned}
    \end{equation}
    are $\tR$-localized symbols. 
Then, using the expansion \eqref{dtGmu} and  
\Cref{lem:stima_comm}, we get that 
\begin{align}
[ \bG, \pa_t \bG]= \bB_{\geq 2}(\tR;\zak),
\label{comm:dtGGmu}
\end{align}
satisfies the quadratic bound \eqref{est:Bexp}.
Then, by the Lie's expansion \eqref{Lie:dt} and \eqref{dtGmu}, \eqref{comm:dtGGmu} 
and noting also that $\vOpbw{\ii \sym{a}{\geq 3}{\mu}}$ satisfies 
\eqref{est:Bexp} thanks to \eqref{stima:exp_opbw}, we eventually get 
\begin{align}
    \pa_t \bF_1(\zak)\bF_1(\zak)^{-1}= \vOpbw{\ii \left[\sym{a}{1}{\mu}
    + \sym{a}{2}{\mu}\right]}+ \bB_{\geq 2}(\tR;\zak)\,.
    \label{conj:dt_mu}
\end{align}

 \smallskip
\noindent{\bfseries Expansion of $\bF_1(\zak) \bR_1(\zak)\bF_1(\zak)^{-1} $.}
 Using \eqref{Lie:1}, get 
\begin{align}
\bF_1(\zak) [\bR_1(\zak)] \bF_1(\zak)^{-1} =& \bR_1(\zak)+ \int_0^1\Phi_{g_1}^\theta (\zak)\left[ \vOpbw{\sym{g}{1}{\mu}},\bR_1(\zak)\right]\Phi_{g_1}^{-\theta}(\zak)\, \di \theta\notag \\
=&\bR_1(\zak)+ \bB_{\geq 2}(\zak),
\label{conj:R_mu}
\end{align}
where $ \bB_{\geq 2}(\zak)$ satisfies \eqref{est:Bexp}.

\smallskip
\noindent
{\bfseries Expansion of $\bF_1(\zak) \bB_{\geq 2}(\zak)\bF_1(\zak)^{-1} $.} Using the bound \eqref{stima:phiexp} for $\bF_1^\pm(\zak) $ , we get 
\begin{align}
\| \bF_1(\zak) \bB_{\geq 2}(\zak)\bF_1(\zak)^{-1} V\|_{s}
&\lesssim_s 
e^{C_s(\tR \| \zak\|_{s_0})^2} 
\| \bB_{\geq 2}(\zak)\bF_1(\zak)^{-1} V\|_{s}\notag 
\\
&\lesssim_s 
e^{C_s(\tR \| \zak\|_{s_0})^2}\left( 
\| \zak\|_{s_0}^2 \|\bF_1(\zak)^{-1} V\|_s
+ \| \zak\|_{s_0}\|\zak\|_s
\| \bF_1(\zak)^{-1} V\|_{s_0}\right) \notag 
\\
&\lesssim_s 
e^{C_s(\tR \| \zak\|_{s_0})^2}\left( 
\| \zak\|_{s_0}^2 
\| V\|_s+ \| \zak\|_{s_0}\|\zak\|_s
\| V\|_{s_0}\right)\,, 
\label{bounf:B2_mu}
\end{align}
proving that 
$\bF_1(\zak) \bB_{\geq 2}(\zak)\bF_1(\zak)^{-1}$ 
can be included in $ \bB_{\geq 2}(\tR;\zak)$.
Finally, collecting \eqref{proof:V1:1}, \eqref{conj:omeghino_mu}, 
\eqref{conj:b12_mu}, \eqref{conj:sigma_mu}, \eqref{conj:dt_mu}, 
\eqref{def:apiccolini}, \eqref{conj:R_mu}, and \eqref{bounf:B2_mu}, 
and applying \eqref{eq:homo_finale1} to the $\tR$-localized symbol 
$\sym{b}{1}{\sla{1}{2}}\theta(\tR^{-1}\langle \xi\rangle)$, we obtain
\begin{align}
\pa_t V_1&= \vOpbw{-\ii \left[ \Lambda(\xi)+ \sym{\Sigma}{\geq 2}{\sla{3}{2}}+ \sym{b}{\geq 2}{1}
+ \sym{a}{2}{\mu}+ \sym{b}{1}{\sla{1}{2}}\left(1-\theta(\tR^{-1}\langle \xi\rangle)\right)\right]}V_1
+ \bR_1(\zak)V_1+ \bB_{\geq 2}(\tR;\zak)V_1
\notag
\\
&+ \vOpbw{-\ii \left[ 
\sym{b}{1}{\sla{1}{2}}\theta(\tR^{-1}\langle \xi \rangle)
+ \sym{a}{1}{\mu}
+ \ii \left( \sym{g}{1}{\mu}\sha{\vr} \Lambda(\xi)
- \Lambda(\xi) \sha{\vr}\sym{g}{1}{\mu}\right)\right]}   
\notag
\\
&\stackrel{\eqref{eq:homo_finale1}}{=}
\vOpbw{-\ii \left[\Lambda(\xi)
+ \sym{\cZ}{1}{\sla{1}{2}}\right]
-\ii\left[ \sym{\Sigma}{2}{\sla{3}{2}}
+ \sym{b}{2}{1}\right]\theta(\tR^{-1}\langle\xi\rangle)
- \ii   \left[ \sym{b}{\geq 2}{\sla{3}{2}}
+ \sym{b}{1}{\sla{1}{2}}\right]\left(1-\theta\left(\tR^{-1}\langle\xi\rangle\right)\right)} 
V_1
\notag
\\
&+ \bR_1(\zak)V_1+ \bB_{\geq 2}(\tR;\zak)V_1\,.
\nonumber
\end{align}
This proves the \Cref{prop:NF_linear} having defined 
in \eqref{QR-linear_normalform} 
$\sym{b}{2}{\sla{3}{2},\tR}:= 
\sym{\Sigma}{2}{\sla{3}{2}}+ \sym{b}{2}{1}$.
\end{proof}

\subsection{Normal form of the quadratic quasi-linear symbol}\label{sec:quadsym}
In order to obtain the final equation \eqref{QR-normalform}, we construct a transformation that puts the low-frequency quadratic quasilinear symbol $\sym{b}{2}{\sla{3}{2},\tR}$, from the intermediate equation \eqref{QR-linear_normalform}, into quasi-resonant normal form.
This transformation is generated by a quadratic symbol of positive order, 
obtained by solving a quadratic homological equation
given in Lemma \ref{homo:solve0_finale1}.
The main result of this subsection is the normal form proposition for the quadratic 
quasilinear symbol
below.
	\begin{proposition}
\label{prop:NF_quadratic}
Fix $\delta, \nu, \tau, \mu$ as in \eqref{410} and let $\vr > 2$.  
There exists $s_0 > 0$ such that, for any $\tR > 1$, there exists a bounded and invertible transformation $\bF_2(\zak)$ with the following property: for any $s \geq s_0$ there exists $r = r_s > 0$ such that, for any $\zak \in B_{s_0,\R}(r)$, the maps $\bF_2(\zak)$ and $\bF_2^{-1}(\zak)$ belong to 
$\cL(H^s_\R; H^s_\R)$ and satisfy the estimate \eqref{stima:admR}.  

Moreover, if $\zak(t) \in B_{s_0,\R}( r)$ solves \eqref{eq:zak}, then the variable $V_2 := \bF_2(\zak)V_1$,
with $V_1$ solving \eqref{QR-linear_normalform}, solves the system

\begin{align}
\pa_t V_2 &= \vOpbw{-\ii 
\left[\Lambda(\xi)
+\sym{\cZ}{1}{\sla{1}{2}}
+\sym{\tm}{2}{\sla{3}{2}}(\xi)+\sym{\cZ}{ 2}{\sla{3}{2}}\right]} V_2
+ \vOpbw{-\ii\,  \sym{b}{3}{1+\mu, \tR}}V_2\notag +\bR_1(\zak)V_2
\\
&+\vOpbw{- \ii  \left[\sym{\Sigma}{\geq 2}{\sla{3}{2}}
+\sym{b}{\geq 2}{1} 
+ \sym{b}{1}{\sla{1}{2}}\right]
\left(1-\theta\left(\tR^{-1}|\xi|\right)\right)}V_2
+ \bB_{\geq 2}(\tR;\zak)V_2\,,
         \label{QR-quadratic_normalform}
\end{align}
    where $\sym{\Sigma}{\geq 2}{\sla{3}{2}}\in \Gamma_{\geq 2}^{\sla{3}{2}}[r]$ is in 
    \eqref{def:sdBIS},
    $ \sym{b}{\geq 2}{1} \in \Gamma_{\geq 2}^{\sla{1}{2}}[r]$ 
   is  in \eqref{def:b1mag} and $\sym{b}{1}{\sla{1}{2}}\in \wt \Gamma_1^{\sla{1}{2}}$ is in  \eqref{NuovoParaprod2} and 
    \begin{itemize}
\item  $\sym{\tm}{2}{\sla{3}{2}}$ is a $2$-homogeneous, $x$-independent,  $\tR$-localized,  real symbol in $\wt \Gamma_2^{\sla{3}{2}, \delta}$;
\item  $\sym{\cZ}{2}{\sla{3}{2}}$ is a $2$-homogeneous, $\tR$-localized,  real symbol in $\wt \Gamma_2^{\sla{3}{2}, \delta}$, in normal form (see \Cref{def:normalform});
        \item $\sym{b}{3}{1+\mu,\tR}$ is a  $3$-homogeneous, $\tR$-localized, real symbol in $\wt \Gamma_2^{\sla{3}{2}, \delta}$, with zero average;
        \item $\bR_1(\zak)$ is a matrix of smoothing operators in $ \wt \cR_1^{-\vr}$;
        \item $\bB_{\geq 2}(\tR;\zak)$ is a family of  matrices of real-to-real quadratic bounded operators, satisfying \eqref{est:Bexp}.
    \end{itemize}
	\end{proposition}
	\begin{proof}
We define the flow
\begin{equation*}
\begin{cases}
\pa_\tau \Phi_{g_2}^\tau(\zak) =  \bG(\zak)\Phi^\tau_{g_2}(\zak)
\\
\Phi^0(\zak)=\id
\end{cases}\,, 
\qquad 
\bG(\zak):= \vOpbw{\ii \sym{g}{2}{1+\mu}(\zak;x,\xi)}\,,
\end{equation*}
        where $\sym{g}{2}{1+\mu}\in \wt \Gamma_2^{1+\mu,\delta}$ is the $\tR$-localized symbol obtained by applying \Cref{homo:solve0_finale1} to the $\tR$-localized, real symbol $\sym{a}{2}{\sla{3}{2}}:=\sym{b}{2}{\sla{3}{2},\tR} \in \wt \Gamma_2^{\sla{3}{2}}$. 
We define $ \bF_2:=  \Phi_{g_2}^\tau(\zak)_{|\tau=1}$. 
Since $V_1$ solves \eqref{QR-linear_normalform}, the variable $V_2=\bF_2(\zak)V_1$  solves 
		\begin{align} 
\partial_t V_2= \bF_2(\zak) \bA(\zak)\bF_2(\zak)^{-1} V_2
+\pa_t \bF_2(\zak)\bF_2(\zak)^{-1}V_2
+\bF_2(\zak)\big[\bR_1(\zak)+\bB_{\geq 2}(\zak)\big]
\bF_2(\zak)^{-1}V_2
\label{proof:V2:1}
\end{align}
 where 
\begin{align*}
\bA(\zak)=&-\ii \vomega(D)
+\vOpbw{ -\ii\left[\sym{\cZ}{1}{\sla{1}{2}}
+ \sym{b}{1}{\sla{1}{2}}(1-\theta\left(\tR^{-1} \langle \xi \rangle\right)\right] }
+ \vOpbw{-\ii \left[\sym{b}{2}{\sla{3}{2},\tR}
+ \sym{b}{\geq 2}{\sla{3}{2}}(1-\theta\left(\tR^{-1}\langle \xi \rangle\right))\right]}\notag
\\
=& -\ii \vomega(D) +\bA_1(\zak)+\bA_2(\zak)\,,
\end{align*}
where (see \eqref{QR-linear_normalform})
\begin{equation}\label{bb32mag}
\sym{b}{\geq 2}{\sla{3}{2}}:=\sym{\Sigma}{\geq 2}{\sla{3}{2}}+ \sym{b}{\geq 2}{1}\,.
\end{equation}
We analyze equation \eqref{proof:V2:1} 
expanding each term with its Lie expansion.
        
        \smallskip
		\noindent{\bfseries Expansion of $\bF_2(\zak) [-\ii \vomega(D)] \bF_2(\zak)^{-1} $.}
        In order to expand $\bF_2(\zak) [-\ii \vomega(D)] \bF_2(\zak)^{-1} $ , we use  the  second order Lie's expansion \eqref{Lie:2}, getting  
 \begin{equation} \label{esp:V2omega}
  \begin{aligned}
   \bF_2(\zak) [-\ii \vomega(D)] \bF_2(\zak)^{-1}&=  
   -\ii \vomega(D) + \left[ \bG(\zak),  -\ii \vomega(D) \right] 
   \\&+\int_{0}^1 (1-\theta) \Phi_{g_2}^\theta(\zak){\mathrm{Ad}}_{\bG}^2
   [ -\ii \vomega(D)]  \Phi_{g_2}^{-\theta}(\zak)\, \di \theta\,.
        \end{aligned}
        \end{equation}
Then we expand the commutator $\left[ \bG(\zak),  -\ii \vomega(D) \right] $ using \Cref{compoparapara0}, 
obtaining 
\begin{align}
 \left[ \bG(\zak),  -\ii \vomega(D) \right]= 
 \vOpbw{-\ii \left[\ii \sym{g}{2}{1+\mu}\sha{4} \Lambda(\xi)
 - \Lambda(\xi) \sha{4} \ii \sym{g}{2}{1+\mu} \right]}   + \bR_2(\zak)\,, 
 \qquad 
 \bR_2 \in \wt \cR^{-\sla{3}{2}+\mu}_2\,.
 \label{primo:V2}
\end{align}
Then, in order to estimate the double commutator in \eqref{esp:V2omega}, we apply \Cref{lem:stima_comm}-(i) to the $\tR$-localized symbol
\[
\sym{g}{2}{1+\mu}\in \wt \Gamma_2^{1+\mu,\delta}\,,
\]
and to the symbol
\[
-\ii\big[\ii \sym{g}{2}{1+\mu}\sha{4}\Lambda(\xi)-\Lambda(\xi)\sha{4}\ii \sym{g}{2}{1+\mu}\big]
\in \wt \Gamma_2^{\sla{3}{2}+2\mu}\,.
\]
We choose the parameters
\[
p_1=p_2=2\,,\qquad m_1=1+\mu\,,\qquad m_2=\tfrac32+2\mu\,,
\qquad 
\sigma:=\tfrac32+4\mu>0\,.
\]
Moreover, we apply \Cref{lem:stima_comm}-(ii) to the commutator $[\bG(\zak),\bR_2(\zak)]$, with
\[
\vr:=\tfrac32-\mu\ge m_2:=1+\mu \,.
\]
As a result, we obtain
\begin{align}
 \left[\bG(\zak), [ \bG(\zak), -\ii \vomega(D)] \right]&=
 \left[\bG(\zak), 
 \vOpbw{-\ii \left[\ii \sym{g}{2}{1+\mu}\sha{4} \Lambda(\xi)
 - \Lambda(\xi) \sha{4} \ii \sym{g}{2}{1+\mu} \right]}\right]
 + [ \bG(\zak), \bR_2(\zak)]\nonumber
 \\
   &= \bB_{\geq 2}(\tR;\zak)\,,
    \label{doppio:V22}
\end{align}
satisfying \eqref{est:Bexp}. Gathering \eqref{esp:V2omega}, \eqref{primo:V2} 
and \eqref{doppio:V22}  and using the bound 
\eqref{stima:phiexp} for $\Phi_{\fun{g}{2}}^\theta$, we get 
\begin{align}
     \bF_2(\zak) [-\ii \vomega(D)] \bF_2(\zak)^{-1}= 
     \vOpbw{-\ii \Lambda(\xi)-\ii \left[\ii \sym{g}{2}{1+\mu}\sha{4} \Lambda(\xi)
     - \Lambda(\xi) \sha{4} \ii \sym{g}{2}{1+\mu} \right]}+ \bB_{\geq 2}(\tR;\zak)\,,
     \label{conj:omeghino_mu2}
\end{align}
        where $ \bB_{\geq 2}(\tR;\zak)$ satisfies \eqref{est:Bexp}.

        \medskip
        
\noindent{\bfseries Expansion of $\bF_2(\zak) \bA_1(\zak)\bF_2(\zak)^{-1} $.}
We use \eqref{Lie:1} getting 
\begin{align*}
\bF_2(\zak) [\bA_1(\zak)] \bF_2(\zak)^{-1} &= 
\vOpbw{-\ii\left[\sym{\cZ}{1}{\sla{1}{2}}
+ \sym{b}{1}{\sla{1}{2}}(1-\theta\left(\tR^{-1} \langle \xi \rangle\right)\right] }
\\&\!\!\!\!\!\!
+ \int_0^1\Phi_{g_2}^\theta (\zak)\left[ \vOpbw{\sym{g}{2}{1+\mu}},
\vOpbw{-\ii\left[\sym{\cZ}{1}{\sla{1}{2}}
+ \sym{b}{1}{\sla{1}{2}}(1-\theta\left(\tR^{-1} \langle \xi \rangle\right)\right] }\right]
\Phi_{g_2}^{-\theta}(\zak)\, \di \theta\,.
\end{align*}
 Then we apply \Cref{lem:stima_comm} with $ m_1=1+\mu$, $m_2= \frac12$ and $\sigma=  \frac12+2\mu>0$ and we get that the commutator 
 \[
 \left[ \vOpbw{\sym{g}{2}{1+\mu}},\vOpbw{-\ii\left[\sym{\cZ}{1}{\sla{1}{2}}+ \sym{b}{1}{\sla{1}{2}}(1-\theta\left(\tR^{-1} \langle \xi \rangle\right)\right] }\right]
 \]
 satisfies \eqref{eq:stima_comm}. Then, using also the bound \eqref{stima:phiexp} for $\Phi_{\fun{g}{2}}^\theta$, we obtain
 \begin{align}
     \bF_2(\zak) [\bA_1(\zak)] \bF_2(\zak)^{-1}= \Opbw{-\ii\left[\sym{\cZ}{1}{\sla{1}{2}}+ \sym{b}{1}{\sla{1}{2}}(1-\theta\left(\tR^{-1} \langle \xi \rangle\right)\right]}+ \bB_{\geq 2}(\tR;\zak),
     \label{conj:b12_(1+mu)}
 \end{align}
 where $ \bB_{\geq 2}(\tR;\zak)$ is a bounded quadratic map satisfying \eqref{est:Bexp}. 
 
 \smallskip
\noindent
{\bfseries Expansion of $\bF_2(\zak) \bA_2(\zak)\bF_2(\zak)^{-1} $.}
Using \eqref{Lie:1}, get 
\begin{equation*}
\begin{aligned}
\bF_2(\zak) [\bA_2(\zak)] &\bF_2(\zak)^{-1} = 
\vOpbw{-\ii \left[\sym{b}{2}{\sla{3}{2},\tR}
+ \sym{b}{\geq 2}{\sla{3}{2}}(1-\theta\left(\tR^{-1}\langle \xi \rangle\right))\right]}
\\&
+ \int_0^1\Phi_{g_2}^\theta (\zak)\left[ \vOpbw{\sym{g}{2}{1+\mu}},
\vOpbw{-\ii \left[\sym{b}{2}{\sla{3}{2},\tR}
+ \sym{b}{\geq 2}{\sla{3}{2}}(1-\theta\left(\tR^{-1}\langle \xi \rangle\right))\right]}\right]
\Phi_{g_2}^{-\theta}(\zak)\, \di \theta\,.
\end{aligned}
\end{equation*}
 Then we apply \Cref{lem:stima_comm} with $p_1=p_2=2$, $ m_1=1+\mu$, $m_2= \tfrac32$ and $\sigma= \tfrac32+2\mu>0$. Since $p_1+p_2=4 > 2+ \sigma$, we get that $\left[ \vOpbw{\sym{g}{1}{\mu}},\vOpbw{-\ii \sym{b}{\geq 2}{\sla{3}{2}}}\right] $ satisfies \eqref{eq:stima_comm}. Then, using the bound \eqref{stima:phiexp} for $\Phi_{\fun{g}{1}}^\theta$, we obtain
 \begin{align}
 \bF_2(\zak) [\bA_2(\zak)] \bF_2(\zak)^{-1}= \Opbw{-\ii \left[\sym{b}{2}{\sla{3}{2},\tR}+ \sym{b}{\geq 2}{\sla{3}{2}}(1-\theta\left(\tR^{-1}\langle \xi \rangle\right))\right]}+ \bB_{\geq 2}(\tR;\zak)\,,
 \label{conj:sigma_(1+mu)}
 \end{align}
 where $ \bB_{\geq 2}(\tR;\zak)$ is a bounded quadratic map satisfying \eqref{est:Bexp}. 

\smallskip
\noindent{\bfseries Expansion of $\pa_t \bF_2(\zak) \bF_2(\zak)^{-1} $.} We apply the first order Lie's expansion \eqref{Lie:dt}. To this end we first compute the time derivative $\pa_t \bG$ of the generator, using that $\bG$ is $2$-homogeneous with respect to the variable $\zak$ and the equation \eqref{eq:zak}-\eqref{X:zak_esp} for $\pa_t\zak$, we get 
\begin{align}
    \pa_t \bG(\zak)= \vOpbw{\ii \di_\zak \sym{g}{2}{1+\mu}[\hamvec{\cH}(\zak)]}=\vOpbw{\ii \left[\di_\zak \sym{g}{2}{1+\mu}[-\ii \vomega(D)\zak]+ \sym{b}{3}{1+\mu,\tR}+\sym{b}{\geq 4}{1+\mu} \right]} 
    \label{dtGmu2}
\end{align}
	where
\begin{gather}
         \di_\zak\sym{g}{2}{1+\mu}[-\ii \vomega(D)\zak]\in 
         \wt \Gamma_2^{1+\mu,\delta}\,, 
         \qquad 
         \sym{b}{3}{1+\mu,\tR}(\zak;x,\xi):= 
         \di_\zak\sym{g}{2}{1+\mu}[\hamvec{2}(\zak)]\in \wt \Gamma_3^{1+\mu,\delta}\,,
         \notag\\
        \qquad 
        \sym{b}{\geq 4}{1+\mu}(\zak;x,\xi)
        :=\di_\zak\sym{g}{2}{1+\mu}[\hamvec{\geq 3}(\zak)] 
        \in \Gamma_{\geq 4}^{1+\mu,\delta}[r]\,,\notag
    \end{gather}
    They are $\tR$-localized, real-valued symbols with zero average, as is $\sym{g}{2}{1+\mu}$. 
Then, using the expansion \eqref{dtGmu2} and  \Cref{lem:stima_comm}, we get that 
\begin{align}
[ \bG, \pa_t \bG]= \bB_{\geq 2}(\tR;\zak),
\label{comm:dtGGmu2}
\end{align}
satisfies the quadratic bound \eqref{est:Bexp}.
Then, by the Lie's expansion \eqref{Lie:dt} and \eqref{dtGmu2}, \eqref{comm:dtGGmu2} and noting also that $\vOpbw{\ii \sym{a}{\geq 4}{1+\mu}}$ satisfies \eqref{est:Bexp} thanks to \eqref{stima:exp_opbw}, we eventually get 
\begin{align}
    \pa_t \bF_2(\zak)\bF_2(\zak)^{-1}= \vOpbw{\ii  \left[\di_\zak \sym{g}{2}{1+\mu}[-\ii \vomega(D)\zak]+ \sym{b}{3}{1+\mu,\tR}\right]}+ \bB_{\geq 2}(\tR;\zak)\,.
    \label{conj:dt_(1+mu)}
\end{align}
 
 \smallskip
\noindent
{\bfseries Expansion of $\bF_2(\zak) \bR_1(\zak)\bF_2(\zak)^{-1} $.}
 Using \eqref{Lie:1} and \eqref{stima:phiexp}, we get 
\begin{align}
\bF_2(\zak) [\bR_1(\zak)] \bF_2(\zak)^{-1} &= \bR_1(\zak)+ \int_0^1\Phi_{g_2}^\theta (\zak)\left[ \vOpbw{\sym{g}{2}{1+\mu}},\bR_1(\zak)\right]\Phi_{g_2}^{-\theta}(\zak)\, \di \theta\notag 
\\
&=\bR_1(\zak)+ \bB_{\geq 2}(\tR;\zak),
\label{conj:R_mu2}
\end{align}
where $ \bB_{\geq 2}(\tR;\zak)$ satisfies \eqref{est:Bexp}.

\smallskip
\noindent{\bfseries Expansion of $\bF_2(\zak) \bB_{\geq 2}(\tR;\zak)\bF_2(\zak)^{-1} $.} Arguing as in  
    \eqref{bounf:B2_mu}
one prove that $\bF_2(\zak) \bB_{\geq 2}(\tR;\zak)\bF_2(\zak)^{-1}$ can be included in $ \bB_{\geq 2}(\tR;\zak)$.

Finally, gathering \eqref{proof:V2:1}, \eqref{conj:omeghino_mu2}, \eqref{conj:b12_(1+mu)}, \eqref{conj:sigma_(1+mu)}, \eqref{conj:dt_(1+mu)} and  \eqref{conj:R_mu2} , we get 
\begin{align*}
    \pa_t V_2=& \vOpbw{-\ii \left[ \Lambda(\xi)+ \sym{\cZ}{
    1}{\sla{1}{2}}+ \sym{b}{3}{1+\mu,\tR}+ [\sym{b}{\geq 2}{\sla{3}{2}}+\sym{b}{1}{\sla{1}{2}}]\left(1-\theta(\tR^{-1}\langle \xi\rangle)\right)\right]}V_2+ \bR_1(\zak)V_2
    + \bB_{\geq 2}(\tR;\zak)V_2\\
    &+ \vOpbw{-\ii \left[ \sym{b}{2}{\sla{3}{2},\tR}+ 
    \di_\zak \sym{g}{2}{1+\mu}[-\ii \vomega(D)\zak]
      + \ii \left( \sym{g}{2}{1+\mu}\sha{4} \Lambda(\xi)- \Lambda(\xi) \sha{4}\sym{g}{2}{1+\mu}\right)\right]} \\
    \stackrel{\eqref{eq:homo_finale2}}{=}& \vOpbw{-\ii \left[ \Lambda(\xi)+\langle \sym{b}{2}{\sla{3}{2},\tR}\rangle(\xi)+ \sym{\cZ}{
    1}{\sla{1}{2}}+\sym{\cZ}{
    2}{\sla{3}{2}}+ \sym{b}{3}{1+\mu,\tR}+ [\sym{b}{\geq 2}{\sla{3}{2}}+\sym{b}{1}{\sla{1}{2}}]\left(1-\theta(\tR^{-1}\langle \xi\rangle)\right)\right]}V_2+ \notag\\
         &+ \bR_1(\zak)+ \bB_{\geq 2}(\tR;\zak)V_2,
\end{align*}
proving \eqref{QR-quadratic_normalform} recalling 
\eqref{bb32mag} and setting 
$\sym{\tm}{2}{\sla{3}{2}}(\xi):=
\langle \sym{b}{2}{\sla{3}{2},\tR}\rangle(\xi)\in \wt \Gamma_2^{\sla{3}{2},\delta}$.
	\end{proof}

	\subsection{Normal form of the cubic super-linear symbol}\label{sec:cubicsym}
In order to obtain the final equation \eqref{QR-normalform}, we construct a transformation that puts the low-frequency quadratic quasilinear symbol $\sym{b}{3}{\sla{3}{2},\tR}$, from the intermediate equation \eqref{QR-quadratic_normalform}, into quasi-resonant normal form.
This transformation is generated by a cubic symbol of positive order, obtained by solving a quadratic homological equation given in Lemma \ref{homo:solve0_finale2}.

The main result of this subsection is
the following cubic super-linear normal form result.
    
	\begin{proposition}
\label{prop:NF_cubicSIMBOLI}
    Fix $\delta, \nu, \tau, \mu$ as in \eqref{410} and let $\vr > 2$.  
There exists $s_0 > 0$ such that, for any $\tR > 1$, there exists a bounded and invertible transformation $\bF_3(\zak)$ with the following property: for any $s \geq s_0$ there exists $r = r_s > 0$ such that, for any $\zak \in B_{s_0,\R}(r)$, the maps $\bF_3(\zak)$ and $\bF_3^{-1}(\zak)$ belong to 
$\cL(H^s_\R; H^s_\R)$ and satisfy the estimate \eqref{stima:admR}.  

Moreover, if $\zak(t) \in B_{s_0,\R}( r)$ solves \eqref{eq:zak}, then the variable $V_3 := \bF_3(\zak)V_2$,
with $V_2$ solving \eqref{QR-quadratic_normalform}, solves the system
  \begin{align}
         \pa_t V_3 &= \vOpbw{-\ii \left[\Lambda(\xi)+\sym{\tm}{2}{\sla{3}{2}}(\xi)+\sym{\cZ}{ 2}{\sla{3}{2}}+\sym{\cZ}{ 3}{1+\mu}+\sym{\cZ}{1}{\sla{1}{2}}\right]} V_3 +\bR_1(\zak)V_3\notag
         \\
         &+\vOpbw{- \ii  \left[\sym{\Sigma}{\geq 2}{\sla{3}{2}}+ \sym{b}{\geq 2}{1} + \sym{b}{1}{\sla{1}{2}}\right]\left(1-\theta\left(\tR^{-1}\langle \xi\rangle\right)\right)}V_3+ \bB_{\geq 2}(\tR;\zak)V_3,
         \label{QR-cubic_normalform}
    \end{align}
    where 
    \begin{itemize}
        \item  $\sym{\cZ}{3}{1+\mu}$ is a  $\tR$-localized, $3$-homogeneous real symbol in normal form (see \Cref{def:normalform}), belonging  to  $\wt \Gamma_3^{1+\mu, \delta}$;
        \item $\bR_1(\zak)$ is a matrix of real-to-real smoothing operators in $ \wt \cR_1^{-\vr}$;
        \item $\bB_{\geq 2}(\tR;\zak)$ is a family of  matrices of real-to-real bounded operators, satisfying \eqref{est:Bexp}.
    \end{itemize}
	\end{proposition}
    \begin{proof}
We define the flow
		\begin{equation*}
		\begin{cases}
			\pa_\tau \Phi_{g_3}^\tau(\zak) =  \bG(\zak)\Phi^\tau_{g_3}(\zak)
			\\
			\Phi^0(\zak)=\id
		\end{cases}\,, 
		\qquad \bG(\zak):= \vOpbw{\ii \sym{g}{3}{\sla{1}{2}+2\mu}(\zak;x,\xi)}\,,
		\end{equation*}
        where $\sym{g}{3}{\sla{1}{2}+2\mu}\in \wt \Gamma_3^{\sla{1}{2}+2\mu,\delta}$ is the $\tR$-localized symbol obtained by applying \Cref{homo:solve0_finale2} to the $\tR$-localized, real, zero average symbol $\sym{a}{3}{1+\mu}:=\sym{b}{3}{1+\mu,\tR} \in \wt \Gamma_3^{1+\mu, \delta}$ appearing in \eqref{QR-quadratic_normalform}. 
		We define $ \bF_3(\zak):=  \Phi_{g_3}^\tau(\zak)_{|\tau=1}$. Since $V_2$ solves \eqref{QR-quadratic_normalform}, the variable $V_3=\bF_3(\zak)V_2$  solves 
		\begin{align} 
			\partial_t V_3= \bF_3(\zak) \bA(\zak)\bF_3(\zak)^{-1} V_3+\pa_t \bF_3(\zak)\bF_3(\zak)^{-1}V_3+\bF_3(\zak)\big[\bR_1(\zak)+\bB_{\geq 2}(\zak)\big]\bF_3(\zak)^{-1}V_3
            \label{proof:V3:1}. 
		\end{align}
 where 
 \begin{equation*}
        \begin{aligned}
        \bA(\zak)&=-\ii \vomega(D)+\vOpbw{ -\ii\left[\sym{\cZ}{1}{\sla{1}{2}}+ \sym{b}{1}{\sla{1}{2}}(1-\theta\left(\tR^{-1} \langle \xi \rangle\right)\right] }
        \\
        &+ \vOpbw{-\ii \left[\sym{\tm}{2}{\sla{3}{2}}(\xi)+\sym{\cZ}{2}{\sla{3}{2}}+ \sym{b}{3}{1+\mu,\tR}+\sym{b}{\geq 2}{\sla{3}{2}}(1-\theta\left(\tR^{-1}\langle \xi \rangle\right))\right]}
        \\
            &= -\ii \vomega(D) +\bA_1(\zak)+\bA_2(\zak)\,,
        \end{aligned}
\end{equation*}
with $\sym{b}{\geq 2}{\sla{3}{2}}$ defined in \eqref{bb32mag}.
We analyze equation \eqref{proof:V3:1} expanding each term with its Lie expansion.
  
  \smallskip
\noindent
{\bfseries Expansion of $\bF_3(\zak) [-\ii \vomega(D)] \bF_3(\zak)^{-1} $.}
        In order to expand $\bF_3(\zak) [-\ii \vomega(D)] \bF_3(\zak)^{-1} $, 
        we use  the  second order Lie's expansion \eqref{Lie:2}, getting  
        \begin{equation}\label{esp:V3omega}
        \begin{aligned}
          \bF_3(\zak) [-\ii \vomega(D)] \bF_3(\zak)^{-1}&=  
          -\ii \vomega(D) + \left[ \bG(\zak),  -\ii \vomega(D) \right] 
          \\&+\int_{0}^1 (1-\theta) \Phi_{g_3}^\theta(\zak){\mathrm{Ad}}_{\bG}^2
          [ -\ii \vomega(D)]  \Phi_{g_3}^{-\theta}(\zak)\, \di \theta\,.
        \end{aligned}
        \end{equation}
Then we expand the commutator 
$\left[ \bG(\zak),  -\ii \vomega(D) \right] $ using \Cref{compoparapara}, obtaining 
\begin{align}
 \left[ \bG(\zak),  -\ii \vomega(D) \right]= 
 \vOpbw{-\ii \left[\ii \sym{g}{3}{\sla{1}{2}+2\mu}\sha{\vr} \Lambda(\xi)
 - \Lambda(\xi) \sha{\vr} \ii \sym{g}{3}{\sla{1}{2}+2\mu} \right]}   
 + \bR_3(\zak), \qquad \bR_3 \in \wt \cR^{-\vr}_3\,.
 \label{primo:V3}
\end{align}
Then applying \Cref{lem:stima_comm} to the $\tR$-localized symbol $ \sym{g}{3}{1+\mu}\in \wt \Gamma_3^{\sla{1}{2}+2\mu,\delta}$ and to the $3$-homogeneous commutator symbol  $-\ii \left[\ii \sym{g}{3}{\sla{1}{2}+2\mu}\sha{\vr} \Lambda(\xi)- \Lambda(\xi) \sha{\vr} \ii \sym{g}{3}{\sla{1}{2}+2\mu} \right] \in \wt \Gamma_3^{1+3\mu}$ and parameters $p_1=p_2=3$, $m_1= \frac12+2\mu$, $m_2= 1+3\mu $ and $\sigma:=\frac12+6\mu>0$, obtaining 
\begin{align}
   {\mathrm{Ad}}_{\bG}^2[ -\ii \vomega(D)]  &=
   \left[\bG(\zak), \vOpbw{-\ii \left[\ii \sym{g}{3}{\sla{1}{2}+2\mu}\sha{\vr} \Lambda(\xi)
   - \Lambda(\xi) \sha{\vr} \ii \sym{g}{3}{\sla{1}{2}+2\mu} \right]}\right]
   + [ \bG(\zak), \bR_3(\zak)]\nonumber
   \\&=\bB_{\geq 2}(\tR;\zak)\,,
    \label{doppio:V2}
\end{align}
satisfying \eqref{est:Bexp}. Gathering \eqref{esp:V3omega}, \eqref{primo:V3} and \eqref{doppio:V2}  and using the bound \eqref{stima:phiexp} for $\Phi_{\fun{g}{3}}^\theta$, we get 
\begin{align}
     \bF_3(\zak) [-\ii \vomega(D)] \bF_3(\zak)^{-1}
     = \vOpbw{-\ii \Lambda(\xi)-\ii \left[\ii \sym{g}{3}{\sla{1}{2}+2\mu}
     \sha{\vr} \Lambda(\xi)- \Lambda(\xi) \sha{\vr} \ii \sym{g}{3}{\sla{1}{2}+2\mu} \right]}
     + \bB_{\geq 2}(\tR;\zak)\,,
     \label{conj:omeghino_mu3}
\end{align}
        where $ \bB_{\geq 2}(\tR;\zak)$ satisfies \eqref{est:Bexp}.
        
        \smallskip
\noindent
{\bfseries Expansion of $\bF_3(\zak) \bA_1(\zak)\bF_3(\zak)^{-1} $.}
We use \eqref{Lie:1} getting 
\begin{equation*}
\begin{aligned}
\bF_3(\zak) [\bA_1(\zak)] &\bF_3(\zak)^{-1} = 
\vOpbw{-\ii\left[\sym{\cZ}{1}{\sla{1}{2}}
+ \sym{b}{1}{\sla{1}{2}}(1-\theta\left(\tR^{-1} \langle \xi \rangle\right)\right] }
\\&
+ \int_0^1\Phi_{g_3}^\theta (\zak)\left[ \vOpbw{\sym{g}{3}{\sla{1}{2}+2\mu}},
\vOpbw{-\ii\left[\sym{\cZ}{1}{\sla{1}{2}}
+ \sym{b}{1}{\sla{1}{2}}(1-\theta\left(\tR^{-1} \langle \xi \rangle\right)\right] }\right]
\Phi_{g_3}^{-\theta}(\zak)\, \di \theta\,.
\end{aligned}
\end{equation*}
 Then we apply \Cref{lem:stima_comm} with $ m_1=\frac12+2\mu$, $m_2= \frac12$ and $\sigma=  3\mu>0$ and we get that the commutator
 \[
 \left[ \vOpbw{\sym{g}{3}{\sla{1}{2}+2\mu}},
 \vOpbw{-\ii \left[\sym{\cZ}{1}{\sla{1}{2}}
 + \sym{b}{1}{\sla{1}{2}}(1-\theta\left(\tR^{-1} \langle \xi \rangle\right)\right] }\right]
 \]
 satisfies \eqref{eq:stima_comm}. 
 Then, using also the bound \eqref{stima:phiexp} 
 for $\Phi_{\fun{g}{3}}^\theta$, we obtain
 \begin{align}
     \bF_3(\zak) [\bA_1(\zak)] \bF_3(\zak)^{-1}= 
     \Opbw{-\ii\left[\sym{\cZ}{1}{\sla{1}{2}}+ \sym{b}{1}{\sla{1}{2}}(1-\theta\left(\tR^{-1} 
     \langle \xi \rangle\right)\right]}
     + \bB_{\geq 2}(\tR;\zak)\,,
     \label{conj:b12_(12+2mu)}
 \end{align}
 where $ \bB_{\geq 2}(\tR;\zak)$ is a bounded quadratic map satisfying \eqref{est:Bexp}. 
 
 \smallskip
\noindent
{\bfseries Expansion of $\bF_3(\zak) \bA_2(\zak)\bF_3(\zak)^{-1} $.}
Using \eqref{Lie:1}, get 
\begin{align}
&\bF_3(\zak) [\bA_2(\zak)] \bF_3(\zak)^{-1} 
= \vOpbw{-\ii \left[\sym{\tm}{2}{\sla{3}{2}}(\xi)
+\sym{\cZ}{2}{\sla{3}{2}}+\sym{b}{3}{1+\mu,\tR}
+ \sym{b}{\geq 2}{\sla{3}{2}}(1-\theta\left(\tR^{-1}\langle \xi \rangle\right))\right]}\notag
\\&
+ \!\!\int_0^1\!\!\Phi_{g_3}^\theta (\zak)\left[ \vOpbw{\sym{g}{3}{\sla{3}{2}+2\mu}},
\vOpbw{-\ii \left[\sym{\tm}{2}{\sla{3}{2}}(\xi)+\sym{\cZ}{2}{\sla{3}{2}}
+ \sym{b}{3}{1+\mu,\tR}+ \sym{b}{\geq 2}{\sla{3}{2}}(1-\theta\left(\tR^{-1}
\langle \xi \rangle\right))\right]}\right]\Phi_{g_3}^{-\theta}(\zak)\, \di \theta\,.
\notag
\end{align}
 Then we apply \Cref{lem:stima_comm} with $p_1=3$, $p_2=2$, 
 $ m_1=\frac12+2\mu$, $m_2= \frac32$ and $\sigma= 1+3\mu>0$. 
 Since $p_1+p_2=5 > 2+ \sigma$, 
 we get that the commutator 
\[
 \left[ \vOpbw{\sym{g}{3}{\sla{1}{2}+2\mu}},
 \vOpbw{-\ii \left[\sym{\tm}{2}{\sla{3}{2}}(\xi)+\sym{\cZ}{2}{\sla{3}{2}}
 +\sym{b}{3}{1+\mu,\tR}
 + \sym{b}{\geq 2}{\sla{3}{2}}(1-\theta\left(\tR^{-1}\langle \xi \rangle\right))\right]}\right]\,,
 \]
 satisfies \eqref{eq:stima_comm}. Then, using the bound \eqref{stima:phiexp} for $\Phi_{\fun{g}{3}}^\theta$, we obtain
 \begin{equation}\label{conj:sigma_(12+2mu)}
 \begin{aligned}
 \bF_3(\zak) [\bA_2(\zak)] &\bF_3(\zak)^{-1}
 \\&= \vOpbw{-\ii \left[\sym{\tm}{2}{\sla{3}{2}}(\xi)+\sym{\cZ}{2}{\sla{3}{2}}
 + \sym{b}{3}{1+\mu,\tR}
 +\sym{b}{\geq 2}{\sla{3}{2}}(1-\theta\left(\tR^{-1}\langle \xi \rangle\right))\right]}
 + \bB_{\geq 2}(\tR;\zak)\,,
 \end{aligned}
 \end{equation}
 where $ \bB_{\geq 2}(\tR;\zak)$ is a bounded quadratic map satisfying \eqref{est:Bexp}. 

\smallskip
\noindent
{\bfseries Expansion of $\pa_t \bF_3(\zak) \bF_3(\zak)^{-1} $.} We apply the first order Lie's expansion \eqref{Lie:dt}. To this end we first compute the time derivative $\pa_t \bG$ of the generator, using that $\bG$ is $3$-homogeneous with respect to the variable $\zak$ and the equation \eqref{eq:zak}-\eqref{X:zak_esp} for $\pa_t\zak$, we get 
\begin{align}
    \pa_t \bG(\zak)= \vOpbw{\ii \di_\zak \sym{g}{3}{\sla{1}{2}+2\mu}[\hamvec{\cH}(\zak)]}, \qquad \di_\zak \sym{g}{3}{\sla{1}{2}+2\mu}[\hamvec{\cH}(\zak)] \in \Gamma_{\geq 3}^{\sla{1}{2}+2\mu}.
    \label{dtGmu3}
\end{align}
	
Then, using the expansion \eqref{dtGmu3} and  \Cref{lem:stima_comm}, we get that 
\begin{align}
[ \bG, \pa_t \bG]= \bB_{\geq 2}(\tR;\zak),
\label{comm:dtGGmu3}
\end{align}
satisfies the quadratic bound \eqref{est:Bexp}.
Then, by the Lie's expansion \eqref{Lie:dt} and \eqref{dtGmu2}, \eqref{comm:dtGGmu3} and noting also that $\vOpbw{\ii\di_\zak \sym{g}{3}{\sla{1}{2}+2\mu}[\hamvec{\cH}(\zak)]}$ satisfies \eqref{est:Bexp} thanks to \eqref{stima:exp_opbw}, we eventually get 
\begin{align}
    \pa_t \bF_3(\zak)\bF_3(\zak)^{-1}=  \bB_{\geq 2}(\tR;\zak).
    \label{conj:dt_(12+2mu)}
\end{align}
 
 \smallskip
\noindent
{\bfseries Expansion of $\bF_3(\zak) \bR_1(\zak)\bF_3(\zak)^{-1} $.}
 Using \eqref{Lie:1} and \eqref{stima:phiexp}, we get 
\begin{align}
\bF_3(\zak) [\bR_1(\zak)] \bF_3(\zak)^{-1} =& \bR_1(\zak)+ \int_0^1\Phi_{g_3}^\theta (\zak)\left[ \vOpbw{\sym{g}{3}{\sla{1}{2}+2\mu}},\bR_1(\zak)\right]\Phi_{g_3}^{-\theta}(\zak)\, \di \theta\notag \\
=&\bR_1(\zak)+ \bB_{\geq 2}(\tR;\zak),
\label{conj:R_mu3}
\end{align}
where $ \bB_{\geq 2}(\tR;\zak)$ satisfies \eqref{est:Bexp}.

\smallskip
\noindent
{\bfseries Bound of $\bF_3(\zak) \bB_{\geq 2}(\tR;\zak)\bF_3(\zak)^{-1} $.} 
Arguing as in  
    \eqref{bounf:B2_mu}
one prove that $\bF_3(\zak) \bB_{\geq 2}(\tR;\zak)\bF_3(\zak)^{-1}$ 
can be included in $ \bB_{\geq 2}(\tR;\zak)$.

\noindent
Finally, gathering \eqref{proof:V3:1}, \eqref{conj:omeghino_mu3}, 
\eqref{conj:b12_(12+2mu)}, \eqref{conj:sigma_(12+2mu)}, 
\eqref{conj:dt_(12+2mu)} and  \eqref{conj:R_mu3}, we get 
\begin{equation*}
\begin{aligned}
    \pa_t V_3&= \vOpbw{-\ii \left[ \Lambda(\xi)
    + \sym{\cZ}{
    1}{\sla{1}{2}}+\sym{\tm}{2}{\sla{3}{2}}(\xi)
    +\sym{\cZ}{2}{\sla{3}{2}}+ [\sym{b}{\geq 2}{\sla{3}{2}}
    +\sym{b}{1}{\sla{1}{2}}]\left(1-\theta(\tR^{-1}\langle \xi\rangle)\right)\right]}V_3
    \\
    &\quad + \vOpbw{-\ii \left[  \sym{b}{3}{1+\mu,\tR}
    + \ii \left( \sym{g}{3}{\sla{1}{2}+2\mu}\sha{\vr} \Lambda(\xi)
    - \Lambda(\xi) \sha{\vr}\sym{g}{3}{\sla{1}{2}+2\mu}\right)\right]}  
    + \bR_1(\zak)V_3+ \bB_{\geq 2}(\tR;\zak)V_3
    \\
   & \stackrel{\eqref{eq:homo_finale3}}{=}
    \vOpbw{-\ii \left[ \Lambda(\xi)+ \sym{\tm}{2}{\sla{3}{2}}(\xi)
    + \sym{\cZ}{
    1}{\sla{1}{2}}+\sym{\cZ}{
    2}{\sla{3}{2}}+ \sym{\cZ}{3}{1+\mu}+ [\sym{b}{\geq 2}{\sla{3}{2}}
    +\sym{b}{1}{\sla{1}{2}}]\left(1-\theta(\tR^{-1}\langle \xi\rangle)\right)\right]}V_3
    \\&
    \quad+ \bR_1(\zak)V_3+ \bB_{\geq 2}(\tR;\zak)V_3\,,
\end{aligned}
\end{equation*}
which implies, recalling \eqref{bb32mag},
formula \eqref{QR-cubic_normalform}.
    \end{proof}
	\subsection{Birkhoff normal form of the smoothing remainder}\label{sezione birkhoff smoothing}
In order to obtain the final equation \eqref{QR-normalform}, we construct a transformation that puts the quadratic smoothing remainder  $ \bR_1(\zak)V_3$
 appearing in \eqref{QR-cubic_normalform} into normal form. 
 The absence of three-wave interactions, proved in \Cref{lem:small_divisors}, 
 implies that no resonant quadratic terms occur, 
 so that the normal form is identically zero and the smoothing 
 term is completely eliminated. 
The main result of this subsection is the following.

	\begin{proposition}
\label{prop:NF_cubic}
     Let $\vr\geq 4$ and $\delta$, $\nu$, $\tau$ fixed as in \eqref{410}. 
     Let $\kap \in \R_{+}\setminus \mathscr{N}$ where $\mathscr{N}$ is the zero measure set
     given by Lemma \ref{lem:small_divisors}.
There exists $s_0>0$ such that for any $\tR>1$ there is a bounded and invertible transformation $\bF_4(\zak)$ (independent of $\tR$) such that for any $s\geq s_0$ there exists $r=r_s>0$ such that for any $\zak \in B_{s_0,\R}(r)\cap H^{s}_{\R}$ the maps $\bF_4(\zak)$ and $\bF_4^{-1}(\zak)$ belong to $\cL\left( H^s_\R;H^s_\R\right)$ and satisfy the estimates 
\begin{align}
\|\bF_4(\zak)V\|_{s}+ \|\bF_4^{-1}(\zak)V\|_{s}
\lesssim_s
\left[\| V\|_s+ \| \zak\|_{s}\| V\|_{s_0}\right], 
\label{stima:F4}
\end{align}
for any $V \in H^s_\R$. Moreover, if $\zak(t) \in B_{s_0, \R}(r)\cap H^s_\R$ solves \eqref{eq:zak}, then the variable
\[
Z:= \bF_4(\zak)V_3,
\]
with $V_3$ solving \eqref{QR-cubic_normalform}, solves the system
  \begin{align}
         \pa_t Z &= \vOpbw{-\ii \left[\Lambda(\xi)+\sym{\cZ}{1}{\sla{1}{2}}+\sym{\tm}{2}{\sla{3}{2}}(\xi)+\sym{\cZ}{ 2}{\sla{3}{2}}+\sym{\cZ}{ 3}{1+\mu}\right]} Z\notag\\
         &+\vOpbw{- \ii  \left[\sym{\Sigma}{\geq 2}{\sla{3}{2}}+ \sym{b}{\geq 2}{1}
         + \sym{b}{1}{\sla{1}{2}}\right]\left(1-\theta\left(\tR^{-1}\langle \xi\rangle\right)\right)}Z
         + \bB_{\geq 2}(\tR;\zak)Z\,,
         \label{QR-birk_normalform}
    \end{align}
    where $\bB_{\geq 2}(\tR;\zak)$ is a family of  matrices of real-to-real bounded operators, satisfying \eqref{est:Bexp}.
	\end{proposition}
	
	In order to prove the proposition above, we shall apply a transformation which is 
 generated by a quadratic smoothing operator 
 $\bG_1(\zak)$, obtained by solving a quadratic homological equation; 
 see the following lemma.
	\begin{lemma}
    \label{homo:solve0_finale_birk}
        Fix $\vr>4$ and $\kap$ outside the zero measure set $ \mathscr{N}$ given in \Cref{lem:small_divisors}. Let $\bR_1(\zak)$ a matrix of $1$-homogeneous, real-to-real, smoothing remainders in $ \wt \cR_1^{-\vr}$ then  there exists a matrix of $1$-homogeneous, real-to-real, smoothing remainders $\bG_1(\zak)\in \wt \cR^{-\vr+2}_1$ such that 
        \begin{align}
\bR_1(\zak)+\bG_1(-\ii \vomega(D) \zak)+ \left[ \bG_1(\zak),\, -\ii \vomega(D)\right] =0.
            \label{eq:homo_birk}
        \end{align}
    \end{lemma}	
    \begin{proof}
        We first Fourier expand the smoothing remainder $ \bR_1(\zak)$ as 
        in \eqref{esp:mat_smoo}, using \Cref{lem:smoohomo} 
        \begin{equation*}
            (\bR_1(\zak)V)^+= 
            \sum_{j,k \in \Z^2\setminus\{0\}} R_{j,k}^{\sigma_1,\sigma_2} 
            \zetina_j^{\sigma_1}v_k^{\sigma_2}e^{\ii (\sigma_1 j+\sigma_2 k)\cdot x}\,, 
            \qquad 
            (\bR_1(\zak)V)^-= \ov{(\bR_1(\zak)V)^+}\,,
        \end{equation*}
        where the coefficients $R_{j,k}^{\sigma_1,\sigma_2}$ satisfy  \eqref{smoocara}. 
        Then we define 
        \begin{equation*}
        \begin{aligned}
        (\bG_1(\zak) V)^+&:= \sum_{j,k \in \Z^2\setminus\{0\}} G_{j,k}^{\sigma_1,\sigma_2}\zetina_j^{\sigma_1}v_k^{\sigma_2}e^{\ii (\sigma_1 j+\sigma_2 k)\cdot x}\,, 
 \\
  G_{j,k}^{\sigma_1,\sigma_2}&:= 
        -\frac{R_{j,k}^{\sigma_1,\sigma_2}}{-\ii \left[\sigma_1 \Lambda(j)
        + \sigma_2\Lambda(k)- \Lambda(\sigma_1 j+ \sigma_2 k)\right]}\, .
        \end{aligned}
        \end{equation*}
        Explicit computation shows that $\bG_1(\zak)$ solves the equation \eqref{eq:homo_birk} as well as the real-to-real property in \eqref{esp:mat_smoo}. 
        By the bound \eqref{smoocara} for  $ \bR_1(\zak)$  and \eqref{eq:lower_bound} we get 
        \begin{align}
            \left|G_{j,k}^{\sigma_1,\sigma_2} \right|\lesssim \frac{\min\{|j|,|k|\}^{\mu}\max\{|j|,|k|\}^{-\vr} }{|\sigma_1 \Lambda(j)+ \sigma_2\Lambda(k)- \Lambda(\sigma_1 j+ \sigma_2 k)|}\lesssim \min\{|j|,|k|\}^{\mu+4}\max\{|j|,|k|\}^{-\vr+2},
            \notag
        \end{align}
        which, using again \Cref{lem:smoohomo}, proves that $\bG_1(\zak)$ is a real-to-real matrix of smoothing operators belonging to the class $ \wt \cR^{-\vr+2}_1$.
    \end{proof}	
	
    \begin{proof}[{\bf Proof of Proposition \ref{prop:NF_cubic}}]
We define the flow
		\be \label{flow_smoo}
		\begin{cases}
			\pa_\tau \Phi^\tau(\zak) =  \bG_1(\zak)\Phi^\tau(\zak)\\
			\Phi^0(\zak)=\id
		\end{cases}, \quad \text{where} \quad \bG_1(\zak)\in \cR^{-\vr+ 2}_1\,,
		\ee
        is the $1$-homogeneous matrix of real-to-real smoothing 
        remainders obtained by applying \Cref{homo:solve0_finale_birk} to the 
        $1$-homogeneous, matrix of real-to-real smoothing remainders 
        $\bR_1(\zak) \in \cR^{-\vr}_1$ of \eqref{QR-cubic_normalform}. 
        We define $ \bF_4(\zak):=  \Phi^\tau(\zak)_{|\tau=1}$. 
        Thanks to \Cref{lem:flow_smoo}, the map $\bF_4(\zak)$ 
        satisfies the required bound \eqref{stima:F4}. Since $V_3$ 
        solves \eqref{QR-cubic_normalform}, the variable $Z=\bF_4(\zak)V_3$  solves 
\begin{equation}\label{proof:V4:1} 
\partial_t Z= \bF_4(\zak) \bA(\zak)\bF_4(\zak)^{-1} Z
+\pa_t \bF_4(\zak)\bF_4(\zak)^{-1}Z
+\bF_4(\zak)\big[\bR_1(\zak)+\bB_{\geq 2}(\zak)\big]\bF_4(\zak)^{-1}Z\,,
\end{equation}
 where (recall \eqref{bb32mag})
 \begin{equation} \label{def:AV4}
    \begin{aligned}
        \bA(\zak)&= -\ii \vomega(D) +\bA_1(\zak)\,,
        \\ 
        \bA_1&:=\vOpbw{-\ii \left[
        \sym{\cZ}{ 2}{\sla{3}{2}}+\sym{\tm}{2}{\sla{3}{2}}(\xi)+\sym{\cZ}{ 3}{1+\mu}
        +\sym{\cZ}{1}{\sla{1}{2}}+\left[\sym{b}{\geq 2}{\sla{3}{2}} 
        + \sym{b}{1}{\sla{1}{2}}\right]
        \left(1-\theta\left(\tR^{-1}\langle \xi\rangle\right)\right)\right]}
        \\&
        = \vOpbw{-\ii \sym{b}{\geq 1}{\sla{3}{2},\delta}},
        \end{aligned}
        \end{equation}
where we have collected in 
$\sym{b}{\geq 1}{\sla{3}{2},\delta} 
\in \Gamma_{\geq 1}^\sla{3}{2}[r]$ 
all the symbols in the definition of $\bA_1$.

\noindent
We analyze equation \eqref{proof:V4:1} expanding 
each term with its Lie expansion.
        
\smallskip
\noindent
{\bfseries Expansion of $\bF_4(\zak) [-\ii \vomega(D)] \bF_4(\zak)^{-1} $.}
In order to expand $\bF_4(\zak) [-\ii \vomega(D)] \bF_4(\zak)^{-1} $, 
we use  the  second order Lie's expansion \eqref{Lie:2}, getting  
\begin{equation}\label{esp:V4omega}
\begin{aligned}
\bF_4(\zak) [-\ii \vomega(D)] \bF_4(\zak)^{-1}&=  
-\ii \vomega(D) + \left[ \bG_1(\zak),  -\ii \vomega(D) \right] 
\\&
+\int_{0}^1 (1-\theta) \Phi^\theta(\zak){\mathrm{Ad}}_{\bG_1}^2[ -\ii \vomega(D)]  
\Phi^{-\theta}(\zak)\, \di \theta\,.
\end{aligned}
\end{equation}
Since $\Lambda(\xi) \in \wt \Gamma_0^\sla{3}{2}$, one has 
\begin{equation}
\label{stima:omegaD}
    \| -\ii \vomega(D)\|_{\cL(H^s_\R;H^s_\R)}\lesssim 1\,, \qquad \forall s \in \R\,.  
\end{equation}
Thus, applying also the estimate \eqref{bound:smoo} for $\bG_1 \in \wt \cR^{-\vr +2}_1$,  
we get the corresponding estimate for the commutator 
$\left[ \bG_1(\zak),  -\ii \vomega(D) \right]$ 
proving that it belongs to $ \wt \cR^{-\vr+\sla{7}{2}}_1$. 
In the same way, we easily get the bound   
\[
\|{\mathrm{Ad}}_{\bG_1}^2[ -\ii \vomega(D)]V\|_{s+\vr-\sla{7}{2}}
\lesssim_s \| \zak\|_{s_0}^2 \| V\|_s + \| \zak\|_{s_0}\| \zak\|_s \| V\|_{s_0}\,.
\]
As $\vr\geq 4>\tfrac72$, ${\mathrm{Ad}}_{\bG_1}^2[ -\ii \vomega(D)]$ 
satisfies the quadratic bound \eqref{est:Bexp}. 
Using also the bound \eqref{stima:flussosmoothing} for $\Phi^\theta(\zak)$, we get 
\begin{align}
     \bF_4(\zak) [-\ii \vomega(D)] \bF_4(\zak)^{-1}=-\ii \vomega(D) 
     + \left[ \bG_1(\zak),  -\ii \vomega(D) \right] +\bB_{\geq 2}(\zak)\,,
     \label{conj:omeghino_birk}
\end{align}
        where $ \bB_{\geq 2}(\zak)$ satisfies \eqref{est:Bexp}.

\smallskip        
\noindent
{\bfseries Expansion of $\bF_4(\zak) \bA_1(\zak)\bF_4(\zak)^{-1} $.}
We use \eqref{Lie:1} getting 
\begin{align}
\bF_4(\zak) [\bA_1(\zak)] \bF_4(\zak)^{-1} = &\vOpbw{-\ii \left[
        \sym{\cZ}{ 2}{\sla{3}{2}}+\sym{\tm}{2}{\sla{3}{2}}(\xi)+\sym{\cZ}{ 3}{1+\mu}+\sym{\cZ}{1}{\sla{1}{2}}+\left[\sym{b}{\geq 2}{\sla{3}{2}} + \sym{b}{1}{\sla{1}{2}}\right]\left(1-\theta\left(\tR^{-1}\langle \xi\rangle\right)\right)\right]}\notag\\
&+ \int_0^1\Phi^\theta (\zak)\left[ \bG_1,\vOpbw{-\ii \sym{b}{\geq 1}{\sla{3}{2}}}\right]\Phi^{-\theta}(\zak)\, \di \theta.\notag
\end{align}
 Then we apply \eqref{eq:stima_comm_smoothing} with $\vr \leadsto \vr-2 \geq \frac{3}{2}$ and we obtain that the commutator $\left[ \bG_1,\vOpbw{-\ii \sym{b}{\geq 1}{\sla{3}{2}}}\right]$ satisfies \eqref{est:Bexp}. Then, using also the bound \eqref{stima:flussosmoothing} for $\Phi^\theta$, 
 we obtain
 \begin{equation}\label{conj:b32_smoo}
 \begin{aligned}
   & \bF_4(\zak) [\bA_1(\zak)] \bF_4(\zak)^{-1}
    \\&\!\!\!\!
    =\vOpbw{-\ii \left[
        \sym{\cZ}{ 2}{\sla{3}{2}}+\sym{\tm}{2}{\sla{3}{2}}(\xi)+\sym{\cZ}{ 3}{1+\mu}
        +\sym{\cZ}{1}{\sla{1}{2}}+\left[\sym{b}{\geq 2}{\sla{3}{2}} 
        + \sym{b}{1}{\sla{1}{2}}\right]
        \left(1-\theta\left(\tR^{-1}\langle \xi\rangle\right)\right)\right]}
        + \bB_{\geq 2}(\zak)\,,
 \end{aligned}
 \end{equation}
 where $ \bB_{\geq 2}(\zak)$ is a bounded quadratic map satisfying \eqref{est:Bexp}.

\smallskip
\noindent
{\bfseries Expansion of $\pa_t \bF_4(\zak) \bF_4(\zak)^{-1} $.} We apply the first order Lie's expansion \eqref{Lie:dt}. To this end we first compute the time derivative $\pa_t \bG_1$ of the generator, using that $\bG_1$ is $1$-homogeneous with respect to the variable $\zak$ and the equation \eqref{eq:zak}-\eqref{X:zak_esp} for $\pa_t\zak$, we get 
\begin{align}
    \pa_t \bG_1(\zak)= \bG_1(\hamvec{\cH}(\zak))=\bG_1(-\ii \vomega(D)\zak)+\bG_1(\hamvec{\geq 2}(\zak)) , 
    \label{dtGmu4}
\end{align}
where since $\bG_1(\zak)\in \wt \cR^{-\vr+2}_1$,  $\bG_1(-\ii \vomega(D)\zak)$ belongs to $\wt  \cR^{-\vr+\sla{7}{2}}_1$ by \eqref{stima:omegaD}. Moreover, combining \eqref{bound:smoo} for  $\bG_1(\bigcdot)$ and \eqref{tameest:campo} for $\hamvec{\geq 2}(\zak)$, we have that $\bG_1(\hamvec{\geq 2}(\zak))$ satisfies the quadratic bound \eqref{est:Bexp}. Using in addition estimate \eqref{est:XH} for $\hamvec{\cH}(\zak)$, we have that  $\pa_t \bG_1$ is a smoothing operator satisfying

\begin{equation*}
    \| \pa_t \bG_1 V \|_{s+\vr -\sla{7}{2}}\lesssim_s \| \zak\|_{s_0}\| V\|_{s}+ \| \zak\|_{s}\| V\|_{s_0}\,.
\end{equation*}
Then, combining the above estimate with the estimate \eqref{bound:smoo} for $ \bG_1$,  we get that 
\begin{align}
[ \bG, \pa_t \bG]= \bB_{\geq 2}(\tR;\zak)\,,
\label{comm:dtGGmu4}
\end{align}
satisfies the quadratic bound \eqref{est:Bexp}.
Then, by the Lie's expansion 
\eqref{Lie:dt} and \eqref{dtGmu4}, \eqref{comm:dtGGmu4} 
and using also the bound \eqref{stima:flussosmoothing} for 
$\Phi^\theta$, we eventually get 
\begin{align}
    \pa_t \bF_4(\zak)\bF_4(\zak)^{-1}= \bG_1(-\ii \vomega(D)\zak)+ \bB_{\geq 2}(\tR;\zak)\,.
    \label{conj:dt_G}
\end{align}
 
 \smallskip
\noindent
{\bfseries Expansion of $\bF_4(\zak) \bR_1(\zak)\bF_4(\zak)^{-1} $.}
 Using \eqref{Lie:1} and \eqref{stima:phiexp}, 
 we get 
 \begin{equation}\label{conj:R_mu4}
\begin{aligned}
\bF_4(\zak) [\bR_1(\zak)] \bF_4(\zak)^{-1} &= \bR_1(\zak)
+ \int_0^1\Phi^\theta (\zak)\left[\bG_1(\zak),
\bR_1(\zak)\right]\Phi^{-\theta}(\zak)\, \di \theta
\\
&= \bR_1(\zak)+ \bB_{\geq 2}(\tR;\zak)\,,
\end{aligned}
\end{equation}
where $ \bB_{\geq 2}(\tR;\zak)$ satisfies \eqref{est:Bexp}.

\smallskip
\noindent
{\bfseries Bound of $\bF_3(\zak) \bB_{\geq 2}(\tR;\zak)\bF_3(\zak)^{-1} $.} Arguing as in  
    \eqref{bounf:B2_mu}
one prove that $\bF_3(\zak) \bB_{\geq 2}(\tR;\zak)\bF_3(\zak)^{-1}$ can be included in $ \bB_{\geq 2}(\tR;\zak)$.

\smallskip
\noindent
Finally, gathering \eqref{proof:V4:1}, \eqref{conj:omeghino_birk}, \eqref{conj:b32_smoo}, \eqref{conj:dt_G} and  \eqref{conj:R_mu4} , we get 

\begin{equation*}
\begin{aligned}
    \pa_t Z&= 
    \vOpbw{-\ii \left[ \Lambda(\xi)+ \sym{\cZ}{ 2}{\sla{3}{2}}+\sym{\cZ}{ 3}{1+\mu}
    +\sym{\cZ}{1}{\sla{1}{2}}+\left[\sym{b}{\geq 2}{\sla{3}{2}} 
    + \sym{b}{1}{\sla{1}{2}}\right]\left(1-\theta\left(\tR^{-1}\langle \xi\rangle\right)\right)\right]}Z
    \\&\qquad\quad
    + \bB_{\geq 2}(\tR;\zak)Z
    + \bR_1(\zak)Z+\bG_1(-\ii \vomega(D)\zak) + \left[ \bG_1(\zak),  -\ii \vomega(D) \right] 
 \\&
    \stackrel{\eqref{eq:homo_birk}}{=} 
    \vOpbw{-\ii \left[ \Lambda(\xi)+ \sym{\cZ}{ 2}{\sla{3}{2}}+\sym{\cZ}{ 3}{1+\mu}
    +\sym{\cZ}{1}{\sla{1}{2}}+\left[\sym{b}{\geq 2}{\sla{3}{2}} 
    + \sym{b}{1}{\sla{1}{2}}\right]\left(1-\theta\left(\tR^{-1}\langle \xi\rangle\right)\right)\right]}Z
    \\&\qquad \quad+ \bB_{\geq 2}(\tR;\zak)Z\,,
\end{aligned}
\end{equation*}
which implies \eqref{QR-birk_normalform}. This concludes the proof.
\end{proof}

\smallskip
\begin{proof}[{\bf Proof of Proposition \ref{prop:NF_finale}}.]
By applying Propositions \ref{prop:NF_linear}, \ref{prop:NF_quadratic}, \ref{prop:NF_cubicSIMBOLI},
\ref{prop:NF_cubic}
we define
\[
{\bf T}(\zak)[\cdot]:=\bF_4(\zak)\circ \bF_3(\zak)\circ \bF_2(\zak)\circ \bF_1(\zak)[\cdot]\,.
\]
In view of the estimates \eqref{stima:admR} on the maps ${\bf F}_{i}(\zak)$, $i=1,2,3$,
and \eqref{stima:F4} on $\bF_4(\zak)$, we deduce the bound
\eqref{stima:admR_nonspec} for the composition map ${\bf T}(\zak)$.
Now by setting 
  $Z:= \bT(\zak)U_2$, 
where $U_2$ solves the problem \eqref{NuovoParaprod2} (see Proposition \ref{prop:nfquadtra}),
formula \eqref{QR-birk_normalform} implies 
\eqref{QR-normalform}.
\end{proof}

\section{ Energy Estimates for the Quasi-Resonant Equation}\label{sez-energy-estimate}
In the previous  sections we showed, under suitable regularity and smallness assumption, 
that if $\zak$ solves \eqref{eq:zak} and $U$ solves the complex water waves system 
\eqref{eq:U}, then
the function (see Propositions \ref{prop:blockdecoupling},  \ref{prop:nfquadtra} and 
\ref{prop:NF_finale})
\begin{equation}\label{varZfinale}
Z:=\bT(\zak)\circ  \mathbf{\Psi}_2\circ \bF(\zak)U
\end{equation}
solves the problem (see \eqref{QR-normalform})
	\be 
	\partial_t Z= \vOpbw{-\ii \left[\Lambda(\xi)+  \sym{\ta}{\geq 2}{\sla{3}{2}}(\zak;x,\xi)+ \sym{\tb}{1}{\sla{1}{2}}(\zak;x,\xi)\right] }Z+ \bB_{\geq 2 }(\tR;\zak)Z\,,
	\label{eq:finale_nuova}
	\ee
	where
    \begin{align}
        &\sym{\ta}{\geq 2}{\sla{3}{2}}=  \sym{\tm}{2}{\sla{3}{2}}(\xi)
        +\sym{\cZ}{ 2}{\sla{3}{2}}+\sym{\cZ}{ 3}{1+\mu} +
      (\sym{\Sigma}{\geq 2}{\sla{3}{2}} 
         + \sym{b}{\geq 2}{1})  \left(1-\theta\left(\tR^{-1}\langle \xi\rangle\right)\right)  
        \in \Gamma^{\sla{3}{2},\delta}_{\geq 2}\label{def:ta32}
        \\
&\sym{\tb}{1}{\sla{1}{2}}=  \sym{\cZ}{ 1}{\sla{1}{2}}
+\sym{b}{1}{\sla{1}{2}}\left(1-\theta\left(\tR^{-1}\langle \xi\rangle\right)\right)  
\in \Gamma^{\sla{1}{2},\delta}_{1}\,.
        \label{def:tb12}
    \end{align} 
    Recall that, in view of \Cref{prop:NF_finale,prop:nfquadtra} and \eqref{stima:diffb3}
    (see also \eqref{def:sdBIS}), 
    the symbol  $\sym{\ta}{\geq 2}{\sla{3}{2}}$ is differentiable 
    with respect to $ \zak$ and there are $s_0,\mu>0$ 
    such that for any $\s\geq s_0$ one has the estimates 
    \begin{equation}\label{stima:derivataa32}
 \Big| 
 \di_\zak \sym{\ta}{\geq 2}{\sla{3}{2}}(\zak;x,\xi)[\wh \zak]
 \Big|_{\sla{3}{2}, \sigma-\mu }
        \lesssim_\s 
        \| \zak\|_{\s}\| \wh \zak\|_{\s}\,.
    \end{equation}
    Moreover we recall that  $\bB_{\geq 2}(\tR;\zak)$ satisfies \eqref{est:Bexp}.
    In this section we prove cubic energy estimate for the 
    equation \eqref{eq:finale_nuova}.

\begin{proposition}[Main energy estimate]
\label{prop:main_est}
Let $T>0$, $I:=[0,T]$, $K>1$ and 
$\kap \in \R_{+}\setminus \mathscr{N}$ (see Lemma \ref{lem:small_divisors}). 
Then there exists $s_0>0$ such that, 
for any $s\ge s_0$, there exist $\und{\vare}_0=\und{\vare}_0(s,K)>0$ 
and two constants $\und{C}_s>0$ (depending only on $s$) and 
$\und{C}_{s,K}>0$ (depending on both $s$ and $K$) such that the following holds. 
For any $0\le \vare\leq \vare_0$ and any solution $\zak \in  C^0(I; H^s_\R)$ with
$\zak(t)\in B_{s_0,\R}(K\vare)$ for any $t \in I$ and
$\|\zak(0)\|_{s_0}\leq \vare$, the variable $Z=Z(t)$  in \eqref{varZfinale}
obtained by applying \Cref{prop:NF_finale} 
with parameter $\tR:=\vare^{-1}$ satisfies 
\begin{align}
	\| Z(t)\|_{s_0}^2 &\leq \und{C} \vare^2+ T \und{C}_{K}\vare^4\,,
    \label{en:est}
    \\
    \| Z(t)\|_{s}^2&\leq \und{C}_s \|Z(0)\|_s^2
    + T \und{C}_{s,K}\vare^2 \sup_{\tau\in [0,T]}\| Z(\tau)\|_s^2\,,   \label{en:est2}
\end{align}
for any $t \in [0,T]$. Moreover, there is $C_s>0$ such that one has the equivalence
  \begin{align}
   C_s^{-1} \| \zak(t)\|_s \leq \|Z(t)\|_s \leq C_s \| \zak(t)\|_s\,, 
   \qquad \forall t \in [0,T]\,.
   \label{equivalenza:Zzak}
   \end{align}
\end{proposition}

In order to prove \Cref{prop:main_est} we make a low-high 
frequencies splitting of the Sobolev norm of $Z$, namely
\begin{align}
\| Z(t)\|_s^2= \| \Pi_{\leq \frac{1}{\vare}} Z(t)\|_s^2
+\| \Pi_{\leq \frac{1}{\vare}}^\bot Z(t)\|_s^2
= \sum_{|j|\in [1,\vare^{-1}]} |j|^{2s} |z_j(t)|^2 
+\sum_{|j|>\vare^{-1}} |j|^{2s} |z_j(t)|^2 \,.
\label{splitting:low_high}
\end{align}
We estimate each term separately: in \Cref{sec:low_est} 
we prove an upper bound for the low frequencies part 
$ \| \Pi_{\leq \frac{1}{\vare}} Z(t)\|_s^2$ while in 
\Cref{sec:high_est} we prove a bound for the high 
frequencies part $ \| \Pi_{\leq \frac{1}{\vare}}^\bot Z(t)\|_s^2$.

\subsection{Low frequencies quasi-resonant block decomposition}\label{sec:low_est}
    In this section we estimate the Sobolev norms of the low frequencies part of the solution $z$ of \eqref{eq:finale_nuova}. The main result is the following.
    \begin{proposition}[Low-frequencies energy estimate]
    \label{prop:stima_low}
		Let $T>0$, $I:=[0,T]$, $K>1$ and
		$\kap \in \R_{+}\setminus \mathscr{N}$ (see Lemma \ref{lem:small_divisors}). 
		Then there exists $s_0>0$ such that, 
		for any $s\ge s_0$, there exist $ {\vare}_0= {\vare}_0(s,K)>0$ 
		and two constants ${C}_s>0$ (depending only on $s$) and $ {C}_{s,K}>0$ 
		(depending on both $s$ and $K$) such that the following holds. 
		For any $0\le \vare\leq \vare_0$ and any solution $\zak\in  C^0(I; H^s_\R)$ with
		$\zak(t)\in B_{s_0,\R}(K\vare)$ for any $t \in I$ and 
		 $\|\zak(0)\|_{s_0}\leq \vare$, 
		the variable $Z=Z(t)$ in \eqref{varZfinale} 
		obtained by applying \Cref{prop:NF_finale} with parameter $\tR:=\vare^{-1}$ satisfies 
\begin{equation}\label{stimascema_low}
\|\Pi_{\leq \frac{1}{\vare}} Z(t)\|_s^2
\leq 
C_s\|Z(0)\|_s^2+ T C_{s,K}\vare^2 \sup_{\tau\in [0,T]}\|Z(\tau)\|_s^2 \,.
\end{equation}
\end{proposition}
    From Corollary 5.11 of \cite{BML2} and Lemma 4.6 in \cite{BML3}, we have the following geometric block decomposition.
    \begin{lemma}
       Let $\delta,\, \tau,\, \nu$ as in \eqref{410}.  
       There is a partition $\cP= \{ \Omega_\alpha\}_{\alpha \in \N_0}$ of 
       $\Z^2$ such that
        \begin{enumerate}
            \item {\bfseries Dyadic blocks:} There 
            is a constant $R=R(\delta)>0$ such that
            \begin{align}
                \Omega_0 \subset B_{\Z^2}(R)\,, 
                \qquad  
                \max_{\xi \in  \Omega_\alpha} |\xi|\leq 2 \min_{\xi \in  \Omega_\alpha}|\xi|\,, 
                \qquad \text{for any }\alpha \in \N\,;
                \label{size:block}
            \end{align}
            \item {\bfseries Invariance of normal form operators:} Let $m\in \R$ 
            and $\cZ(x,\xi) \in \cN^{m,\delta}_{s_0}$ 
            a normal form symbol according to Definition \ref{def:normalform}, 
            then 
            \begin{align}
            \Pi_{\Omega_{\alpha}} \Opbw{\cZ(x,\xi)}= 
            \Opbw{\cZ(x,\xi)}\Pi_{\Omega_{\alpha}}\,, 
            \label{block_invariant}
            \end{align}
        \end{enumerate}
    \end{lemma}
We define the set of indices corresponding to 
blocks that intersect the ball $ B_{\mathbb{Z}^2}\big(1/\varepsilon\big)$ 
as
\begin{align}\label{I_vare}
\mathbb{I}_\vare := \left\{ \alpha \in \mathbb{N}_0 \,\middle|\, 
\Omega_\alpha \cap B_{\mathbb{Z}^2}\big( {\varepsilon^{-1}} \big) 
\neq \emptyset \right\}\,.
\end{align}

    \begin{lemma}
    \label{lem:barriera}
There is $\vare_0>0$ such that for any $ 0<\vare \leq \vare_0$ and  any 
$\alpha \in \mathbb{I}_\vare$ one has 
        \begin{align}\label{inclu}
            \Omega_\alpha \subset B_{\mathbb{Z}^2}\big(2\varepsilon^{-1}\big)\,.
        \end{align}
    \end{lemma}
    \begin{proof}
    First of all, by \eqref{size:block}, there is $R>0$ such that, 
    if $\vare\leq 2 R^{-1}=:\vare_0$, 
    $\Omega_0\subset B_{\mathbb{Z}^2}\big(2\varepsilon^{-1}\big)$.
    Moreover, by its definition in \eqref{I_vare}, we have that if 
    $\alpha\not=0$ belongs to $\mathbb{I}_\vare$, then there is 
    $\xi_\alpha\in \Omega_\alpha$ with $|\xi_\alpha|\leq\vare^{-1}$. 
    Then, for any $\xi \in \Omega_\alpha$, we have
        \begin{align*}
            |\xi|\leq \max_{\xi \in \Omega_\alpha} |\xi| \leq 2  \min_{\xi \in \Omega_\alpha} |\xi| 
            \leq 
            2 |\xi_\alpha|\leq {2}{\vare^{-1}}\,.
        \end{align*}
        This proves that $\xi \in B_{\mathbb{Z}^2}\big(2\varepsilon^{-1}\big) $, 
        thus proving the inclusion in \eqref{inclu}.
    \end{proof}

\begin{lemma}
    Let $s_0>1$ and $a(x,\xi) \in \cN^\sla{3}{2}_{s_0}$ a 
    real symbol such that 
    \begin{equation} \label{ass_Lem}
    \supp a(x,\xi) \subset \left\{ |\xi|\geq{9}{\vare^{-1}}\right\}\,.
    \end{equation} 
    There is $\vare_0>0$ such that for any 
    $ 0<\vare \leq \vare_0$ and  any $\alpha \in \mathbb{I}_\vare$ 
    one has 
     \begin{align}
         \Pi_{\Omega_\alpha} \Opbw{a(x,\xi)}
         \equiv
         \Opbw{a(x,\xi)}\Pi_{\Omega_\alpha}\equiv 0\,.
         \label{vanish_high_freq}
     \end{align}
\end{lemma}
\begin{proof}
We fix $\vare_0>0$ as in \Cref{lem:barriera}
     Since $\alpha \in \mathbb{I}_\vare$, we obtain 
    \begin{align*}
       j\in \Omega_\alpha 
       \quad 
       \implies 
       \quad  
       \chi\Big(j-k,\frac{j+k}{2}\Big)
       \hat{a}\left(j-k, \frac{j+k}{2}\right)=0\,,
    \end{align*}
    where $\chi$ is in \eqref{cut off defin}.
    Indeed, assume by contradiction that 
    $ \chi\Big(j-k,\frac{j+k}{2}\Big)
       \hat{a}\left(j-k, \frac{j+k}{2}\right)\neq0$,
    then $|k|\leq 2|j|$ and,  by \eqref{inclu}, we get
    \begin{align*}
         \frac{|j+k|}{2}\leq \frac{3}{2} | j| \leq {3}{\vare^{-1}}\,.
    \end{align*}
    On the other hand by the assumption \eqref{ass_Lem} 
    we have also $\frac{|j+k|}{2}\geq {9}{\vare^{-1}}$, 
    which is a contradiction. 
    Thus, for any $u \in H^s(\T^2;\C)$ we have 
    \begin{align*}
        \Pi_{\Omega_\alpha} \Opbw{a(x,\xi)}u= 
        \sum_{\substack{j\in \Omega_\alpha
        \\
        k \in \Z^2\setminus \{0\}}} 
        \chi\Big(j-k,\frac{j+k}{2}\Big)
        \widehat{a}\left(j-k, j-\frac{j-k}{2}\right)
        u_k e^{\ii j\cdot x}=0\,.
    \end{align*}
 This proves that $\Pi_{\Omega_\alpha} \Opbw{a(x,\xi)}\equiv 0$ and also
 $  \Opbw{a(x,\xi)}\Pi_{\Omega_\alpha}\equiv 
 \Big( \Pi_{\Omega_\alpha}\Opbw{a(x,\xi)}\Big)^*\equiv 0$.
\end{proof}

	We define the block invariant modified Sobolev norms 
	\begin{align}
	\tE^{(s)}(u):= \sum_{\alpha\in \mathbb{I}_\vare} \tM_\alpha^{2s} 
	\| \Pi_{\Omega_\alpha}u\|_{L^2}^2\,, 
	\qquad 
	\tM_{\alpha}:= \max |\Omega_\alpha|\,.
    \label{def:Energy_low}
	\end{align}
	One has the following equivalence
    \begin{lemma}
        For any $s\geq 0$ there is a constant $C_s>0$ such that 
        \begin{align}
            \| \Pi_{\leq \frac{1}{\vare}}u\|_s^2
            \leq 
            \tE^{(s)}(u)\leq C_s\|u\|_s^2\,, 
            \qquad \forall u \in \dot H^s(\T^2;\C)\,.
            \label{stima:lem:equi}
        \end{align}
    \end{lemma}
    \begin{proof}
    Thanks to \eqref{size:block}, there is a constant $R\geq 2$ 
    such that for any $\alpha\in \mathbb{I}_\vare$ one has 
        \begin{align}
            |k |\leq \tM_\alpha\leq R |k|\,, 
            \qquad 
            \text{for any }0\not=k \in \Omega_\alpha\,.
            \label{equi_tM}
        \end{align}
        Moreover, recalling the definition 
        \eqref{I_vare} of $\mathbb{I}_\vare$, one has 
        \begin{align}
            B_{\Z^2}(\vare^{-1})
            \subset 
            \bigcup_{\alpha \in \mathbb{I}_\vare} \Omega_\alpha\,. 
        \end{align}
        Then, using also \eqref{equi_tM}, we get
        \begin{align*}
 \| \Pi_{\leq {\vare^{-1}}}u\|_s^2= 
 \sum_{|k|\leq\vare^{-1}}|k|^{2s} |u_k|^2 
 \leq 
 \sum_{\alpha \in \mathbb{I}_\vare}
 \sum_{k \in \Omega_\alpha}|k|^{2s}|u_k|^2 
 \leq \tE^{(s)}(u)\,,
        \end{align*}
        which is the first inequality in \eqref{stima:lem:equi}.
        Moreover, using the second inequality in \eqref{equi_tM}, we get 
        \begin{align*}
            \tE^{(s)}(u)=
            \sum_{\alpha \in \mathbb{I}_\vare}
            \sum_{k \in \Omega_\alpha\setminus\{0\}}
            \tM_\alpha^{2s}|u_k|^2\leq R^s 
            \sum_{\alpha \in \mathbb{I}_\vare}
            \sum_{k \in \Omega_\alpha}|k|^{2s}|u_k|^2
            \leq R^s \| u\|_s^2\,,
        \end{align*}
       implying the second inequality in \eqref{stima:lem:equi}. This concludes the proof.
    \end{proof}
    
    We finally prove the main low-frequencies energy estimate.
    \begin{lemma}\label{lem:est_dt_E_low}
        Let $T>0$,  $ I:=[0,T]$, $K>1$ and
        $\kap \in \R_{+}\setminus \mathscr{N}$ (see Lemma \ref{lem:small_divisors}).  
        There are $s_0>0 $  
        such that for any $s\geq s_0$ there are $\vare_0=\vare(s,K)>0$ 
        and  a constant $C_{s,K}>0$  such that for any  
        $0\leq \vare\leq \vare_0$ and any solution $ \zak\in  C(I;H^s_\R)$ with
        $ \zak(t)\in B_{s_0,\R}(K\vare)$, $t \in I$ and $\| \zak(0) \|_{s_0} \leq \varepsilon$, 
        the variable $Z=Z(t)$ in \eqref{varZfinale}
        obtained applying  \Cref{prop:NF_finale} with parameter $\tR:=\vare^{-1}$, 
        satisfies 
   \begin{equation}\label{est_dt_E_low}
            \Big|\frac{d}{d t } \tE^{(s)}[Z](t)\Big|
            \leq 
            C_{s,K} \varepsilon^2 \| Z(t)\|_s^2\,.
    \end{equation}
    \end{lemma}
    \begin{proof}
        First of all, from the definition \eqref{def:Energy_low} 
        and the equation \eqref{eq:finale_nuova} 
        satisfied by $Z(t)$ 
        with $\tR:=\vare^{-1}$, we get
   \begin{align}
\frac{d}{d t } \tE^{(s)}[Z]&= 
\sum_{\alpha \in \mathbb{I}_\vare}\tM_\alpha^{2s}
2 \Re \langle 
\Pi_{\Omega_\alpha}\Opbw{-\ii \left[\Lambda(\xi)
+  \ta_{\geq 2}^{(\sla{3}{2})}(\zak;x,\xi)
+ \tb_1^{(\sla{1}{2})}(\zak;x,\xi)\right]}z,\Pi_{\Omega_\alpha}z
\rangle_{L^2}
\notag
\\&\qquad
+ \sum_{\alpha \in \mathbb{I}_\vare}\tM_\alpha^{2s}
2 \Re \langle 
\Pi_{\Omega_\alpha}\big(\bB_{\geq 2}(\vare^{-1};\zak)Z\big)^+,\Pi_{\Omega_\alpha}z
\rangle_{L^2}
\notag
\\&
\stackrel{\eqref{vanish_high_freq}}{=} 
\sum_{\alpha \in \mathbb{I}_\vare}\tM_\alpha^{2s}
2 \Re \langle 
\Pi_{\Omega_\alpha}\Opbw{-\ii \left[  \cZ_{\geq 2}^{(\sla{3}{2})}(\zak;x,\xi)
+ \cZ_1^{(\sla{1}{2})}(\zak;x,\xi)\right]}z,\Pi_{\Omega_\alpha}z
\rangle_{L^2}
\notag
\\&\qquad
+  \sum_{\alpha \in \mathbb{I}_\vare}\sum_{k \in \Omega_\alpha}\tM_\alpha^{2s}
2 \Re \Big[\big(\bB_{\geq 2}(\vare^{-1};\zak)Z\big)^+_k\bar{z_k}\Big]
\notag
\\&
\stackrel{\eqref{block_invariant}}{=}
\sum_{\alpha \in \mathbb{I}_\vare}\tM_\alpha^{2s}
2 \Re \langle 
\Opbw{-\ii \left[  \cZ_{\geq 2}^{(\sla{3}{2})}(\zak;x,\xi)
+ \cZ_1^{(\sla{1}{2})}(\zak;x,\xi)\right]}\Pi_{\Omega_\alpha}z,\Pi_{\Omega_\alpha}z
\rangle_{L^2}
\label{en:normal_form}
\\&\qquad
+  \sum_{\alpha \in \mathbb{I}_\vare}
\sum_{k \in \Omega_\alpha}\tM_\alpha^{2s}
2 \Re \Big[\big(\bB_{\geq 2}(\vare^{-1};\zak)Z\big)^+_k\bar{z_k}\Big] \,.
\label{en:resto}
\end{align}
        Since $\cZ_{\geq 2}^{(\sla{3}{2})}$ and $ \cZ_1^{(\sla{1}{2})} $ 
        are real we have that \eqref{en:normal_form} is identically zero. Indeed, the corresponding paradifferential operators are self-adjoint, hence their contribution is purely imaginary and vanishes after taking the real part.
        We then estimate \eqref{en:resto} using \eqref{equi_tM},  
        Cauchy–Schwarz inequality, \eqref{est:Bexp} with 
        $\tR=\vare^{-1}$, the hypothesis $ \zak(t)
        \in B_{s,\R}(K\vare)$ for $t \in I$ and \eqref{equivalenza:Zzak}, getting
  \begin{align*}
 | \eqref{en:resto}| &\leq 
 2  R^s  \sum_{\alpha \in \mathbb{I}_\vare}
 \sum_{k \in \Omega_\alpha}|k|^{2s}
 \Big|\big(\bB_{\geq 2}(\vare^{-1};\zak)Z\big)^+_k\Big|\,|z_k|
 \\&
 \leq  2\cdot  R^s \| z\|_s\| \bB_{\geq 2}(\vare^{-1};\zak)Z\|_s
 \leq
 C_{s,K}  \vare^2 \|Z(t)\|_s^2\,.
        \end{align*}
        This concludes the proof of \eqref{est_dt_E_low}.
    \end{proof}
    \paragraph{Proof of \Cref{prop:stima_low}.} Using first \eqref{stima:lem:equi} to bound $\| \Pi_{\leq \frac{1}{\vare}}z\|_s^2$,  
    one gets
    \begin{align*}
        \| \Pi_{\leq \frac{1}{\vare}}z\|_s^2 
        \leq \tE^{(s)}(z)
        \leq \tE^{(s)}(z(0))
        + \int_0^T  \Big| \frac{\di}{\di t}\tE^{(s)}(z)\Big| \,\di t\,,
    \end{align*}
    which, combined with  \eqref{stima:lem:equi} 
    and  \eqref{est_dt_E_low} for the time derivative implies
 \eqref{stimascema_low}.

\subsection{The large frequencies energy estimate}\label{sec:high_est}
The main result of this section is the following energy estimate for 
the high frequencies portion of $Z$ (recall the splitting \eqref{splitting:low_high}). 
\begin{proposition}[High-frequencies energy estimate]
\label{prop:stima_high}
Let $T>0$, $I:=[0,T]$, $K>1$ and 
$\kap \in \R_{+}\setminus \mathscr{N}$ (see Lemma \ref{lem:small_divisors}). 
 Then there exists $s_0>0$ 
such that, for any $s\ge s_0$, there exist $ {\vare}_0= {\vare}_0(s,K)>0$ 
and two constants ${C}_s>0$ (depending only on $s$) and 
$ {C}_{s,K}>0$ (depending on both $s$ and $K$) 
such that the following holds. 
For any $0\le \vare\leq \vare_0$ and any solution $\zak\in  C^0(I;H^s_\R)$
with $\zak(t)\in B_{s_0,\R}(K\vare)$, $t \in I$ and $\|\zak(0)\|_{s_0}\leq \vare$, 
the variable $Z=Z(t)$ in \eqref{varZfinale} obtained by applying 
\Cref{prop:NF_finale} with parameter $\tR:=\vare^{-1}$ satisfies 
		\begin{align}
			\|\Pi_{\leq \frac{1}{\vare}}^\bot Z(t)\|_s^2
			\leq 
			C_s\| Z(0)\|_s^2+ C_s K \vare \| Z(t)\|_s^2 
			+ T C_{s,K}\vare^2 \sup_{\tau \in [0,T]} \| Z(\tau)\|_s^2\,.
            \label{stimascema_high}
		\end{align}
        \end{proposition}

    The rest of this section is devoted to the proof of \Cref{prop:stima_high}. To this end,
	for $s\geq 0$, we define the modified $s$ power 
	of the Laplacian's symbol (see \eqref{eq:finale_nuova}-\eqref{def:tb12})
	\begin{align}
	L^{(2s)}=L^{(2s)}(\zak;x,\xi):= 
	\left(\Lambda(\xi)+  \ta_{\geq 2}^{(\sla{3}{2})}(\zak;x,\xi)
	+ \tb_1^{(\sla{1}{2})}(\zak;x,\xi)\right)^{\sla{4}{3}s}\,.
    \label{def:L2s}
\end{align}
We define also a large frequencies cut-off and the dispersion relation function 
\begin{align}
	\tilde \theta(y)= \begin{cases}0 & |y|\leq \frac{1}{3} \\ 1 & | y|\geq \frac{5}{12}.
	\end{cases}\qquad \Omega(y):= \sqrt{\kap y^3+y}\,, \ y\geq 0.
\label{def:tilde_theta}
\end{align}
	We collect in the following lemma some fundamental 
	properties of the modified $s$ power of 
	the Laplacian in \eqref{def:L2s}.
\begin{lemma}
  Let $T>0$,  $ I:=[0,T]$ and $K>1$. There is $s_0>1 $  
  such that for any $s\geq\s \geq s_0$ there are 
  $\vare_0=\vare(s,K)>0$ and  two constants $C_{s},\,\mu>0$ (depending only on $s$)  
  such that for any  $0\leq \vare\leq \vare_0$ and any solution 
  $ \zak(t)\in B_{\s,\R}(K\vare)$, $t \in I$, one has
\begin{enumerate}
\item \label{item:L2s} {\bfseries Ellipticity:} $ L^{(2s)}\in \cN_{\s-\mu}^{2s,\delta}$ 
with bounds
\begin{align}
\left|L^{(2s)}\right|_{2s,\s-\mu}\leq C_s\,, 
\qquad \left|L^{(2s)}- \Lambda(\xi)^{\sla{4}{3}s} \right|_{2s, \s-\mu}\leq C_s K \varepsilon\,.
       \label{est:L32}
       \end{align}
       
     \item \label{item:specloc} {\bfseries Spectral equivalence:}
      let  $\wt \theta(\cdot )$ be the cut-off in \eqref{def:tilde_theta}, 
      then $\wt\theta\Big( \frac{L^{(\sla{3}{2})}}{\Omega(1/\vare)} \Big) \in \cN_{\s-\mu}^0$, 
      and one has
\begin{align}
\left|\wt\theta\Big( \frac{L^{(\sla{3}{2})}}{\Omega(1/\vare)} \Big)
-\wt\theta\left( \frac{ \Lambda(\xi)}{\Omega(1/\vare)}\right)\right|_{0,\s-\mu}
\leq C_s K \varepsilon\,, 
\qquad 
\supp \wt\theta\Big( \frac{L^{(\sla{3}{2})}}{\Omega(1/\vare)} \Big) \subset 
\left\{ |\xi| \geq \frac{1}{ 36\vare}\right\}\,.
 \label{specloc:theta_tilde}
 \end{align}
Moreover, for any $r\geq 0$, 
\begin{align}
\left|L^{(2s)}\wt\theta\Big( 
\frac{L^{(\sla{3}{2})}}{\Omega(1/\vare)}
\Big)\right|_{2s+r,\s-\mu}
&\leq C_s \vare^{r}\,, 
\label{eta:L2sdiff1}
\\
\left|L^{(2s)}\wt\theta\Big( 
\frac{L^{(\sla{3}{2})}}{\Omega(1/\vare)}\Big)
- \Lambda(\xi)^{\sla{4}{3}s}\wt\theta\left( 
\frac{\Lambda(\xi)}{\Omega(1/\vare)}\right) \right|_{2s+r, \s-\mu}
&\leq C_s K\varepsilon^{1+r}\,.
\label{eta:L2sdiff2}
\end{align}

\item {\bfseries Derivative:} $L^{(2s)}$ 
is differentiable with respect to $ \zak$ with estimates 
			\begin{align}
			\left|\di_\zak \left(L^{(2s)} \tilde \theta \Big(\frac{L^{(\sla{3}{2})}}{\Omega(1/\vare)}\Big)\right)[ \wh \zak]\right|_{2s,\s-\mu}\leq C_s K \varepsilon \|\wh \zak\|_{\s}.
            \label{est:diffL2s}
			\end{align}
		\end{enumerate}
	\end{lemma}
	\begin{proof}
We first prove \Cref{item:L2s}.    
First of all, by \eqref{def:ta32}, \eqref{def:tb12}, 
\eqref{stima:derivataa32} and the estimate 
\eqref{bound:homosym} for $\sym{\tb}{1}{\sla{1}{2}}\in \wt \Gamma_1^{\sla{1}{2},\delta}$, 
one has that  the map 
\begin{align*}
 B_{\s,\R}(K\vare) \ni \zak \to 
 L^{(\sla{3}{2})}(\zak;x,\xi)= 
 \Lambda(\xi)+\sym{a}{\geq 2}{\sla{3}{2}}(\zak;x,\xi)
 + \sym{b}{1}{\sla{1}{2}}(\zak;x,\xi) \in C^{1}\left(B_{\s,\R}(K\vare); 
 \cN_{\s-\mu}^{\sla{3}{2},\delta}\right)\,,
        \end{align*}
satisfies the estimates 
\begin{equation}\label{est:diffproof}
\begin{aligned}
   \left|L^{(\sla{3}{2})}(\zak;x,\xi)\right|_{\sla{3}{2},\s-\mu}&\lesssim_s 1\,, 
\\   \max_{|\alpha|+|\beta|\leq \s-\mu} 
   \left| 
   \pa_x^\alpha \pa_\xi^\beta \di_{\zak}L^{(\sla{3}{2})}(\zak;x,\xi)[\hat \zak]\right|
   &\lesssim_s 
   K \vare \| \hat \zak\|_{\s} \langle \xi \rangle^{\sla{3}{2}-\delta | \beta|} 
   + \| \hat \zak\|_{\s} \langle \xi\rangle^{\sla{1}{2}-\delta |\beta|}\,.
\end{aligned}
\end{equation}
Moreover, by construction, we have 
$ L^{(2s)}= \left( L^{(\sla{3}{2})}\right)^{\sla{4}{3} s}$. 
Then Faà di Bruno formula gives the following point-wise 
estimates: for 
$(x,\xi)\in \T^2\times \R^2$ and $|\alpha|+ |\beta|\leq \s-\mu$ one has
\begin{equation}\label{est:L2sproof}
\begin{aligned}
\left| \pa_x^\alpha \pa_\xi^\beta  L^{(2s)}(x,\xi)\right| 
&\lesssim_s  
\max_{0\leq d\leq |\alpha|+ |\beta|} 
\left( \left|L^{(\sla{3}{2})}(x,\xi)\right|\right)^{\sla{4}{3} s- d}
\max_{ \substack{\alpha_1+ \ldots +\alpha_d=\alpha\\\beta_1+ \ldots +\beta_d=\beta}}\ \prod_{j=1}^d \left| \pa_x^{\alpha_j}\pa_\xi^{\beta_j}  L^{(\sla{3}{2})}(x,\xi)\right|
\\
    &\lesssim_s  
    \left|L^{(\sla{3}{2})}\right|_{\sla{3}{2},\s-\mu}^{\sla{4}{3} s} \ \max_{0\leq d\leq |\alpha|+ |\beta|} 
    \langle \xi\rangle^{2s- \sla{3}{2} d}
    \max_{ \substack{\alpha_1+ \ldots +\alpha_d=\alpha\\ \beta_1+ \ldots +\beta_d=\beta}}
    \ \prod_{j=1}^d \langle \xi\rangle^{\sla{3}{2}- \delta |\beta_j|}
    \\
   & \lesssim_s 
    \langle\xi\rangle^{2s- \delta |\beta|}\,,
\end{aligned}
\end{equation}
where we also used that $d\leq |\alpha|+ |\beta|\leq \s-\mu \leq s\leq \tfrac43 s$. This proves the first estimate in \eqref{est:L32}.
Moreover one has 
\begin{align*}
    \di_\zak L^{(2s)}(\zak;x,\xi)[\hat \zak]=\tfrac{4}{3}s L^{(2s-\sla{3}{2})}(\zak;x,\xi)\di_{\zak}L^{(\sla{3}{2})}(\zak;x,\xi)[\hat \zak].
\end{align*}
Then, by \eqref{est:diffproof} and \eqref{est:L2sproof}, for any 
$ |\alpha|+|\beta|\leq \s-\mu$, one has 
\begin{equation}\label{est.dl2s}
\begin{aligned}
    \left|\pa_x^\alpha\pa_\xi^\beta \di_\zak L^{(2s)}(\zak;x,\xi)[\hat \zak]\right|
    &\lesssim_s \max_{\substack{\alpha_1+\alpha_2=\alpha \\ \beta_1+ \beta_2=\beta}} 
    \left| \pa_x^{\alpha_1}\pa_\xi^{\beta_1}L^{(2s-\sla{3}{2})}\right|\ 
    \left| \pa_x^{\alpha_2}\pa_\xi^{\beta_2} \di_{\zak}L^{(\sla{3}{2})}(\zak;x,\xi)[\hat \zak]\right|
    \\
   & \lesssim_s 
   K\vare  \| \hat \zak\|_{\s} \langle \xi\rangle^{2s-\delta|\beta|}
   + \| \hat \zak\|_{\s} \langle \xi\rangle^{2s-1-\delta|\beta|}\ .
\end{aligned}
\end{equation}
A first order Taylor's expansion together with the bound \eqref{est.dl2s}, 
proves also the second bound in \eqref{est:L32}.   
Next we prove \Cref{item:specloc}. 
We first note that (recalling \eqref{def:L2s} 
and \Cref{rem:symbols_modozero})
        \begin{align*}
      L^{(\sla{3}{2})}
    = \Lambda(\xi) \Big( 1+ \sym{a}{\geq 1}{0}\Big)\,, 
    \qquad 
    \sym{a}{\geq 1}{0}:=\frac{\sym{\ta}{\geq 2}{\sla{3}{2}}
    + \sym{\tb}{1}{\sla{1}{2}}}{\Lambda(\xi)}\in \cN^0_{\s-\mu}\,,
        \end{align*}
        with estimates $\left|\sym{a}{\geq 1}{0}\right|_{0, \s-\mu}\lesssim_s \vare$. 
        This implies, for $\vare$ small enough, that
        \begin{align*}
            \frac12 \Lambda(\xi) \leq L^{(\sla{3}{2})} \leq 2 \Lambda(\xi)\,.
        \end{align*}
    Then, recalling \eqref{def:dispersion}, 
    we have that, on the support of 
    $\wt \theta \Big(\frac{L^{(\sla{3}{2})}}{\Omega(1/\vare)}\Big)$, 
    \begin{align*}
         \Omega(36|\xi|)=
         \Lambda(36\x)\geq 6\Lambda(\xi)\geq {3\Omega(1/\vare)}
         \frac{L^{(\sla{3}{2})}}{\Omega(1/\vare)}\geq {\Omega(1/\vare)}\,.
    \end{align*}
    Since $\Omega(\cdot)$ is increasing we get  
    \begin{align}\label{supportotildetheta}
       \supp\left(\wt \theta \Big(\frac{L^{(\sla{3}{2})}}{\Omega(1/\vare)}\Big)\right)
       \subseteq 
       \Big\{  |\xi|\geq \frac{1}{36 \vare}\Big\}\,.
    \end{align}
    This proves the second property in \eqref{specloc:theta_tilde}.
    Moreover Faà di Bruno formula gives 
 \begin{align}
\left| \pa_x^\alpha \pa_\xi^\beta\wt \theta 
\Big(\frac{L^{(\sla{3}{2})}}{\Omega(1/\vare)}\Big) 
\right|
&\lesssim 
\max_{0\leq d\leq |\alpha|+ |\beta|} \wt \theta^{(d)}
\Big(\frac{L^{(\sla{3}{2})}}{\Omega(1/\vare)}\Big)\Omega(1/\vare)^{-d}
\max_{ \substack{\alpha_1+ \ldots +\alpha_d=\alpha\\\beta_1+ \ldots +\beta_d=\beta}}\ 
\prod_{j=1}^d \left| \pa_x^{\alpha_j}\pa_\xi^{\beta_j}  L^{(\sla{3}{2})}\right|
\nonumber
\\
&\lesssim 
\max_{0\leq d\leq |\alpha|+ |\beta|} \wt \theta^{(d)}
\Big(\frac{L^{(\sla{3}{2})}}{\Omega(1/\vare)}\Big)\Omega(1/\vare)^{-d}
\langle \xi \rangle^{\sla{3}{2} d-\delta |\beta|}
\lesssim_\kap \langle \xi \rangle^{-\delta |\beta|}\,.
 \label{thetatildino}
\end{align}
 By combining estimates \eqref{est:L2sproof} and \eqref{thetatildino}
   with \eqref{supportotildetheta}, and arguing as in \eqref{est.dl2s}, one deduces 
   the  bound in \eqref{eta:L2sdiff1}.
   Let us now check the first  in \eqref{specloc:theta_tilde} and the  bound in
   \eqref{eta:L2sdiff2}.
    By using  that
\begin{align*}
\di_\zak \wt \theta \Big(\frac{L^{(\sla{3}{2})}}{\Omega(1/\vare)}\Big)[ \hat \zak]
= 
\wt \theta' \Big(\frac{L^{(\sla{3}{2})}}{\Omega(1/\vare)}\Big)\Omega(1/\vare)^{-1}
\di_\zak L^{(\sla{3}{2})}(\zak;x,\xi)[\hat \zak]\,,
\end{align*}
 we get the estimates
    \begin{align}
 \left| \pa_x^\alpha \pa_\xi^\beta\di_\zak \wt \theta 
 \Big(\frac{L^{(\sla{3}{2})}}{\Omega(1/\vare)}\Big)[ \hat \zak]\right|
 &\lesssim 
 \max_{\substack{\alpha_1+\alpha_2=\alpha \\ \beta_1+ \beta_2=\beta}} 
 \left| \pa_x^{\alpha_1}\pa_\xi^{\beta_1} \wt \theta' 
 \Big(\frac{L^{(\sla{3}{2})}}{\Omega(1/\vare)}\Big)\right|\ 
 \left| \pa_x^{\alpha_2}\pa_\xi^{\beta_2} 
 \di_{\zak}L^{(\sla{3}{2})}(\zak;x,\xi)[\hat \zak]\right|\Omega(1/\vare)^{-1}
  \nonumber
  \\&
 \lesssim_\kap  \left(K\vare   \langle \xi\rangle^{-\delta|\beta|}
 + \langle \xi\rangle^{-1-\delta|\beta|}\right) \| \hat \zak\|_{\s}\,.
        \label{est:dtildetheta}
    \end{align}
    Gathering \eqref{est:L2sproof},\eqref{est.dl2s} and \eqref{thetatildino}, 
    \eqref{est:dtildetheta} we get \eqref{est:diffL2s}. Finally the first of 
    \eqref{specloc:theta_tilde} and  \eqref{eta:L2sdiff2} follow by a first 
    order Taylor expansion and \eqref{est:dtildetheta}, \eqref{est:diffL2s}, 
    using also 
    \eqref{supportotildetheta}.
    \end{proof}
	In order to prove Proposition \ref{prop:stima_high}
	we define the energy
	\begin{align}
	\cE^{(s)}(\zak)[Z]:= 
	\left\langle \Opbw{L^{(2s)}\wt\theta
	\Big( \frac{L^{(\sla{3}{2})}}{\Omega(1/\vare)}\Big)} z, 
	z\right\rangle_{L^2(\T^2;\C)}
	= 
	\int_{\T^2}\Opbw{L^{(2s)}\wt\theta
	\Big( \frac{L^{(\sla{3}{2})}}{\Omega(1/\vare)}\Big)} z\,\cdot \, \overline{z}\, \di x \,.
    \label{def:mod_ene}
	\end{align}
	In the next lemma we analyse an important property of the 
	modified energy introduced above.

    \begin{lemma}\label{lem:est_dt_E}
         Let $T>0$,  $ I:=[0,T]$ and $K>1$. There are $s_0>0 $  
         such that for any $s\geq s_0$ there are $\vare_0=\vare(s,K)>0$ 
         and  a constant $C_{s,K}>0$  such that for any  $0\leq \vare\leq \vare_0$ 
         and any solution $ \zak\in  C^0(I;H^s_\R)$ with $ \zak(t)\in B_{s_0,\R}(K\vare)$, $t \in I$, 
         the variable $Z=Z(t)$ in \eqref{varZfinale}
         obtained applying  \Cref{prop:NF_finale} 
         with parameter $\tR:=\vare^{-1}$, satisfies 
        \begin{align}
            \Big|\frac{d}{d t } \cE^{(s)}(\zak)[Z](t)\Big|\leq C_{s,K}\varepsilon^2 \|Z(t)\|_{s}^2.
            \label{est:lem_dtE}
        \end{align}
    \end{lemma}
    \begin{proof}
  By \eqref{def:L2s}, and fixing $\tR=\varepsilon^{-1}$, the equation \eqref{eq:finale_nuova} can be rewritten as
		\begin{equation}
			\pa_t Z =  \vOpbw{-\ii L^{(\sla{3}{2})}(\zak;x,\xi)}Z+ B_{\geq 2}(\vare^{-1};\zak)Z\,.
            \label{eq:L32}
		\end{equation} 
	We then compute the time derivative of the modified energy $\cE^{(s)}$. 
	Using \eqref{eq:L32}, \eqref{def:mod_ene} and the 
	water waves equation \eqref{eq:zak} for $ \zak$, 
	we get  
	\begin{align}
	\frac{d}{d t } \cE^{(s)}(\zak)[Z]= &\left\langle \ii \left[ \Opbw{L^{(2s)}\wt \theta \Big(\frac{L^{(\sla{3}{2})}}{\Omega(1/\vare)}\Big)},\Opbw{L^{(\sla{3}{2})}}\right]   z, 
	z\right\rangle_{L^2(\T^2;\C)}\label{dt:comm}
	\\&
	+ \left\langle \Opbw{\di_\zak\left[L^{(2s)}\wt \theta 
	\Big(\frac{L^{(\sla{3}{2})}}{\Omega(1/\vare)}\Big)\right]
	[ \hamvec{\cH}(\zak)]} z, z\right\rangle_{L^2(\T^2;\C)}
	\label{dt:dt}
	\\&
	+ 2 \Re\left\langle  \Opbw{L^{(2s)}\wt \theta 
	\Big(\frac{L^{(\sla{3}{2})}}{\Omega(1/\vare)}\Big)}  
	(\bB_{\geq 2}(\vare^{-1};\zak)Z)^+, z\right\rangle_{L^2(\T^2;\C)}\,.
	\label{dt:resto}
	\end{align}
    We estimate each summand in the r.h.s. above separately.
    First of all we claim that 
    	\begin{align}
	    \left\| \left[ \Opbw{L^{(2s)}\wt\theta\Big( 
	    \frac{L^{(\sla{3}{2})}}{\Omega(1/\vare)}\Big)},\Opbw{L^{(\sla{3}{2})}}\right]   
	    z\right\|_{-s} \leq C_s \vare^{2+\sla{1}{16}} \| z\|_s\,.
        \label{est:comm}
	\end{align}
    Therefore, by Cauchy-Schwartz,  \eqref{est:comm} and \eqref{equivalenza:Zzak}, 
    we get 
    \begin{align}
        |\eqref{dt:comm}|\leq &  \left\| 
        \left[ \Opbw{L^{(2s)}\wt\theta\Big( \frac{L^{(\sla{3}{2})}}{\Omega(1/\vare)}\Big)},
        \Opbw{L^{(\sla{3}{2})}}\right]   z\right\|_{-s} \| z\|_s
        \lesssim_s 
        \vare^{2+\sla{1}{16}} \| z\|_s^2\,.
        \label{est:dt:comm}
    \end{align}
    Moreover, by Cauchy-Schwartz, \Cref{thm:action},  
    \eqref{est:diffL2s}, \eqref{est:XH} and \eqref{equivalenza:Zzak}, we get 
\begin{align}
    |\eqref{dt:dt}|\lesssim_s 
    K\vare \| \hamvec{\cH}(\zak)\|_{\s_0+\mu} \| z\|_{s}^2
    \lesssim_s K^2\vare^2 \|z\|_s^2\,,
    \label{est:dt:dt}
\end{align}
   for some $\s_0>1$. Here, in the last inequality we also used that 
   $\| \hamvec{\cH}(\zak)\|_{\s_0+\mu}\lesssim 
   \| \zak\|_{\s_0+\mu+\sla{3}{2}}\lesssim K\vare $, 
   for $s_0$ sufficiently large. 
    Finally, by Cauchy-Schwartz, \Cref{thm:action}, \eqref{eta:L2sdiff1}-\eqref{eta:L2sdiff2} 
    and \eqref{est:Bexp} with $\tR:=\vare^{-1}$, we have
    \begin{align}
        |\eqref{dt:resto}|\leq C_s\| \bB_{\geq 2}(\vare^{-1};\zak)Z\|_s\| z\|_s  
        \leq C_{s,K} \vare^2 \|z\|_s^2\,. 
        \label{est:dt_Bione}
    \end{align}
    Gathering \eqref{est:dt:comm}, \eqref{est:dt:dt} and \eqref{est:dt_Bione} 
    we get the claimed \eqref{est:lem_dtE}.

  \smallskip
  \noindent  
    In order to conclude the proof it remains to prove
    the claim \eqref{est:comm}.
    We denote by $\ta:=L^{(2s)}\wt\theta\left( \varepsilon^{\sla{3}{2}} L^{(\sla{3}{2})}\right)$ and $\tb:=L^{(\sla{3}{2})}$. First, by \eqref{eta:L2sdiff1}-\eqref{eta:L2sdiff2} (with $r=2+\sla{1}{16}$, $\ta \in \cN_{s_0-\mu}^{2s+2+\sla{1}{16}}$ and $\tb \in \cN_{s_0-\mu}^{\sla{3}{2}}$.
	   By the composition \Cref{compoparapara0} applied with $\vr=4$, we have
        \begin{align}
            \left[ \Opbw{L^{(2s)}\wt\theta\Big(\frac{L^{(\sla{3}{2})}}{\Omega(1/\vare)}\Big)},
            \Opbw{L^{(\sla{3}{2})}}\right]=\overbrace{\Opbw{\frac{1}{\ii} \big\{\ta, 
            \tb \big\}}}^{=0} +\Opbw{ 2p_3(\ta, \tb)}+ R(\ta,\tb)\,,
            \label{comm:esp}
        \end{align}
        where $p_3$ is defined in \eqref{def:pk} and 
        $R(\ta,\tb):= \cQ(\ta, \tb)-\cQ(\tb,\ta)$ satisfies, for $\sigma_0>0$ 
        large enough and taking $s_0\geq \s_0+\mu$, 
        \begin{align}
            \| R(\ta,\tb)z\|_{-s}\lesssim_s 
            | \ta|_{2s+2+\sla{1}{16},\s_0}| \tb|_{\sla{3}{2}, \s_0} 
            \| z\|_{s}\lesssim_s | \ta|_{2s+2+\sla{1}{16},s_0-\mu}
            | \tb|_{\sla{3}{2}, s_0-\mu} \| z\|_{s}
            \stackrel{\eqref{eta:L2sdiff1}}{\lesssim_s}K^2 \vare^2 \| z\|_s\,.
            \label{est:Rab}
        \end{align}
        Next we denote by $ \ta_{\geq 1}:= \ta-\Lambda(\xi)^{\sla{4}{3}s}\wt\theta\left( \frac{\Lambda(\xi)}{\Omega(1/\vare)}\right)$ and $ \tb_{\geq 1}:= \tb- \Lambda(\xi)$. 
        By its expression \eqref{def:pk}, we note that 
        \begin{align*}
            p_3(\ta,\tb)=p_3\Big(\ta_{\geq 1}  ,\tb\Big) + p_3\Big(\Lambda(\xi)^{\sla{4}{3}s}\wt\theta\left(\frac{\Lambda(\xi)}{\Omega(1/\vare)}\right) , \tb_{\geq 1}\Big)+\overbrace{p_3\Big(\Lambda(\xi)^{\sla{4}{3}s}\wt\theta
            \Big( \frac{\Lambda(\xi)}{\Omega(1/\vare)}\Big) , \Lambda(\xi)\Big)}^{=0}\,.
        \end{align*}
        Then, using also \eqref{est:pk}, \eqref{eta:L2sdiff1}-\eqref{eta:L2sdiff2}, \eqref{est:L32} and the fact that $3\delta> \frac98+\frac32$, we get 
 \begin{align}
 \big| p_3(\ta,\tb)\big|_{2s, 3}
 &\leq \big| p_3(\ta,\tb)\big|_{2s+\sla{9}{8}+\sla{3}{2}-3\delta, 3}
 \nonumber
\\
&\leq 
            \big|p_3\big(\ta_{\geq 1} ,\tb\big)  \big|_{2s+\sla{9}{8}+\sla{3}{2}-3\delta, 3}
            + \Big| p_3\Big(\Lambda(\xi)^{\sla{4}{3}s}\wt\theta
            \Big( \frac{\Lambda(\xi)}{\Omega(1/\vare)}\Big), 
            \tb_{\geq 1}\Big)\Big|_{2s+\sla{9}{8}+\sla{3}{2}-3\delta, 3}
            \notag
            \\
            &\lesssim 
            \Big| \ta_{\geq 1}\Big|_{2s+\sla{9}{8},3+q} |\tb|_{\sla{3}{2}, 3+q}
            +\Big| \Lambda(\xi)^{\sla{4}{3}s}\wt\theta
            \Big( \frac{\Lambda(\xi)}{\Omega(1/\vare)}\Big)\Big)\Big|_{2s+\sla{9}{8},3+\tq}\Big| 
            \tb_{\geq 1}\big|_{\sla{3}{2},3+q}
            \notag
            \\
            &\lesssim_s  \ K \vare^{1+\sla{9}{8}}\lesssim_s \vare^{2+\sla{1}{16}} \,,
            \label{est:p3}
        \end{align}
        provided $ \vare^{\sla{1}{16}}K\leq 1$ and $s_0\geq 3+\tq$.
        Then, eventually applying \Cref{thm:action} and \eqref{est:p3}, we get 
        \begin{align}
           \|  \Opbw{ 2p_3(a, b)}z\|_{-s}\lesssim_{s}
           | p_3(\ta, \tb)|_{2s, s_0}\| z\|_{s}\lesssim_s \vare^{2+\sla{1}{16}} \| z\|_{s}\,.
           \label{est:opp3}
        \end{align}
In view of \eqref{comm:esp}, we gather \eqref{est:opp3} and 
\eqref{est:Rab} to obtain \eqref{est:comm}. 
This concludes the proof of the Lemma.
	\end{proof}
    \paragraph{End of the proof of \Cref{prop:stima_high}.}
In order to obtain estimate \eqref{stimascema_high}
we first show that
		\begin{align}
			\|\Pi_{\leq \frac{1}{\vare}}^\bot Z(t)\|_s^2\leq 
			C_s\left( \|Z(0)\|_s^2+ K\vare \|Z(t)\|_s^2
			+\int_0^t \Big|\frac{d}{d t } \cE^{(s)}(\zak)[Z](\tau)\Big|\,\di \tau 
			\right)\,,
            \label{stimascema}
		\end{align}
	where $\cE^{(s)}$
 is the modified energy introduced in \eqref{def:mod_ene}. 
We begin by proving the first inequality. 
By monotonicity of 
$ \Omega(\cdot)$, for any $\xi \in \Z^2$ with 
$ |\xi|\geq \frac{1}{\vare}$, 
we have $\Lambda(\xi)\geq \Omega(1/\vare)$ 
and thus $\wt\theta
\Big( \frac{\Lambda(\xi)}{\Omega(1/\vare)}\Big)\equiv 1$. 
Then 
\begin{align*}
\|\Pi_{\leq \frac{1}{\vare}}^\bot Z\|_s^2 
&= 
\sum_{|\xi|>\frac{1}{\vare}} |\xi|^{2s} |\hat z (\xi)|^2
= \sum_{|\xi|>\frac{1}{\vare}} |\xi|^{2s} 
\wt\theta\Big( \frac{\Lambda(\xi)}{\Omega(1/\vare)}\Big) 
|\hat z (\xi)|^2
\\&
\lesssim_{s,\kap}  \sum_{\xi \in \Z^2} 
\Lambda(\xi)^{\sla{4}{3} s} 
\wt\theta\Big( \frac{\Lambda(\xi)}{\Omega(1/\vare)}\Big) 
|\hat z (\xi)|^2
= \langle \Opbw{\Lambda(\xi)^{\sla{4}{3} s}
\wt\theta\Big( \frac{\Lambda(\xi)}{\Omega(1/\vare)}\Big)} 
z, z\rangle\,.
\end{align*}
Then, using \eqref{eta:L2sdiff1}-\eqref{eta:L2sdiff2}, 
we get 
\begin{equation}\label{est:energia}
\begin{aligned}
 \langle &\Opbw{\Lambda(\xi)^{\sla{4}{3} s}
\wt\theta
\Big( \frac{\Lambda(\xi)}{\Omega(1/\vare)}\Big)} 
z, z\rangle 
\\&= 
\cE^{(s)}(Z)
+ \langle 
\Opbw{\Lambda(\xi)^{\sla{4}{3} s}\wt\theta
\Big( \frac{\Lambda(\xi)}{\Omega(1/\vare)}\Big)
- L^{(2s)}\wt\theta
\Big( \frac{L^{(\sla{3}{2})}}{\Omega(1/\vare)}\Big) } z, 
z \rangle 
    \\ & 
    \lesssim_s \cE^{(s)}(Z) 
+ \Big\| \Opbw{\Lambda(\xi)^{\sla{4}{3} s}\wt\theta
\Big( \frac{\Lambda(\xi)}{\Omega(1/\vare)}\Big)
- L^{(2s)}\wt\theta
\Big( \frac{L^{(\sla{3}{2})}}{\Omega(1/\vare)}\Big) } 
z \Big\|_{- s} \| z \|_s 
    \\
 &  \lesssim_s \cE^{(s)}(Z)
 + \left| \Lambda(\xi)^{\sla{4}{3} s}\wt\theta
\Big( \frac{\Lambda(\xi)}{\Omega(1/\vare)}\Big)
- L^{(2s)}\wt\theta
\Big( \frac{L^{(\sla{3}{2})}}{\Omega(1/\vare)}\Big) 
\right|_{2s, s_0} \| z \|_s^2
 \\
 & \lesssim_s  \cE^{(s)}(Z)+ K \vare \| z\|_s^2\,,
 \end{aligned}
 \end{equation}
where in the last inequality we used \eqref{equivalenza:Zzak}. 
Then, imposing $  K^3 \vare_0\leq 1 $ and  
using \eqref{est:L32}, we have 
\begin{equation}\label{FTC}
\begin{aligned}
\cE^{(s)}(Z)= &  \left\langle \Opbw{\Big[L^{(2s)}
\wt\theta\Big( \frac{L^{(\sla{3}{2})}}{\Omega(1/\vare)}\Big)
\Big]_{|t=0}} z_0, z_0\right\rangle
+ \int_0^t \frac{\di}{\di t} \cE^{(s)}(Z)\, \di \tau
\\&
\lesssim_s 
\|\Opbw{\Big[L^{(2s)}\wt\theta\Big( 
\frac{L^{(\sla{3}{2})}}{\Omega(1/\vare)}\Big)\Big]_{|t=0}} 
z_0\|_{-s}
\|z_0\|_s
+\int_0^t \frac{\di}{\di t} \cE^{(s)}(Z)\, \di \tau
\\&
\lesssim_s 
\|Z(0)\|_s^2
+ \int_0^t \frac{\di}{\di t} \cE^{(s)}(Z)\, \di \tau\,.
\end{aligned}
\end{equation}
        Gathering \eqref{est:energia} and \eqref{FTC}, we get \eqref{stimascema}.
         Combining 
   \eqref{stimascema} and \eqref{est:lem_dtE} we get \eqref{stimascema_high}.
    \qed

	
	\begin{proof}[{\bf Proof of \Cref{prop:main_est}}]. We 
	first define $\und{C}_s$ as twice the 
	maximum of the constants $C_s$ given in \Cref{prop:stima_low} 
	and \Cref{prop:stima_high}. In the 
	same way we define $\und{C}_{s,K}$ as twice the 
	maximum of the constants $C_{s,K}$ given in 
	\Cref{prop:stima_low} and \Cref{prop:stima_high}. 
	Then, by \eqref{splitting:low_high}, \eqref{stimascema_low} 
	and \eqref{stimascema_high}, we get \eqref{en:est}-\eqref{en:est2}.
	The bound \eqref{equivalenza:Zzak}
	simply follows 
	   By  \Cref{prop:blockdecoupling,prop:nfquadtra,prop:NF_finale} 
   with estimates \eqref{equiv:U_etaomega}, \eqref{stima:adm} 
   and \eqref{stima:admR_nonspec}.
	\end{proof}

\subsection{Main bootstrap and proof of Theorem \ref{thm:main}}\label{sezione bootstrap argument}
The next bootstrap Proposition \ref{prop:boot}  is the main ingredient 
for the proof of the long time existence Theorem \ref{thm:main}.
 Proposition \ref{prop:boot} is a consequence of a priori energy estimates of 
Proposition  \ref{prop:main_est}.

\begin{proposition}[Bootstrap]\label{prop:boot}
   Let $\kap \in \R_{+}\setminus \mathscr{N}$, where 
$\mathscr{N}$ is the zero measure set given by Lemma \ref{lem:small_divisors}. 
There exist $s_0>0$, $\tK>1$ and $c>0$ with the following property.

Let $\vare_0=\vare_0(s_0,\tK)>0$ be given by Proposition \ref{prop:main_est} (with $s=s_0$). 
For any $0\le \bar{\vare}\le \vare_0$, assume that $\zak(t)$ is a solution of \eqref{eq:zak} satisfying
\begin{equation}\label{ipobootstrap}
\| \zak(0)\|_{s_0} \le \bar{\vare}, 
\qquad  
\| \zak(t)\|_{s_0}\le \tK \bar{\vare}, 
\quad \forall t \in [0,T], 
\quad T\le c\,\bar{\vare}^{-2}.
\end{equation}
Then the following improved estimate holds
\begin{align}
\| \zak(t)\|_{s_0}\le \frac{\tK}{2}\,\bar{\vare}, 
\quad \forall t \in [0,T].
\label{improve:boot}
\end{align}
Moreover, for any $s\ge s_0$, there exist constants $\tK_s>1$ and $c_s>0$ such that the following holds.
Let $\vare_0=\vare_0(s,\tK_s)>0$ be given by Proposition \ref{prop:main_est}. 
If $0\le \bar{\vare}\le \vare_0$ and, in addition to \eqref{ipobootstrap},
\begin{equation}\label{ipobootstraps}
\zak(0)\in H^s_\R, \qquad
\| \zak(t)\|_{s}\le \tK_s \|\zak(0)\|_s,
\quad \forall t \in [0,T], 
\quad T\le c_s\,\bar{\vare}^{-2},
\end{equation}
then
\begin{align}
\| \zak(t)\|_{s}\le \frac{\tK_s}{2}\,\|\zak(0)\|_s,
\quad \forall t \in [0,T].
\label{improve:boots}
\end{align}
\end{proposition}

\begin{proof}
We fix $ s_0>0$ as the one given by  \Cref{prop:main_est}\,.
Consider $\und{C}>0$ given by \Cref{prop:main_est} 
and $C=C_{s_0}>0$ given in \eqref{equivalenza:Zzak} (with $s=s_0$).
 Then we define $\tK:= \sqrt{8C\, \und{C}}$. 
Then, applying \Cref{prop:main_est} 
with 
$ K\equiv \tK$ and $s=s_0$, there is $\vare_0=\vare_0(s_0,\mathtt{K})>0$  such that, if \eqref{ipobootstrap}
 holds, then (see \eqref{en:est})
 \[
\| \zak(t)\|_{s_0}^2 \leq C \und{C} \bar{\vare}^2+ T C \und{C}_{\tK} \mathtt{K}^{2} \bar{\vare}^{4}
=\frac{\tK^2}{8} \bar{\vare}^2+T C \und{C}_{\tK} \mathtt{K}^2\bar{\vare}^{4}, \qquad \forall t \in [0,T].
\]
We now define $c:= \frac{1}{8 C \und{C}_{\tK}}>0$ 
so that the above bound gives 
\begin{align*}
    \| \zak(t)\|_{s_0}^2 \leq \frac{\tK^2}{4} \bar{\vare}^2\,, 
    \qquad \forall t \in [0, c\vare^{-2}]\,,
\end{align*}
    as claimed in \eqref{improve:boot}. The proof of the second part 
    of the statement follows in the same way, using also \eqref{en:est2}.
    \end{proof}

We now prove the long-time existence  Theorem \ref{thm:main}.
Let us consider $s_0>0$,  
$ \varepsilon_0>0$, $\tK>1$ and $c>0$ given by 
Proposition \ref{prop:boot},
and let $\kappa\in \mathscr{N}$ where $\mathscr{N}$ 
is the set of parameters provided by Lemma
\ref{lem:small_divisors}.
By the assumption \eqref{piccolezzaDati}, local existence theory
for the Water Waves system 
   \eqref{eq:1.2} (see \cite{ABZ2011_2} and \cite{ABZ2014})
guarantees the existence of a time $T_{\mathrm{loc}}>0$ 
and a unique classical solution 
$(\eta,\psi)\in C^{0}([0,T_{\mathrm{loc}}]; 
H_0^{s+\frac{1}{4}}(\T^2;\R)\times\dot{H }^{s-\frac{1}{4}}(\T^2;\R))$  
of \eqref{eq:1.2}, with initial data as in \eqref{piccolezzaDati}, 
such that 
\begin{equation}
\label{prLTE1}
       \sup_{t \in [0,T_{\rm loc}]} \Big(
       \|  \eta(t)\|_{H^{s_0+\sla{1}{4}}_0(\T^2;\R)}
       + \| \psi(t)\|_{\dot H^{s_0-\sla{1}{4}}(\T^2;\R)}\Big)
       \leq C\vare\,.
\end{equation}
Passing to the complex coordinates 
\eqref{zak}, by \eqref{prLTE1} we deduce that $\zak=\mathcal{M}(\eta,\psi)^{T}$
solves the problem \eqref{eq:zak} (see Proposition \eqref{prop:compZak})
and that (see \eqref{Sobnorm})
\[
\begin{aligned}
\|\zak(0)\|_{s_0}&\leq c_s \Big( \|  \eta_0\|_{H^{s_0+\sla{1}{4}}_0(\T^2;\R)}
       + \| \psi_0\|_{\dot H^{s_0-\sla{1}{4}}(\T^2;\R)}\Big)\stackrel{\eqref{piccolezzaDati}}{\leq}
       c_{s}\varepsilon\,,
       \\
    \|\zak(t)\|_{s_0}&\leq c_s
      \Big(
       \|  \eta(t)\|_{H^{s_0+\sla{1}{4}}_0(\T^2;\R)}
       + \| \psi(t)\|_{\dot H^{s_0-\sla{1}{4}}(\T^2;\R)}\Big)
       \stackrel{\eqref{prLTE1}}{\leq} c_{s}C\vare\,,\qquad t\in[0,T_{loc}]\,,
    \end{aligned}
\]
for some $c_s>0$. We now choose $\e_0$ in \eqref{piccolezzaDati} in such a way that,
for $0\leq \vare\leq \e_0$ we have
$c_{s}\vare \leq \bar{\vare}$
 and $c_{s}C\vare\leq \mathtt{K}\bar{\vare}$.

 Moreover we take $s\geq s_0$, therefore the two
 assumption in \eqref{ipobootstrap}
hold
with $\bar{\vare} = \vare \max\{2c_s,c_{s}C\mathtt{K}^{-1} \} $ on the time interval 
$ [0, T_{ loc}] $.
Then Proposition \ref{prop:boot} 
and  a standard bootstrap argument guarantee that 
$ \zak(t)  $ can be extended up to a time 
$
T_\vare:= c\bar{\vare}^{\,-2} 
\, ,
$
consistently with \eqref{timetime}, 
and that
\begin{equation}
\sup_{t\in[0,T_{\vare}]}  \| \zak(t)\|_{s_0}\leq \tK \bar{\vare}\,.
\label{stimafinaletta}
\end{equation}
The latter bound, together with the equivalence \eqref{equiequi},
implies \eqref{stimaFinalemain}. 
Moreover, \eqref{stimafinaletta} and \eqref{equiequi}, combined with the same bootstrap
argument based on \eqref{ipobootstraps}–\eqref{improve:boots},
yields the high-energy estimate \eqref{stimaFinalemain2}. 
This concludes the proof of Theorem \ref{thm:main}.
     
    \appendix
\addtocontents{toc}{\protect\setcounter{tocdepth}{1}}

\section{ The Dirichlet-Neumann Operator}\label{app:DN}
We begin by recalling some fundamental properties of the Dirichlet–Neumann operator in \eqref{eq:112aINTRO} that will be used throughout the paper. Our main focus is on its paradifferential structure.
This aspect has been extensively studied in the literature (see, for instance, \cite{AlDe} and the monograph by Lannes \cite{Lannes} and the references therein).
For our purposes, however, we require a more refined description of $G(\eta)$, which is stated in Proposition \ref{para:DN}.

	Let $\Phi$ the harmonic extension of $\psi$ in the domain 
    \begin{align}
        \Omega= \{ (x,y)\in \T^2\times \R \colon y\leq \eta(x)\},
        \label{Omega_domain}
    \end{align}
    namely
\begin{align}
	\begin{cases} 
		\Delta_{x,y} \Phi(x,y)=0 , & \text{on } y<\eta(x);\\
		\Phi(x, \eta(x))=\psi(x);\\
		\pa_y\Phi(x,y)\to 0 & \text{as } y\to -\infty.
	\end{cases}
    \label{original_elliptic}
	\end{align}
	Then the Dirichlet-Neumann operator 
	is defined as 
	\begin{equation*}
	G(\eta)\psi:= \pa_y \Phi\,(x,\eta(x)) -\nabla \eta(x)\cdot \nabla_x \Phi \, (x, \eta(x)) = \tB-\tV\cdot \nabla \eta\,.
	\end{equation*}
	Note that one has 
\begin{align}
	G(\eta)\psi= (1+|\nabla \eta|^2) \tB- \nabla \eta\cdot \nabla \psi= (1+|\nabla \eta|^2) \pa_y \Phi_{|y=\eta(x)}- \nabla \eta\cdot \nabla \psi.
    \label{eq:112a}
\end{align}
Define the good unknown of Alinhac:
		\be 
		\omega:= \psi - \Opbw{\tB} \eta.
        \label{def:good_unknown}
		\ee
    The main result of this section is the following paralinearization formula.
    \begin{proposition}
    \label{para:DN}
       Let  $ \vr \geq 0$. There are $s_0>1$ and $r>0$ such that for any $ \eta \in B_{s_0}(r)$ one has 
       \begin{align}
           G(\eta)\psi= \Opbw{\lambda(\eta;x,\xi)}\omega
           + \Opbw{-\ii \tV\cdot \xi- \frac12\div(\tV)}\eta
           + R_{\geq 1}(\eta)\psi
           \label{eq_para:DN}
       \end{align}
       where 
       \begin{itemize}
           \item   The symbol $\lambda(\eta;x,\xi) $, recalling \eqref{def:lam1} and \eqref{def:lam0}, expands as 
            \begin{align}
            \label{lambda:DN}
           \lambda(\eta;x,\xi)= \sym{\lambda}{}{1}(\eta;x,\xi)+ \sym{\lambda}{}{0}(\eta;x,\xi)+\sym{\lambda}{}{-1}(\eta;x,\xi),
            \end{align}
           where $\sym{\lambda}{}{-1} \in \Gamma_{\geq 1}^{-1}[r]$, namely there is $\mu>0$ such that 
           \begin{align}
           \lambda^{(-1)}= \lambda^{(-1)}_1+\lambda^{(-1)}_{\geq 2}\,, 
           \qquad 
           \lambda^{(-1)}_1\in \wt \Gamma_1^{-1}\,, 
           \qquad 
               \big| \lambda^{(-1)}_{\geq 2}\big|_{-1,\s-\mu}
               \lesssim \| \eta\|_{\s+\frac14}^2\,.
               \label{esp:lambda-1}
           \end{align}
           Moreover the symbol $ \sym{\lambda}{1}{-1}$ is real valued.
           \item There is $s_0>0$ such that one has  $R_{\geq 1}(\eta)[\bigcdot]\colon B_{s_0}(r)\cap H^s \mapsto\cL(H^s;H^{s+\vr})$ for any $ s\geq s_0$ with estimates 
           \begin{align*}
               \| R_{\geq 1}(\eta)\psi\|_{{s+\vr}}\lesssim_s \| \eta\|_{s_0}\| \psi\|_s+ \| \eta\|_s\| \psi\|_{s_0}.
           \end{align*}
           \item There is $s_0>0$ such that one has the expansion 
           \begin{align}\label{est_R:DN2}
               R_{\geq 1}(\eta)= R_{1}(\eta)+R_{\geq 2}(\eta)\,, 
               \qquad \wt R_{1}(\eta)\in \cR^{-\vr}_1\,,
           \end{align}
        and $R_{\geq 2}(\eta)$ belongs to $\cL(H^s;H^{s+\vr})$ 
        for any $s\geq s_0$ with estimates
           \begin{align}
           \label{est_R:DN}
               \| R_{\geq 2}(\eta)\psi\|_{s+\vr}
               \lesssim_s 
               \| \eta\|_{s_0}^2 \| \psi\|_{s}+ \| \eta\|_{s_0}\|\eta \|_{s}\|\psi\|_{s_0}\,.
           \end{align}
       \end{itemize}
    \end{proposition}
The following lemma is a direct consequence of \cite[Theorem 1.2]{BMV22}. Before turning to the proof of \Cref{para:DN}, we record it here.
 
\begin{lemma}
\label{lem:an_DN}
Let $s_0>3/2$. For any $s\geq s_0+3/2$ there is $r=r(s)>0$
such that the following holds.
Let $\eta\in B_{s_0+3/2}(r)\cap H^{s}$, $\psi\in H^{s}$.
Then the Dirichlet-Neumann $G(\eta)[\psi]$ in \eqref{eq:112aTRIS} and the functions $\tV,\tB$ in 
 \eqref{def:V-B} are analytic in $\eta$, linear in $\psi$ and satisfy
\begin{align}
\|G(\eta)[\psi]\|_{s-1}, \|\tV\|_{s-1}, \|\tB\|_{s-1}
&\lesssim_s \| \psi \|_{s} + \| \eta \|_{s} \| \psi \|_{s_0+3/2}\,,
\label{stimaDNa perdere}
\end{align}
\end{lemma}
 \begin{remark}[Momentum condition]\label{rem:trans}
By \eqref{def:good_unknown} the variable 
$\omega$ can be expressed as 
     $\omega= \omega(\eta)[\psi]$
     where $\omega(\eta)$ satisfies 
     the translation invariant property \eqref{tra:smoo} 
\[
\tau_\upsilon \omega(\eta)= 
\omega(\tau_\upsilon\eta)\tau_\upsilon\,.
\]
Then once one proves that the paradifferential 
operators in \eqref{eq_para:DN} satisfy the 
translation invariant property \eqref{def:sym_mome}, 
the translation invariant property on $R(\eta)$ 
can be deduced a posteriori by difference, 
since $G(\eta)$ is itself translation invariant, i.e. 
$\tau_\upsilon G(\eta)= G(\tau_\upsilon\eta)\tau_\upsilon$. 
For this reason, in the remainder of the section 
we shall refer to smoothing remainders 
in the class $\wt \cR^{-\vr}_1$ just 
verifying the estimate \eqref{bound:smoo}.
 \end{remark}

\subsection{Preliminaries}
In order to study the elliptic problem \eqref{original_elliptic}, 
the first step is to perform the change of variables 
\begin{equation*}
    z  = y-\eta(x)
\end{equation*}
    to reduce the problem to the fixed half-space
\[
\T^2\times \R_{\leq 0}=\{(x,z)\in\T^2\times\R\,:\, z\leq0\}\,.
\]
Then the boundary condition in \eqref{original_elliptic} becomes $\varphi=\psi(x)$ on $\{z=0\}$.
Moreover the interior equation in \eqref{original_elliptic} reads 
\begin{equation} \label{perturba}
\Delta_{x,z} \varphi + \tF(\eta)[\varphi]=0\,, 
\quad 
\tF(\eta)[\varphi]:= |\nabla  \eta|^2  \partial_z ^2 \varphi
 -2 \nabla  \eta \cdot\nabla_x \partial_z \varphi - \partial_z \varphi \Delta\eta\,.
\end{equation}
In order to study the elliptic problem \eqref{perturba} we  need some preliminary results collected in this section.

\subsubsection{Space of Functions on the Strip}\label{sec:funzionistrip}
We introduce 
appropriate spaces (and norms) of functions defined on the strip.
We define for any $s \in \N\cup\{0\}$,
and $a\geq0$ the space
\begin{equation*}
\begin{aligned}
& {\cal H}^{s,a}  := \Big\{ u(y,x)=\sum_{j\in \Z}u_{j}(y)e^{\ii jx}: 
(- \infty, 0] \times \T^2 \to \C :    \| u \|_{{\cal H}^{s,a}} 
 < + \infty \Big\}\,, 
 \\
& \| u \|_{{s,a}}^2 := 
 \sum_{k = 0}^s \| \partial_y^k u \|^2_{L^{2,a}_y H_{x}^{s - k}}
=\sum_{k = 0}^s \int_{- \infty}^0 \| \partial_y^k u(y, \cdot) \|_{s - k}^2 e^{-2ay}\, d y  \,.
\end{aligned}
\end{equation*}
With abuse of notation, for $a=0$, we shall write $\mathcal{H}^{s}$ 
instead of $\mathcal{H}^{s,0}$. We denote also with 
\begin{align*}
    \C\oplus{\cal H}^{s,a}:= 
    \left\{ c+ u\colon c \in \C, \ u \in {\cal H}^{s,a}\right\}\,, 
    \qquad 
    \Pi: \C\oplus {\cal H}^{s,a}\to {\cal H}^{s,a}, \quad \Pi[c+u]=u \,.
\end{align*}
For $\phi= c+u\in \C\oplus {\cal H}^{s,a} $ we define the norm 
$ \| \phi\|_{s,a}:= |c| +\| u\|_{s,a}.$ 
\begin{lemma}\label{algebra striscia}
Let $s_0 > 1$, $s \geq s_0$, $a\geq 0$, $v\equiv v(x,y)$, $u\equiv u(x,y)\in {\cal H}^{s,a}$ 
and $g\equiv g(x) \in H^{s_0}_x$.
Then $vu$ and $g u$ are in $ {\cal H}^{s,a}$  and
\begin{align}
\| v u  \|_{s,a} &\lesssim_s 
\| v \|_{s,0} \| u\|_{s_0,a} + \| v \|_{s_0,0} \| u \|_{s,a}\,,\label{stimatamecalBMV}
\\
\| g u  \|_{s,a} &\lesssim_s \| g\|_{s_0}
\| u \|_{s,a}  + \| g \|_{s} \| u \|_{s_0,a} \,.\label{stimatamecalH}
\end{align}
\end{lemma}

\begin{proof}
First of all, estimate \eqref{stimatamecalBMV} follows from \cite[Lemma 2.4]{BMV22}. 
So we prove \eqref{stimatamecalH}
Using the tame estimates for the product \eqref{tameHsx}, we have 
\begin{align*}
\normLa{\| g u\|_{s}}&\lesssim_s 
\normLa{\| g\|_{s_0}\| u\|_{s}} +
\normLa{\| g\|_{s}\| u\|_{s_0}}\\
&\lesssim_s  \|  g \|_{s_0}\|u\|_{s,a} +
\|  g \|_{s}\|  u\|_{s_0,a}.
\end{align*}
Then we estimate the higher order derivatives  
$ \pa_y^k \big(g u \big)= g(\pa_y^k u)$ by induction on $ k=1, \ldots, s$. 
First we define the step function
\[
\chi(y)= \begin{cases}
1 & y\geq -1 \\ 0 & y \leq -2
\end{cases}\,, 
\qquad \chi \in C^\infty (\R;\R)\,.
\]
Then, using that $\left| \pa_y^j \chi(y)\right|\lesssim 1$, 
$\left| \pa_y^j\left( e^{|y| a} (1-\chi(y))\right)\right|\lesssim e^{|y| a}$ 
and denoting $f:= e^{ya}\left( 1-\chi(y)\right)u$, we get 
\begin{align}
 \normLa{\| g (\pa_y^k u)\|_{s-k}}
 &\leq 
 \normLa{\left\| g(x) \chi(y)  (\pa_y^k u)\right\|_{s-k}}
 +\normLa{\left\| g(x) \big(1-\chi(y)\big)  (\pa_y^k u)\right\|_{s-k}} \notag
 \\
&\lesssim  \| g(x) \chi(y) u\|_{s,a}
+ \left\| \left\| g(x) (\pa_y^k f)\right\|_{s-k}\right\|_{L^2_y}
+ \sum_{j=1}^k\normLa{\left\| g(x)   (\pa_y^{k-j} u)\right\|_{s-k}}\notag
\\
&=: I+II+III\,.\notag
\end{align}
We bound each of the three summands above separately. 
First, 
we apply \eqref{stimatamecalBMV} with $ v \leadsto g(x) \chi(y) \in \cH^{s, a}$, getting 
\begin{align*}
I \lesssim 
\| g(x)\chi(y) \|_{s,0} \| u\|_{s_0,a} 
+ \| g(x)\chi(y)\|_{\fs_0,0} \| u \|_{s,a}
\lesssim_s  \| g\|_{s_0}
\| u \|_{s,a}  + \| g \|_{s} \| u \|_{s_0,a}\,.
\end{align*}
On the other hand, by inductive hypothesis, we get 
\begin{align*}
III 
\lesssim \sum_{j=1}^k\normLa{\left\| g(x)   (\pa_y^{k-j} u)\right\|_{s-(k-j)}}
\lesssim_s \| g\|_{s_0}
\| u \|_{s,a}  + \| g \|_{s} \| u \|_{s_0,a}\,.
\end{align*}
Finally we note that $  f $ extends to the whole 
real line as a function in  $L^2\left(\R;H^s(\T^2;\C)\right)$, 
then by Young's inequality for product we get 
\begin{equation}\label{III.est}
\begin{aligned}
 II&\leq 
\Big(\int_\R \sum_{ \ell \in \Z^2 } | \zeta|^{2k} \langle \ell \rangle^{2(s-k)}
 \Big| \sum_{j\in \Z^2} g_{\ell-j} \hat f_j(\zeta) \Big|^2\, \di \zeta  \Big)^\frac12
 \lesssim_s 
 \| g(x) \pa_y^s f\|_{L^2_{x,y}}+ \left\| \left\| g(x) f \right\|_{s}\right\|_{L^2_y}
 \\
&\lesssim  
\| g \|_{s_0} \| f\|_{s,0} 
+ \| g \|_{s} \| f\|_{s_0,0}
\lesssim_s  \| g \|_{s_0} \| u\|_{s,a} + \| g \|_{s} \| u\|_{s_0,a}\,,
\end{aligned}
\end{equation}
where, in \eqref{III.est}, we used the notation
\[
\hat f_j(\zeta):= 
\frac{1}{(2\pi)^3} \int_{\R}\, e^{- \ii \zeta y}\int_{ \T^2} f(y, x) 
e^{-\ii j\cdot x } \, \di x \,\di y\,.
\]
This concludes the proof of \eqref{stimatamecalH}.
\end{proof}

\begin{lemma}\label{lem:techcal}
One has the following.

\noindent
$(i)$ For $s\in \mathbb{N}$, $a\geq 0$ one has 
\begin{equation}\label{abbiategrasso}
\begin{aligned}
\|\pa_{x_i}u\|_{\mathcal{H}^{s-1,a}}&\lesssim \|u\|_{\mathcal{H}^{s,a}}\,,
\qquad 
\|\pa_{y}u\|_{\mathcal{H}^{s-1,a}}\lesssim \|u\|_{\mathcal{H}^{s,a}}\,,\qquad \forall\,u\in \mathcal{H}^{s,a}\,,
\end{aligned}
\end{equation}
$i=1,2$.

\noindent
$(ii)$ For $s\in \R$ one has
\begin{equation}\label{abbiategrasso2}
\sup_{y\leq0}\|u\|_{H_{x}^{s}}\leq \|u\|_{L_{y}^2H_{x}^{s+\frac{1}{2}}}
+ \|\pa_{y}u\|_{L_{y}^2H_{x}^{s-\frac{1}{2}}}\,.
\end{equation}
\end{lemma}

\begin{proof}
It follows from Lemmata $2.2$ and $2.3$ in \cite{BMV22}.
\end{proof}

\begin{remark}{\bfseries (Trace).}\label{rmk:trace}
By \eqref{abbiategrasso2} we deduce that the trace operator
\[
T(u):=u(\cdot,0):=u_{|y=0}\,,
\]
is, for $s\in \N$, $a\geq 0$ a bounded linear map from ${H}^{1,a}_y H^{s+\frac12}_x$ to $H_{x}^{s}$ and 
\[
\|T(u)\|_{s}\leq \|u\|_{H^{1,a}_y H^{s+\frac12}_x}\,.
\]
\end{remark}

\medskip
\noindent
{\bfseries The Dirichlet-Neumann problem at the flat surface.}
At $\eta=0$ the problem \eqref{original_elliptic} becomes
\begin{equation}\label{elliptic0}
\left\{\begin{aligned}
\Delta_{y,x} \vphi_0&=0\,,\qquad x\in\T^2\,,\;\; -\infty<y<0\,,\\
 \vphi_0(x,0)&=\psi(x)\,,\\
(\pa_{y}  \vphi_0)(x,y)&\to0\,,\;\;\;y\to-\infty
\end{aligned}\right.
\end{equation}
By an explicit computation one can check that the solution of the problem above is 
given by
\begin{equation}\label{def vphi 0}
\vphi_0(x,y)=\mathcal{L}_0[\psi]\,,\quad 
\mathcal{L}_0[\cdot]:=\mathcal{L}_0(y)[\cdot]:=e^{y|D|}[\cdot]\,.
\end{equation}

We have the following result.
\begin{lemma}\label{lemma stima cal L0 sol omogenea laplace}
For any $s \geq 0$, $\psi \in H^{s }_x$ ($s$ such that $s+1/2\in \N$), one has  $\mathcal{L}_0[\psi]\in \C\oplus\cH^{s+\frac12,a}$ with estimates 
\begin{align}
\| {\cal L}_0 [ \psi] \|_{s+\frac{1}{2} ,a } \lesssim \| \psi \|_{s}\,, \quad a\in(0,1)\,.
\label{stima:omognea}
\end{align}
\end{lemma}
\begin{proof}
See Lemma $2.5$ in \cite{BMV22}.
\end{proof}

\subsubsection{Para-differential calculus on the strip}\label{sec:basicpara}
We collect here some results about para-differential calculus for periodic symbols introduced in \Cref{quantizationtotale} and study its properties when they act on functions introduced in \Cref{sec:funzionistrip}.
We follow the notation introduced in \cite{FI2022}, \cite{BMM2021}. 
For a symbol $ a(x,\xi)\in \cN_{s_0}^m$ we define the truncated paradifferential operator
\begin{align*}
    \cOpbw{a(x,\xi)}u:= \sum_{j,k \in \Z^2} \chi\left(\frac{j-k}{j+k}\right) 
    \hat a_{j-k}\left( \frac{j+k}{2}\right) u_k\,  e^{\ii j\cdot x}\,.
\end{align*}

The following is the classical theorem that explains how paradifferential operators act on the scale of Sobolev spaces.

\begin{lemma}{\bfseries (Action on Sobolev spaces).}\label{azioneSimboo}
Let ${s}_0>2$, $ m \in \R$ and $a\geq 0$.
 Consider the linear map 
\begin{align*}
    \cN_{s_0}^m\to \cL(H^{s}; H^{s-m}), \quad  \ta \mapsto \Opbw{\ta}\,, 
\end{align*}
which is continuous thanks to \eqref{actionSob}.

For any $s\in \R$, one has:
\begin{enumerate}[(i)]
    \item  $\Opbw{\ta}$ extends to an operator in 
    $\cL\left(L^{2,a}_y H^{s-m}_x;L^{2,a}_y H^{s-m}_x\right)$ 
    with estimates
    \begin{align}
        \| \cOpbw{\ta}u\|_{L^{2,a}_y H^{s-m}_x} 
        \lesssim_{s}   
        |\ta|_{m, s_0} \| u\|_{L^{2,a}_y H^{s}_x}\,, 
        \qquad \forall u \in L^{2,a}_y H^{s}_x\,.
        \label{action:striscia1}
    \end{align}
    \item If $\ta\equiv \ta(x,y)$ belongs to $ \cH^{s_0,a}$ 
    then $\cOpbw{\ta}$ belongs to $\cL\left(H^{s}; L^{2,a}_y H^{s}_x\right)$ 
    with estimates
    \begin{align}
         \| \cOpbw{\ta}h\|_{L^{2,a}_y H^{s}_x} 
         \lesssim_{s}   
         \|\ta\|_{L^{2,a}_y H^{s_0}_x} \|h\|_{s}\,, 
         \qquad \forall h \in H^{s}.
         \label{action:striscia2}
    \end{align}
\end{enumerate}
 \end{lemma}
 \begin{proof}
  We first prove  \eqref{action:striscia1}. By \eqref{actionSob} one has 
 \begin{align*}
      \left\| e^{|y| a} \|  \cOpbw{\ta}u\|_{s-m} \right\|_{L^2_y}
      \lesssim_s 
      | \ta |_{m,s_0} \left\| e^{|y| a} \|  u\|_{s} \right\|_{L^2_y}
      = | \ta |_{m,s_0} \| u\|_{L^{2,a}_y H^{s}_x}\,.
 \end{align*}
 In the same way, using again \eqref{actionSob}, we have 
 \begin{align*}
      \left\| e^{|y| a} \|  \cOpbw{\ta}h \|_{s-m} \right\|_{L^2_y}
      \lesssim_s 
      \left\| e^{|y| a} \|  \ta\|_{s_0}\| h\|_s \right\|_{L^2_y}
      = \| \ta\|_{L^{2,a}_y H^{s_0}_x}\| h\|_s\,,
 \end{align*}
 which implies \eqref{action:striscia2}.
 \end{proof}
We have the following result.
 \begin{lemma}{\bfseries (Products).}\label{lem:prodotto}
Let $ s_0>1$ and $\ta\in H^{s}$, $\tb\in H^{r}$ with $s+r\geq0$.
Then
\begin{align}
    \ta\cdot \tb=\cOpbw{\ta}\tb+\cOpbw{\tb}\ta+\mathcal{R}(\ta,\tb)\,,
    \label{alg:paraprod}
\end{align}
where the bilinear operator $\mathcal{R}: H^{s}\times H^{r}\to H^{s+r-s_0}$ 
is symmetric and satisfies the estimate
\begin{equation}\label{stimarestoparaproduct}
\|\mathcal{R}(\ta,\tb)\|_{s+r-s_0}\lesssim_{s,r}\|\ta\|_{s}\|\tb\|_{r}\,.
\end{equation}
Moreover $\mathcal{R}(\ta,\tb)=\mathcal{R}(\Pi_{0}^{\perp}\ta,\Pi_0^{\perp}\tb)-\ta_0\tb_0$ 
and 
\[
\|\Pi_0^{\perp}\mathcal{R}(\ta,\tb)\|_{s+r-s_0}
\lesssim_{s,r}\|\Pi_0^{\perp}\ta\|_{s}\|\Pi_0^{\perp}\tb\|_{r}\,.
\]
Moreover, for any $a \geq 0$, the bilinear operator extends as 
$\cR: L^{2,a}_y H^{s}_x\times H^{r}\to L^{2,a}_y H^{s+r-s_0,a}_x$ with estimates 
\begin{equation}\label{stimarestoparaproduct:striscia}
\|\mathcal{R}(u,\tb)\|_{L^{2,a}_y H^{s+r-s_0,a}_x}
\lesssim_{s,r}
\|u\|_{L^{2,a}_y H^{s,a}_x}\|\tb\|_{r}\,
\qquad \forall u \in \cH^{s,a},\, \tb \in H^{r}\,.
\end{equation}
\end{lemma}
\begin{proof}
In \cite[Lemma 2.7]{BMM2021} it is proven that 
\begin{equation*}
a\cdot b=\Opbw{a}b+\Opbw{b}a+\breve{R}(a,b)\,,
\end{equation*}
where $\breve{R}(a,b)$ is a remainder that satisfies \eqref{stimarestoparaproduct} 
and 
$\breve{R}(a,b)=\breve{R}(\Pi_{0}^{\perp}a,\Pi_0^{\perp}b)
-a_0b_0 $.  
Then \eqref{alg:paraprod} follows from defining 
defining $\cR(a,b):= \breve{R}(a,b)+2 a_0 b_0$. 
Finally the proof of \eqref{stimarestoparaproduct:striscia} 
follows the same lines as the proof of \eqref{action:striscia1} 
and \eqref{action:striscia2} 
using \eqref{stimarestoparaproduct} instead of \eqref{actionSob}.
\end{proof}

The following lemma is crucial.

\begin{lemma}{\bfseries (Composition on the strip).}\label{compoparapara}
Let $\vr\geq 0$ and $ a\geq 0$. There is $s_0 > 0$ large enough such that for any
$\ta \in H^{s_0+\vr}$ and $ u \in L^{2,a}_y H^{s_0+\vr}_x $, one has:
\begin{enumerate}[(i)]
    \item The operator $\cQ(\ta,u)$ defined as in \eqref{def:Q} 
    extends as an operator in $
\mathcal{L}\big(H^{s}_x; L^{2,a}_y H^{s+\vr}_x \big)$, 
with estimates 
\begin{align}
    \|\mathcal{Q}(\ta,u)[h]\|_{L^{2,a}_y H^{s+\vr}_x}\lesssim_{s,\vr}
\|\ta\|_{s_0+\vr}\|\tb\|_{L^{2,a}_y H^{s_0+\vr}_x}
\|h\|_{s}\,,
\qquad 
\forall\, h\in H^{s}(\T^2;\C)\,.
\label{calma2:striscia1}
\end{align}
 \item The operator $\cQ(\ta,\tb)$ defined as in \eqref{def:Q} extends as an operator in $
\mathcal{L}\big( L^{2,a}_y H^{s}_x; L^{2,a}_y H^{s+\vr}_x \big)$, with estimates 
\begin{align}
    \|\mathcal{Q}(\ta,\tb)[\phi]\|_{L^{2,a}_y H^{s+\vr}_x}\lesssim_{s,\vr}
\|\ta\|_{s_0+\vr}\|\tb\|_{s_0+\vr}
\|\phi\|_{ L^{2,a}_y H^{s}_x}\,,
\qquad 
\forall\, \phi\in L^{2,a}_y H^{s}_x\,.
\label{calma2:striscia2}
\end{align}
\end{enumerate}
\end{lemma}
\begin{proof}
    In order to prove \eqref{calma2:striscia1} and \eqref{calma2:striscia2} one can follows the same line of the proof of \eqref{action:striscia1} and \eqref{action:striscia2} using estimate \eqref{calma1} instead of \eqref{actionSob}.
\end{proof}

\begin{remark}
    \label{rem:strip_symbols}
Similarly to \Cref{rem:symbols_modozero}, the action of a paradifferential operator $\Opbw{a(x,\xi)}$ on $L^{2,a}_y H^s_x$ depends only on the value $\hat a(0,0)$ and  $a(x,\xi)$ for $|\xi| \geq \tfrac{1}{2}$. Accordingly, we identify symbols up to this equivalence.
\end{remark}
\subsection{Tame estimates for the Laplace equation}\label{sec:tameLaplace}
Consider the new vertical variable $y \in \R_{\leq0}:=(- \infty, 0]$
\begin{equation}\label{change1}
z=\eta(x)+y\,,\quad -\infty< y\leq 0\,,
\end{equation}
so that the closure of $\Omega$ in \eqref{Omega_domain} becomes
\begin{equation*}
\mathcal{S}_1=\{(x,y)\in \T^2 \times \R\; :\; -\infty< y\leq 0\}\,.
\end{equation*}
Consider the function
\begin{equation}\label{func:phi}
\varphi(x,y):=\Phi(x,y+\eta(x))\,.
\end{equation}
We have the following result.
\begin{lemma}
If $\Phi(x,z)$ solves the problem \eqref{original_elliptic} then the function
$\varphi(x,y)$ in
\eqref{func:phi} solves 
\begin{equation}
\left\{\begin{aligned}
\mathcal{L}\varphi&=0\,,\qquad (x,y)\in\cS_1\,,\\
\varphi(x,0)&=\psi(x)\,,\\
(\pa_{y}\varphi)(x,y)&\to0\,,\;\;\;y\to-\infty
\end{aligned}\right.
\label{elliptic2}
\end{equation}
where
\begin{align}\label{op:L}
{\mathcal{L}}:={\mathcal{L}}(\eta):=\pa_{yy}+\tb\Delta_x-2\tb\nabla\eta\cdot\nabla_x\pa_{y}
-\tb\Delta\eta\pa_{y}\,,
\end{align}
with
\begin{equation}
\tb\equiv \tb(\eta;x):=\frac{1}{1+|\nabla\eta|^2}\,.
 \label{formaalphai}
\end{equation}
We define also the coefficients 
\begin{align*}
    \beta_1:= \tb-1= - \tb |\nabla \eta|^2 \,, 
    \qquad 
    \beta_2:= -2\tb \nabla \eta, \qquad \beta_3:= -\tb \Delta \eta\,.
\end{align*}
\noindent
Let $s_0>1$. 
For any $s\geq s_0+2$ there is $r:=r(s)\ll1$
such that, for any $ \eta \in B_{\mathfrak{s}_0+2}(r)\cap H^{s}$, one has the bounds
\begin{equation}\label{stimealphai}
\begin{aligned}
\| \beta_1(\eta)\|_{s-1}
\lesssim_s 
\|\eta\|_{s_0+1}\|\eta\|_{s}\,, 
\qquad\|\beta_{2}(\eta)\|_{s-1}
+\|\beta_{3}(\eta)\|_{s-2}
&\lesssim_{s}\|\eta\|_{H^{s}}\,.
\end{aligned}
\end{equation}

\end{lemma}

\begin{proof}
By \eqref{func:phi} and \eqref{change1} we shall write
\[
\Phi(x,z)=\varphi(x,z-\eta(x))\,.
\]
Therefore at $y=z-\eta(x)$ we obtain 
\begin{align*}
\pa_{z}\Phi(x,z)&=(\pa_{y}\varphi)(x,y)\,,
\qquad \pa_{zz}\Phi(x,z)=(\pa_{yy}\varphi)(x,y)\,,
\\
\nabla_x\Phi(x,z)&=(\nabla_x\varphi)(x,y)-\nabla\eta(\pa_{y}\varphi)(x,y)\,,
\\
\Delta_x\Phi(x,z)&=(\Delta_x\varphi)(x,y)-2\nabla\eta\cdot(\pa_{y}\nabla_x\varphi)(x,y)
-\Delta\eta(\pa_{y}\phi)(x,y)+|\nabla\eta|^{2}(\pa_{yy}\phi)(x,y)\,.
\end{align*}
Moreover
\[
\varphi(x,0)=\Phi(x,\eta)=\psi\,,\qquad
(\pa_{y}\varphi)(x,y)\to 0\;\;{\mathrm{as}}\;\;y\to-\infty\,.
\]
Hence by  \eqref{original_elliptic} we deduce
\[
\begin{aligned}
0&=(\pa_{zz}+\Delta)\Phi
=(1+|\nabla\eta|^{2})(\pa_{yy}\phi)+\Delta\phi
-2\nabla\eta\cdot(\pa_{y}\nabla\phi)
-\Delta\eta(\pa_{y}\phi)\,,
\end{aligned}
\]
which implies \eqref{elliptic2}-\eqref{op:L}-\eqref{formaalphai}.
The estimates \eqref{stimealphai} follow by direct inspection using the explicit definitions in 
\eqref{formaalphai} and the tame estimate \eqref{tameHsx}.
\end{proof}

\begin{remark}
Notice that formula \eqref{eq:112a} 
becomes
\begin{equation}\label{eq:112aTRIS}
 G(\eta)\psi = \Big[
(1+|\nabla\eta|^2)(\pa_{y}\varphi)(x,y)-\nabla\eta\cdot(\nabla_x\varphi)(x,y)\Big]_{|y=0}\,.
\end{equation}
Moreover, formulas\,  \eqref{def:V-B}
become (recall \eqref{func:phi})
    \begin{align} 
\label{def:VBIS}
& \tV =    (\nabla_x\Phi(x,z))_{|z=\eta(x)}  = (\nabla_x\varphi(x,y))_{|y=0}-\nabla\eta(x)(\pa_{y}\varphi(x,y))_{|y=0} \,,
\\
\label{form-of-BBIS}
& \tB = (\pa_z \Phi(x,z))_{|z=\eta(x)} = (\pa_{y}\varphi(x,y))_{|y=0}\,.
\end{align}
\end{remark}

\vspace{0.5em}
\noindent
{\bfseries Perturbative setting.}
The solution $\varphi(x,y)$
of the problem \eqref{elliptic2}  can be written as 
\begin{equation}\label{solphiphi}
\varphi (x, y) = \vphi_0(x, y) + u(x, y)
\end{equation}
where $\vphi_0(x,y)$ solves \eqref{elliptic0}
while 
$u(x, y)$ solves the problem
\begin{equation}\label{prob omogeneop forzato}
\begin{cases}
\partial_{yy} u +  \Delta_x u + F(\eta)[u]= f \\
u( x,0) = 0 \\
\partial_y u( x,y) \to  0\,,\;\;\;y\to-\infty
\end{cases}
\end{equation}
where we defined the differential operator  (recall \eqref{op:L})
\begin{equation*}
\begin{aligned}
F(\eta):={F}(\eta;y)[\cdot ] 
&:= 
-\Big( 2b\nabla\eta\cdot \nabla_x +b\Delta\eta\Big) \partial_y   +(b-1)\Delta_x  
\\&=
\beta_{1}(\eta)\cdot\nabla_x\pa_{y}
+\beta_{2}(\eta)\pa_{y}+\beta_3(\eta)\Delta_x\,,
\end{aligned}
\end{equation*}
and the ``error''
\begin{equation*}
\begin{aligned}
& f := - {F}(\eta)[\vphi_0] =-{\mathtt{L}}[\psi]\,,
\end{aligned}
\end{equation*}
where recalling \eqref{def vphi 0} we defined
\begin{equation*}
{\mathtt{L}}:={\mathtt{L}}(\eta,y):={F}(\eta;y)\circ\mathcal{L}_0(y)[\cdot]\,.
\end{equation*}

\begin{proposition}\label{stima cal Kn equazione di laplace}
Let $N\in \N_0$ and $a\geq 0$.
There is $s_0> 0$ such that for any $s\geq s_0+3/2$ there are $r:=r(s)>0$ and a map 
\[
\varphi_{\geq 1}(\eta;y)[\cdot]: B_{s_0+3/2}(r)\cap H^{s+\frac12}
\to 
\mathcal{L}(H^{s}, \C\oplus\mathcal{H}^{s+1/2,a})
\]
such that $u:=\varphi_{\geq 1}(\eta;y)[\psi]$,  
is the unique solution of \eqref{prob omogeneop forzato}.
Moreover one has the estimate
\begin{align}
\|  \Pi \varphi_{\geq 1}(\eta;y)[\psi] \|_{s+1/2,a} 
&\lesssim_s  
 \| \eta\|_{s_0  + 3/2}
\| \psi \|_{s} + \| \eta \|_{s+\frac12} \| \psi \|_{s_0+3/2}\,.\label{bottle10}
\end{align}
As a consequence, the function (see \eqref{solphiphi}) 
\begin{align}
\varphi(x,y)=e^{y|D|}\psi (x)+\varphi_{\geq 1}(\eta;y)[\psi]
\label{def:varphi}
\end{align}
is the unique solution of 
\eqref{elliptic2} and satisfies, for any $s\geq s_0+3/2$, $m=0,1,2,3$,
\begin{equation}\label{bottle110}
\|  \pa_{x_i ,y}^m\varphi\|_{s+1/2-m, a} 
\lesssim_s  
\| \psi \|_{s} + \| \eta \|_{s+\frac12} \| \psi \|_{s_0+3/2}\,.
\end{equation}
\end{proposition}
\begin{proof}
The proof follows the line of \cite{BMV22} and is divided in three steps. 

\noindent{\bfseries Step 1:} 
Let $s_0=4$ and $ r=1$. If $ \eta \in H^{s+\frac12}(\T^2)\cap B_{s_0}(r)$ then 
\begin{equation}\label{step1}
F(\eta): \C \oplus \cH^{s+\frac12,a}\mapsto \cH^{s-\frac32,a}\,,
\quad 
\| F(\eta)[ \phi]\|_{s-\frac32,a}
\lesssim_{s} 
\| \eta\|_{s_0}\| \phi \|_{s+\frac12 ,a}+\| \eta\|_{s+{\frac12}}\| \phi \|_{s_0,a}\,.
\end{equation}
To prove \eqref{step1}, we estimate the three terms in \eqref{perturba} separately. 

\noindent
First, by \eqref{stimatamecalH}, 
tame estimates for the product and \eqref{abbiategrasso}, we have 
\begin{align}
\| |\nabla \eta|^2 \pa_z^2 \phi\|_{s-\frac32,a}
&\lesssim_s  \|  \eta \|_{H^{3}_x}^2\| \pa_z^2 \phi\|_{s-\frac32, a} +
\|  \eta \|_{H^{s-\frac12}_x}\| \eta\|_{H^3_x}\| \pa_z^2 \phi\|_{2, a}\notag 
\\
&\lesssim_s 
\|  \eta \|_{H^{3}_x}\| \phi\|_{s+\frac12, a} +
\|  \eta \|_{H^{s-\frac12}_x}\| \phi\|_{4, a}\,,\label{step1_1}
\\
\|  \Delta \eta \pa_z \phi\|_{s-\frac32,a}
&\lesssim_s \|  \eta \|_{H^{4}_x}\| \pa_z \phi\|_{s-\frac32, a} +
\|  \eta \|_{H^{s+\frac12}_x}\| \pa_z \phi\|_{2, a}\notag
\\&
\lesssim_{s}  \|  \eta \|_{H^{4}_x}\|  \phi\|_{{s-\frac12},a} +
\|  \eta \|_{H^{s+\frac12}_x}\|  \phi\|_{{3},a}\,, \label{step1_2}
\\
\| \nabla  \eta \cdot\nabla_x \partial_z \phi\|_{s-\frac32,a}
&\lesssim_s  
\|  \eta \|_{H^{3}_x}\|  \nabla_x \partial_z\|_{{s-\frac32},a} +
\|  \eta \|_{H^{s-\frac12}_x}\| \nabla_x \partial_z\phi\|_{{2},a}\notag 
\\
&\lesssim_s   \|  \eta \|_{H^{3}_x}\|  \phi\|_{{s+\frac12},a} +
\|  \eta \|_{H^{s-\frac12}_x}\| \phi\|_{{4},a}\,.\label{step1_3}
\end{align}
Gathering \cref{step1_1,step1_2,step1_3}, 
in view of \eqref{perturba}, we get \eqref{step1}.

\noindent{\bfseries Step 2:} 
For $a\in (0,1)$ and  any $g=g(x,z) \in \cH^{s-\frac32,a}$, the elliptic problem 
\[
\begin{cases} 
\Delta_{x,z} u(x,z) = g (x,z)\\ 
u(x,0)=0= \lim_{z\to -\infty} \pa_z u 
\end{cases}
\]
has a unique solution $u=L(g) \in \C \oplus \cH^{s+\frac12,a}$ 
with estimate 
\begin{align}
L: \cH^{s-\frac32,a}\mapsto \C\oplus\cH^{s+\frac12,a}\,, 
\qquad 
\| L(g)\|_{s+\frac12,a}\lesssim \|g\|_{s-\frac32,a}\,.
\label{Linv}
\end{align}
This is  \cite[Lemma 2.10]{BMV22}. 
Then, using also $ F(\eta) = F(\eta) \Pi $, we deduce that  $\varphi_{\geq 1}$ solves 
    \begin{equation*}
\begin{cases}
\Big(\uno + L\circ F(\eta)\Pi \Big)[ \varphi_{\geq 1}]
= -L\circ  F(\eta)[ e^{z|D_x|}\psi(x)]
\\
\varphi_{\geq 1}( \eta;x,0) =0
 \\ 
\lim_{z\to -\infty} \pa_z\varphi_{\geq 1}( \eta;x,z)=0\,.
\end{cases}
    \end{equation*}
By defining $ v = \Pi \varphi_{\geq 1} $, we have
\[
\Big(\uno + P(\eta) \Big)[ v]= -P(\eta) [ e^{z|D_x|}\psi(x)]\,,
\]
where, combining \eqref{Linv} and \eqref{step1}, 
one has
\begin{align}
P(\eta)&:=\Pi\circ L\circ F(\eta): \cH^{s+\frac12,a} \mapsto  \cH^{s+\frac12,a}\,,\notag
 \\
 \| P(\eta)[\phi]\|_{s+\frac12,a}&\lesssim_{s} \| \eta\|_{4}\| \Pi\phi \|_{s+\frac12 ,a}
 +\| \eta\|_{s+{\frac12}}\| \Pi\phi \|_{4,a}\, ,
 \label{altona}
 \\ 
\| P(\eta)^j[\phi] \|_{ 4, a} &\leq \big(C'(s_0) \| \eta \|_{4}\big)^j \|\Pi\phi \|_{ 4, a}\,.  \label{indbassa}
 \end{align}

 \begin{remark}
 The variable $\varphi_{\geq 1} $ is recovered by $ v$ solving 
\[
\varphi_{\geq 1}= -L\circ F(\eta)[  [ e^{z|D_x|}\psi(x)]+v]\,.
\]
From the view point of the estimate 
we only need to bound  $v=\Pi \varphi_{\geq 1}$.
\end{remark}

 \noindent{\bfseries Step 3:} 
 It follows directly from \eqref{altona} and \eqref{indbassa}  and an induction argument, 
 that for any $j\in \N$, one has the estimate
 \begin{equation}\label{indalta}
\| P(\eta)^j[v] \|_{ s+\frac12, a} 
\leq 
C(s)^j \|\eta \|_{4}^{j-1}  \big(  \| \eta \|_{4}\|v \|_{ s+\frac12, a}\!
+\!j\| \eta \|_{{s+{\frac12}}}\|v \|_{4, a} \big) \, ,
\end{equation}
where $C(s)>0$ is the maximum between $C'(s_0)$ in \eqref{indbassa} 
and the constant in estimate \eqref{altona}. 
Enlarging $C(s)$, if needed, so that $ C(s)\geq 2$, 
from \eqref{indalta} one deduce also
\begin{align*}
\| P(\eta)^j[v] \|_{ s+\frac12, a} 
\leq 
C(s)^{2j} \|\eta \|_{4}^{j-1}  \big(  \| \eta \|_{4}\|v \|_{ s+\frac12, a}\!
+\| \eta \|_{{s+{\frac12}}}\|v \|_{4, a} \big) \, .
\end{align*}
Therefore, if $0<r<\frac{1}{2 C(s)^2}$, 
for any  $ \eta \in B_r(H^{4})\cap H^{s+\frac12}$, the operator
 $ \uno + P(\eta)$ is invertible and 
\begin{align}\label{invfin}
\|\big(\uno + P(\eta)\big)^{-1}P(\eta)[\phi]\|_{s+\frac12,a}
\leq 
2 C(s)^2\left(\| \eta \|_{4}\|\Pi \phi \|_{ s+\frac12, a}
 +\| \eta \|_{{s+{\frac12}}}\|\Pi \phi \|_{4, a} \right)\,.
\end{align}
Finally we set 
$ v =\Pi \varphi_{\geq 1}= \big(\uno + P(\eta)\big)^{-1}P(\eta)[e^{y |D_x|}\psi] $. 
By \eqref{invfin} with $\phi= e^{y |D_x|}\psi \in \C\oplus \cH^{s+\frac12,a}$ 
and by \eqref{stima:omognea}, we get
 \begin{align*}
\| v \|_{s+\frac12,a}
\lesssim   
\| \eta \|_{4}\|\Pi e^{y |D_x|}\psi \|_{ s+\frac12, a}
+\| \eta \|_{{s+{\frac12}}}\|\Pi e^{y |D_x|}\psi\|_{4, a} 
\lesssim 
\| \eta \|_{4}\|\psi \|_{ s}+\| \eta \|_{{s+{\frac12}}}\|\psi\|_{4, a}\,.
\end{align*}
This concludes the proof.
\end{proof}

\begin{remark}\label{rmk:phielliptic}
Notice that, in view of Proposition 
\ref{stima cal Kn equazione di laplace}, 
we shall write
$\varphi(x,y)=\varphi(\eta;y)[\psi]$
with 
$\varphi(\eta;y)[\psi]:=e^{y|D|}\psi +\varphi_{\geq 1}(\eta;y)[\psi]$. 
\end{remark}

\subsection{Para-differential structure of the Dirichlet-Neumann operator}
In this section we shall also show that actually the 
Dirichlet-Neumann operator \eqref{eq:112a} 
(see also \eqref{eq:112aTRIS})
has a para-differential structure.

\paragraph{Good unknown}
We introduce the variable
\begin{equation}\label{good1}
w=\varphi-\cOpbw{\pa_{y}\varphi}[\eta]\,,\qquad y\leq 0\,.
\end{equation}
\begin{remark}{\bfseries (Good unknown of Alinhac).}\label{rmk:goodali}
Recalling \eqref{def:V-B} and \eqref{func:phi}
we note that the good unknown in \eqref{def:good} 
correspond to the trace on $\{y=0\}$ of $w$.
Indeed
\[
\begin{aligned}
w(x,0)&=\varphi(x,0)-\cOpbw{(\pa_{y}\varphi)(x,0)}[\eta]
\\&
\stackrel{\eqref{func:phi}}{=}\Phi(x,\eta(x))-\cOpbw{(\pa_{z}\Phi)(x,\eta(x))}[\eta]
{=}\psi-\cOpbw{B(\eta,\psi)}[\eta]
\stackrel{\eqref{def:good}}{=}\omega\,.
\end{aligned}
\]
\end{remark}
\noindent
Moreover, in view of \Cref{rmk:phielliptic} 
and \eqref{good1} one has 
\begin{align*}
w= e^{y|D|}\psi+ w_{\geq 1}(\eta;y)[\psi]\,, 
\qquad 
w_{\geq 1}(\eta;y)[\psi]:= 
-\cOpbw{\pa_{y}\varphi(\eta;y)[\psi]}[\eta]+ \varphi_{\geq 1}(\eta;y)[\psi]\,.
\end{align*}
In particular there is $\sigma_0>0$ such that for any $ s\geq \sigma_0$ there is $r>0$ such that  
\begin{align*}
B_{s_0+\sigma}(r)\cap H^{s+\frac12} \ni \eta 
\mapsto 
w_{\geq 1}(\eta;y)[\psi] \in \mathcal{L}(H^s_x, \C\oplus H^{1,a}_{y}H_x^{s-\frac12})\,.
\end{align*}
Moreover, by
using \eqref{bottle10}, \eqref{bottle110} and \eqref{action:striscia2}, 
one has the estimates
\begin{align}\label{stima:w}
\| w_{\geq 1}(\eta;y)[\psi]\|_{H^{1,a}_y H^{s-\frac12}_x}
\lesssim 
\| \eta\|_{s_0+\frac32}\| \psi\|_s+ \| \eta\|_{s+\frac12}\|\psi\|_{s_0+\frac32}\,. 
\end{align}
We prove the following important preliminary result.

\begin{proposition}\label{prop:eq:112aquatuor}
Let $\vr\geq 0$, $s_0>3/2$ and $\s:=2+3/2+\vr$. 
For any $s\geq s_0+\s$ 
there is $\delta=\delta(s)$ such that for any
$\eta\in B_{\delta}(H^{s_0+\s})\cap H^{s}$,  $\psi\in H^{s}$ the following holds.
There exists a map
\[
\begin{aligned}
\mathcal{S}_{\geq 1}: B_{\delta}(H^{s_0+\s})\cap H^{s}
&\longrightarrow
\mathcal{L}(H^{s}_x, H^{1,a}_y H^{s+\vr}_x)
\\
\qquad \eta\;\;\;&\mapsto\;\; \mathcal{S}_{\geq 1}(\eta;y)[\cdot]
\end{aligned}
\]
such that 
 the Dirichlet-Neumann operator in 
 \eqref{eq:112a}-\eqref{eq:112aTRIS}
can be written as
\begin{equation}\label{eq:112aquatuor}
\begin{aligned}
G(\eta)\psi  =& \Big[
\cOpbw{1+|\nabla\eta|^2}[\pa_{y}w]-\cOpbw{\nabla\eta}[\nabla w]
-\cOpbw{\nabla\eta-\nabla\eta\pa_{y}\varphi}[\nabla\eta]\Big]_{|y=0}
\\
&+
\Big[-\cOpbw{\div(\nabla\varphi-\nabla\eta\pa_{y}\varphi)}[\eta]
+\mathcal{S}_{\geq 1}(\eta;y)[\psi]\Big]_{|y=0}\,,
\end{aligned}
\end{equation}
where $w$ is in \eqref{good1}, $\phi$ solve the problem \eqref{elliptic2}. 
Moreover one has the expansion
\begin{align}\label{esp:sgeq1}
\mathcal{S}_{\geq 1}(\eta;y)[\cdot]= 
\mathcal{S}_1(\eta;y)[\cdot]+\mathcal{S}_{\geq 2}(\eta;y)[\cdot]\,, 
\end{align}
where 
$\mathcal{S}_1(\eta;y)[\cdot]$ is linear with respect to $\eta$ and it 
satisfies  the estimate
\begin{equation}
\| \mathcal{S}_1(\eta;y)[\psi]\|_{H^{1,a}_y H_{x}^{s+\vr}}
\lesssim_{s}
\|\psi\|_{s}
\|\eta\|_{s_0+\s}+
\|\eta\|_{s}\|\psi\|_{s_0+\s}\,,
\label{leggelegge1}
\end{equation}
while $\mathcal{S}_{\geq 2}(\eta;y)[\cdot]$ 
is non-homogeneous and it satisfies 
\begin{align}
\| \mathcal{S}_{\geq 2}(\eta;y)[\psi]\|_{H^{1,a}_y H_{x}^{s+\vr}}
\lesssim_{s}
\|\psi\|_{s}
\|\eta\|_{\fs_0+\s}^2+
\|\eta\|_{s}\| \eta\|_{s_0+\s}\|\psi\|_{s_0+\s}\,.
    \label{leggelegge2}
\end{align}
In particular one has 
\begin{align}
G(\eta)[\psi]
&=\cOpbw{1+|\nabla\eta|^2}[(\pa_{y}w)(x,0)]
-\cOpbw{\nabla\eta}[\nabla \omega]-\cOpbw{\tV}[\nabla\eta]  
\notag
\\
&\quad-\cOpbw{\div(\tV)}[\eta]
+\mathcal{S}_{\geq 1}(\eta;0)[\psi]\,,
\label{eq:112aquatuorBIS}
\end{align}
where $\tV,\, \tB$ are in \eqref{def:V-B} 
and $\omega$ is in \eqref{def:good}.
\end{proposition}

\begin{proof}
Consider formula \eqref{eq:112aTRIS} and
define 
\begin{equation}\label{accendino2}
\mathcal{Z}(y):=(1+|\nabla\eta|^2)(\pa_{y}\varphi)(x,y)
-\nabla\eta\cdot(\nabla\varphi)(x,y)\,,\qquad y\leq 0\,.
\end{equation}
By applying Lemmata \ref{lem:prodotto}  and \ref{lem:prodotto} one gets
\begin{equation}
\begin{aligned}
\mathcal{Z}(y)&=\cOpbw{1+|\nabla\eta|^2}[\pa_{y}\varphi]
+2\cOpbw{\pa_{y}\varphi\nabla\eta}[\nabla\eta]
-\cOpbw{\nabla\eta}[\nabla\varphi]-\cOpbw{\nabla\varphi}[\nabla\eta]
\\
&\quad+Q_1+Q_2+Q_3+Q_4
\end{aligned}
\label{Zparalinearizzato}
\end{equation}
with
\begin{equation}\label{resti1234}
\begin{aligned}
Q_1&:=-\mathcal{R}(\nabla\eta,\nabla\varphi)\,,
\qquad 
Q_2:=\mathcal{R}(|\nabla\eta|^2,\pa_{y}\varphi)\,,
\qquad 
Q_{3}:=\cOpbw{\pa_{y}\varphi}\circ \mathcal{R}(\nabla\eta,\nabla\eta)\,,
\\
Q_{4}&:=2\mathcal{Q}(\pa_{y}\varphi,\nabla\eta)[\nabla\eta]\,,
\end{aligned}
\end{equation}
where $\mathcal{R}(\cdot,\cdot)$ denotes the remainders coming 
form para-product Lemma \ref{lem:prodotto},
while 
$\mathcal{Q}(\cdot,\cdot)$ denotes the remainders 
coming form the composition
Lemma \ref{compoparapara}.
Note that, by \eqref{good1} we have
\[
\nabla\varphi=\nabla w+\cOpbw{\nabla\pa_{y}\varphi}[\eta]
+\cOpbw{\pa_{y}\varphi}[\nabla\eta]\,,
\qquad
\pa_{y}\varphi=\pa_{y}w+\cOpbw{\pa_{yy}\varphi}[\eta]\,.
\]
By inserting the latter formulas\, in \eqref{Zparalinearizzato} 
and using again the composition
Lemma \ref{compoparapara}, we obtain
\begin{equation}\label{accendino1}
\begin{aligned}
\mathcal{Z}(y)&
=\cOpbw{1+|\nabla\eta|^2}[\pa_{y}w]-\cOpbw{\nabla\eta}[\nabla w]
\\&
-\cOpbw{\nabla\varphi-\pa_{y}\varphi\nabla\eta}[\nabla\eta]
+\cOpbw{(1+|\nabla\eta|^2)\pa_{yy}\varphi-\nabla\eta\cdot\nabla\pa_{y}\varphi}[\eta]
\\&
+Q_1+Q_2+Q_3+Q_4+Q_5+Q_6+Q_7
\end{aligned}
\end{equation}
where
\begin{equation}\label{resti567}
Q_{5}:=\mathcal{Q}(|\nabla\eta|^2,\pa_{yy}\varphi)[\eta]\,,
\qquad
Q_{6}:=\mathcal{Q}(\nabla\eta,\nabla\pa_{y}\varphi)[\eta]\,,
\qquad
Q_{7}:=-\mathcal{Q}(\nabla\eta,\pa_{y}\varphi)[\nabla\eta]\,.
\end{equation}

\noindent
Now, recalling that $\varphi$ solves the problem \eqref{elliptic2} 
(see also \eqref{op:L}-\eqref{formaalphai}),
we note that 
\[
(1+|\nabla\eta|^2)\pa_{yy}\varphi-\nabla\eta\cdot\nabla\pa_{y}\varphi
=
-\div\big(\nabla\varphi-\pa_{y}\varphi\nabla\eta\big)\,.
\]
Hence, by \eqref{accendino1}, we get
\begin{align*}
\mathcal{Z}(y)&=
\cOpbw{1+|\nabla\eta|^2}[\pa_{y}w]-\cOpbw{\nabla\eta}[\nabla w]
-\cOpbw{\nabla\varphi-\pa_{y}\varphi\nabla\eta}[\nabla\eta]
\\
&-
\cOpbw{\div\big(\nabla\varphi-\pa_{y}\varphi\nabla\eta\big)}[\eta]+
\mathcal{T}_{\geq 1}(\eta;y)[\varphi]
\end{align*}
where (recalling \eqref{resti1234},  \eqref{resti567}) we defined
\[
\mathcal{T}_{\geq 1}(\eta)[\varphi]
=
\mathcal{T}_{ 1}(\eta)[\varphi]+\mathcal{T}_{\geq 2}(\eta)[\varphi]\,, 
\qquad 
\mathcal{T}_{ 1}(\eta)[\varphi]:= Q_1\,, 
\qquad 
\mathcal{T}_{\geq 2}(\eta)[\varphi]:=\sum_{j=2}^{7}Q_j\,.
\]
First of all, since $\mathcal{T}_{\geq 1}(\eta)[\varphi]$
depends only on the derivatives of $\varphi$, we note that 
$\mathcal{T}_{\geq 1}(\eta;y)[\varphi]=\mathcal{T}_{\geq 1}(\eta;y)[\Pi \varphi]$. 
Moreover, in view of \Cref{stima cal Kn equazione di laplace} (see formula \eqref{def:varphi}), 
we define  
\begin{align*}
\mathcal{S}_{\geq 1}(\eta;y)[\psi]&:= 
\mathcal{S}_{ 1}(\eta;y)[\psi]+ \mathcal{S}_{\geq 2}(\eta;y)[\psi]\,,
\\
\mathcal{S}_{ 1}(\eta;y)[\psi]&:= \mathcal{T}_{1}(\eta)[ e^{y |D|}\psi]\,, 
\qquad 
\mathcal{S}_{\geq 2}(\eta;y)[\psi]:= 
\mathcal{T}_{1}(\eta)[ \Pi \varphi_{\geq 1}(\eta;y)[\psi]]
+\mathcal{T}_{\geq 2}(\eta)[\Pi \varphi[\psi]]\,.
\end{align*}
The discussion above, together with \eqref{accendino2} and \eqref{eq:112aTRIS}
implies formula
\eqref{eq:112aquatuor}.
The bounds \eqref{leggelegge1} and \eqref{leggelegge2} follow by using the explicit definitions of the remainders $Q_j$ in \eqref{resti1234} and \eqref{resti567}, \Cref{lem:prodotto,compoparapara} and together with the bounds \eqref{stima:omognea} and \eqref{bottle10}.
 The thesis on the operator $\mathcal{S}_1$ and the bound 
\eqref{leggelegge1} follow by recalling Remark \ref{rmk:phielliptic}.
Finally, in view of \eqref{def:VBIS}-\eqref{form-of-BBIS}
we have that 
\[
\big(\nabla\varphi-\pa_{y}\varphi\nabla\eta\big)_{|y=0}=\tV=\tV(\eta,\psi)\,.
\]
with $\tV$ in \eqref{def:V-B}, hence 
 \eqref{eq:112aquatuorBIS} follows.
\end{proof}
In view of the Proposition above we will concentrate in the study of the structure of
$w(x,y)$ introduced in \eqref{good1}.

\paragraph{Paralinearization of the interior equation}
In the following we show that 
$w$ in \eqref{good1}, with $\varphi$ solving \eqref{elliptic2}, satisfies a \emph{para-differential} equation.

\begin{proposition}\label{prop:interior}
Let $\vr\geq 0$, $s_0>3/2$. There is $\mu>0$ such that for any 
$s\geq s_0+\mu$ 
there is $\delta=\delta(s)$ such that for any
$\eta\in B_{s_0+\s}(\delta)\cap H^{s}$,  
$\psi\in H^{s}$ the following holds.
There exists a map
\[
\begin{aligned}
\breve{\mathcal{S}}_{\geq 1}: B_{s_0+\mu}(\delta)\cap H^{s}
&\longrightarrow\mathcal{L}(H^s_x, H^{1,a}_{y}H_x^{s+\vr})
\\
\qquad \eta\;\;\;&\mapsto\;\; \breve{\mathcal{S}}_{\geq 1}(\eta;y)[\cdot]
\end{aligned}
\]
such that
the good unknown $w$ in \eqref{good1} satisfies the equation
\begin{equation}\label{eq:paraw}
\pa_{y}^2w+\cOpbw{1+\beta_{1}}\Delta w+\cOpbw{\beta_{2}}\nabla\pa_{y}w
+\cOpbw{\beta_3}\pa_{y}w
=\breve{\mathcal{S}}_{\geq 1}(\eta;y)[\psi]
\end{equation}
where the functions $\beta_i$, $i=1,2,3$ are in \eqref{formaalphai}. 
Moreover, one has the expansion
\[
\breve{\mathcal{S}}_{\geq 1}(\eta;y)[\cdot]=\breve{\mathcal{S}}_{ 1}(\eta;y)[\cdot]
+\breve{\mathcal{S}}_{\geq 2}(\eta;y)[\cdot]\,,
\]
where $ \breve{\mathcal{S}}_{ 1}(\eta;y)[\cdot]$ 
is linear with respect to $ \eta$ and it
satisfies,
for $s\geq s_0+\s$,  the estimate
\begin{align}
\|\breve{\mathcal{S}}_{ 1}(\eta;y)[\psi]\|_{H^{1,a}_yH_{x}^{s+\vr}}
&\lesssim_{s}
\|\psi\|_{s}
\|\eta\|_{s_0+\mu}+
\|\eta\|_{s}\|\psi\|_{s_0+\mu}\,,
\label{legge1}
\end{align}
while $\breve{\mathcal{S}}_{\geq 2}(\eta;y)[\cdot]$ 
is non-homogeneous and satisfies 
\begin{align}
\| \breve{\mathcal{S}}_{\geq 2}(\eta;y)[\psi]\|_{H^{1,a}_y H_{x}^{s+\vr}}
\lesssim_{s}
\|\psi\|_{s}
\|\eta\|_{s_0+\mu}^2+
\|\eta\|_{s}\| \eta\|_{s_0+\mu}\|\psi\|_{s_0+\mu}\,.
    \label{legge2}
\end{align}
\end{proposition}

\begin{proof}
Recalling \eqref{op:L}-\eqref{formaalphai} we note that $\varphi$ 
solves the equivalent equation (recall \eqref{elliptic2})
\[
\widetilde{\mathcal{L}}\varphi=0\,,
\qquad y\leq0\,,
\qquad 
\widetilde{\mathcal{L}}=(1+|\nabla\eta|^2)\pa_{yy}+\Delta
-2\nabla\eta\cdot\nabla\pa_{y}-\Delta\eta\pa_{y}= \tb^{-1} \cL\,.
\]
By applying Lemma \ref{lem:prodotto} to 
$\widetilde{\mathcal{L}}\varphi$,
and using 
$|\nabla\eta|^{2}=\cOpbw{2\nabla\eta}[\nabla\eta]+R(\nabla\eta,\nabla\eta)$, 
we get
\begin{align}
\widetilde{\mathcal{L}}\varphi&=
\cOpbw{1+|\nabla\eta|^{2}}\pa_{yy}\varphi
+\Delta\varphi-2\cOpbw{\nabla\eta}\cdot\nabla\pa_{y}\varphi
-\cOpbw{\Delta\eta}\pa_{y}\varphi
\nonumber
\\&
\qquad
+\cOpbw{\pa_{y}^2\varphi}\cOpbw{2\nabla\eta}\big[\nabla\eta]
-2\cOpbw{\nabla\pa_{y}\varphi}\big[\nabla\eta\big]
-\cOpbw{\pa_{y}\varphi}\big[\Delta\eta\big]\nonumber 
\\&\qquad +Q_1+Q_2+Q_3+Q_{4}\,,
\nonumber
\\&
\stackrel{\eqref{good1}}{=}
\cOpbw{1+|\nabla\eta|^{2}}\pa_{yy}w
+\Delta w-2\cOpbw{\nabla\eta}\cdot\nabla\pa_{y}w
-\cOpbw{\Delta\eta}\pa_{y}w
\label{winston1}
\\&
\qquad
+\cOpbw{1+|\nabla\eta|^{2}}\cOpbw{\pa_{y}^3\varphi}[\eta]
+\Delta \cOpbw{\pa_{y}\varphi}[\eta]\notag \\
& \qquad -2\cOpbw{\nabla\eta}\cdot\nabla \cOpbw{\pa_{y}^2\varphi}[\eta]
-\cOpbw{\Delta\eta}\cOpbw{\pa_{y}^2\varphi}[\eta]
\label{winston2}
\\&
\qquad
+2\cOpbw{\pa_{y}^2\varphi}\circ \cOpbw{\nabla\eta}[\nabla\eta]
-2\cOpbw{\nabla\pa_{y}\varphi}\big[\nabla\eta\big]-\cOpbw{\pa_{y}\varphi}\big[\Delta\eta\big]
\label{winston3}
\\&
\qquad
+Q_1+Q_2+Q_3+Q_4\,,
\label{winston4}
\end{align}
where the remainders in \eqref{winston4} are defined as
\begin{equation}\label{primorestoQ}
\begin{aligned}
Q_1&:=\mathcal{R}(\pa_{y}\varphi,\Delta\eta)\,,
\qquad \;\;
Q_2:=\mathcal{R}(\nabla\eta,\nabla\pa_{y}\varphi)\,,
\\
Q_3&:=\mathcal{R}(\pa_{yy}\varphi,|\nabla\eta|^2)\,,
\qquad\qquad
Q_{4}:=\cOpbw{\pa_{y}^2\varphi}\big[\mathcal{R}(\nabla\eta,\nabla\eta)\big]\,.
\end{aligned}
\end{equation}
By the composition
Lemma \ref{compoparapara} and using also the 
Leibniz rule for the differential operators $ \Delta$ and $ \nabla$, 
we deduce
\[
\begin{aligned}
\cOpbw{1+|\nabla\eta|^{2}}\cOpbw{\pa_{y}^3\varphi}[\cdot]&=
\cOpbw{(1+|\nabla\eta|^{2})\pa_{y}^3\varphi}
+\mathcal{Q}(|\nabla\eta|^2,\pa_{y}^{3}\varphi)\,,
\\
\Delta \cOpbw{\pa_{y}\varphi}[\cdot]&=\cOpbw{\pa_{y}\varphi}\Delta+2\cOpbw{\nabla\pa_{y}\varphi}\cdot\nabla+\cOpbw{\Delta\pa_{y}\varphi}\,,
\\
\nabla \cOpbw{\pa_{y}^2\varphi}[\cdot]&=\cOpbw{\pa_{y}^2\varphi}\nabla
+\cOpbw{\nabla\pa_{y}^2\varphi}\,,
\\
-2\cOpbw{\nabla\eta}\cdot\nabla \cOpbw{\pa_{y}^2\varphi}[\cdot]&=
-2\cOpbw{\nabla\eta\pa_{y}^{2}\varphi}\nabla-2\cOpbw{\nabla\eta\cdot\nabla\pa_{y}^2\varphi}
\\&\quad
+\mathcal{Q}(-2\nabla\eta,\pa_{y}^2\varphi)\nabla
+\mathcal{Q}(-2\nabla\eta,\nabla\pa_{y}^2\varphi)\,,
\\
-\cOpbw{\Delta\eta}\cOpbw{\pa_{y}^2\varphi}[\cdot]&=-\cOpbw{\Delta\eta\pa_{y}^2\varphi}+\mathcal{Q}(\Delta\eta,\pa_{y}^2\varphi)\,,
\\
2\cOpbw{\pa_{y}^2\varphi}\circ \cOpbw{\nabla\eta}[\nabla\cdot]&=2\cOpbw{\nabla\eta\pa_{y}^2\varphi}\nabla
+\mathcal{Q}(2\nabla\eta,\pa_{y}^2\varphi)\nabla\,.
\end{aligned}
\]
where $\mathcal{Q}(\cdot,\cdot)$  denotes remainders 
coming from the Lemma \ref{compoparapara}
between para-differential operators.
As a consequence we get
\[
\begin{aligned}
\eqref{winston2}+\eqref{winston3}&=
\cOpbw{(1+|\nabla\eta|^{2})\pa_{y}^3\varphi}\eta
+\cOpbw{\pa_{y}\varphi}\Delta\eta+2\cOpbw{\nabla\pa_{y}\varphi}\cdot\nabla\eta
+\cOpbw{\Delta\pa_{y}\varphi}\eta
\\&
-2\cOpbw{\nabla\eta\pa_{y}^{2}\varphi}\nabla\eta
-2\cOpbw{\nabla\eta\cdot\nabla\pa_{y}^2\varphi}\eta
-\cOpbw{\Delta\eta\pa_{y}^2\varphi}\eta
+2\cOpbw{\nabla\eta\pa_{y}^2\varphi}\nabla\eta
\\&
-2\cOpbw{\nabla\pa_{y}\varphi}\big[\nabla\eta\big]-\cOpbw{\pa_{y}\varphi}\big[\Delta\eta\big]+\Gamma
\\&
=\cOpbw{f}[\eta]+\Gamma
\end{aligned}
\]
where
$f=f(y,,\eta,\psi;x,\x)$ is the symbol 
\begin{equation}\label{crucial1}
f=(1+|\nabla\eta|^{2})\pa_{y}^3\varphi+\Delta\pa_{y}\varphi
-2\nabla\eta\cdot\nabla\pa_{y}^2\varphi-\Delta\eta\pa_{y}^2\varphi\,,
\end{equation}
and the remainder $\Gamma$
is given by
\begin{align}
\Gamma&:=\sum_{i=5}^{9}Q_{i}\,,
\notag\\
Q_5&:=\mathcal{Q}(|\nabla\eta|^2,\pa_{y}^{3}\varphi)[\eta]\,,
\qquad
Q_6:=\mathcal{Q}(-2\nabla\eta,\pa_{y}^2\varphi)[\nabla\eta]\,,
\qquad 
Q_7:=\mathcal{Q}(-2\nabla\eta,\nabla\pa_{y}^2\varphi)[\eta]\,,\notag\\
Q_{8}&:=\mathcal{Q}(\Delta\eta,\pa_{y}^2\varphi)[\eta]\,,
\qquad \;\;\;\,Q_{9}:=\mathcal{Q}(2\nabla\eta,\pa_{y}^2\varphi)[\nabla\eta]\,.
\label{terzorestoQ}
\end{align}
Notice that, since $\varphi$ solves \eqref{elliptic2} (and $\eta$ is independent of $y$), 
we have the crucial cancellation
\begin{equation}
f\stackrel{\eqref{crucial1}}{=}
\pa_{y}\Big[(1+|\nabla\eta|^{2})\pa_{y}^2\varphi+\Delta\varphi
-2\nabla\eta\cdot\nabla\pa_{y}\varphi-\Delta\eta\pa_{y}\varphi\Big]
\stackrel{\eqref{elliptic2}}{\equiv}0\,.
\label{crucial2}
\end{equation}
Now recalling $\tb=(1+|\nabla\eta|^{2})^{-1}$ in \eqref{formaalphai}, formulas\, 
\eqref{winston1}-\eqref{winston4}, the cancellation \eqref{crucial2}, 
the identity $ \tb-1= -|\nabla \eta|^2 \tb$ and using 
Lemma \ref{compoparapara} we get
\begin{align}
\cOpbw{\tb}\circ\widetilde{\mathcal{L}}\varphi
&=
\pa_{yy}w
+\cOpbw{\tb}\Delta w-2\cOpbw{\tb\nabla\eta}\cdot\nabla\pa_{y}w
-\cOpbw{\tb\Delta\eta}\pa_{y}w\notag
\\
&\quad+\cOpbw{\tb}[\sum_{i=1}^{9}Q_i]+Q_{10}+Q_{11}+Q_{12}\,,
\label{ostiense3}
\end{align}
where $Q_j$, $j=1,\ldots,9$ are in \eqref{primorestoQ}-\eqref{terzorestoQ} 
and where we defined (recall again \eqref{good1})
\begin{equation}\label{quartorestoQ}
\begin{aligned}
Q_{10}&:=
\mathcal{Q}(\tb,|\nabla\eta|^2)[\pa_{yy}\varphi-\cOpbw{\pa_{y}^3\varphi}[\eta]]\,,
\\
Q_{11}&:=\mathcal{Q}(\tb|\nabla \eta|^2,2\nabla\eta)\big[\nabla\pa_{y}\varphi-\nabla\circ \cOpbw{\pa_{y}^2\varphi}[\eta]\big]\,,
\\
Q_{12}&:=\mathcal{Q}(|\nabla \eta|^2\tb,\Delta\eta)\big[\pa_{y}\varphi-\cOpbw{\pa_{y}^2\varphi}[\eta]\big]\,.
\end{aligned}
\end{equation}
Since $\mathcal{L}\varphi=0$ if and only if $\widetilde{\mathcal{L}}\varphi=0$, 
we have that \eqref{ostiense3} implies
\[
\cOpbw{\tb}\Delta w-2\cOpbw{\tb\nabla\eta}\cdot\nabla\pa_{y}w
-\cOpbw{\tb\Delta\eta}\pa_{y}w
=\mathcal{T}_{\geq1}(\eta)[\varphi]\,,
\]
where (recalling \eqref{primorestoQ},  \eqref{terzorestoQ} and \eqref{quartorestoQ}) 
we have defined
\begin{gather*}
\mathcal{T}_{\geq 1}(\eta)[\varphi]=\mathcal{T}_{ 1}(\eta)[\varphi]+\mathcal{T}_{\geq 2}(\eta)[\varphi]\,,
\\
 \mathcal{T}_{ 1}(\eta)[\varphi]:= -(Q_1+Q_2)= 
 -\mathcal{R}(\pa_{y}\varphi,\Delta\eta)-\mathcal{R}(\nabla\eta,\nabla\pa_{y}\varphi)\,,
 \\
  \mathcal{T}_{\geq 2}(\eta)[\varphi]:=- \cOpbw{|\nabla \eta|^2 \tb }[Q_1+Q_2]
  -\cOpbw{\tb}\left[\sum_{j=3}^{9}Q_j\right]-Q_{10}-Q_{11}-Q_{12}\,.
\end{gather*}
First we note that, as $\mathcal{T}_{\geq 1}(\eta)[\varphi]$ depends only on the derivatives of $\varphi$, 
we have $\mathcal{T}_{\geq 1}(\eta;y)[\varphi]=\mathcal{T}_{\geq 1}(\eta;y)[\Pi \varphi]$. 
Moreover, in view of \Cref{stima cal Kn equazione di laplace} 
(see formula \eqref{def:varphi}), we define  
\begin{align*}
\breve{\mathcal{S}}_{\geq 1}(\eta;y)[\psi]&:= 
\breve{\mathcal{S}}_{ 1}(\eta;y)[\psi]
+ \breve{\mathcal{S}}_{\geq 2}(\eta;y)[\psi]\,,
\\
\breve{\mathcal{S}}_{ 1}(\eta;y)[\psi]&:= \mathcal{T}_{1}(\eta)[ e^{y |D|}\psi]\,, 
\qquad 
\mathcal{S}_{\geq 2}(\eta;y)[\psi]:= 
\mathcal{T}_{1}(\eta)[ \Pi \varphi_{\geq 1}(\eta;y)[\psi]]
+\mathcal{T}_{\geq 2}(\eta)[\Pi \varphi[\psi]]\,.
\end{align*}
The discussion prove the formula \eqref{eq:paraw}.
The bounds \eqref{legge1} and \eqref{legge2} follow 
by using the explicit definitions of the remainders 
$Q_j$ in \cref{primorestoQ,terzorestoQ,quartorestoQ}, \Cref{lem:prodotto} 
and \Cref{compoparapara,azioneSimboo} 
and also \eqref{tameHsx} and 
together with the bounds \eqref{stima:omognea} and \eqref{bottle10}.
\end{proof}

\paragraph{Reduction to the boundary.}
Recalling \eqref{eq:paraw} we define
\begin{equation}
\begin{aligned}
\mathcal{D}&:=\pa_{y}^2+\cOpbw{1+\beta_{1}}\Delta 
+\cOpbw{\beta_{2}}\nabla\pa_{y}+\cOpbw{\beta_3}\pa_{y}
\\&
\stackrel{\eqref{formaalphai}}{=}
\pa_{y}^2+\cOpbw{\tb}\Delta -\cOpbw{2\tb\nabla\eta}\nabla\pa_{y}-\cOpbw{\tb\Delta\eta}\pa_{y}
\\&=
\pa_{yy}-\cOpbw{\tb|\x|^2+\ii \nabla \tb \cdot \xi +\tfrac{1}{4} \Delta \tb }-
\cOpbw{2\ii b\nabla\eta\cdot\x
- \nabla \tb\cdot \nabla \eta}\pa_{y}\,.
\end{aligned}
\label{splittingbello2}
\end{equation}
We prove the following decomposition.

\begin{lemma}\label{divo1}
Fix $s_0>3/2$, $\vr\geq0$. There is $\mu>0$ such that, 
for any $s\geq s_0+\mu$ there is $r=r(s)>0$ such that if 
$\eta \in B_{s_0+\mu}(r)$ then the following holds.
There exist  maps

\[
\begin{aligned}
\sym{a}{}{j},\, \sym{A}{}{j}: B_{s_0+\mu}(r)\cap H^{s}
&\longrightarrow \mathcal{N}_{s+j-2}^{j}\,,
\\
\breve{\mathcal{R}}_{\geq 1}: B_{s_0+\mu}(r)&\longrightarrow\mathcal{L}(H^{s'}, H^{1,a}_y H^{s'+\vr}_x)\,,
\quad \forall s'\in \R\, \,,
\end{aligned}
\]
for $j=1,0,-1,\ldots,-\vr-1$\,, 
such that one has (see \eqref{splittingbello2})
\begin{equation*}
(\pa_{y}+\cOpbw{a})\circ (\pa_{y}-\cOpbw{A})[w]=
\mathcal{D}[w]
+\breve{\mathcal{R}}_{\geq 1}(\eta;y)[\psi]\,,
\end{equation*}
where we set
\begin{equation}\label{laghetto3}
a:=a(\eta;x,\x):=\sum_{j=-\vr}^{1}\sym{a}{}{j}\,,\qquad
A:=A(\eta;x,\x):=\sum_{j=-\vr}^{1}\sym{A}{}{j}\,.
\end{equation}
Moreover, one has the following.

\noindent
(i) \emph{(Symbols).} One has the expansion
\begin{equation}\label{simboliordine1}
\begin{aligned}
\sym{a}{}{1}&:=  \mathtt{b} \left( - \ii\nabla\eta\cdot\x
+   \sym{\lambda}{}{1}\right)(1-\chi(\xi))\,, 
\qquad 
\sym{A}{}{1}:= \mathtt{b} \left(\ii\nabla\eta\cdot\x
+  \sym{\lambda}{}{1}\right)(1-\chi(\xi))\,,
\\
\sym{A}{}{0}&:=\left(\frac{\ii \nabla \tb \cdot \xi 
+\{   \tb \nabla \eta \cdot \xi, \tb \sym{\lambda}{}{1}\} 
-\ii \tb (\nabla \eta \cdot \xi ) \nabla \tb \cdot \nabla \eta}{2 \tb \sym{\lambda}{}{1}}
-\frac{\nabla\tb \cdot \nabla \eta}{2}\right) (1-\chi(\xi))\,,
\\
\sym{a}{}{0}&:=A^{(0)}+\nabla\tb \cdot \nabla \eta
\,,
\qquad
\sym{\lambda}{}{1}:=\sqrt{(1+|\nabla \eta|^2)|\x|^{2}-(\nabla\eta\cdot\x)^{2}}\,.
\end{aligned}
\end{equation}
Moreover $\sym{A}{}{1}-|\xi|, \, \sym{a}{}{1}-|\xi| \in \Sigma \Gamma_1^{1}[r,4]$, 
$\sym{\lambda}{}{1}- |\xi| \in \Sigma\Gamma_2^1[r,4]$, 
$ \sym{a}{}{0},\, \sym{A}{}{0} \in \Sigma \Gamma_1^0[r,4]$  
and for any $j=-\vr, \ldots, -1$ the symbols 
$A^{(j)}$ and $ a^{(j)}$ belong to $ \Sigma\Gamma_1^{j}[r,2]$ 
verifying \Cref{def:sfr}  with $\mu:=3+\vr$.

\noindent
$(ii)$ \emph{(Remainder).}
The operator $\breve{\mathcal{R}}_{\geq 1}(\eta)$, 
satisfies the expansion
\begin{align}\label{def:Rbreve}
\breve{\mathcal{R}}_{\geq 1}(\eta;y)[\cdot]
=
\breve{\mathcal{R}}_{ 1}(\eta;y)[\cdot]
+\breve{\mathcal{R}}_{\geq 2}(\eta;y)[\cdot]\,,
\end{align}
where $ \breve{\mathcal{R}}_{ 1}(\eta;y)[\cdot]$ is linear 
with respect to $ \eta$ and 
satisfies,
for  any $s'\in \R$,  the estimate
\begin{align}
\|\breve{\mathcal{R}}_{ 1}(\eta;y)[\psi]\|_{H^{1,a}_yH_{x}^{s'+\vr}}&\lesssim_{s'}
\|\eta\|_{{s_0+\s}}\|\psi\|_{{s'}}\,,
\label{pettirosso1}
\end{align}
while $\breve{\mathcal{R}}_{\geq 2}(\eta;y)[\cdot]$ is 
non-homogeneous and  satisfies 
\begin{align}
\| \breve{\mathcal{R}}_{\geq 2}(\eta;y)[\psi]\|_{H^{1,a}_y H_{x}^{s'+\vr}}
\lesssim_{s'}
\|\eta\|_{\fs_0+\s}^2
\|\psi\|_{{s'}}\,.
    \label{pettirosso2}
\end{align}

\end{lemma}

\begin{proof}
We look for symbols $a,A$ of the form \eqref{laghetto3}. 
Moreover we note that
\begin{gather}
(\pa_{y}+\cOpbw{a})\circ (\pa_{y}-\cOpbw{A})=\pa_{yy}+(\cOpbw{a}
-\cOpbw{A})\pa_{y}-\cOpbw{a}\circ \cOpbw{A}\,,\notag
\\
\cOpbw{a}\circ \cOpbw{A}=\sum_{j,k=-\vr}^{1}\cOpbw{a^{(j)}}\circ \cOpbw{A^{(k)}}\,.
\label{destinazione1}
\end{gather}
By applying  Lemma \ref{compoparapara} with 
$\vr\rightsquigarrow j+k+\vr$, 
and denoting the remainder in \eqref{calma2:striscia2}  
$\cQ^{(j+k+\vr)}(\cdot, \cdot)\equiv \cQ(\cdot , \cdot) $, 
we make the expansion
\begin{align*}
&\cOpbw{a^{(j)}}\circ \cOpbw{A^{(k)}}
=\cOpbw{\sum_{n=0}^{j+k+\vr-1} p_n(a^{(j)},A^{(k)})}+
\cQ^{(j+k+\vr)}(a^{(j)},A^{(k)})\,.
\end{align*}
Therefore, we shall write
\begin{equation}\label{laghetto1}
\begin{aligned}
\cOpbw{a}\circ \cOpbw{A}&=
\sum_{j,k=-\vr-1}^{1}\cOpbw{a^{(j)}}\circ \cOpbw{A^{(k)}}=\Opbw{ c_{\vr}}+\mathcal{Q}
\end{aligned}
\end{equation}
where
\begin{equation}\label{laghetto2}
\begin{aligned}
c_{\vr}&:=
\sum_{\substack{
-\vr\leq j,k\leq 1
\\
j+k\geq -\vr+1}
}
\sum_{n=0}^{j+k+\vr-1} p_{n}(a^{(j)},A^{(k)})\,,
\\
\mathcal{Q}&:=
\sum_{\substack{
-\vr\leq j,k\leq 1
\\
j+k\leq  -\vr}
}\cOpbw{a^{(j)}}\circ \cOpbw{A^{(k)}}
+
\sum_{\substack{
-\vr\leq j,k\leq 1
\\
j+k\geq -\vr+1}
}\cQ^{(j+k+\vr)}(a^{(j)},A^{(k)})\,.
\end{aligned}
\end{equation}
In view of \eqref{destinazione1} we look for symbols which solve the equations
\[
\begin{aligned}
c_{\vr}&=\tb|\x|^{2}+ \ii \nabla \tb \cdot \xi+ \frac14 \Delta \tb \,,
\\
a-A&=-2\ii \tb\nabla\eta\cdot\x+ \nabla \tb\cdot \nabla \eta\,.
\end{aligned}
\]
Recalling \eqref{laghetto3} the equations above
can be written as
\begin{align}
{\mathrm{ord.}\,1}\;&
\left\{
\begin{aligned}
a^{(1)} A^{(1)}&=\tb|\x|^2
\\
a^{(1)}-A^{(1)}&=-2\ii \tb\nabla\eta\cdot\x
\end{aligned}
\right.
\label{laghetto10}
\\
{\mathrm{ord.}\,0}\;&
\left\{
\begin{aligned}
a^{(0)} A^{(1)}+a^{(1)}A^{(0)}+p_1(a^{(1)},A^{(1)})&=\ii \nabla \tb \cdot \xi
\\
a^{(0)}-A^{(0)}&=\nabla\tb\cdot \nabla\eta
\end{aligned}
\right.
\label{laghetto11}
\\
{\mathrm{ord.}\,p}\;&
\left\{
\begin{aligned}
a^{(p)} A^{(1)}+a^{(1)}A^{(p)}&= \frac{\delta_{p,-1}}{4} \Delta \tb - \sum_{
\substack{j+k-n=p+1
\\ p+1\leq j,k\leq 1}
}p_n(a^{(j)},A^{(k)})
\\
a^{(p)}-A^{(p)}&=0
\end{aligned}
\right.\qquad -\vr\leq p\leq -1\,,
\label{laghetto12}
\end{align}
where the Kronecker delta $ \delta_{p,-1}=1$ if $p=-1$, and $0$ otherwise. 
We iteratively solve the equations above, 
taking into account the identification in \Cref{rem:strip_symbols}. 
\vspace{0.5em}

\noindent
{\bfseries Case $p=1$.} The solutions of 
\eqref{laghetto10} are given by the symbols $A^{(1)}$ and $a^{(1)}$ 
defined in \eqref{simboliordine1}.

\noindent
We note that $a^{(1)}$ and $A^{(1)}$ solve the equation \eqref{laghetto10} 
in the sense of \Cref{rem:strip_symbols}, 
namely both sides of the equation define the same paradifferential operator.
By using the explicit expressions of the symbols in 
\eqref{simboliordine1} and recalling  \Cref{def:sfr} 
one  directly check that 
actually $A^{(1)}-|\xi|, \, a^{(1)}-|\xi| \in \Sigma \Gamma_1^{1}[r,4]$, $\sym{\lambda}{}{1}- |\xi| \in \Sigma\Gamma_2^1[r,4]$.

\vspace{0.5em}
\noindent
{\bfseries Case $p=0$.} 
Consider \eqref{laghetto11} which reads
\[
\begin{aligned}
A^{(0)}&=a^{(0)}-\nabla\tb\cdot \nabla \eta
\\
a^{(0)} A^{(1)}+a^{(1)}A^{(0)}&=\ii \nabla \tb \cdot \xi +\{   \tb \nabla \eta \cdot \xi, \tb \sym{\lambda}{}{1}\}\,.
\end{aligned}
\]
Moreover, by \eqref{simboliordine1} and using 
the smallness condition on $\eta$, one can note that
\[
2{\Re}(a^{(1)})=a^{(1)}+A^{(1)}=2\tb \sym{\lambda}{}{1}\gtrsim |\x|\,.
\]
Therefore the solutions of 
\eqref{laghetto11} are given by the symbols $A^{(0)}$ and $a^{(0)}$ defined in \eqref{simboliordine1}.
We note that $a^{(0)}$ and $A^{(0)}$ solve the equation \eqref{laghetto11} 
in the sense of \Cref{rem:strip_symbols}, namely both sides 
of the equation define the same paradifferential operator.
 Moreover $ {A}^{(0)}\equiv 0$ when $ |\xi|\leq \frac14$. 
By using the explicit expressions of the symbols in 
\eqref{simboliordine1} and recalling  \Cref{def:sfr} 
one  directly check that 
actually $A^{(0)}, \, a^{(0)} \in \Sigma \Gamma_1^{0}[r,4]$.

\vspace{0.5em}
\noindent
{\bfseries Case $p\leq -1$.} First of all \eqref{laghetto12}
implies $a^{(p)}=A^{(p)}$. Therefore a solution is given by
\[
a^{(p)}=A^{(p)}=\left(\frac{\delta_{p,-1}}{8\sym{\lambda}{}{1}} \Delta \tb -\frac{1}{2\sym{\lambda}{}{1}}\sum_{
\substack{j+k-n=p+1
\\ p+1\leq j,k\leq 1}
}p_n(a^{(j)},A^{(k)})\right) (1-\chi(\xi))\,.
\]
We note that $a^{(p)}$ and $A^{(p)}$ solve the equation \eqref{laghetto12} 
in the sense of \Cref{rem:strip_symbols}, 
namely both sides of the equation define the same paradifferential operator, 
using also that $ \widehat{(\Delta \tb)}(0)=0$.
Since the right hand side of the equation above depends only 
on symbols $a^{(j)},A^{(j)}$ with $j\geq p+1$, 
the claim follows by induction on $j \geq p+1$, 
using \eqref{def:pk} and \Cref{rmk:inclusioneclassi}.

\noindent
By the discussion above, recalling \eqref{destinazione1}, \eqref{laghetto1},
and \eqref{splittingbello2} we have obtained
\[
(\pa_{y}+\cOpbw{a})\circ (\pa_{y}-\cOpbw{A})[w]=\mathcal{D}[w]-\mathcal{Q}[w]\,,
\]
with $\mathcal{Q}$ in  \eqref{laghetto2}. 
We now prove the Taylor expansion of $\cQ$. 
First of all we Taylor expand the symbols $a_j$ and $A_k$ as 
\begin{align}
   &\sym{a}{}{j}= \sym{a}{0}{j}+ \sym{a}{1}{j}+ \sym{a}{2}{j}\,, 
   \qquad 
   \sym{a}{0}{j}\in \wt \Gamma_{0}^j\,,\  
   \sym{a}{1}{j}\in \wt \Gamma_{1}^j\,,\ \sym{a}{2}{j}\in  \Gamma_{\geq 2}^j[r]\,;
   \\
   & \sym{A}{}{k}= \sym{A}{0}{k}+ \sym{A}{1}{k}+ \sym{A}{2}{k}\,, 
   \qquad 
   \sym{A}{0}{k}\in \wt \Gamma_{0}^k\,,\  
   \sym{A}{1}{k}\in \wt \Gamma_{1}^k\,,\ \sym{A}{2}{k}\in  \Gamma_{\geq 2}^k[r]\,.
\end{align}
Then we note that $\sym{a}{0}{j}=\sym{A}{0}{k}=0$ whenever 
$j,k \not=1$ and $ \sym{a}{0}{1}= \sym{A}{0}{1}= |\xi|$. 
Plugging this expansion in \eqref{laghetto2}, we get 
\begin{align*}
    \mathcal{Q}&:= \cQ_0+\cQ_1+ \cQ_{\geq 2}\,, 
    \\
\cQ_p&:= \sum_{p_1+p_2=p}\sum_{\substack{
-\vr\leq j,k\leq 1
\\
j+k\leq  -\vr}
}\cOpbw{\sym{a}{p_1}{j}}\circ \cOpbw{\sym{A}{p_2}{k}}
+
\sum_{p_1+p_2=p}
\sum_{\substack{
-\vr\leq j,k\leq 1
\\
j+k\geq -\vr+1}
}\cQ^{(j+k+\vr)}(\sym{a}{p_1}{j},\sym{A}{p_2}{k})\,, \quad p=0,1\,;
\\
 \cQ_{\geq 2}&:= \sum_{p_1+p_2\geq 2}\sum_{\substack{
-\vr\leq j,k\leq 1
\\
j+k\leq  -\vr}
}\cOpbw{\sym{a}{p_1}{j}}\circ \cOpbw{\sym{A}{p_2}{k}}
+
\sum_{p_1+p_2\geq 2}
\sum_{\substack{
-\vr\leq j,k\leq 1
\\
j+k\geq -\vr+1}
}\cQ^{(j+k+\vr)}(\sym{a}{p_1}{j},\sym{A}{p_2}{k})\,.
\end{align*}
By \Cref{rem:compo_Fou} one has $\cQ_0\equiv 0$. In addition we define
$\breve \cR_{\geq 1}$ as in \eqref{def:Rbreve} where, recalling \Cref{rmk:goodali}, we set
\begin{align*}
 \breve \cR_1(\eta;y)[\psi]:= \cQ_1 [e^{y|D|}\psi]\,, 
 \qquad  
 \breve \cR_{\geq 2}(\eta;y)[\psi]:= \cQ_2 [w(\eta;y)[\psi]]+ \cQ_1[ w_{\geq 1}(\eta;y)[\psi]]\,.
\end{align*}
Then estimates \eqref{pettirosso1} and \eqref{pettirosso2} 
follow combining  \eqref{calma2:striscia2} and  \eqref{action:striscia1} 
with estimates \eqref{stima:omognea} and \eqref{stima:w}.
\end{proof}

\noindent
In view of Lemma \ref{divo1} and recalling \eqref{splittingbello2}, 
 we shall rewrite problem \eqref{eq:paraw} as
\begin{equation}\label{elliptic44}
\left\{\begin{aligned}
(\pa_{y} + \cOpbw{a})\circ(\pa_{y}-\cOpbw{A})w&= \breve \cR_{\geq 1}(\eta;y)[\psi] 
+\breve{\mathcal{S}}_{\geq 1}(\eta)[\psi] \,,
\qquad x\in\T^{2}\,,\;\; y<0\,,
\\
w(x,0)&=\omega\,,
\end{aligned}\right.
\end{equation}
where $w$ is given  in \eqref{good1}.
We recall again that we need to provide a precise 
structure of the normal derivative 
$\pa_{y}w$ close to the boundary $y\sim0$. 
To do this we introduce 
the function
\begin{equation}\label{funz:ww}
\mathcal{W}:=\big(\pa_{y}-\cOpbw{A}\big)w\,,
\end{equation}
and note that
\[
(\pa_{y}w)(x,0)={(\Opbw{A} w)_{|y=0}}+\mathcal{W}(x,0)\,.
\]
Then, recalling that $w(x,0)=\omega$  (see also Remark \ref{rmk:goodali}), we have
\begin{equation}\label{pontecoperto}
(\pa_{y}w)(x,0)=\cOpbw{A}[\omega]+\mathcal{W}(x,0)\,.
\end{equation}
\noindent
A direct computation implies the following simple lemma.
\begin{lemma}\label{lem:strutturaDNconw}
Let $\vr\geq 0$, $s_0>3/2$. There is $\mu>0$ such that  for any $s\geq s_0+\mu$ 
there is $r=r(s)>0$ such that for any
$\eta\in B_{s_0+\mu}(r)\cap H^{s}$,  $\psi\in H^{s}$ the following holds.
There exists a smoothing remainder 
\begin{align}
S_{\geq 1}(\eta)=S_{1}(\eta)+S_{\geq 2}(\eta)\,, 
\qquad 
S_1(\eta)\in\wt\cR^{-\vr}_1\,, \;\; S_{\geq 2}(\eta)\in \cR^{-\vr}_{\geq 2}[r]\,,
    \label{resto:DNsemifinale}
\end{align}
such that 
the Dirichlet-Neumann operator in \eqref{eq:112aquatuor}
(see also 
\eqref{eq:112a}, \eqref{eq:112aTRIS}) has the form
\begin{equation}\label{eq:112TOTALE}
\begin{aligned}
G(\eta)[\psi]
&=\cOpbw{(1+|\nabla\eta|^2)\sha{\vr} A-\ii\nabla\eta\cdot\x+  \frac12 \Delta \eta}[\omega]+
\cOpbw{1+|\nabla\eta|^2}[\mathcal{W}]_{|y=0}
\\&-\cOpbw{\tV}[\nabla\eta]-\cOpbw{\div(\tV)}[\eta]
+S_{\geq 1}(\eta)[\psi]\,,
\end{aligned}
\end{equation}
where $A$ is the symbol given by Lemma \ref{divo1}, $\mathcal{W}$ is given in \eqref{funz:ww}.
\end{lemma}

\begin{proof}
By recalling \eqref{eq:112aquatuorBIS} in
Proposition \ref{prop:eq:112aquatuor}, we get
\begin{equation*}
\begin{aligned}
G(\eta)[\psi]
&=\cOpbw{1+|\nabla\eta|^2}[(\pa_{y}w)(x,0)]-\cOpbw{\nabla\eta}[\nabla \omega]
\\&\qquad -\cOpbw{\tV}[\nabla\eta]-\cOpbw{\div(\tV)}[\eta]+\mathcal{S}_1(\eta;0)[\psi]
\\&
\stackrel{\eqref{pontecoperto}}{=}
\cOpbw{1+|\nabla\eta|^2}\circ \cOpbw{A}[\omega]-\cOpbw{\nabla\eta}[\nabla \omega]
-\cOpbw{\tV}[\nabla\eta]-\cOpbw{\div(\tV)}[\eta]\\
&\qquad\qquad+
\cOpbw{1+|\nabla\eta|^2}[\mathcal{W}]_{y=0}
+\mathcal{S}_{\geq 1}(\eta;0)[\psi]
\\&=
\cOpbw{(1+|\nabla\eta|^2)\sha{\vr} A-\ii\nabla\eta\cdot\x+ \frac12 \Delta \eta}[\omega]-\cOpbw{\tV}[\nabla\eta]-\cOpbw{\div(\tV)}[\eta]
\\& \qquad \qquad
+\cOpbw{1+|\nabla\eta|^2}[\mathcal{W}]_{y=0}+\mathcal{Q}(|\nabla\eta|^2,A)[\omega]
+\mathcal{S}_{\geq 1}(\eta;0)[\psi]\,.
\end{aligned}
\end{equation*}
Then formula \eqref{eq:112TOTALE} follows setting (recall \eqref{esp:sgeq1})
\[
{S}_1(\eta)[\psi]:={\cS}_1(\eta;0)[\psi]\,, 
\qquad 
S_{\geq 2}(\eta)[\psi]:= \mathcal{Q}(|\nabla\eta|^2,A)[\omega]+\mathcal{S}_{\geq 2}(\eta;0)[\psi]\,.
\]
The thesis in \eqref{resto:DNsemifinale} follows by \Cref{compoparapara0}, 
the properties of $A$ given in \Cref{divo1}, 
estimates \eqref{leggelegge1} and \eqref{leggelegge2} 
for $\cS_{\geq 1}$ combined with \Cref{rmk:trace}.
\end{proof}

\noindent
Our aim is to 
prove 
a formula like
\begin{equation*}
\mathcal{W}(x,0)=\mathrm{is \; a \; smoothing \;remainder}\,.
\end{equation*}
By \eqref{elliptic44}-\eqref{funz:ww}  the function $\mathcal{W}$ 
solves the (parabolic) problem
\begin{equation}\label{Cauchy problem backward calore}
\left\{\begin{aligned}
&(\pa_{y} + \Opbw{a})\mathcal{W}= F\,,
\\ 
&\mathcal{W}(x,y)\to 0\,,\; \mathrm{as}\; y\to-\infty\,,
\end{aligned}\right.
\end{equation}
where the forcing $F(x,y)$ is given by 
\begin{equation}\label{forcingFF}
F(x, y)  =\breve{\mathcal{R}}_{\geq 1}(\eta;y)[\psi] +\breve{\mathcal{S}}_{\geq 1}(\eta;y)[\psi]\,,
\end{equation}
with $\breve{\mathcal{R}}_{\geq 1}$ given in Lemma \ref{divo1}
and $\breve{\mathcal{S}}_{\geq 1}$ in Proposition \ref{prop:interior}.
The solution of \eqref{Cauchy problem backward calore} has the form 
\begin{equation}\label{soluzione parametrica dirichlet Neumann}
\mathcal{W}(\cdot, y) = \int_{- \infty}^y{\cal U}( \eta; y-z) F(\cdot, z)\, d z\,,\qquad \forall\,y \in (-\infty, 0]\,,
\end{equation}
where the operator ${\cal U}(\eta;\bigcdot)$,
is the  flow of the (homogeneous) parabolic equation 
\begin{equation}\label{calore striscia}
\begin{cases}
\partial_r {\cal U}(\eta;\zeta) = -\cOpbw{a } {\cal U}(\eta; \zeta)\,, 
\quad \zeta\in[0,+\infty)  \\
{\cal U}(\eta;  0) = { \id} \,, 
\end{cases}
\end{equation}
where 
\begin{align}
a= |\xi|+ a_{\geq 1}(\eta;x,\xi)\,, 
\qquad 
a_{\geq 1}(\eta;x,\xi)\in \Sigma \Gamma_1^1[r,2]\,,
\label{esp:a1}
\end{align}
is the symbol given in \Cref{divo1}.
We prove the following a priori estimate for the operator $\cU$.
\begin{proposition}[A priori estimates]\label{prop:apriori} 
Let  $s\geq0$, 
$ {s}_0>1$. There is $\mu>0$ such that for any $s\geq 0$ there is 
$ r_0=r_0(s)>0$ such that 
\[
B_{s_0+\mu}(r_0) \cap H^s \ni \eta \mapsto \cU(\eta;\zeta) \in \cL\left(H^s;C^0(I;H^{s}) \right)\,, 
\qquad I:=[0,+\infty)
\]
is bounded with the following estimates: 
for any $r\in (0,r_0)$ and $\eta \in B_{s_0+\mu}(r)\cap H^s$, one has
\begin{align}
\|\mathcal{U}(\eta; \zeta) g\|_{{s}}\lesssim_{s} e^{C r^2 \zeta} \|g\|_{{s}}\,,
\qquad 
\forall g \in H^s\,. \label{guardian1}
\end{align}
Moreover one has the expansion 
\begin{align}
\mathcal{U}(\eta; \zeta)= e^{-\zeta|D|}
+ \cU_{\geq 1}(\eta;\zeta)\,, 
    \label{esp:cU}
\end{align}
where $ \cU_{\geq 1}(\eta;\zeta)$ fulfills the quadratic estimates:  
for any $r\in (0,r_0)$ and $\eta \in B_{s_0+\mu}(r)\cap H^s$,   
\begin{align}
\|\cU_{\geq 1}(\eta;\zeta)[g]\|_{{s-1}}\lesssim_{s} \zeta e^{Cr^2 \zeta}
\| \eta\|_{s_0+\mu}\|g\|_{s},\qquad \forall g \in H^s\,.
\label{guardian2}
\end{align}
\end{proposition}
We first need the following preliminary result whose proof is postponed. 
\begin{lemma}\label{lemm:supergard}
Let $\vr\geq0$, ${s}_0>1$. There is $\mu>0$ such that 
for  any $s>0$ there is $r=r(s)>0$ such that if 
$\eta \in B_{s_0+\mu}(r)$,
then one has
\begin{equation}\label{eq:supergard}
2{ \Re}\langle |D|^{2s}\cOpbw{a}\Pi_0^\perp h, \Pi_0^\perp h\rangle_{L^2_x}\geq 0
\end{equation}
for any $h\in H^{s}_x$.
\end{lemma}

\begin{proof}[{\bfseries Proof of Proposition \ref{prop:apriori}}]
We denote with $ v(\zeta):= \mathcal{U}(\eta;\zeta)g$.
In view of \Cref{rem:symbols_modozero}, equation \eqref{calore striscia} 
splits into two independent equations: one for $\Pi_0 v$ 
and the other for $\Pi_0^\bot v$.
First, thanks to \eqref{comm_copbw} and \eqref{calore striscia}, we have
\begin{align}
\pa_r \Pi_0 v(\zeta) = -\hat a (0,0) \Pi_0 v(\zeta)
= -\left(\int_{\T^2} \nabla \tb \cdot \nabla \eta\, \di x \right)\Pi_0 v(\zeta)\,.
\label{modozero}
\end{align}
Since $ \| \eta\|_{s_0+\mu} \leq r$, one has 
\begin{align*}
| \Pi_0 v(\zeta)|^2 \leq e^{Cr^2 \zeta} |\Pi_0 v(0)|^2\,.
\end{align*}
Then it remains to bound 
\begin{align*}
\| \Pi_0^\bot v(\zeta)\|_s^2= \langle |D|^{2s} v(\zeta), v(\zeta)\rangle_{L^2}\,.
\end{align*}
Using \eqref{calore striscia} we get
\begin{equation}
\begin{aligned}
\pa_{\zeta}\|\Pi_0^\bot v(\zeta)\|^2_{s}&=
2{\Re}\langle |D|^{2s}\pa_r{v}(\zeta),v(\zeta)\rangle_{L^2}
=
-2{\Re}\langle | D|^{2s}\cOpbw{a} v(\zeta), v(\zeta)\rangle_{L^2}
\stackrel{\eqref{eq:supergard}}{\leq}0 \,.
\end{aligned}
\label{gard1}
\end{equation}
Then, being $\zeta \mapsto \|\Pi_0^\bot v(\zeta)\|_{s}$ 
a decreasing map, it follows that 
\begin{align*}
\|\Pi_0^\bot v(\zeta)\|_{s}\leq \| \Pi_0^\bot v(0)\|_s =  \| \Pi_0^\bot g\|_s \leq \| g\|_s\,.
\end{align*}
Therefore the estimate \eqref{guardian1} follows 
by \eqref{modozero} and \eqref{gard1}.
We now prove \eqref{esp:cU}. We use \eqref{esp:a1} and 
the Duhamel formula to obtain the expansion 
\begin{align*}
\cU(\eta;\zeta)g= e^{-\zeta |D|} g
+ \underbrace{\int_0^\zeta e^{-(\zeta- \tau)|D|} 
\Opbw{\fun{a}{\geq 1}(\eta;x,\xi)}\cU(\eta;\tau)g\, \di \tau\,.}_{:= \, \cU_{\geq 1}(\eta;\zeta)g}
\end{align*}
As the Fourier multiplier $e^{-(\zeta- \tau)|\xi|}\leq 1$ for any $\xi \in \R$, 
we use the bound in \Cref{thm:action}, 
the bound \eqref{stima:nonhom} for $\fun{a}{\geq 1}\in \Gamma_{\geq 1}^1[r]$ 
and the bound \eqref{guardian1} for $\cU(\eta;\tau)$ 
to obtain the estimate \eqref{guardian2} for $\cU_{\geq 1}(\eta;\zeta)g$.
\end{proof}

\begin{proof}[{\bfseries Proof of Lemma \ref{lemm:supergard}}]\label{lemmagarding}
Recalling \eqref{laghetto3}-\eqref{simboliordine1} we write
\begin{equation*}
a= |\xi|+ a_{\geq 1}(\eta;x,\xi),\qquad a_{\geq 1} \in \Sigma \Gamma_1^1[r]\,.
\end{equation*}
In particular there is $\mu>0$ such that  
\begin{equation}\label{stimesimboliAAAppendix}
|a_{\geq 1}|_{1,s_0}\lesssim_{s}\|\eta\|_{s_0+\mu}\,.
\end{equation}
Secondly, we note that 
\begin{align*}
2{\Re}( |D|^{s}\Opbw{a}&\Pi_0^\perp h,|D|^{s}\Pi_0^\perp h)_{L^2_x} 
\\
&=  2\big( | D|^{s+1}  \Pi_0^\perp h , |D|^s \Pi_0^\perp h \big)_{L^2_x}  
+ 2\Re\big(| D|^s\Opbw{a_{\geq 1}}\Pi_0^\bot  h ,| D|^s\Pi_0^\bot h \big)_{L^2_x}   
\\ 
&= 2 \| \Pi_0^\bot h \|_{s+\frac12}^2
-  2\Re\big(| D|^s\Opbw{a_{\geq 1}}\Pi_0^\bot  h ,| D|^s\Pi_0^\bot h \big)_{L^2_x}\,. 
\end{align*}    
Using Cauchy-Schwartz, \Cref{thm:action} 
and \eqref{stimesimboliAAAppendix} we get 
\begin{align*}
\left|2\Re\big(| D|^s\Opbw{a_{\geq 1}}\Pi_0^\bot  h ,| D|^s\Pi_0^\bot h \big)_{L^2_x}\right|
\lesssim_s 
\| \eta\|_{s_0+\mu}\| \Pi_0^\bot h \|_{s+\frac12}^2\,,
\end{align*}
which implies 
\begin{align*}
2{\Re}( &|D|^{s}\Opbw{a}\Pi_0^\perp h,|D|^{s}\Pi_0^\perp h)_{L^2_x}
\geq 
(2- C_s r) \| \Pi_0^\bot h \|_{s+\frac12}^2\geq 0\,,
\end{align*}
 provided $C_s r\leq 2$. This proves \eqref{eq:supergard}.
 \end{proof}

\paragraph{Conclusions}

In this section we analyze the structure of the term $\mathcal{W}(x,y)$ 
in \eqref{funz:ww}, by exploiting the result on the flow 
$\mathcal{U}(x,y)$ in \Cref{prop:apriori}.
More precisely we prove the following.

\begin{proposition}[Expansion of $\mathcal{W}_{|y=0}$]\label{lem:WWsmoothing}
Fix $\vr\geq 1$, $s_0>3/2$. There is $\mu>0$ such that,  
for any $s\geq s_0+\mu$ there is $r=r(s)>0$ such that if 
$\eta\in B_{s_0+\mu}(r)$ 
 then the following holds. 
 One has that (recall \eqref{funz:ww}-\eqref{soluzione parametrica dirichlet Neumann})
 \begin{equation}\label{provaclaimsuWW}
\mathcal{W}(x,0)
=
\sym{\mathcal{R}}{}{\mathcal{W}}(\eta)[\psi]
=\sym{\mathcal{R}}{1}{\mathcal{W}}(\eta)[\psi]
+\sym{\mathcal{R}}{\geq 2}{\mathcal{W}}(\eta)[\psi] \,,
\qquad 
\sym{\mathcal{R}}{1}{\mathcal{W}}(\eta)\in \wt \cR^{-\vr}_1\,,\; 
\sym{\mathcal{R}}{\geq 2}{\mathcal{W}}(\eta)\in \cR^{-\vr}_{\geq 2}[r]\,.
\end{equation}
\end{proposition}
\begin{proof}
Recalling \eqref{forcingFF}-\eqref{soluzione parametrica dirichlet Neumann}
we write
\[
\begin{aligned}
&\mathcal{W}(x, 0)= \int_{- \infty}^{0}{\cal U}(-z) F(\cdot, z)\, d z
=\int_{- \infty}^{0}{\cal U}(-z) \big(\breve{\mathcal{R}}_{\geq 1}(\eta;z)[\psi] 
+\breve{\mathcal{S}}_{\geq 1}(\eta;z)[\psi]\big)\, d z
\\
&\stackrel{\eqref{esp:cU}}{=} \overbrace{\int_{-\infty}^0 e^{z|D|}
\left( \breve{\mathcal{R}}_{ 1}(\eta;z)[\psi] +\breve{\mathcal{S}}_{ 1}(\eta;z)[\psi]\right)\, \di z}^{:= \sym{\cR}{1}{\cW}}
\\
&\underbrace{+ \int_{-\infty}^0 e^{z|D|}\left( \breve{\mathcal{R}}_{ \geq 2}(\eta;z)[\psi] 
+\breve{\mathcal{S}}_{ \geq 2}(\eta;z)[\psi]\right)\, \di z
+ \int_{-\infty}^0 \cU_{\geq 1}(\eta;-z)\left( \breve{\mathcal{R}}_{\geq 1}(\eta;z)[\psi] 
+\breve{\mathcal{S}}_{ \geq 1}(\eta;z)[\psi]\right)\, \di z.}_{:= \sym{\cR}{\geq 2}{\cW}}
\end{aligned}
\]
Then the expansion \eqref{provaclaimsuWW} follows. 
The fact that $\sym{\mathcal{R}}{1}{\mathcal{W}}(\eta)\in \cR^{-\vr}_1$ and 
$ \sym{\mathcal{R}}{\geq 2}{\mathcal{W}}(\eta)\in \cR^{-\vr}_{\geq 2}[r]$ 
(recall \Cref{def:smoothing}) follow combining estimates \eqref{stima:omognea} for $e^{z|D|}$, 
\eqref{pettirosso1} and  \eqref{pettirosso2} for $\breve{\mathcal{R}}_{\geq 1}(\eta)$, 
\eqref{legge1} and \eqref{legge2} for $\breve{\mathcal{S}}_{\geq 1}(\eta)$ 
and estimate \eqref{esp:cU} and \eqref{guardian2} for the parabolic flow $ \cU$. 
For example we detail the estimate for the last term
\begin{align*}
\left\| \int_{-\infty}^0 \cU_{\geq 1}(\eta;-z)
\left( \breve{\mathcal{S}}_{ \geq 1}(\eta;z)[\psi]\right)\, \di z\right\|_{s+\vr}
\leq &\int_{-\infty}^0 \left\| \cU_{\geq 1}(\eta;-z)
\left( \breve{\mathcal{S}}_{ \geq 1}(\eta;z)[\psi]\right)\right\|_{s+\vr}\, \di z
\\
&\lesssim_s  
\|\eta\|_{s_0+\mu} \int_{-\infty}^0 |z| e^{C r^2 |z|}\|\breve{\mathcal{S}}_{ \geq 1}(\eta;z)[\psi]\|_{s+\vr}\, \di z
\\
&\lesssim_s  
\|\eta\|_{s_0+\mu}\| \breve{\mathcal{S}}_{ \geq 1}(\eta;z)[\psi]\|_{L^{2,a}_zH^{s+\vr}_x} 
\left(\int_{-\infty}^0 |z| e^{C r^2 |z|}e^{-a|z|}\, \di z\right)^{\frac12}
\\
&\lesssim_s 
\|\eta\|_{s_0+\mu}^2 \| \psi\|_s + \|\eta\|_{s_0+\mu}\| \eta\|_s\| \psi\|_{s_0+\mu}\,,
\end{align*}
where to prove that the integral of $|z| e^{C r^2 |z|}e^{-a|z|}$ is finite 
we used the restriction $Cr^2<a$.
 \end{proof}

We are now in position to prove our main result on the Dirichlet-Neuman operator.

\begin{proof}[{\bfseries Proof of Theorem \ref{para:DN}}]
In view of Lemma \ref{lem:strutturaDNconw} and recalling  \eqref{funz:ww} and \eqref{provaclaimsuWW},
 we have
\begin{equation}\label{pornichet20}
\begin{aligned}
G(\eta)[\psi]&=\Opbw{(1+|\nabla\eta|^2)\sha{\vr} A-\ii\nabla\eta\cdot\x+\tfrac12 \Delta \eta}[\omega]
-\cOpbw{\tV}[\nabla\eta]-\cOpbw{\div(\tV)}[\eta]
\\&
\qquad+\cOpbw{1+|\nabla\eta|^2}[\sym{\mathcal{R}}{}{\mathcal{W}}(\eta)[\psi]]+S_{\geq 1}(\eta)\psi\,.
\end{aligned}
\end{equation}
We now expand $(1+|\nabla\eta|^2)\sha{\vr} A$ recalling \eqref{espansione2} and the formulas \,\eqref{laghetto3}-\eqref{simboliordine1} for $A$.
By explicit computation, using also that  $\tb^{-1}=(1+|\nabla\eta|^2)$ , one gets
\begin{align}
(1+|\nabla\eta|^2)\sha{\vr} A-\ii\nabla\eta\cdot\x+\tfrac12 \Delta \eta=\lambda&:=\lambda(\eta;x,\x)=\sym{\lambda}{}{1}+\sym{\lambda}{}{0}+\sym{\lambda}{}{-1}\,,
\end{align}
where
\begin{align}
\sym{\lambda}{}{1}&:=\tb^{-1}A^{(1)}-\ii\nabla\eta\cdot\x= \sqrt{(1+|\nabla \eta |^2)|\xi|^2 - (\nabla \eta \cdot \xi)^2}\,,\label{def:lambda1}\\ 
\sym{\lambda}{}{0}&:=\tb^{-1} A^{(0)}+ \frac{1}{2 \ii}\big\{\tb^{-1}, A^{(1)}\big\} +\tfrac12 \Delta \eta= \frac{\big\{ \tb \nabla \eta \cdot \xi, \tb \sym{\lambda}{}{1}\big \}}{2 \tb^2 \sym{\lambda}{}{1}} + \tfrac12 \Delta \eta\,,\label{def:lambda0}\\
\sym{\lambda}{}{-1}&:= \sum_{\substack{ 0\leq n\leq \vr-1 \\
\\ 
-\vr \leq j\leq 1 \\ j- n \leq -1 }} p_n( \tb^{-1}, A^{(j)})=\underbrace{\sum_{
-\vr \leq j\leq -1  } \sym{A}{1}{j}}_{:= \sym{\lambda}{1}{-1}}+\underbrace{\sum_{
-\vr \leq j\leq -1  } \sym{A}{\geq 2}{j}+\sum_{\substack{ 0\leq n\leq \vr-1 \\
\\ 
-\vr \leq j\leq 1 \\ j- n \leq -1 }} p_n( |\nabla \eta|^2, A^{(j)})}_{:= \sym{\lambda}{\geq 2}{-1}}\,.
\label{def:lambda-1}
\end{align}
We note that the equality above are meant in the sense 
of \Cref{rem:strip_symbols}, namely both sides 
of the equations define the same paradifferential operator.
Therefore, by \eqref{pornichet20}, one gets
\[
G(\eta)[\psi]=\cOpbw{\lambda}[\omega]-\cOpbw{\tV}[\nabla\eta]
-\cOpbw{\div(\tV)}[\eta]+R_{\geq 1}(\eta)[\psi]\,,
\]
where we defined 
\[
R_{\geq 1}(\eta):=S_{\geq 1}(\eta)
+ \cOpbw{1+|\nabla\eta|^2}[\sym{\mathcal{R}}{}{\mathcal{W}}(\eta)
= R_1(\eta)[\psi]+ R_{\geq 2}(\eta),
\]
and 
\[
R_{1}(\eta)= S_1(\eta)+ \sym{\mathcal{R}}{1}{\mathcal{W}}(\eta)\,, 
\qquad 
R_{\geq 2}(\eta)= S_{\geq 2}(\eta) + \Opbw{|\nabla \eta|^2}\sym{\mathcal{R}}{}{\mathcal{W}}(\eta)
+ \sym{\mathcal{R}}{\geq 2}{\mathcal{W}}(\eta)\,.
\]
This implies the expansion
\eqref{eq_para:DN}.
To prove that $R_{1}(\eta)\in \wt\cR^{-\vr}_1$ and 
$R_{\geq 2}(\eta)\in \cR_{\geq 2}^{-\vr}[r]$ (cfr. \Cref{def:smoothing}), 
we prove their bounds using 
\Cref{lem:strutturaDNconw} for $S_{\geq 1}(\eta)$, \Cref{lem:WWsmoothing} for $\sym{R}{}{\cW}$ 
and \Cref{thm:action} for $\Opbw{|\nabla \eta|^2}$ 
and we deduce the translation invariant property \eqref{tra:smoo} 
as explained in \Cref{rem:trans}. 
Formulas \, \eqref{def:lambda1}-\eqref{def:lambda-1} 
imply the expansion \eqref{lambda:DN}, 
see the explicit expressions of $\sym{\lambda}{}{1}$ and 
$\sym{\lambda}{}{0}$ in \eqref{def:lam1}, \eqref{def:lam0}. 
In view of \Cref{divo1}, $\sym{\lambda}{\geq 2}{-1}$ 
belongs to $\Gamma_{\geq2}^{-1}[r]$ and $\sym{\lambda}{1}{-1}$ belongs to $\wt \Gamma_1^{-1}$. 
It remains to prove the reality of $\sym{\lambda}{1}{-1}$ 
which is the content of the following a posteriori lemma. 
\end{proof}
\begin{lemma}\label{lem:lambda-1}
The symbol $ \sym{\lambda}{1}{-1}\in \wt \Gamma_1^{-1}$ in \eqref{esp:lambda-1} 
is real and even in $\xi$, namely 
        \begin{align}
            \begin{array}{cc}
                  \ov{\sym{\lambda}{1}{-1}(\eta;x,\xi)}&= \sym{\lambda}{1}{-1}(\eta;x,\xi) 
                  \\
                  \sym{\lambda}{1}{-1}(\eta;x,-\xi)&= \sym{\lambda}{1}{-1}(\eta;x,\xi)
            \end{array}\,, 
            \qquad 
            \forall \eta \in H^{2}(\T^2;\R)\,, \, (x,\xi) \in \T^2\times \R^2\,.
        \end{align}
    \end{lemma}
    \begin{proof}
        It is well-known that the Dirichlet-Neumann operator 
        is a self-adjoint and real operator, 
        namely $G(\eta)^\top= G(\eta)$ and $G(\eta)= \ov{G(\eta)}$.  
        Moreover, since it is analytic, we expand it as
        \begin{align*}
        G(\eta)= |D|+ G_1(\eta)+ G_{\geq 2}(\eta)\,.    
        \end{align*}
        Expanding also the right hand side of \eqref{eq_para:DN} we get 
        \begin{align*}
            G_1(\eta)= \Opbw{\sym{\lambda}{1}{0}+ \sym{\lambda}{1}{-1}} + L_1(\eta) + R_1(\eta)\,,
        \end{align*}
        where 
        \begin{align*}
            L_1(\eta)\psi:=|D|\Opbw{|D|\psi}\eta+ \Opbw{-\ii \nabla \psi\cdot \xi- \frac12 \Delta \psi}\eta\,.
        \end{align*}
        By definition one has
\begin{align*}
L_1(\eta)\psi= \sum_{j,k\in \Z^{2}\setminus\{0\}} L_{j,k}\eta_j \psi_k e^{\ii (j+k)\cdot x}\,, 
\qquad 
L_{j,k}:= \chi\left( \frac{| k|}{\langle k+2j\rangle}\right) \left[ |k|^2 + k\cdot j -|k| |j+k| \right]\,.
\end{align*}
In the same way, we expand the transpose operator $ L_1(\eta)^\top$ as
        \begin{equation*}
        \begin{aligned}
            L_1(\eta)^\top\psi &= 
            \sum_{j,k\in \Z^{2}\setminus\{0\}} 
            (L^\top)_{j,k}\eta_j \psi_k e^{\ii (j+k)\cdot x}\,, 
            \\
            (L^\top)_{j,k}&:= 
            L_{j,-(j+k)}= \chi
            \left( \frac{| j+k|}{\langle j-k\rangle}\right) 
            \left[ |k|^2 + k\cdot j -|k| |j+k| \right]\,.
        \end{aligned}
        \end{equation*}
Thanks to the cut-off $\chi$ one has the restrictions 
\begin{align*}
     L_{j,k}\not=0 \implies  |k| \ll |j|\sim |j+k|\,, 
     \qquad  
     (L^\top_1)_{j,k}\not= 0 \implies |j+k|\ll | j| \sim |k|\,.
\end{align*}
As a consequence one has the bounds 
\begin{equation}\label{bound:L}
\begin{aligned}
|L_{j,k}| &\lesssim \min\{ |j|, |j+k|\}^{2+\vr}\max\{ |j|, |j+k|\}^{-\vr}\,, 
\\
|(L^\top_1)_{j,k}| &\lesssim \min\{ |j|, |k|\}^{2+\vr}\max\{ |j|, |k|\}^{-\vr}\,.
\end{aligned}
\end{equation}
Then we Fourier expand the smoothing remainder and its transpose, 
\begin{align}
R_1(\eta)\psi= \sum_{j,k\in \Z^{2}\setminus\{0\}} R_{j,k}\eta_j \psi_k e^{\ii (j+k)\cdot x}\,, 
\qquad   
R_1(\eta)^\top \psi =\sum_{j,k\in \Z^{2}\setminus\{0\}} (R^\top)_{j,k}\eta_j \psi_k e^{\ii (j+k)\cdot x} \,,
\label{bound:R}
\end{align}
where $(R^\top)_{j,k}:= R_{j,-(j+k)}$. 
Since $R_1(\eta) \in \wt \cR^{-\vr}_1$ 
one has the bounds
\begin{align*}
|(R^\top)_{j,k}| \lesssim \min\{ |j|, |j+k|\}^{2+\vr}\max\{ |j|, |j+k|\}^{-\vr}\,, 
\qquad 
|R_{j,k}| \lesssim \min\{ |j|, |k|\}^{2+\vr}\max\{ |j|, |k|\}^{-\vr}\,.
 \end{align*}
We note that, as $G_1(\eta)$ is real, $L_1(\eta)$ 
is real  as well as $ \Opbw{\sym{\lambda}{1}{0}}$, 
by difference, the operator 
\[
\Opbw{\sym{\lambda}{1}{-1}}+ R_1(\eta)
\]
is real. Then, replacing if necessary 
$\sym{\lambda}{1}{-1}\leadsto\frac{\sym{\lambda}{1}{-1}+\ov{\sym{\lambda}{1}{-1}}^\vee}{2}$ 
and $R(\eta)\leadsto \frac{R_1(\eta)+ \ov{R_1(\eta)}}{2}$, 
we can always assume that 
\begin{align*}
    \sym{\lambda}{1}{-1}= \ov{\sym{\lambda}{1}{-1}}^\vee.
\end{align*}
Moreover, since $ G(\eta)^\top= G(\eta)$ and 
$ \Opbw{\sym{\lambda}{1}{0}}^\top=\Opbw{\sym{\lambda}{1}{0}}$, 
we get 
\begin{align}
\Opbw{\sym{\lambda}{1}{-1}- [\sym{\lambda}{1}{-1}]^\vee}
= L_1(\eta)^\top- R_1(\eta)+R_1(\eta)^\top-L_1(\eta)\,.
    \label{eq:GGtop}
\end{align}
We denote 
$\sym{\lambda}{1}{-1}- [\sym{\lambda}{1}{-1}]^\vee=: \sym{a}{1}{-1}(\eta;x,\xi)
= \sum_{j\in \Z^{2}\setminus\{0\}} a_j(\xi) \eta_j e^{\ii j\cdot x}$. 
Then 
\begin{align}
 \Opbw{\sym{\lambda}{1}{-1}- [\sym{\lambda}{1}{-1}]^\vee}\psi= \sum_{j,k\in \Z^{2}\setminus\{0\}} \Theta_{j,k}\eta_j \psi_k e^{\ii (j+k)\cdot x}, \qquad \Theta_{j,k}:= \chi\left( \frac{|j|}{\langle j+2k\rangle}\right)a_j\left(k+\frac{j}{2}\right).
 \label{def:lambdone}
\end{align}
On the one hand, from the identity \eqref{eq:GGtop}, one has 
\begin{align}
    \Theta_{j,k}= (L^\top)_{j,k}- R_{j,k}+(R^\top)_{j,k}-L_{j,k}\,.
    \label{lam:LR}
\end{align}
On the other hand, by \eqref{def:lambdone}, we have the restriction 
\begin{align}
    \Theta_{j,k}\not=0 \;\;\; \implies\;\;\; |j|\ll |k| \sim |j+k|\,. 
    \label{lam:restr}
\end{align}
We then estimate the coefficients $\Theta_{j,k}$ for all indices 
$j,k \in \Z^{2}\setminus\{0\}$ which satisfy \eqref{lam:restr} using \eqref{lam:LR}. 
By \eqref{bound:L} and \eqref{bound:R} and 
the restriction \eqref{lam:restr}, we have 
\begin{align*}
     |\Theta_{j,k}| &\lesssim  \min\{ |j|, |j+k|\}^{2+\vr}\max\{ |j|, |j+k|\}^{-\vr}
     +\min\{ |j|, |k|\}^{2+\vr}\max\{ |j|, |k|\}^{-\vr}
     \\
     &\lesssim  \min\{ |j|, |k|\}^{2+\vr}\max\{ |j|, |k|\}^{-\vr}\,,,
\end{align*}
proving that $\Opbw{\sym{\lambda}{1}{-1}- [\sym{\lambda}{1}{-1}]^\vee} \in \wt\cR^{-\vr}_1$. 
Then, without loss of generality, one can replace if necessary
\begin{align*}
\Opbw{\sym{\lambda}{1}{-1}}\leadsto 
\Opbw{\frac{\sym{\lambda}{1}{-1}
+(\sym{\lambda}{1}{-1})^\vee }{2}}
\end{align*}
by including $\Opbw{\frac{\sym{\lambda}{1}{-1}-(\sym{\lambda}{1}{-1})^\vee }{2}} $ in the smoothing remainder $R_1(\eta)$.
    \end{proof}

\section{ Proof of the Paralinearization  Proposition \ref{prop:paraWW}}\label{app:paraWW}
In this section, we apply paradifferential calculus and the paralinearization formula \eqref{eq_para:DN} to obtain a paralinearized form of the full water waves system \eqref{eq:1.2}.

\begin{proof}[Proof of \Cref{prop:paraWW}]
The equation \eqref{eq:eta} for $\pa_t\eta$ follows directly from the 
paralinearization of the Dirichlet-Neumann operator in \eqref{eq_para:DN}. 
Then it only remains to study the equation \eqref{eq:omega} for 
\begin{align}
     \pa_t \omega= \pa_t\psi- \Opbw{\pa_t \tB}\eta-\Opbw{\tB}\pa_t \eta
     =\pa_t\psi- \Opbw{\pa_t \tB}\eta-\Opbw{\tB}G(\eta)\psi\,.
     \label{pat:omega}
\end{align}
In view of \eqref{pat:omega}, we start to paralinearize 
the second equation in \eqref{eq:1.2} for $\pa_t\psi$.

\noindent {\bf Paralinearization of $\tb$:}
One has 
\begin{gather}
    \tb= 1- \Opbw{\tb^2}|\nabla\eta|^2+ \sym{r}{\geq 2}{\tb}(\eta;x)\label{para:tb}
    \\
    \sym{r}{\geq 2}{\tb}(\eta;x):=-\frac{1}{(2\pi)^2}\int_{\T^2} \frac{|\nabla \eta|^2}{1+|\nabla \eta|^2}\, \di x
    + \Opbw{\tb}\cR(\tb,|\nabla\eta|^2)- \tQ(\tb,\tb)|\nabla \eta|^2- \cQ(\tb,|\nabla \eta|^2)\tb\,. \notag
\end{gather}
By \Cref{thm:action}, \Cref{compoparapara0} and \Cref{lem:prodotto} for any $\vr>0$, 
there is $s_0>\vr $ such that one has the estimates 
\begin{align}
    \| \sym{r}{\geq 2}{\tb}\|_{s+\vr}
    \lesssim_s \| \eta\|_{s_0} \| \eta \|_s\,, 
    \qquad \forall s\geq s_0\,.
    \label{est:rb}
\end{align}
\noindent{\bf Paralinearization of the convective term:}
One has\footnote{Recall that, by \Cref{lem:prodotto}, one has 
$\cR(1,a)=-\frac{1}{(2\pi)^2}\int_{\T^2} a\, \di x$ and by \Cref{compoparapara0} one has $\cQ(1,a)\equiv0$.}
\begin{equation}\label{mattone1}
\begin{aligned}
    -\tfrac12|\nabla \psi|^2 &+  \tfrac12 \frac{(G(\eta)\psi
    + \nabla \eta \cdot\nabla \psi)^2}{1+ | \nabla \eta|^2} = -\tfrac12|\nabla \psi|^2 + \tfrac12 \tb^{-1} \tB^2
    \\
    &= \Opbw{\tb^{-1}\tB}\tB + \Opbw{\tB^2\nabla \eta }\cdot \nabla \eta
    - \Opbw{\nabla \psi}\cdot \nabla \psi + \sym{R}{\geq 2}{E}(\eta,\psi)\psi
    \\
    \sym{R}{\geq 1}{E}(\eta,\psi)\psi&:= -\tfrac{1}{2}\cR(\nabla \psi, \nabla \psi) 
    + \tfrac{1}{2}\Opbw{\tb^{-1}}\cR(\tB,\tB)+\tfrac{1}{2}\cR(\tB^2,|\nabla \eta|^2)
    \\&
    + \tfrac{1}{2}\cQ(|\nabla \eta|^2,\tB)\tB+ \cQ(\tB^2,\nabla \eta)\cdot \nabla \eta
    + \tfrac{1}{2}\Opbw{\tB^2 }\cR(\nabla \eta, \nabla \eta)\,.
\end{aligned}
\end{equation}
Combining \Cref{thm:action}, \Cref{compoparapara0} and \Cref{lem:prodotto}, for any $\vr>0$, there is $s_0>\vr $ such that one has the estimates we get 
\begin{align*}
    \| \sym{R}{\geq 1}{E}(\eta,\psi)\psi\|_{s+\vr}
    \lesssim_s \|\psi\|_{s_0}\| \psi\|_s+ \| \eta\|_s\| \psi\|_{s_0}^2\,, 
    \qquad 
    \forall s\geq s_0\,.
\end{align*}
 We now paralinearize $\tB$ using that by \eqref{def:V-B}, one has 
 $ \tB= \tb (G(\eta)\psi+ \nabla \eta \cdot \nabla \psi)$. By \Cref{lem:prodotto} we get 
 \begin{equation}\label{mattone2}
 \begin{aligned}
     \tB
     &= \Opbw{\tb}G(\eta)\psi+ \Opbw{\tb}\nabla \eta\cdot \nabla \psi
     - \Opbw{\tb \tB}|\nabla \eta|^2  + \sym{\cR}{\geq 1}{\tB}(\eta)\psi\,,
     \\
     \sym{\cR}{\geq 1}{\tB}(\eta)\psi&:= \cR(\tb-1,G(\eta)\psi)
     + \cR(\tb,\nabla \eta\cdot \nabla \psi)+\Opbw{G(\eta)\psi+ \nabla \eta\cdot \nabla \psi}\sym{r}{\geq 2}{\tb} 
     \\
    &
    + \frac{1}{(2\pi)^2}\int_{\T^2} \nabla \eta\cdot \nabla \psi \, \di x
    + \cQ(\tb^{-1}\tB, \tb^2) |\nabla \eta |^2\,,
 \end{aligned}
 \end{equation}
 where we used that $\int_{\T^2} G(\eta)\psi\, \di x=0$.
Thanks to \Cref{lem:prodotto} and \Cref{lem:an_DN} we have the estimates
\begin{align*}
\|\sym{\cR}{\geq 1}{\tB}(\eta)\psi\|_{s+\vr}
\lesssim_s 
\| \eta \|_{s_0} \|\psi\|_s + \| \eta \|_s \| \psi\|_{s_0}\,, 
\qquad \forall s\geq s_0 \,.
\end{align*}
By collecting \eqref{mattone1}-\eqref{mattone2} 
and using \Cref{compoparapara0,lem:prodotto} 
we get
\begin{equation}\label{mattone3}
\begin{aligned}
 -\tfrac12|\nabla \psi|^2 + \tfrac12 \tb^{-1} \tB^2 
 &=
 \Opbw{\tB}G(\eta)\psi
 -\Opbw{\underbrace{\nabla\psi-\tB\nabla\eta}_{\equiv\tV}}\nabla\psi
 \\&
 +\Opbw{\underbrace{\tB\nabla\psi-\tB^2\nabla\eta}_{\equiv\tV \tB}}\nabla\eta
 +\tR_{\geq1}(\eta,\psi)\psi\,,
\end{aligned}
\end{equation}
where
\[
\begin{aligned}
\tR_{\geq1}(\eta,\psi)\psi&:=
\tQ(\tb^{-1}\tB,\tb)G(\eta)\psi+
\tQ(\tb^{-1}\tB,\tb)\nabla\eta\cdot\nabla\psi+
\tQ(\tb^{-1}\tB,\tB\tb)|\nabla\eta|^2+\sym{\cR}{\geq1}{\tE}(\eta,\psi)\psi
\\&
+\Opbw{\tb^{-1}\tB}\sym{\cR}{\geq1}{\tB}(\eta)\psi+\Opbw{\tB}\cR(\nabla\eta,\nabla\psi)
+\Opbw{\tB^2}\cR(\nabla\eta,\nabla\eta)
\\&+\cQ(\tB,\nabla\eta)\nabla\psi+\cQ(\tB,\nabla\psi)\nabla\eta+2\cQ(\tB^2,\nabla\eta)\nabla\eta\,.
\end{aligned}
\]
Recalling \eqref{def:good} (and also Remark \ref{rem:compo_Fou}) we get
\[
\begin{aligned}
\Opbw{\tV}\nabla\psi&=\Opbw{\tV}\nabla\omega+\Opbw{\tV}\nabla\Opbw{\tB}\eta
\\&
=\Opbw{\tV}\nabla\omega+\Opbw{\tV}\Opbw{\tB}\nabla\eta
+\Opbw{\tV}\Opbw{\nabla\tB}\eta
\\
&=
\Opbw{\tV}\nabla\omega+\Opbw{\tV\tB}\nabla\eta+
\Opbw{V\nabla\tB}\eta
+\tQ(\tV,\tB)\nabla\eta+\tQ(\tV,\nabla\tB)\eta
\\&
=\Opbw{\ii \tV\cdot\x-\tfrac{1}{2}\div(\tV)}\omega +
\Opbw{\tV\tB}\nabla\eta+
\Opbw{\tV\nabla\tB}\eta
+\tQ(\tV,\tB)\nabla\eta+\tQ(\tV,\nabla\tB)\eta\,,
\end{aligned}
\]
which, together with \eqref{mattone3}, implies
\begin{equation}\label{mattone4}
\begin{aligned}
 -\tfrac12|\nabla \psi|^2 + \tfrac12 \tb^{-1} \tB^2 &=
  \Opbw{\tB}G(\eta)\psi
  \\&-\Opbw{\ii \tV\cdot\x-\tfrac{1}{2}\div(\tV)}\omega-\Opbw{\tV\cdot \nabla\tB}\eta
  +\sym{\tR}{\geq 1}{\mathtt{good}}(\eta,\psi)\psi\,,
\end{aligned}
\end{equation}
with
\[
\sym{\tR}{\geq 1}{\mathtt{good}}(\eta,\psi)\psi=\tR_{\geq1}(\eta,\psi)\psi
+\tQ(\tV,\tB)\nabla\eta+\tQ(\tV,\nabla\tB)\eta\,.
\]
\noindent{\bfseries Paralinearization of the capillary term:} recalling \eqref{eq:1.2}, 
we now paralinearize the capillary term $ \div( \tb^{\sla{1}{2}} \nabla \eta)$.  
First of all using  \Cref{lem:prodotto,compoparapara0} and \eqref{para:tb} 
we have 
\begin{align*}
    \tb^{\sla{1}{2}}-1=& -\Opbw{\tb^{\sla{3}{2}}\nabla \eta}\nabla \eta+ \sym{\breve{r}}{\geq 2}{\tb}\,,
    \\
    \sym{\breve{r}}{\geq 2}{\tb}:=& -\tfrac{1}{2}\Opbw{\tb^{-\frac{1}{2}}}\cR(\tb^{\sla{1}{2}}-1,\tb^{\sla{1}{2}}-1) - \tfrac12\Opbw{\tb^{-\sla{1}{2}}}\sym{r}{\geq 2}{\tb} 
    \notag
    \\
    &- \tfrac{1}{2}\cQ(\tb^{-\sla{1}{2}}, \tb^{\sla{1}{2}}) [\tb^{\sla{1}{2}}-1]
    - \tfrac12\cQ(\tb^{-\sla{1}{2}},\tb^2)|\nabla \eta|^2
    \notag
    \\&
    - \tfrac12 \Opbw{\tb^{\sla{3}{2}}}\cR(\nabla \eta,\nabla \eta)- \cQ(\tb^{\sla{3}{2}}, \nabla \eta) \nabla \eta\,.
    \notag
\end{align*}
Moreover, using again \Cref{thm:action}, \Cref{compoparapara0} 
and \Cref{lem:prodotto} we have that $\sym{\breve{r}}{\geq 2}{\tb}$ 
satisfies the same estimate of $\sym{r}{\geq 2}{\tb}$ in \eqref{est:rb}. 
In view of the above formula we have
\begin{equation*}
\begin{aligned}
\div\left( \tb^{\sla{1}{2}}\nabla \eta\right)
&=\div\left(\Opbw{\tb^{\sla{1}{2}}}\nabla \eta 
+ \Opbw{\nabla \eta}[\tb^{\sla{1}{2}}-1]\right)+ \div\Big( \cR( \tb^{\sla{1}{2}}-1, \nabla \eta)\Big)
\\
&= \div\left(\Opbw{\tb^{\sla{1}{2}}\ii\xi- \tfrac12\nabla(\tb^{\sla{1}{2}})
- \tb^{\sla{3}{2}}\nabla \eta (\nabla \eta\cdot \ii \xi) 
+\tfrac12 \nabla \eta \div(\tb^{\sla{3}{2}}\nabla \eta)
+\tfrac14 \tb^{\sla{3}{2}}\nabla (|\nabla \eta|^2)}\eta \right)
\\
&+ \div\left(\Opbw{\nabla \eta}\sym{\breve{r}}{\geq 2}{\tb} + \cR(\tb^{\sla{1}{2}}-1, \nabla \eta) 
- \cQ\Big(\nabla \eta, \tb^{\sla{3}{2}} (\nabla \eta\cdot \ii \xi)
- \tfrac12\div(\tb^{\sla{3}{2}}\nabla \eta)\Big) \right)\,. 
    \end{aligned}
        \end{equation*}
    Therefore applying also 
    $\div\left(\Opbw{\vec a (x,\xi)}\eta \right)= \Opbw{\vec a\cdot \ii \xi+ \tfrac12 \div(\vec a)}\eta$ 
    to the formula above,
    we obtain that 
\begin{gather}
\div\left( \tb^{\sla{1}{2}}\nabla \eta\right)=    
\Opbw{- \sym{\th}{}{2}+ \th^{(0)}_{\geq 2}(\eta;x) }\eta
+ \sym{\cR}{\geq 2}{\mathtt{div}}(\eta)\eta\,,
\label{mattone5}
\\
\sym{\th}{}{2}:=
\left(\tb^{\sla{1}{2}} |\xi|^2-\tb^{\sla{3}{2}}(\nabla \eta\cdot\xi)^2\right)
= \tb^{\sla{3}{2}}[\sym{\lambda}{}{1}]^2\,,
 \label{def:h2}
\\
\sym{\th}{\geq 2}{0}(\eta;x):=
- \tfrac14 \Delta (\tb^{\sla{1}{2}})
+ \tfrac14 \div\left(  \nabla \eta \div( \tb^{\sla{3}{2}}\nabla \eta)\right) 
+\tfrac18\div\left(   \tb^{\sla{3}{2}}\nabla (|\nabla \eta|^2) \right)\,, 
\label{def:h0}
\\
\sym{\cR}{\geq 2}{\mathtt{div}}(\eta)\eta
:= \div\left(\Opbw{\nabla \eta}\sym{\breve{r}}{\geq 2}{\tb} 
+ \cR(\tb^{\sla{1}{2}}-1, \nabla \eta) 
- \cQ\Big(\nabla \eta, \tb^{\sla{3}{2}} (\nabla \eta\cdot \ii \xi)- \tfrac12\div(\tb^{\sla{3}{2}}\nabla \eta)\Big) \right)\,.
\notag
\end{gather}
Note that, for any $\vr>1$, there is $s_0>0$ such that, for $\s,s\geq s_0$, one has 
\begin{align*}
\| \th^{(0)}_{\geq 2}\|_\s 
\lesssim_\s \| \eta \|_{\s+3}^2\,, 
\\
\| \sym{\cR}{\geq 2}{\mathtt{div}}(\eta)\eta\|_{s+\vr}
\lesssim_s \| \eta\|_{s_0}^2\| \eta\|_s\,.
\end{align*}
Moreover, note that \eqref{def:h2}, together with \eqref{funz:bpiccolo} 
and \eqref{def:lam1}, implies formula \eqref{def:h}.
Then, gathering \eqref{mattone4}, \eqref{mattone5} 
and recalling the definition of $\ta$ in \eqref{def:a-Tay}, 
we eventually get  
    \begin{align*}
        \pa_t \omega &=  \pa_t \psi- \Opbw{\pa_t \tB } \eta - \Opbw{\tB}\pa_t \eta 
        \\
        &= - \Opbw{\kap \th^{(2)}+ 1+ \ta +  \th^{(0)}_{\geq 2} }\eta
        +\Opbw{-\ii \tV\cdot \xi + \tfrac12 \div(\tV)} \omega
        +\sym{\tR}{\geq 1}{\mathtt{good}}(\eta,\psi)\psi
        + \sym{\cR}{\geq 2}{\mathtt{div}}(\eta)\eta\,.
    \end{align*}
Recalling the properties of   $ \sym{\tR}{\geq 1}{\mathtt{good}}$ and  $\sym{\cR}{\geq 2}{\mathtt{div}}$ and inserting formula \eqref{def:good}, we obtain \eqref{eq:omega}.
\end{proof}

\begin{footnotesize}
\addcontentsline{toc}{section}{References}

  \end{footnotesize}

\end{document}